\documentclass[10pt,a4paper,oneside,reqno]{amsart}

\input{packages_and_notation.tex}
\usepackage{pgfplots}    
\pgfplotsset{compat=newest} 
\pgfplotsset{plot coordinates/math parser=false}

\title{Cusp Universality for Random Matrices II: The Real Symmetric Case} 
\author{Giorgio Cipolloni$^{\dagger\ddagger}$}
\author{L\'aszl\'o Erd\H{o}s$^{\dagger}$}
\address[G.~Cipolloni, L.~Erd\H{o}s and D.~Schr\"oder]{IST Austria, Am Campus 1, A-3400 Klosterneuburg, Austria}
\email{dschroed@ist.ac.at}
\email{giorgio.cipolloni@ist.ac.at}
\email{lerdos@ist.ac.at}
\thanks{$^\dagger$Partially supported by ERC Advanced Grant No.~338804}\thanks{$^\ddagger$This project has received funding from the European Union's Horizon 2020 research and innovation programme under the Marie Sk\l odowska-Curie Grant Agreement No. 665385.} 
\author{Torben Kr\"uger$^{\ast}$}
\address[T. Kr\"uger]{University of Bonn, Endenicher Allee 60, 53115 Bonn, Germany}
\email{torben-krueger@uni-bonn.de}
\thanks{$^\ast$Partially supported by the Hausdorff Center for Mathematics} 
\author{Dominik Schr\"oder$^{\dagger}$}
\date{\today}

\newcommand{\om}{\omega}

\newcommand{\ov}{\overline}

\begin{document}

\begin{abstract}
We prove that the local eigenvalue statistics  of real symmetric Wigner-type matrices near the cusp points of the eigenvalue density
are universal.
Together with the companion paper~\cite{1809.03971}, which proves the same result for the complex Hermitian symmetry class,
this completes the last remaining case of the Wigner-Dyson-Mehta universality conjecture  
after bulk and edge universalities have been established in the last years.
We extend the recent Dyson Brownian motion analysis at the edge~\cite{1712.03881}
to the cusp regime using the optimal local law from~\cite{1809.03971} 
and the  accurate local shape analysis of the density from~\cite{1506.05095, 1804.07752}. 
We also present a novel PDE-based  method to improve the estimate on eigenvalue
rigidity via  the maximum principle of the heat flow related to the Dyson Brownian motion. 
\end{abstract} 
\maketitle

\tableofcontents

\section{Introduction}

We consider \emph{Wigner-type} matrices, i.e.~$N\times N$ Hermitian random matrices $H$ with  independent, not necessarily identically
distributed entries above the diagonal; a natural generalization of the standard Wigner ensembles that have i.i.d.~entries. 
The Wigner-Dyson-Mehta (WDM) conjecture asserts that the local eigenvalue statistics are universal, i.e.~they are independent of the details
of the ensemble and depend only on the \emph{symmetry type}, i.e.~on whether $H$ is real symmetric or complex Hermitian.
Moreover, different statistics emerge in the bulk of the spectrum and at the spectral edges with a square root vanishing behavior of the eigenvalue density.
 The WDM conjecture for both symmetry classes has
 been proven for Wigner matrices, see~\cite{MR3699468} for  complete historical references. Recently it has been
 extended to more general ensembles including Wigner-type matrices in the bulk and edge regimes; we refer to the companion paper~\cite{1809.03971} for up to date
 references.

The key tool for the recent proofs of the WDM conjecture is the Dyson Brownian motion (DBM), a system of coupled stochastic differential equations.
The DBM method has evolved during the last years. The original version, presented in the monograph~\cite{MR3699468},  was in the spirit of 
a high dimensional analysis of a strongly correlated Gibbs measure and its dynamics. Starting in~\cite{MR3372074}  with  the analysis of the underlying 
parabolic equation and its short range approximation,
 the  PDE component of the theory became  prominent. With the coupling idea, introduced in~\cite{MR3541852, MR3606475}, the essential 
part of the proofs became fully deterministic, greatly simplifying the technical aspects.  In the current paper we extend  this trend and use
PDE methods even for the proof of the rigidity bound, a key technical input,  that earlier was obtained with  direct random matrix methods.

The historical focus on the  bulk and edge universalities has been  motivated by the Wigner ensemble since, apart from the natural bulk regime,
 its semicircle density vanishes as  a square root  near the edges, giving rise to the Tracy-Widom statistics.
  Beyond the Wigner ensemble,  however, the density profile shows a  much richer structure. Already Wigner matrices with nonzero expectation
  on the diagonal, also called \emph{deformed Wigner ensemble},  may have a density supported on several intervals and
  a cubic root cusp singularity in the density arises whenever two such intervals touch each other as some deformation parameter varies.  
  Since local spectral universality is ultimately determined by the local behavior of the density near its vanishing points,
  the appearance of the cusp gives rise to a new type of universality.
  This was first observed in~\cite{MR1618958} and the local eigenvalue statistics at the cusp
   can be explicitly described by the Pearcey process in the complex Hermitian case~\cite{MR2207649}. The corresponding explicit formulas
  for the real symmetric case have not  yet been established. 
 
 The key classification theorem~\cite{MR3684307} for the density of Wigner-type matrices showed that the density may vanish only 
as a square root (at regular edges) or as a  cubic root (at cusps); no other singularity may occur.  This result has recently been extended to a  large
class of matrices with correlated entries~\cite{1804.07752}. In other words, the cusp universality is the third and last 
universal spectral statistics for random matrix ensembles arising from natural generalizations of the Wigner matrices.
We note that invariant $\beta$-ensembles may exhibit further universality classes, see~\cite{MR3833603}.
  
In the companion paper~\cite{1809.03971} we established  cusp universality for Wigner-type matrices in the complex Hermitian symmetry class. In
the present work we extend this result to the real symmetric class and even to certain space-time correlation functions. In fact, we 
show the appearance of a natural  one-parameter family of universal statistics associated to a family of singularities of the eigenvalue density that we call \emph{{physical} cusps}.
  In both works we follow the \emph{three step strategy}, 
a general method developed for proving local spectral universality for random matrices, see~\cite{MR3699468} for a pedagogical introduction.
The first step is the  \emph{local law} or \emph{rigidity},  establishing the location of the eigenvalues with a precision slightly above
the typical local eigenvalue spacing.
The second step is to establish universality for ensembles with a tiny Gaussian component. The third step is a perturbative argument
to remove this tiny Gaussian component relying on the optimal local law.  The first and third steps are insensitive to the symmetry type,
in fact the optimal local law  in the cusp regime has been established for both symmetry classes in~\cite{1809.03971} and it completes also the third step
in both cases.  

There are two different strategies for  the second step. In the complex Hermitian symmetry class, the  Br\'ezin-Hikami formula~\cite{MR1662382}
turns the problem into a saddle point analysis for a contour integral.  This direct path was followed in~\cite{1809.03971}
relying on the optimal local law. 
In the real symmetric case, lacking the Br\'ezin-Hikami formula, only the second strategy via
the analysis of  Dyson Brownian motion (DBM) is feasible.   This approach exploits the very fast decay to local equilibrium of  DBM. It
 is the most robust and powerful method  up to now
to establish local spectral universality. In this paper we present a version of this method adjusted to the cusp situation. 
We will work in the real symmetric case for definiteness. The proof can easily be modified 
for the complex Hermitian case as well. The DBM method does not explicitly yield the local correlation kernel.
 Instead it establishes that the local statistics are universal and therefore can be identified from a reference ensemble
  that we will choose as the simplest Gaussian ensemble exhibiting a cusp singularity.

 In this paper we partly  follow the recent DBM analysis at the 
regular edges~\cite{1712.03881} and we extend it to the cusp regime, using the optimal local law from the companion paper~\cite{1809.03971} and
the precise control of the density near the cusps~\cite{1506.05095, 1804.07752}. 
 The main conceptual difference between~\cite{1712.03881} and the current work is
that we obtain  the necessary local law  along the time evolution of DBM  via novel DBM methods in Section~\ref{sec:rigid}. 
Some other steps, such as the Sobolev inequality,  heat kernel estimates
from~\cite{MR3253704} and the finite speed of propagation~\cite{MR3372074, MR3606475, 1712.03881},  require only moderate adjustments for
 the cusp regime, but for completeness we include them in the
Appendix.  The comparison of the short range approximation of the DBM with the full evolution, Lemma~\ref{se}  and  Lemma~\ref{seb111},
will be presented in detail in Section~\ref{DBMS} and in Appendix~\ref{SLL}  since it is more involved in the cusp setup, 
 after the necessary estimates on the semicircular flow near the cusp are
proven in Section~\ref{sec  scflow}. 

We now outline the novelties and main difficulties at the cusp compared with the edge analysis in~\cite{1712.03881}.
The  basic idea  is  to interpolate between   the time evolution of two DBM's, with initial conditions given by the original ensemble
and the reference ensemble, respectively, after their local densities have been matched by shift and scaling.
Beyond this  common idea there are several differences.

The first  difficulty lies in the rigidity  analysis of the DBM starting from the interpolated  initial conditions.
The optimal  rigidity  from~\cite{1809.03971}, that
 holds for very general Wigner-type matrices,   applies for the flows of both  the  original and the reference matrices,
 but it does not directly apply to the interpolating process. The latter starts from a regular initial data but it
  runs for a very short time, violating the \emph{flatness}  (i.e.~effective mean-field) assumption of~\cite{1809.03971}.
 While it is possible to extend the analysis of~\cite{1809.03971} to this case, here we chose a  technically
 lighter and conceptually  more interesting route. We use the maximum principle of the DBM to transfer
  rigidity information on the reference process to the interpolating one after an appropriate localization. Similar ideas for proving rigidity of the $\beta$-DBM flow has been used in the bulk~\cite{1612.06306} and at the edge~\cite{1810.08308}.

The second difficulty in the cusp regime  is that the  shape of the density  is highly unstable under
the semicircular flow that describes the evolution of the density under the DBM. 
The regular edge analysed in~\cite{1712.03881} remains of square root type along its dynamics and it
 can be simply described by its location and its multiplicative \emph{slope parameter} --- both vary regularly with time.
In contrast, the evolution of the cusp is a relatively complicated process: it  starts with 
a small gap that shrinks to zero as the cusp forms and then continues  developing  a small local minimum. Heavily relying on the main results of~\cite{1804.07752}, the density is described by  quite involved shape functions, see~\eqref{Psi edge},~\eqref{Psi min}, that
have a two-scale structure, given in terms of a total of three parameters, each varying on different time scales.
For example, the location of the gap moves linearly with time, the length of gap shrinks as the $3/2$-th power
of the time, while the local minimum after the cusp increases as the $1/2$-th power
of the time. The scaling behavior of the corresponding quantiles, that approximate  the eigenvalues by rigidity, 
follows the same complicated pattern of the density. All these require a very precise description of the semicircular
flow near the cusp as well as the optimal rigidity.

The third difficulty is that we need to run the DBM for a relatively long time in order to exploit the local decay; in
fact this time scale, $N^{-1/2+\epsilon}$ is considerably longer than the characteristic time scale $N^{-3/4}$ on which the {physical} cusp varies under the semicircular flow. 
We need to tune the initial condition very precisely so that after a relatively long time it develops 
a cusp exactly at the right location with the right slope.

The fourth difficulty is that, unlike for the regular edge regime, the eigenvalues or quantiles on both sides of the ({physical}) cusp  contribute to
the short range approximation of the dynamics,  their effect cannot be treated as mean-field. Moreover, 
there are two scaling regimes for quantiles  corresponding to the two-scale structure of the density.

Finally, we note that  the analysis of the semicircular flow around the cusp, partly completed already 
 in the companion paper~\cite{1809.03971},
 is relatively short and transparent despite its considerably
more complex pattern compared to the corresponding analysis around  the regular edge. 
  This is mostly   due to strong  results imported from the general shape analysis~\cite{1506.05095}. Not only 
the exact formulas for the density shapes are taken over, but we also heavily rely on the  $1/3$-H\"older continuity
in space and time of the density and its Stieltjes transform, established in the strongest form in~\cite{1804.07752}.

\bigskip

\noindent\textbf{Notations and conventions.} 
We now introduce some custom notations we use throughout the paper. 
For integers $n$ we define $[n]\defeq\{1,\dots,n\}$.
For positive quantities $f,g$, we write $f\lesssim g$ and $f\sim g$ if $f\le Cg$ or, respectively, $cg\le f\le Cg$ for some constants $c,C$ that depend only on the \emph{model parameters}, 
i.e.~on the constants appearing in the basic Assumptions~\ref{bdd moments}--\ref{bdd m} listed in Section~\ref{sec: main results} below. Similarly, we write $f\ll g$ if $f\le cg$ 
for some tiny constant $c>0$ depending on the model parameters.  We denote vectors by bold-faced lower case Roman letters $\vx,\vy\in\C^N$, and matrices by upper case Roman letters $A,B\in\C^{N\times N}$. We write $\braket{A}\defeq N^{-1}\Tr A$ and $\braket{\vx}\defeq N^{-1}\sum_{a\in[N]}x_a$ for the averaged trace and the average of a vector. We often identify diagonal matrices with the vector of its diagonal elements. Accordingly, for any matrix $R$, 
we denote by $\diag(R)$ the vector of its diagonal elements, and for any vector $\vr$ we denote by $\diag(\vr)$ the 
corresponding diagonal matrix.

We will frequently use the concept of ``with very high probability'' meaning that for any fixed $D>0$  the probability of the event is bigger than $1-N^{-D}$ if $N\ge N_0(D)$.

\bigskip
\noindent\textbf{Acknowledgement.} The authors are very grateful to Johannes Alt for his invaluable contribution in helping improve several  results of~\cite{1804.07752} tailored to the needs of this paper.

\section{Main results}\label{sec: main results}
For definiteness we consider the real symmetric case $H\in\R^{N\times N}$. With small modifications the proof presented in this paper works for complex Hermitian case as well, but 
this case was already considered in~\cite{1809.03971} with a contour integral analysis. Let $W=W^* \in \R^{N \times N}$ be a symmetric random matrix and $A=\diag(\bm{a})$ be a deterministic diagonal matrix with entries $\bm{a}=(a_i)_{i=1}^N \in \R^N$. We say that $W$ is of \emph{Wigner-type}~\cite{MR3719056} if its entries $w_{ij}$ for $i \le j$ are centred, $\E  w_{ij} =0$, independent random variables. We define the \emph{variance matrix} or \emph{self-energy matrix} $S=(s_{ij})_{i,j=1}^N$, $s_{ij}\defeq \E w_{ij}^2$. In~\cite{MR3719056} it was shown that as $N$ tends to infinity, the resolvent $G(z)\defeq (H-z)^{-1}$ of the \emph{deformed Wigner-type matrix} $H=A+W$ entrywise approaches a diagonal matrix $M(z)\defeq \diag(\vm(z))$ for $z\in\HC\defeq\Set{z\in\C|\Im z>0}$. The entries $\vm=(m_1 \dots , m_N)\colon \HC \to \HC^N$  of $M$ have positive imaginary parts and solve the \emph{Dyson equation}
\begin{equation} \label{Dyson equation} -\frac{1}{m_i(z)}= z-a_i +\sum_{j=1}^Ns_{ij}m_j(z),\qquad z \in \HC\defeq\Set{z\in\C|\Im z>0}, \quad i\in[N]. \end{equation}
We call $M$ or $\vm$ the \emph{self-consistent Green's function}. 
The normalised trace $\braket{M}$ of $M$ is the Stieltjes transform $\braket{M(z)}=\int_\R (\tau-z)^{-1}\rho(\diff\tau)$ of a unique probability measure $\rho$ on $\R$ that approximates the empirical eigenvalue distribution of $A+W$ increasingly well as $N \to \infty$. We call $\rho$ the \emph{self-consistent density of states} (scDOS). Accordingly, its support $\supp\rho$ is called the \emph{self-consistent spectrum}. It was proven in~\cite{1506.05095} that under very general conditions, $\rho(\diff\tau)$ is an absolutely continuous measure with a $1/3$-H\"older continuous density, $\rho(\tau)$. 
Furthermore, the self-consistent spectrum consists of finitely many intervals with square root growth of $\rho$ at the \emph{edges}, i.e.~at the points in $\partial \supp \rho$.

We call a point $\cu\in\R$ a cusp of $\rho$ if $\cu\in\interior\supp\rho$ and $\rho(\cu)=0$. Cusps naturally emerge when we consider a one-parameter family of ensembles and two support intervals of $\rho$ merge as the parameter value changes. The cusp universality phenomenon is not restricted to the exact cusp; it also occurs for situations shortly before and after the merging of two such support intervals, giving rise to a one parameter family of universal statistics. More precisely, universality emerges if $\rho$ has a \emph{{physical} cusp}. 
The terminology indicates that all these singularities become indistinguishable from the exact cusp if the density is resolved with a local precision above the typical eigenvalue spacing.
 We say that $\rho$ exhibits a physical cusp if it has a small gap  $(\ed_-, \ed_+) \subset \R\setminus \supp \rho$  with $\ed_+,\ed_- \in \supp \rho$ in its support of size  $\ed_+-\ed_-\lesssim N^{-3/4}$ or a local minimum $\mi\in\interior\supp\rho$ of size $\rho(\mi)\lesssim N^{-1/4}$, cf.~Figure~\ref{fig phys cusp}.  Correspondingly, we call the points 
 $\bu\defeq \frac{1}{2}
 (\ed_++\ed_-)$ and $\bu\defeq\mi$ \emph{{physical} cusp points}, respectively. One of the simplest models exhibiting a physical cusp point is the deformed Wigner matrix
\begin{equation}\label{eq deformed wigner} H=\diag(1,\dots,1,-1,\dots,-1)+\sqrt{1+t}W \end{equation}
 with equal numbers of \(\pm1\), and where \(W\) is a Wigner matrix of variance \(\  E\abs{w_{ij}}^2=N^{-1}\). The ensemble \(H\) from~\eqref{eq deformed wigner} exhibits an exact cusp if \(t=0\) and a physical cusp if \(\abs{t}\lesssim N^{-1/2}\), with \(t>0\) corresponding to a small non-zero local minimum and \(t<0\) corresponding to a small gap in the support of the self-consistent density. For the proof of universality in the real symmetric symmetry class we will use~\eqref{eq deformed wigner} with \(W\sim\mathrm{GOE}\) as a Gaussian reference ensemble.
\begin{figure}[htbp]
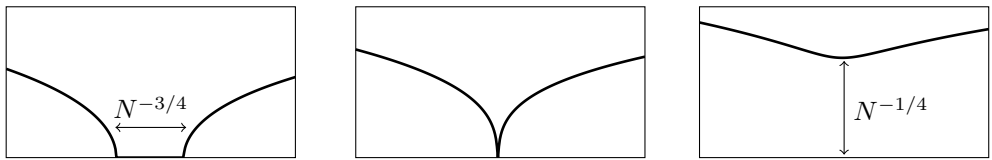

  \centering
    \setlength{\figurewidth}{4cm}\setlength{\figureheight}{2cm}\input{almost0.tikz} \qquad \input{almost1.tikz}\qquad \input{almost2.tikz}
  \caption{The cusp universality class can be observed in a 1-parameter family of \emph{physical cusps}.}
  \label{fig phys cusp}
\end{figure}

Our main result is cusp universality under the real symmetric analogues of the assumptions of~\cite{1809.03971}. Throughout this paper we make the following three assumptions:  

\begin{assumption}[Bounded moments]\label{bdd moments}
The entries of the  matrix $\sqrt{N}W$ have bounded moments and the expectation $A$ is bounded, i.e.\ there are positive $C_k$ such that 
\[
\abs{a_i}\le C_0, \qquad 
\E\abs{w_{ij}}^k \le C_kN^{-k/2} , \qquad k \in \N.
\]
\end{assumption}
\begin{assumption}[Flatness]\label{fullness}
We assume that the matrix $S$ is flat in the sense $s_{ij}=\E w_{ij}^2 \ge c/N$ for some constant $c>0$. 
\end{assumption}
\begin{assumption}[Bounded self-consistent Green's function] \label{bdd m}
The scDOS $\rho$ has a {physical} cusp point $\bu$, and in a neighbourhood of the {physical} cusp point $\bu \in \R$ the self-consistent Green's function is bounded, i.e.\ for positive $C,\kappa$ we have 
\[
\abs{m_{i}(z)}\le C, \qquad z \in [\bu-\kappa,\bu+\kappa]+ \ii \R^+.
\]
\end{assumption}

We call the constants appearing in Assumptions~\ref{bdd moments}--\ref{bdd m} \emph{model parameters}. All generic constants in this paper may implicitly depend on these model parameters. Dependence on further parameters, however, will be indicated.

\begin{remark}
The boundedness of $\vm$ in Assumption~\ref{bdd m} can be, for example, ensured by assuming some regularity of the variance matrix $S$. For more details we refer to~\cite[Chapter~6]{1506.05095}.
\end{remark}

According to the extensive analysis in~\cite{1506.05095,1804.07752} it follows\footnote{The claimed expansions~\eqref{gamma def} and~\eqref{gamma def min} follow directly from~\cite[Theorem 7.2(c),(d)]{1804.07752}. The error term in~\eqref{gamma def edge} follows from~\cite[Theorem 7.1(a)]{1804.07752}, where we define $\gamma$ according to $h$ therein.} that there exists some small $\delta_\ast\sim 1$ such that the self-consistent density $\rho$ around the points where it is small exhibits one of the following three types of behaviours.
\begin{subequations}\label{gamma def eqs}
\begin{enumerate}[(i)]
\item \emph{Exact cusp}. There is a cusp point $\cu\in\R$ in the sense that $\rho(\cu)=0$ and $\rho(\cu\pm\delta)>0$ for $0\ne\delta\ll1$. In this case the self-consistent density is locally around $\cu$ given by 
\begin{equation}\label{gamma def}
\rho(\cu +\omega) = \frac{\sqrt 3\gamma^{4/3}\abs{\omega}^{1/3}}{2\pi}\Big[1+\landauO{\abs{\omega}^{1/3}}\Big]
\end{equation}
for $\omega\in [-\delta_\ast,\delta_\ast]$ and some $\gamma>0$.
\item \emph{Small gap.} There is a maximal interval $[\ed_-,\ed_+]$ of size $0<\Delta \defeq  \ed_+-\ed_-\ll1$ such that $\rho\rvert_{[\ed_-,\ed_+]}\equiv 0$. In this case the density around $\ed_\pm$ is, for some $\gamma>0$, locally given by
\begin{equation}\label{gamma def edge}
\rho(\ed_\pm\pm \omega)=\frac{\sqrt{3}(2\gamma)^{4/3}\Delta^{1/3}}{2\pi}\Psi_{\mathrm{edge}}(\omega/\Delta)\left[1+\landauO{\min\Bigl\{\omega^{1/3},\frac{\omega^{1/2}}{\Delta^{1/6}}\Bigr\}}\right]
\end{equation}
for $\omega\in[0,\delta_\ast]$, where 
\begin{equation}\label{Psi edge}
\Psi_{\mathrm{edge}}(\lambda)\defeq \frac{\sqrt{\lambda(1+\lambda)}}{(1+2\lambda+2\sqrt{\lambda(1+\lambda)})^{2/3}+(1+2\lambda-2\sqrt{\lambda(1+\lambda)})^{2/3}+1},\quad \lambda\ge0.
\end{equation}
\item \emph{Non-zero local minimum.} There is a local minimum at $\mi\in\R$ of $\rho$ such that $0<\rho(\mi)\ll1$. In this case there exists some $\gamma>0$ such that
\begin{equation}\label{gamma def min}
\rho(\mi + \omega) = \rho(\mi) + \rho(\mi) \Psi_{\mathrm{min}}\left(\frac{3\sqrt 3 \gamma^4 \omega}{2(\pi\rho(\mi))^3 }\right) \left[1+\landauO{\min\Bigl\{\rho(\mi)^{1/2},\frac{\rho(\mi)^4}{\abs{\omega}}\Bigr\}+\min\Bigl\{\frac{\omega^2}{\rho(\mi)^5},\abs{\omega}^{1/3}\Bigr\}}\right]
\end{equation}
for $\omega\in[-\delta_\ast,\delta_\ast]$, where 
\begin{equation}\label{Psi min}
\Psi_{\mathrm{min}}(\lambda) \defeq \frac{\sqrt{1+\lambda^2}}{(\sqrt{1+\lambda^2}+\lambda)^{2/3}+(\sqrt{1+\lambda^2}-\lambda)^{2/3}-1}-1,\qquad \lambda\in\R.
\end{equation}
\end{enumerate}
\end{subequations}
We note that the choices for the \emph{slope} parameter $\gamma$ in~\eqref{gamma def edge}--\eqref{gamma def min} are consistent with~\eqref{gamma def} in the sense that in the regimes $\Delta\ll \omega\ll 1$ and $\rho(\mi)^3\ll \abs{\omega}\ll 1$ the respective formulae asymptotically agree. The precise form of the pre-factors in~\eqref{gamma def eqs} is also chosen such that in the universality statement $\gamma$ is a linear rescaling parameter.

It is natural to express universality in terms of a rescaled $k$-point function $p_k^{(N)}$ which we define implicitly by 
\begin{equation} \E\binom{N}{k}^{-1} \sum_{\{i_1,\dots,i_k\}\subset[N]} f(\lambda_{i_1},\dots,\lambda_{i_k}) = \int_{\R^k} f(\bm x)p_k^{(N)}(\bm x)\diff \bm x \label{eq k pt func}\end{equation}
for test functions $f$, where the summation is over all subsets of $k$ distinct integers from $[N]$.

\begin{theorem}\label{thr:cusp universality}
Let $H$ be a real symmetric or complex Hermitian deformed Wigner-type matrix whose scDOS $\rho$ has a {physical} cusp point $\bu$ such that Assumptions~\ref{bdd moments}--\ref{bdd m} are satisfied. Let $\gamma>0$ be the slope parameter at $\bu$, i.e.~such that $\rho$ is locally around $\bu$ given by~\eqref{gamma def eqs}. Then the local $k$-point correlation function at $\bu$ is universal, i.e.~for any $k\in\N$ there exists a $k$-point correlation function $p_{k,\alpha}^{\mathrm{GOE}/\mathrm{GUE}}$ such that for any test function $F\in C_c^1(\overline{\Omega})$, with $\Omega\subset\R^k$ some bounded open set, it holds that
\[
\int_{\R^k} F(\bx)\left[ \frac{N^{k/4}}{\gamma^k} p_{k}^{(N)}\left( \bu + \frac{\bx}{\gamma N^{3/4}}\right)- p_{k,\alpha}^{\mathrm{GOE}/\mathrm{GUE}}\left(\bx\right)\right] \diff \bx = \mathcal{O}_{k,\Omega}\big( N^{-c(k)}\lVert F\rVert_{C^1} \big),
\]
where the parameter $\alpha$ and the physical cusp \(\bu\) are given by 
\begin{equation}
\alpha \defeq \begin{cases}
0 & \text{in case (i)}\\
3 \left(\gamma\Delta/4\right)^{2/3} N^{1/2} & \text{in case (ii)}\\
-\left(\pi\rho(\mi)/\gamma\right)^2 N^{1/2} & \text{in case (iii)},
\end{cases}, \qquad 
\bu\defeq \begin{cases}
  \cu & \text{in case (i)}\\
  (\ed_-+\ed_+)/2 & \text{in case (ii)}\\
  \mi & \text{in case (iii)},
  \end{cases} \label{eq pearcey param choice}
\end{equation}
and $c(k)>0$ is a small constant only depending on $k$. The implicit constant in the error term depends on \(k\) and the diameter of the set \(\Omega\).
\end{theorem} 

\begin{remark}\label{remark kpt func}
\begin{enumerate}[(i)]
\item In the complex Hermitian symmetry class the $k$-point function is given by
\[ p_{k,\alpha}^{\mathrm{GUE}}(\bm x)= \det \Bigl(K_{\alpha,\alpha}(x_i,x_j)\Bigr)_{i,j=1}^k.\]
Here the \emph{extended Pearcey kernel} $K_{\alpha,\beta}$ is given by 
\begin{equation}\label{Pearcey kernel}
\begin{split}
K_{\alpha,\beta}(x,y) &= \frac{1}{(2\pi\ii)^2} \int_\Xi \diff z \int_\Phi \diff w \frac{\exp(-w^4/4 + \beta w^2/2-yw + z^4/4-\alpha z^2/2 + xz)}{w-z}\\
&\quad - \frac{\1_{\beta>\alpha}}{\sqrt{2\pi(\beta-\alpha)}} \exp\Bigl(-\frac{(y-x)^2}{2(\beta-\alpha)}\Bigr),
\end{split}
\end{equation}
where $\Xi$ is a contour consisting of rays from $\pm\infty e^{\ii\pi/4}$ to $0$ and rays from $0$ to $\pm \infty e^{-\ii\pi/4}$, and $\Phi$ is the ray from $-\ii\infty$ to $\ii\infty$. For more details we refer to~\cite{MR2207649,MR2642890,MR1618958} and the references in~\cite{1809.03971}.

\item The real symmetric $k$-point function (possibly only a distribution) \smash{$p_{k,\alpha}^{\mathrm{GOE}}$} is not known explicitly. In fact, it is not even known whether \smash{$p_{k,\alpha}^{\mathrm{GOE}}$} is Pfaffian.
 We will nevertheless establish the existence of $p_{k, \alpha}^\mathrm{GOE}$ as a distribution in the dual of the
 $C^1$ functions in Section~\ref{sec OU flow} as the limit of
the correlation functions of a one parameter family of Gaussian comparison models.
\end{enumerate}
\end{remark}

Theorem~\ref{thr:cusp universality} is a universality result about the spatial correlations of eigenvalues. Our method also allows us to prove the corresponding statement on space-time universality when we consider the time evolution of eigenvalues $(\lambda_i^t)_{i\in[N]}$ according to the Dyson Brownian motion \smash{$\diff H^{(t)}=\diff \mathfrak{B}_t$} with initial condition \smash{$H^{(0)}=H$}, where, depending on the symmetry class, $\mathfrak{B}_t$ is a complex Hermitian or real symmetric matrix valued Brownian motion. For any ordered $k$-tuple $\bm\tau=(\tau_1,\dots,\tau_k)$ with $0\le\tau_1\le\dots\le\tau_k\lesssim N^{-1/2}$ we then define the \emph{time-dependent $k$-point function} as follows. Denote the unique values in the tuple $\bm\tau$ by $\sigma_1<\dots<\sigma_l$ such that $\{\tau_1,\dots,\tau_k\}=\{\sigma_1,\dots,\sigma_l\}$ and denote the multiplicity of $\sigma_j$ in $\bm\tau$ by $k_j$ and note that $\sum k_j=k$. We then define \smash{$p_{k,\bm\tau}^{(N)}$} implicitly via
\begin{equation}\label{eq time k pt func}
\E \prod_{j=1}^l \left[ \binom{N}{k_j}^{-1}\sum_{\{i_1^j,\dots,i_{k_j}^j\}\subset[N]} \right] f(\lambda_{i_1^1}^{\sigma_1},\dots,\lambda_{i_{k_1}^1}^{\sigma_1},\dots,\lambda_{i_1^l}^{\sigma_l},\dots,\lambda_{i_{k_l}^l}^{\sigma_l})=\int_{\R^k} f(\bm x) p_{k,\bm\tau}^{(N)}(\bm x)\diff \bm x
\end{equation}
for test functions $f$ and note that~\eqref{eq time k pt func} reduces to~\eqref{eq k pt func} in the case $\tau_1=\dots=\tau_k=0$. We note that in~\eqref{eq time k pt func} coinciding indices are allowed only for eigenvalues at different times. If the scDOS $\rho$ of $H$ has a physical cusp in $\bu$, then for $\tau\lesssim N^{-1/2}$ the scDOS $\rho_\tau$ of \smash{$H^{(\tau)}$} also has a physical cusp $\bu_\tau$ close to $\bu$ and we can prove space-time universality in the sense of the following theorem, whose proof we defer to Appendix~\ref{sec proof space time}.

\begin{theorem}\label{thr:time cusp universality}
Let $H$ be a real symmetric or complex Hermitian deformed Wigner-type matrix whose scDOS $\rho$ has a {physical} cusp point $\bu$ such that Assumptions~\ref{bdd moments}--\ref{bdd m} are satisfied. Let $\gamma>0$ be the slope parameter at $\bu$, i.e.~such that $\rho$ is locally around $\bu$ given by~\eqref{gamma def eqs}. Then there exists a $k$-point correlation function $p_{k,\bm\alpha}^{\mathrm{GOE}/\mathrm{GUE}}$ such that for any $0\le\tau_1\le\dots\le\tau_k\lesssim N^{-1/2}$ and for any test function $F\in C_c^1(\overline{\Omega})$, with $\Omega\subset\R^k$ some bounded open set, it holds that
\[ \int_{\R^k} F(\bx)\left[ \frac{N^{k/4}}{\gamma^k} p_{k,\bm\tau/\gamma^2}^{(N)}\left( \bu_{\bm\tau/\gamma^2} + \frac{\bx}{\gamma N^{3/4}}\right)- p_{k,\bm\alpha}^{\mathrm{GOE}/\mathrm{GUE}}\left(\bx\right)\right] \diff \bx = \mathcal{O}_{k,\Omega}\left(N^{-c(k)}\lVert F\rVert_{C^1}\right),\]
where $\bm\tau=(\tau_1,\dots,\tau_k)$, $\bu_{\bm\tau}=(\bu_{\tau_1},\dots,\bu_{\tau_k})$ and $\bm\alpha=\alpha-\bm\tau N^{1/2}$ with $\alpha$ from~\eqref{eq pearcey param choice} and $c(k)>0$ is a small constant only depending on $k$. In the case of the complex Hermitian symmetry class the $k$-point correlation function is known to be determinantal of the form
\[ p_{\alpha_1,\dots,\alpha_k}^{\mathrm{GUE}}(\bm x)=\det\Bigl( K_{\alpha_i,\alpha_j}(x_i,x_j)\Bigr)_{i,j=1}^k\]
with $K_{\alpha,\beta}$ as in~\eqref{Pearcey kernel}.
\end{theorem}

The analogous version of Theorem~\ref{thr:time cusp universality} for fixed energy bulk multitime universality has been proven in~\cite[Sec.~2.3.1.]{MR3914908}. 

\begin{remark}
The extended Pearcey kernel $K_{\alpha,\beta}$ in Theorem~\ref{thr:time cusp universality} has already been observed for the double-scaling limit of non-intersecting Brownian bridges~\cite{MR2642890,MR2207649}. However, in the random matrix setting our methods also allow us to prove that the space-time universality of Theorem~\ref{thr:time cusp universality} extends beyond the Gaussian DBM flow. If the times $0\le \tau_1\le\dots\le\tau_k\lesssim N^{-1/2}$ are ordered, then the $k$-point correlation function of the DBM flow asymptotically agrees with the $k$-point correlation function of eigenvalues of the matrices 
\[H+\sqrt{\tau_1}W_1,H+\sqrt{\tau_1} W_1 + \sqrt{\tau_2-\tau_1} W_2, \dots, H+\sqrt{\tau_1}W_1+\dots+\sqrt{\tau_k-\tau_{k-1}}W_k \]
for independent standard Wigner matrices $W_1,\dots,W_k$. 
\end{remark}

\section{Ornstein-Uhlenbeck flow}\label{sec OU flow}
Starting from  this section we consider  a more general framework that allows for random matrix ensembles with certain correlation among the entries.
In this way we  stress that our proofs  regarding the semicircular flow and the Dyson Brownian motion are largely model independent,
assuming the optimal local law holds. The  independence assumption on the entries  of $W$  is made
only  because we rely on the local law from~\cite{1809.03971} that was proven for deformed Wigner-type matrices. 
We therefore present the flow directly in the more general framework of the \emph{matrix Dyson equation} (MDE)
\begin{equation} \label{MDE matrix form} 1 + (z-A+\SS[M(z)])M(z) =0,\qquad A\defeq \E H,\qquad \SS[R]\defeq \E WRW, \end{equation}
with spectral parameter in the complex upper half plane, $\Im z>0$, and positive definite imaginary part, $\frac{1}{2\ii}(M(z)-M(z)^\ast)>0$, of the solution $M$. The MDE generalizes~\eqref{Dyson equation}. Note that  in the deformed
 Wigner-type case the \emph{self-energy operator}  $\SS\colon\C^{N\times N}\to \C^{N\times N}$ is related to the variance matrix $S$
by $\SS[\diag\vr] = \diag(S\vr)$.

As in~\cite{1809.03971} we consider the Ornstein-Uhlenbeck flow 
\begin{equation}
\label{ORNULflow} \diff \wt H_s=-\frac{1}{2}(\wt H_s-A)\diff s+ \Sigma^{1/2}[\diff \mathfrak{B}_s] , \qquad \Sigma[R]\defeq\frac{\beta}{2}\E W\Tr WR, \qquad \wt H_0\defeq H,
\end{equation}
which preserves expectation and self-energy operator $\SS$. Since we consider real symmetric $H$, the parameter $\beta$ indicating the symmetry class is $\beta=1$.  In~\eqref{ORNULflow} with $\mathfrak{B}_s\in\mathbb{R}^{N\times N}$ we denote a real symmetric 
matrix valued standard  (GOE) Brownian motion, i.e.~$(\mathfrak{B}_s)_{ij}$  for $i<j$ and $(\mathfrak{B}_s)_{ii}/\sqrt{2}$ are independent standard 
Brownian motions and $(\mathfrak{B}_s)_{ji}=(\mathfrak{B}_s)_{ij}$. In case $H$ were complex Hermitian, we would have $\beta =2$ and  $\diff \mathfrak{B}_s$ would be an infinitesimal GUE matrix. This was the setting in~\cite{1809.03971}.
 The OU flow effectively adds a small Gaussian component of size $\sqrt s$ to $\wt H_s$. More precisely, we can construct a Wigner-type matrix $H_s$, satisfying Assumptions~\ref{bdd moments}--\ref{bdd m}, such that, for any fixed $s$, 
\begin{equation} \label{OU decomp}\wt H_s= H_s + \sqrt{cs} U,\qquad \SS_s=\SS-cs\SS^{\mathrm{GOE}}, \qquad \E H_s=A, \qquad U\sim \mathrm{GOE},\end{equation}
where $U$ is independent of $H_s$. Here $c>0$ is a small universal constant which depends on the constant in Assumption~\ref{fullness}, $\SS_s$ is the self-energy operator corresponding to $H_s$ and 
$\SS^\mathrm{GOE}[R]\defeq \braket{R}+R^t/N$, where $\langle \cdot \rangle\defeq N^{-1}\text{Tr}(\cdot)$ and  $R^t$ denotes the transpose of $R$. 
Since $\SS$ is flat in the sense $\SS[R]\gtrsim \braket{R}$ and $s$ is small it follows that also $\SS_s$ is flat.

As a consequence of the well established Green function comparison technique  the $k$-point function of $H=\wt H_0$ is comparable with the one of $\wt H_s$ as long as $s\le N^{-1/4-\epsilon}$ for some $\epsilon>0$. Indeed, from~\cite[Eq.~(116)]{1809.03971} for any $F\in C_c^1(\overline{\Omega})$, compactly supported $C^1$ test function 
on a bounded open set $\Omega\subset \mathbb{R}^k$, we find
\begin{equation}\label{gft}
\int_{\R^k} F(\bx) N^{k/4} \biggl[p_k^{(N)}\Bigl(\bu+\frac{\bx}{\gamma N^{3/4}}\Bigr)-\wt p_{k,s}^{(N)}\Bigl(\bu+\frac{\bx}{\gamma N^{3/4}}\Bigr)\biggr]\diff\bx = \mathcal{O}_{k,\Omega}\big( N^{-c} \lVert F\rVert_{C^1} \big),
\end{equation} 
where $\wt p_{k,s}^{(N)}$ is the $k$-point correlation function of $\wt H_s$,  $c=c(k)>0$ is some constant.\nc

It follows from the flatness assumption that the matrix $H_s$ satisfies the assumptions of the local law from~\cite[Theorem 2.5]{1809.03971} uniformly in $s\ll1$. Therefore~\cite[Corollary 2.6]{1809.03971} implies that the eigenvalues of $H_s$ are rigid down to the optimal scale. It remains to prove that for long enough times $s$ the local eigenvalue statistics of $H_s+\sqrt{cs}U$ on a scale of $1/\gamma N^{3/4}$ around $\bu$ agree with the local eigenvalue statistics of the Gaussian reference ensemble around $0$ at a scale of $1/N^{3/4}$. By a simple rescaling Theorem~\ref{thr:cusp universality} then follows from~\eqref{gft} together with the following proposition.

\begin{proposition}\label{DBM prop}
Let $t_1\defeq N^{-1/2+\omega_1}$ with some small $\omega_1>0$ and let $t_\ast$ be such that $\abs{t_\ast-t_1}\lesssim N^{-1/2}$. Assume that $H^{(\lambda)}$ and $H^{(\mu)}$ \footnote{We use the notation $H^{(\lambda)}$ and $H^{(\mu)}$ since we denote the eigenvalues of  $H^{(\lambda)}$ and $H^{(\mu)}$ by $\lambda_i$ and $\mu_i$ respectively, with $1\le i\le N$ respectively.} are Wigner-type matrices satisfying Assumptions~\ref{bdd moments}--\ref{bdd m} such that the scDOSs $\rho_{\lambda,t_\ast},\rho_{\mu,t_\ast}$ of $H^{(\lambda)}+\sqrt{t_\ast}U^{(\lambda)}$ and $H^{(\mu)}+\sqrt{t_\ast}U^{(\mu)}$ with independent $U^{(\lambda)},U^{(\mu)}\sim\mathrm{GOE}$ have cusps in some points $\cu_\lambda$, $\cu_\mu$ such that locally around $\cu_r$, $r=\lambda,\mu$, the densities $\rho_{r,t_\ast}$ are given by~\eqref{gamma def} with $\gamma=1$. Then the local $k$-point correlation functions $p_{k,t_1}^{(N,r)}$ of $H^{(r)}+\sqrt{t_1}U^{(r)}$ around the respective {physical} cusps $\bu_{r,t_1}$ of $\rho_{r,t_1}$, $j=1,2$, asymptotically agree in the sense 
\[ \int_{\R^k} F(\bx)\left[ N^{k/4} p_{k,t_1}^{(N,\lambda)}\left( \bu_{\lambda,t_1} + \frac{\bx}{ N^{3/4}}\right)-N^{k/4} p_{k,t_1}^{(N,\mu)}\left( \bu_{\mu,t_1} + \frac{\bx}{ N^{3/4}}\right)\right] \diff \bx =  \mathcal{O}_{k,\Omega}
\big( N^{-c(k)} \lVert F\rVert_{C^1} \big)  \]
 for any $F\in C_c^1(\overline{\Omega})$, with $\Omega\subset \mathbb{R}^k$ a bounded open set. 
\end{proposition}

\begin{proof}[Proof of Theorem~\ref{thr:cusp universality}]Set $s\defeq t_1/c\theta^2$ and $H^{(\lambda)}\defeq\theta H_s$ where $c$ is the constant from~\eqref{OU decomp} and $\theta\sim 1$ is yet to be chosen. Note that $H^{(\lambda)}+\sqrt{t} U=\theta(H_s+ \sqrt{t/\theta^2} U)$, and in particular $H^{(\lambda)}+\sqrt{t_1}U=\wt H_s$. Moreover, it follows from the semicircular flow analysis in Section~\ref{sec scflow} that for some $t_\ast$ with $\abs{t_\ast-t_1}\lesssim N^{-1/2}$, the scDOS $\theta\rho_{\lambda,t_\ast}(\lambda\cdot)$ of $H_s+\sqrt{t_\ast/\theta^2}U$ and thereby also $\rho_{\lambda,t_\ast}$, the one of $H^{(\lambda)}+\sqrt{t_\ast}U$, have exact cusps in $\cu_\lambda/\theta$ and $\cu_\lambda$, respectively. It follows from the $1/3$-H\"older continuity of the slope parameter, cf.~\cite[Lemma 10.5, Eq.~(7.5a)]{1804.07752}, that locally around $\cu_\lambda/\theta$ the scDOS of $H_s+\sqrt{t_\ast/\theta^2}U$ is given by
\[ \theta\rho_{\lambda,t_\ast}\bigl(\cu_\lambda+ \theta \omega\bigr)= \theta\rho_{\lambda,t_\ast}\Bigl(\theta\Bigl(\frac{\cu_\lambda}{\theta}+ \omega\Bigr)\Bigr)=\frac{\sqrt 3\gamma^{4/3}\abs{\omega}^{1/3}}{2\pi}\Bigl[1+\landauO{\abs{\omega}^{1/3}+\abs{t_\ast-t_1}^{1/3}}\Bigr].\]
Whence we can choose $\theta=\gamma\bigl[1+\landauO[1]{\abs{t_1-t_\ast}^{1/3}}\bigr]$ appropriately such that 
\[\rho_{\lambda,t_\ast}(\cu_\lambda+\omega) = \frac{\sqrt{3}\abs{\omega}^{1/3}}{2\pi}\Bigl[1+\landauO{\abs{\omega}^{1/3}}\Bigr]\]
and it follows that $H^{(\lambda)}$ satisfies the assumptions of Proposition~\ref{DBM prop}, in particular the slope parameter of $H^{(\lambda)}+\sqrt{t_\ast}U$ is normalized to $1$. Furthermore, the almost cusp $\bu_{\lambda,t_1}$ of $H^{(\lambda)}+\sqrt{t_1}U$ is given by $\bu_{\lambda,t_1}=\theta \bu$ with $\bu$ as in Theorem~\ref{thr:cusp universality}. 

We now choose our Gaussian comparison model. For $\alpha\in\R$ we consider the \emph{reference ensemble} 
\begin{equation} U_\alpha =U_\alpha^{(N)}
\defeq \diag(1,\dots,1,-1,\dots,-1)+\sqrt{1-\alpha N^{-1/2}}U\in\R^{N\times N}, \qquad U\sim\mathrm{GOE}\label{Ualpha def}\end{equation}
with $\floor{N/2}$ and $\lceil N/2\rceil$ times $\pm1$ in the deterministic diagonal. An elementary computation shows that for even $N$ and $\alpha=0$, the self-consistent density of $U_\alpha$ has an exact cusp of slope $\gamma=1$ in $\cu=0$, i.e.~it is given by~\eqref{gamma def}. For odd $N$ the exact cusp is at distance  $\lesssim N^{-1}$ away from $0$ 
 which is well below the natural scale of order $N^{-3/4}$  of the  eigenvalue fluctuation  and therefore has no influence on the $k$-point correlation function. The reference ensemble $U_\alpha$ has for $0\ne\abs{\alpha}\sim 1$ a small gap of size $N^{-3/4}$ or small local minimum of size $N^{-1/4}$ at the {physical} cusp point $\abs{\bu}\lesssim \frac{1}{N}$, depending on the sign of $\alpha$. Using the definition in~\eqref{Ualpha def}, let $H^{(\mu)}\defeq U_{N^{1/2}t_\ast}$ from which it follows that $H^{(\mu)}+\sqrt{t_\ast}U\sim U_0$ has an exact cusp in $0$ whose slope is $1$ by an easy explicit computation in the case of even $N$. For odd $N$ the cusp emerges at a distance of $\lesssim N^{-1}$ away from $0$, which is well below the investigated scale. Thus also $H^{(2)}$ satisfies the assumptions of Proposition~\ref{DBM prop}. The almost cusp $\bu_{\mu,t_1}$ is given by $\bu_{\mu,t_1}=0$ by symmetry of the density $\rho_{\mu,t_1}$ in the case of even $N$ and at a distance of $\abs[0]{\bu_{\mu,t_1}}\lesssim N^{-1}$ in the case of odd $N$. This fact follows, for example, from explicitly solving the $2$d-quadratic equation. The perturbation of size $1/N$ is not visible on the scale of the $k$-point correlation functions.

Now Proposition~\ref{DBM prop} together with~\eqref{gft} and $s\sim N^{-1/2+\omega_1}$ implies that 
\begin{equation}\label{univ comp Ualpha} \int_{\R^k} F(\bx)\left[ \frac{N^{k/4}}{\theta^k} p_{k}^{(N)}\left(\bu + \frac{\bx}{\theta N^{3/4}}\right)-N^{k/4} p_{k,\alpha,\mathrm{GOE}}^{(N)}\left(\frac{\bx}{ N^{3/4}}\right)\right] \diff \bx = \mathcal{O}_{k,\Omega} \big( N^{-c(k)} \lVert F\rVert_{C^1(\Omega)} \big)
\end{equation}
with $\alpha=N^{1/2}(t_\ast-t_1)$, where $p_{k,\alpha,\mathrm{GOE}}^{(N)}$ denotes the $k$-point function of the comparison model $U_\alpha$. This completes the proof of Theorem~\ref{thr:cusp universality} modulo the comparison of \smash{$p_{k,\alpha,\mathrm{GOE}}^{(N)}$} with its limit by relating $t_\ast-t_1$ to the size of the gap and the local minimum of $\rho$ via~\cite[Lemma 5.1]{1809.03971} (or~\eqref{eq Delta size}--\eqref{eq min size} later) and recalling that \smash{$\theta=\gamma\bigl[1+\landauO[1]{\abs{t_1-t_\ast}^{1/3}}\bigr]$}.

To complete the proof we claim that for any fixed $k$ and $\alpha$ there exists a distribution
 $p_{k,\alpha}^{\mathrm{GOE}}$  on $\R^k$, locally in the dual of $C^1_c(\overline{\Omega})$ for every open bounded $\Omega\subset\R^k$,
such that
\begin{equation}\label{infty}
\int_{\R^k} F(\bx)  \left[ N^{k/4} 
p_{k,\alpha,\mathrm{GOE}}^{(N)}\left(\frac{\bx}{N^{3/4}}\right) - 
p_{k,\alpha}^{\mathrm{GOE}}\left(\bx\right)
\right] \diff \bx = \mathcal{O}_{k,\Omega}\big( N^{-c(k)} \lVert F\rVert_{C^1}\big)
\end{equation}
holds for any $F\in C^1_c(\Omega)$. We now show that
~\eqref{infty} is a straightforward consequence of~\eqref{univ comp Ualpha}. 

 First notice that,  for notational simplicity, we gave the proof of~\eqref{univ comp Ualpha}  only
 for the case when $H$ and $U_\alpha$ are of the same dimension, but 
  it works without any modification when their dimensions are only comparable, see Remark~\ref{notequal}. 
   Hence, applying this result to a sequence  of GOE ensembles $U_\alpha^{(N_n)}$ with $N_n\defeq(4/3)^n$, for any compactly supported $F\in C^1_c(\overline{\Omega}) $ we have 
\begin{equation}
\label{cauchyseq}
\int_{\R^k} F(\bx)  \left[ N_n^{k/4} 
p_{k,\alpha,\mathrm{GOE}}^{(N_n)}\left(\frac{\bx}{N_n^{3/4}}\right) -  N_{n+1}^{k/4}
p_{k,\alpha,\mathrm{GOE}}^{(N_{n+1})}\left(\frac{\bx}{N_{n+1}^{3/4}}\right)
\right] \diff \bx = \mathcal{O}_{k,\Omega}\Big( \left( \frac{3}{4}\right)^{n c(k)}\lVert F\rVert_{C^1} \Big).
\end{equation}
Fix a bounded open set $\Omega\subset \R^k$ and 
 define the sequence of functionals $\{\mathcal{J}_n\}_{n\in\N}$ in the dual space $C^1_c(\overline{\Omega})^*$ as follows
\[
\mathcal{J}_n(F)\defeq \int_{\R^k} F(\bx)  N_n^{k/4} p_{k,\alpha,\mathrm{GOE}}^{(N_n)}\left(\frac{\bx}{N_n^{3/4}}\right) \diff \bx,
\]
for any $F\in C^1_c(\overline{\Omega})$. 
Then, by~\eqref{cauchyseq} it easily follows that $\{\mathcal{J}_n\}_{n\in\N}$ is a Cauchy sequence on $C^1_c(\overline{\Omega})^*$.  
Indeed, for any $M>L$ we have by a telescopic sum
\begin{equation}
\label{limitcau}
\begin{split}
\left| (\mathcal{J}_M-\mathcal{J}_L)(F)\right| &=\left|\sum_{n=L}^{M-1} \int_{\R^k} F(\bx)  \left[ N_{n+1}^{k/4} 
p_{k,\alpha,\mathrm{GOE}}^{(N_{n+1})}\left(\frac{\bx}{N_{n+1}^{3/4}}\right) -  N_n^{k/4}
p_{k,\alpha,\mathrm{GOE}}^{(N_n)}\left(\frac{\bx}{N_n^{3/4}}\right)
\right] \diff \bx \right| \\
&\le C_{k,\Omega}\left(\frac{3}{4} \right)^{Lc(k)}\lVert F\rVert_{C^1}.
\end{split}
\end{equation}
Thus, we conclude that there exists a unique $\mathcal{J}_\infty  \in C_c^1(\overline{\Omega})^*$ such that $\mathcal{J}_n\to \mathcal{J}_\infty$ as $n\to \infty$ in norm. Then,~\eqref{limitcau} clearly concludes the proof of~\eqref{infty}, 
identifying $\mathcal{J}_\infty=\mathcal{J}_\infty^{(\Omega)}$ with $p_{k,\alpha}^{\mathrm{GOE}}$ restricted to $\overline{\Omega}$.
Since this holds for any open bounded set $\Omega\subset \R^k$, the distribution 
 $p_{k,\alpha}^{\mathrm{GOE}}$ can be identified with the inductive limit of 
the consistent family of functionals $\{ \mathcal{J}_\infty^{(\Omega_m)}\}_{m\ge 1}$, where, say, $\Omega_m$ is the ball of radius $m$. This completes the proof of Theorem~\ref{thr:cusp universality}.

\end{proof}

\section{Semicircular flow analysis}\label{sec scflow}
In this section we analyse various properties of the semicircular flow in order to prepare the Dyson Brownian motion argument in Section~\ref{sec:rigid} and Section~\ref{DBMS}. If $\rho$ is a probability density on $\R$ with Stieltjes transform $m$, then the free semicircular evolution $\rho_t^\mathrm{fc}=\rho\boxplus\sqrt{t}\rho_\mathrm{sc}$ of $\rho$ is defined as the unique probability measure whose Stieltjes transform $m_t^\mathrm{fc}$ solves the implicit equation
\begin{equation}\label{eq free sc conv def}
m_t^\mathrm{fc}(\zeta)=m(\zeta+ t m_t^\mathrm{fc}(\zeta)), \qquad \zeta\in\HC,\quad t\ge 0. 
\end{equation}
Here $\sqrt{t}\rho_\mathrm{sc}$ is the semicircular distribution of variance $t$. 

We now prepare the Dyson Brownian motion argument in Section~\ref{DBMS} by providing a detailed analysis of the scDOS along the semicircular flow. As in Proposition~\ref{DBM prop} we consider the setting of two densities $\rho_\lambda,\rho_\mu$ whose semicircular evolutions reach a cusp of the same slope at the same time. Within the whole section we shall assume the following setup: Let $\rho_\lambda,\rho_\mu$ be densities associated with solutions $M_\lambda,M_\mu$ to some Dyson equations satisfying Assumptions~\ref{bdd moments}--\ref{bdd m} (or their matrix counterparts). We consider the free convolutions $\rho_{\lambda,t}\defeq \rho_\lambda\boxplus \sqrt{t}\rho_\mathrm{sc}$, $\rho_{\mu,t}\defeq \rho_\mu\boxplus\sqrt{t}\rho_\mathrm{sc}$ of $\rho_\lambda,\rho_\mu$ with semicircular distributions of variance $t$ and assume that after a time $t_\ast\sim N^{-1/2+\omega_1}$ both densities $\rho_{\lambda,t_\ast}$, $\rho_{\mu,t_\ast}$ have cusps in points $\cu_\lambda,\cu_\mu$ around which they can be approximated by~\eqref{gamma def} with the same $\gamma=\gamma_\lambda(t_\ast)=\gamma_\mu(t_\ast)$. It follows from the semicircular flow analysis in~\cite[Lemma 5.1]{1809.03971} that for $0\le t\le t_\ast$ both 
densities have small gaps $[\ed_{r,t}^-,\ed_{r,t}^+]$, $r=\lambda,\mu$
 in their supports, while for $t_\ast\le t\le 2t_\ast$ they have non-zero local minima in some points $\mi_{r,t}$, $r=\lambda,\mu$. Instead of comparing the eigenvalue flows corresponding to \(\rho_\lambda,\rho_\mu\) directly, we rather consider a continuous interpolation \(\rho_\alpha\) for \(\alpha\in[0,1]\) of \(\rho_\lambda\) and \(\rho_\mu\). For technical reasons we define this interpolated density \(\rho_{\alpha,t}\) as an interpolation of \(\rho_{\lambda,t}\) and \(\rho_{\mu,t}\) separately for each time \(t\), rather than considering the evolution \(\rho_{\alpha,0}\boxplus\sqrt{t}\rho_\mathrm{sc}\) of the initial interpolation \(\rho_{\alpha,0}\). We warn the reader that semicircular evolution and interpolation do not commute, i.e.~\(\rho_{\alpha,t}\ne\rho_{\alpha,0}\boxplus\sqrt{t}\rho_\mathrm{sc}\). We now define the concept of \emph{interpolating densities} following~\cite[Section 3.1.1]{1712.03881}.

\begin{definition}\label{def interpolating density}
For $\alpha\in [0,1]$ define the \emph{$\alpha$-interpolating density} $\rho_{\alpha,t}$ as follows. For any $0\le E\le \delta_*$ and $r=\lambda,\mu$ let 
\[
      n_{r,t}(E) \defeq  \int_{\ed_{r,t}^+}^{\ed_{r,t}^++E} \rho_{r,t}(\om)\diff \om,\quad 0\le t\le t_\ast, \qquad  n_{r,t}(E) \defeq  \int_{\mi_{r,t}}^{\mi_{r,t}+E} \rho_{r,t}(\om)\diff \om, \quad t_\ast\le t\le 2t_\ast
\]
be the counting functions and $\varphi_{\lambda,t}$, $\varphi_{\mu,t}$ their  inverses, i.e.~$n_{r,t}(\varphi_{r,t}(s))=s$. Define now
\begin{equation}\label{interpolating phi}
  \varphi_{\alpha,t}(s)\defeq \alpha\varphi_{\lambda,t}(s)+ (1-\alpha)\varphi_{\mu,t}(s)
\end{equation}
for $s\in [0, \delta_{**}]$ where $\delta_{**}\sim 1$ depends on $\delta_*$ and is chosen in such a way that $\varphi_{\alpha,t}$ is invertible\footnote{Invertibility in a small neighbourhood follows from the form of the explicit shape functions in~\eqref{gamma def edge} and~\eqref{gamma def min}}. We thus define $n_{\alpha,t}(E)$ to be the inverse of $\varphi_{\alpha,t}(s)$ near zero. Furthermore,
for $0 \le t\le t_*$ set 
\begin{align}\label{rhoz}
\ed^\pm_{\alpha,t}&\defeq \alpha \ed_{\lambda,t}^\pm + (1-\alpha)\ed_{\mu,t}^\pm, \\
\label{dernal}
 \rho_{\alpha,t}(\ed^+_{\alpha,t}+E) & \defeq \frac{\diff}{\diff E} n_{\alpha,t}(E),\quad E\in[0,\delta_\ast]
\end{align}
and for $t\ge t_*$ set
\begin{align}
  \mi_{\alpha,t}&\defeq\alpha\mi_{\lambda,t}+(1-\alpha)\mi_{\mu,t}, \\
   \rho_{\alpha,t}(\mi_{\alpha,t}+E)&\defeq \alpha\rho_{\lambda,t}(\mi_{\lambda,t})+(1-\alpha)\rho_{\mu,t}(\mi_{\mu,t}) + \frac{\diff}{\diff E} n_{\alpha,t}(E),\quad E\in[-\delta_\ast,\delta_\ast].\nonumber
\end{align}
We define $ \rho_{\alpha,t}(E)$ for $0\le t\le t_\ast$ and $E\in [\ed_{\alpha,t}^--\delta_*, \ed_{\alpha,t}^-]$ analogously.  
\end{definition}
The motivation for the interpolation mode 
 in Definition~\ref{def interpolating density} is that~\eqref{interpolating phi} ensures that the quantiles of $\rho_{\alpha,t}$ 
are the convex combination of the quantiles of $\rho_{\lambda,t}$ and $\rho_{\mu,t}$,  see~\eqref{quantiles convex comb} later. 
The following two lemmas collect various properties of the interpolating density. Recall that 
$\rho_{\lambda,t}$ and $\rho_{\mu,t}$ are asymptotically close near the cusp regime,  up 
to a trivial shift, since they develop a cusp with the same slope at the same time.
 In  Lemma~\ref{lm:xyflow} we show that  $\rho_{\alpha,t}$ shares this property. 
   Lemma~\ref{lemma holderhsh} shows that $\rho_{\alpha,t}$ inherits the regularity properties of $\rho_{\lambda,t}$ and $\rho_{\mu,t}$ from~\cite{1804.07752}.

\begin{lemma}[Size of gaps and minima along the flow]\label{lm:xyflow}
For $t\le t_\ast$ and $r=\alpha,\lambda,\mu$ the supports of $\rho_{r,t}$ have small gaps $[\ed_{r,t}^-,\ed_{r,t}^+]$ near $\cu_\ast$ of size
\begin{subequations}
\begin{equation}\label{eq Delta size}
\Delta_{r,t}\defeq\ed_{r,t}^+-\ed_{r,t}^- = (2\gamma)^2 \Bigl(\frac{t_\ast-t}{3}\Bigr)^{3/2}\bigl[1+\landauO[1]{(t_\ast-t)^{1/3}}\bigr], \qquad \Delta_{r,t} = \Delta_{\mu,t} \bigl[1+\landauO[1]{(t_\ast-t)^{1/3}}\bigr]
\end{equation}
and the densities are close in the sense 
\begin{equation}\label{rho rho gap}
\rho_{r,t}(\ed_{r,t}^\pm\pm \omega) = \rho_{\mu,t}(\ed_{\mu,t}^\pm\pm \omega)\biggl[1+\landauO[2]{(t_\ast-t)^{1/3}+\min\Bigl\{\omega^{1/3},\frac{\omega^{1/2}}{(t_\ast-t)^{1/4}}\Bigr\} }\biggr]
\end{equation}
for $0\le\omega\le\delta_\ast$. For $t_\ast < t \le 2t_\ast$ the densities $\rho_{r,t}$ have small local minima $\mi_{r,t}$ of size
\begin{equation}\label{eq min size}
\rho_{r,t}(\mi_{r,t})=\frac{\gamma^2\sqrt{t-t_\ast}}{\pi}\bigl[1+\landauO[1]{(t-t_\ast)^{1/2}}\bigr], \qquad 
\rho_{r,t}(\mi_{r,t}) = \rho_{\mu,t}(\mi_{\mu,t}) \bigl[1+\landauO[1]{(t-t_\ast)^{1/2}}\bigr]
\end{equation}
and the densities are close in the sense 
\begin{equation}\label{rho rho min}
\frac{\rho_{r,t}(\mi_{r,t}+\omega)}{\rho_{\mu,t}(\mi_{\mu,t}+\omega)} = 1+\landauO[3]{(t-t_\ast)^{1/2}+\min\Bigl\{ (t-t_\ast)^{1/4},\frac{(t-t_\ast)^2}{\abs{\omega}}\Bigr\}+\min\Bigl\{\frac{\omega^2}{(t-t_\ast)^{5/2}},\abs{\omega}^{1/3}\Bigr\} }
\end{equation}
for $\omega\in[-\delta_\ast,\delta_\ast]$. Here $\delta_\ast,\delta_{\ast\ast}\sim1$ are small constants depending on the model parameters in Assumptions~\ref{bdd moments}--\ref{bdd m}.
\end{subequations}
\end{lemma}

\begin{lemma}\label{lemma holderhsh}
The density $\rho_{\alpha,t}$ from Definition~\ref{def interpolating density} is well defined and is a $1/3$-H\"older continuous density. More precisely, in the pre-cusp regime, i.e.~for $t\le t_\ast$, we have
\begin{subequations}
\begin{equation}\label{rho' difference hash} \abs{\rho_{\alpha,t}'(\ed_{\alpha,t}^\pm\pm x)} \lesssim \frac{1}{\rho_{\alpha,t}(\ed_{\alpha,t}^\pm\pm x)\Big(\rho_{\alpha,t}(\ed_{\alpha,t}^\pm\pm x)+\Delta_{\alpha,t}^{1/3}\Big)}\end{equation}
for $0\le x\le \delta_\ast$. Moreover, the Stieltjes transform $m_{\alpha,t}$ satisfies the bounds
\begin{equation}\label{weakm}
\begin{split}
\abs{m_{\alpha,t}(\ed_{\alpha,t}^\pm\pm x)}\lesssim 1,\quad \abs{m_{\alpha,t}(\ed_{\alpha,t}^\pm\pm (x+y))-m_{\alpha,t}(\ed_{\alpha,t}^\pm\pm x)}\lesssim \frac{\abs{y} \abs{\log \abs{y} }}{\rho_{\alpha,t}(\ed_{\alpha,t}^\pm\pm x)(\rho_{\alpha,t}(\ed_{\alpha,t}^\pm\pm x)+\Delta^{1/3}_{\alpha,t})}
\end{split}
\end{equation}
\end{subequations}
for $\abs{x}\le\delta_\ast/2$, $\abs{y}\ll x$. In the small minimum case, i.e.~for $t\ge t_\ast$, we similarly have
\begin{subequations}
\begin{equation}\label{rho' difference min}
\abs{\rho_{\alpha,t}'(\mi_{\alpha,t}+ x)} \lesssim \frac{1}{\rho_{\alpha,t}^2(\mi_{\alpha,t}+ x)}
\end{equation}
for $\abs{x}\le \delta_\ast$ and 
\begin{equation}\label{weakm min}
\qquad \abs{m_{\alpha,t}(\mi_{\alpha,t}+ x)}\lesssim 1,\quad\abs{m_{\alpha,t}(\mi_{\alpha,t}+ (x+y))-m_{\alpha,t}(\mi_{\alpha,t}+ x)}\lesssim \frac{\abs{y} \abs{\log\abs{y}}}{\rho_{\alpha,t}^2(\mi_{\alpha,t}+ x)}
\end{equation}
for $\abs{x}\le \delta_\ast$ and $\abs{y}\ll \abs{x}$.
\end{subequations}
\end{lemma}

\begin{proof}[Proof of Lemma~\ref{lm:xyflow}]
We first consider the two densities $r=\lambda,\mu$ only. The first claims in~\eqref{eq Delta size} and~\eqref{eq min size} follow directly from~\cite[Lemma 5.1]{1809.03971}, while the second claims follow immediately from the first ones. For the proof of~\eqref{rho rho gap} and~\eqref{rho rho min} we first note that by elementary calculus
\[\Psi_\edge((1+\epsilon)\lambda)=\Psi_\edge(\lambda)\bigl[1+\landauO{\epsilon}\bigr],\qquad\Psi_{\min}((1+\epsilon)\lambda) = \Psi_{\min}(\lambda)\bigl[1+\landauO{\epsilon}\bigr]\]
so that 
\[\Delta_{\lambda,t}^{1/3}\Psi_\edge(\omega/\Delta_{\lambda,t}) = \Delta_{\mu,t}^{1/3}\Psi_\edge(\omega/\Delta_{\mu,t})\biggl[1+\landauO[2]{(t_\ast-t)^{1/3}}\biggr] \]
and the claimed approximations follow together with~\eqref{gamma def edge} and~\eqref{gamma def min}. Here the exact cusp case $t=t_\ast$ is also covered by interpreting $0^{1/3}\Psi_\edge(\omega/0)=\omega^{1/3}/2^{4/3}$. 

 In order to prove the corresponding statements for the interpolating densities $\rho_{\alpha,t}$, 
 we first have to establish a quantitative understanding of the counting function $n_{r,t}$ and its inverse. 
 We claim that for $r=\alpha,\lambda,\mu$ they satisfy for $0\le E\le \delta_\ast$, $0\le s\le \delta_{\ast\ast}$ that
\begin{subequations}
\label{conting functions inverses}
\begin{equation}\label{rho counting fcts gap}
\begin{split}
n_{r,t}(E) &\sim\min\bigg\{ \frac{E^{3/2}}{\Delta_{r,t}^{1/6}},E^{4/3}\bigg\},\qquad  \varphi_{r,t}(s)\sim \max\bigg\{ s^{3/4},s^{2/3}\Delta^{1/9}_{r,t} \bigg\}, \\
 \frac{\varphi_{r,t}(s)}{\varphi_{\lambda,t}(s)} &\sim \min\bigg\{ \varphi_{\lambda,t}^{1/3}(s),\frac{\varphi_{\lambda,t}^{1/2}(s)}{\Delta_{\lambda,t}^{1/6}} \bigg\}
\end{split}
\end{equation}
for $t\le t_\ast$ and 
\begin{equation}\label{rho conting fcts min}
\begin{split}
n_{r,t}(E)&\sim\max\{ E^{4/3}, E \rho_{r,t}(\mi_{r,t}) \}, \qquad \varphi_{r,t}(s)\sim\min\bigg\{s^{3/4}, \frac{s}{\rho_{r,t}(\mi_{r,t})} \bigg\}\\ 
\frac{\varphi_{r,t}(s)}{\varphi_{\lambda,t}(s)} &\sim \min\bigg\{ \varphi_{\lambda,t}^{1/3}(s), \frac{\varphi_{\lambda,t}(s)}{\rho_{r,t}^2(\mi_{r,t})},\frac{\varphi^2_{\lambda,t}(s)}{\rho_{r,t}^{11/2}(\mi_{r,t})}\bigg\}
\end{split}
\end{equation}
for $t\ge t_\ast$. 
\end{subequations}
\begin{proof}[Proof of~\eqref{conting functions inverses}]
We begin with the proof of~\eqref{rho counting fcts gap} for $r=\lambda,\mu$. Recall that the shape function $\Psi_\mathrm{edge}$ satisfies the scaling $\Delta^{1/3}\Psi_\mathrm{edge}(\omega/\Delta)\sim\min\{\omega^{1/3},\omega^{1/2}/\Delta^{1/6}\}$. We first find by elementary integration that
\[ \int_0^q \min\Bigl\{\omega^{1/3},\frac{\omega^{1/2}}{\Delta^{1/6}}\Bigr\}\diff \omega = \frac{9q^{4/3}\min\{q,\Delta\}^{1/6}-\min\{q,\Delta\}^{3/2}}{12\Delta^{1/6}}\sim \min\Bigl\{\frac{q^{3/2}}{\Delta^{1/6}},q^{4/3}\Bigr\} \] 
from which we conclude the first relation in~\eqref{rho counting fcts gap}, and by inversion also the second relation. Together with the estimate for the error integral for $\rho_{\lambda,t}(\ed_{\lambda,t}^++\omega)-\rho_{\mu,t}(\ed_{\mu,t}^++\omega)\lesssim \min\{\omega^{2/3},\omega/\Delta_{\lambda,t}^{1/3}\}$,
\[  
\int_0^q \min\Bigl\{ \omega^{2/3}, \frac{\omega}{\Delta^{1/3}}\Bigr\}\diff \omega = \frac{6 q^{5/3} \min\{q,\Delta\}^{1/3}-\min\{q,\Delta\}^2}{10\Delta^{1/3}} \sim \min\Bigl\{\frac{q^2}{\Delta^{1/3}},q^{5/3}\Bigr\}
\]
we can thus conclude also the third relation in~\eqref{rho counting fcts gap}.

We now turn to the case $t>t_\ast$ where both densities $\rho_{\lambda,t},\rho_{\mu,t}$ exhibit a small local minimum. We first record the elementary integral
\[ \int_0^q \Bigl(\rho  + \min\Bigl\{\omega^{1/3},\frac{\omega^2}{\rho^5}\Bigr\}\Bigr)\diff \omega = \frac{q^{4/3}\min\{\rho^3,q\}^{5/3}+12q\rho^6-5\min\{q,\rho^3\}^3}{12\rho^5} \sim \max\{q^{4/3},q\rho \} \] 
for $q,\rho\ge 0$ and easily conclude the first two relation in~\eqref{rho conting fcts min}. For the error integral we obtain
\[ \int_0^q \min\Bigl\{\omega^{1/3},\frac{\omega^2}{\rho^5}\Bigr\}\Bigl[ \min\Bigl\{\rho^{1/2},\frac{\rho^4}{\omega}\Bigr\}+\min\Bigl\{\omega^{1/3},\frac{\omega^2}{\rho^5}\Bigr\} \Bigr] \diff \omega \sim\min\Bigl\{ q^{5/3}, \frac{q^2}{\rho}, \frac{q^3}{\rho^{9/2}}\Bigr\}\]
from which the third relation in~\eqref{rho conting fcts min} follows. Finally, the claims~\eqref{rho counting fcts gap} and~\eqref{rho conting fcts min} for $r=\alpha$ follow immediately from Definition~\ref{def interpolating density} and the corresponding statements for $r=\lambda,\mu$. This completes the proof of~\eqref{conting functions inverses}.
\end{proof}

We now turn to the density $\rho_{\alpha,t}$ for which the claims~\eqref{eq Delta size},~\eqref{eq min size} follow immediately from Definition~\ref{def interpolating density} and the corresponding statements for $\rho_{\lambda,t}$ and $\rho_{\mu,t}$. For $t\le t_\ast$ we now continue by differentiating $E=\varphi_{r,t}(n_{r,t}(E))$ to obtain
\begin{align}
\rho_{\alpha,t}(\ed_{\alpha,t}^++\varphi_{\alpha,t}(s)) &= \frac{1}{\varphi_{\alpha,t}'(s)} = \frac{1}{\alpha\varphi_{\lambda,t}'(s)+(1-\alpha)\varphi_{\mu,t}'(s)}=\bigg(\frac{\alpha}{\rho_{\lambda,t}(\ed_{\lambda,t}^++\varphi_{\lambda,t}(s))}+\frac{1-\alpha}{\rho_{\mu,t}(\ed_{\mu,t}^++\varphi_{\mu,t}(s))}\bigg)^{-1} \nonumber\\
&= \rho_{\lambda,t}(\ed_{\lambda,t}^++\varphi_{\lambda,t}(s)) \bigg(\alpha+(1-\alpha)\frac{\rho_{\lambda,t}(\ed_{\lambda,t}^++\varphi_{\lambda,t}(s))}{\rho_{\mu,t}(\ed_{\mu,t}^++\varphi_{\mu,t}(s))}\bigg)^{-1},
\label{ealphder}
\end{align}
from which we can easily conclude~\eqref{rho rho gap} for $r=\alpha$ together with~\eqref{rho rho gap} for $r=\lambda$ and~\eqref{rho counting fcts gap}. The proof of~\eqref{rho rho min} for $r=\alpha$ follows by the same argument and replacing $\ed_{r,t}^+$ by $\mi_{r,t}$. This finishes the proof of Lemma~\ref{lm:xyflow}
\end{proof}
\begin{proof}[Proof of Lemma~\ref{lemma holderhsh}]
By differentiating we find 
\begin{align*}
 &\frac{\rho_{\alpha,t}'(\ed_{\alpha,t}^++ \varphi_{\alpha,t}(s)) }{\rho_{\alpha,t}(\ed_{\alpha,t}^++ \varphi_{\alpha,t}(s)) }  =  - \frac{\alpha\varphi''_{\lambda,t}(s)+(1-\alpha)\varphi''_{\mu,t}(s)}{\Big(\alpha\varphi_{\lambda,t}'(s)+(1-\alpha)\varphi'_{\mu,t}(s)\Big)^2}  \\
& \qquad = \Bigg[ \alpha\frac{\rho_{\lambda,t}'(\ed_{\lambda,t}^++ \varphi_{\lambda,t}(s))}{\rho_{\lambda,t}^3(\ed_{\lambda,t}^++ \varphi_{\lambda,t}(s))}+(1-\alpha)\frac{\rho_{\mu,t}'(\ed_{\mu,t}^++ \varphi_{\mu,t}(s))}{\rho_{\mu,t}^3(\ed_{\mu,t}^++ \varphi_{\mu,t}(s))}\Bigg]
\Bigg(\frac{\alpha}{\rho_{\lambda,t}(\ed_{\lambda,t}^++\varphi_{\lambda,t}(s))}+\frac{1-\alpha}{\rho_{\mu,t}(\ed_{\mu,t}^++\varphi_{\mu,t}(s))}\Bigg)^{-2},
\end{align*}
from which we conclude the claimed bound~\eqref{rho' difference hash} together with the fact that the densities $\rho_\lambda$ and $\rho_\mu$ fulfil the same bound according to~\cite[Remark 10.7]{1804.07752}, and the estimates from Lemma~\ref{lm:xyflow}. Similarly, the bound in~\eqref{rho' difference min} follows by the same argument by replacing $\ed_{\alpha,t}^\pm$ by $\mi_{\alpha,t}$. The bound $\abs{\rho'}\le \rho^{-2}$  on the derivative implies $\frac{1}{3}$-H\"older continuity.

We now turn to the claimed bound on the Stieltjes transform and compute 
\[ m_{\alpha,t}(\ed_{\alpha,t}^+ + x) = \int_{0}^{\delta_\ast} \frac{\rho_{\alpha,t}(\ed_{\alpha,t}^++\om)}{\om-x}\diff\om + \int_{-\delta_\ast}^{0} \frac{\rho_{\alpha,t}(\ed_{\alpha,t}^-+\om)}{\om-\Delta_{\alpha,t}-x}\diff\om,\]
out of which for $x>0$ the first term can be bounded by 
\[\int_{0}^{\delta_\ast} \frac{\rho_{\alpha,t}(\ed_{\alpha,t}^++\om)}{\om-x}\diff\om \lesssim \int_{0}^{\delta_\ast} \frac{\abs{\om-x}^{1/3}}{\om-x}\diff\om + \int_{2x}^{\delta_\ast} \frac{\rho_{\alpha,t}(\ed_{\alpha,t}^++x)}{\om-x}\diff\om\lesssim \abs{x}^{1/3}\abs{\log x} + \abs{\delta_\ast-x}^{1/3},\]
while the second term can be bounded by 
\[\abs[3]{\int_{-\delta_\ast}^{0} \frac{\rho_{\alpha,t}(\ed_{\alpha,t}^-+\om)}{\om-\Delta_{\alpha,t}-x}\diff\om}\lesssim \abs{\delta_\ast-\Delta_{\alpha,t}-x}^{1/3}+\abs{\Delta_{\alpha,t}+x}^{1/3}\abs{\log(\Delta_{\alpha,t}+x)},\]
both using the $1/3$-H\"older continuity of $\rho_{\alpha,t}$. The corresponding bounds for $x<0$ are similar, completing the proof of the first bound in~\eqref{weakm}. 

The proof of the first bound in~\eqref{weakm min} is very similar and follows from
\[ \abs{m_{\alpha,t}(\mi_{\alpha,t}+x)}\lesssim  \abs{\int_{-\delta_\ast}^{\delta_\ast} \frac{\abs{\omega-x}^{1/3}}{\omega-x}\diff\omega} + \abs{\int_{[-\delta_\ast,\delta_\ast]\setminus[x-\delta_\ast/2,x+\delta_\ast/2]} \frac{\rho_{\alpha,t}(\mi_{\alpha,t}+x)}{\omega-x}\diff\omega}\lesssim 1. \]

We now turn to the second bound in~\eqref{weakm} which is only non-trivial in the case $x>0$. To simplify the following integrals we temporarily use the short-hand notations $m=m_{\alpha,t}, \ed^+=\ed_{\alpha,t}^+, \rho=\rho_{\alpha,t}$, $\Delta=\Delta_{\alpha,t}$ and compute
\[ m(\ed^++x+y)-m(\ed^++x) = \int_{-\Delta-\delta_\ast}^{\delta_\ast}\frac{\rho(\ed^++\omega)}{\omega-x-y}\diff\om - \int_{-\Delta-\delta_\ast}^{\delta_\ast}\frac{\rho(\ed^++\omega)}{\omega-x}\diff\om \]
where we now focus on the integration regime $\om\ge 0$ as this is the regime containing the two critical singularities. We first observe that
\[ 
\int_{-y}^{\delta_\ast-y}\frac{\rho(\ed^++\omega+y)}{\omega-x}\diff\om - \int_{0}^{\delta_\ast}\frac{\rho(\ed^++\omega)}{\omega-x}\diff\om = \int_{0}^{\delta_\ast}\frac{\rho(\ed^++\omega+y)-\rho(\ed^++\omega)}{\omega-x}\diff\om +\int_{-y}^0\frac{\rho(\ed^++\om+y)}{\om-x}\diff\om + \landauO{y},
  \]
where the second integral is easily bounded by 
\[ \int_{-y}^0\frac{\rho(\ed+\om+y)}{\om-x}\diff\om  \lesssim \frac{1}{x}\min\Big\{ y^{4/3},y^{3/2} \Delta^{-1/6}\Big\} \lesssim  \frac{y}{\rho(\ed^++x)(\rho(\ed^++x)+\Delta^{1/3})}.\]
We split the remaining integral into three regimes $[0,x/2]$, $[x/2,3x/2]$ and $[3x/2,\delta_\ast]$. In the first one we use~\eqref{rho' difference hash} as well as the scaling relation $\rho(\ed^++\omega)\sim\min\{\omega^{1/3},\omega^{1/2}\Delta^{-1/6}\}$ to obtain
\[ \begin{split}
&\int_{0}^{x/2}\frac{\rho(\ed^++\omega+y)-\rho(\ed^++\omega)}{\omega-x}\diff\om\lesssim \frac{y}{x}\int_0^{x/2} \frac{1}{\rho(\ed^++\om)\Big(\rho(\ed^++\om)+\Delta^{1/3}\Big)}\diff\om\\
&\quad\lesssim \frac{y}{x} \min\Big\{\frac{x^{1/2}}{\Delta^{1/6}},x^{1/3}\Big\}   \sim \frac{y}{\max\{x^{2/3},x^{1/2}\Delta^{1/6}\}} \lesssim \frac{y}{\rho(\ed^++x)(\rho(\ed^++x)+\Delta^{1/3})}.
\end{split}  \]
The integral in the regime $[3x/2,\delta_\ast]$ is completely analogous and contributes the same bound. Finally, we are left with the regime $[x/2,3x/2]$ which we again subdivide into $[x-y,x+y]$ and $[x/2,3x/2]\setminus[x-y,x+y]$. In the first of those we have
\[ \begin{split}
\int_{x-y}^{x+y} \frac{\rho(\ed^++\omega+y)-\rho(\ed^++\omega)}{\omega-x}\diff\om &= \int_{x-y}^{x+y} \frac{\rho(\ed^++\omega+y)-\rho(\ed^++x+y)-\rho(\ed^++\omega)+\rho(\ed^++x)}{\omega-x}\diff\om \\
&\lesssim \frac{y}{\rho(\ed^++x)(\rho(\ed^++x)+\Delta^{1/3})},
\end{split} \]
while in the second one we obtain 
\[ \begin{split}
&\int_{[x/2,3x/2]\setminus[x-y,x+y]} \frac{\rho(\ed^++\omega+y)-\rho(\ed^++x+y)-\rho(\ed^++\omega)+\rho(\ed^++x)}{\omega-x}\diff\om \\
&\lesssim \frac{y}{\rho(\ed^++x)(\rho(\ed^++x)+\Delta^{1/3})} \int_{[x/2,3x/2]\setminus[x-y,x+y]} \abs{\omega-x}^{-1}\diff\om \lesssim \frac{y\abs{\log y}}{\rho(\ed^++x)(\rho(\ed^++x)+\Delta^{1/3})}.
\end{split} \]
Collecting the various estimates completes the proof of~\eqref{weakm}.

The second bound in~\eqref{weakm min} follows by a similar argument and we focus on the most critical term 
\[ \int_{-\delta_\ast/2}^{\delta_\ast/2}\frac{\rho(\mi+\omega+y)-\rho(\mi+\omega)}{\om-x}\diff\omega = \bigg(\int_{-\delta_\ast/2}^{x-y}+\int_{x-y}^{x+y}+\int_{x+y}^{\delta_\ast/2}\bigg)\frac{\rho(\mi+\omega+y)-\rho(\mi+\omega)}{\om-x}\diff\omega.\]
Here we can bound the middle integral by 
\[ \begin{split}
\abs{\int_{x-y}^{x+y}\frac{\rho(\mi+\omega+y)-\rho(\mi+\omega)}{\om-x}\diff\omega }&=\abs{ \int_{x-y}^{x+y}\frac{\rho(\mi+\omega+y)-\rho(\mi+x+y)-\rho(\mi+\omega)+\rho(\mi+x)}{\om-x}\diff\omega} \\
&\lesssim \frac{\abs{y}}{\rho^2(\mi+x)},
\end{split}
\]
while for the first integral we have 
\[ \begin{split}
\abs{\int_{-\delta_\ast/2}^{x-y}\frac{\rho(\mi+\omega+y)-\rho(\mi+x+y)-\rho(\mi+\omega)+\rho(\mi+x)}{\om-x}\diff\omega} \lesssim \frac{\abs{y}}{\rho^2(\mi+x)} \int_{-\delta_\ast/2}^{x-y}\frac{1}{\abs{\omega-x}}\diff\omega\lesssim \frac{\abs{y}\abs{\log\abs{y}}}{\rho^2(\mi+x)}.
\end{split}
\]
The third integral is completely analogous, completing the proof of~\eqref{weakm min}. 
\end{proof}

\subsection{Quantiles}\label{sec:quantiles}
Finally we consider the locations of quantiles of $\rho_{r,t}$ for $r=\alpha,\lambda,\mu$ and their fluctuation scales. For $0\le t\le t_\ast$ we define the shifted quantiles $\wh\gamma_{r,i}(t)$, and for $t_* \le t\le 2t_\ast$ the shifted quantiles\footnote{We use a separate variable name $\widecheck\gamma$ because in Section~\ref{UNMIN} the name $\wh\gamma$ is used for the quantiles with respect to the base point $\wt\mi$ instead of $\mi$.} $\widecheck\gamma_{r,i}(t)$ in such a way that 
\begin{equation}\label{def:quan}
 \int_{0}^{\wh\gamma_{r,i}(t)}\rho_{r,t}(\ed_{r,t}^++\omega)\diff\omega=\frac{i}{N}, \qquad \int_{0}^{\widecheck\gamma_{r,i}(t)}\rho_{r,t}(\mi_{r,t}+\omega)\diff\omega=\frac{i}{N}, \qquad  \abs{i}\ll N .
\end{equation}
Notice that  for $i=0$ we always have $\wh \gamma_{r,0}(t) =\widecheck \gamma_{r,0}(t)=0$.
We will also need to define the semiquantiles, distinguished by star from the quantiles, defined as follows:
\begin{equation}\label{def:semiquan1}
 \int_{0}^{\wh\gamma_{r,i}^*(t)}\rho_{r,t}(\ed_{r,t}^++\omega)\diff\omega=\frac{i-\frac{1}{2}}{N}, 
 \qquad \int_{0}^{\widecheck\gamma_{r,i}^*(t)}\rho_{r,t}(\mi_{r,t}+\omega)\diff\omega=\frac{i-\frac{1}{2}}{N}, \qquad  1\le i \ll N
\end{equation}
and
\begin{equation}\label{def:semiquan2}
 \int_{0}^{\wh\gamma_{r,i}^*(t)}\rho_{r,t}(\ed_{r,t}^++\omega)\diff\omega=\frac{i+\frac{1}{2}}{N}, 
 \qquad \int_{0}^{\widecheck\gamma_{r,i}^*(t)}\rho_{r,t}(\mi_{r,t}+\omega)\diff\omega=\frac{i+\frac{1}{2}}{N}, \qquad  -N\ll i\le -1
\end{equation}
Note that the definition is slightly different for positive and negative $i$'s, in particular $\wh\gamma_i^* \in [\wh\gamma_{i-1}, \wh \gamma_i]$
for $i\ge 1$ and $\wh\gamma_i^* \in [\wh\gamma_{i}, \wh \gamma_{i+1}]$ for $i<0$. The semiquantiles are not defined for $i=0$.

\begin{lemma}\label{def:reg}
For $1\le \abs{i}\ll N$, $r=\alpha,\lambda,\mu$ and $0\le t\le t_\ast$ we have  
\begin{subequations}\label{gamma eqs}
\begin{equation}\label{gamma gap}
\begin{split}
\wh\gamma_{r,i}(t)&\sim  \sgn(i)\max\Bigl\{\Bigl(\frac{\abs{i}}{N}\Bigr)^{3/4},\Bigl(\frac{\abs{i}}{N}\Bigr)^{2/3}(t_\ast-t)^{1/6}\Bigr\}
-\begin{cases}
0,&i> 0\\
\Delta_{r,t}, & i<0
\end{cases}  \\
\wh\gamma_{r,i}(t) &= \wh\gamma_{\mu,i}(t) \Bigl[1+\landauO[2]{(t_\ast-t)^{1/3} + \min\Bigl\{\frac{\wh\gamma_{\mu,i}(t)^{1/2}}{(t_\ast-t)^{1/4}},\wh\gamma_{\mu,i}(t)^{1/3}\Bigr\}}\Bigr],
\end{split}
\end{equation}
while for $t_*\le t\le 2t_\ast$ we have
\begin{equation}\label{gamma min}
\begin{split}
\widecheck\gamma_{r,i}(t) &\sim \sgn(i)\min\Bigl\{\Bigl(\frac{\abs{i}}{N}\Bigr)^{3/4},\frac{\abs{i}}{N}(t_\ast-t)^{-1/2}\Bigr\},\\ 
\widecheck\gamma_{r,i}(t) &= \widecheck\gamma_{\mu,i}(t) \Bigl[1+\landauO[2]{(t_\ast-t)^{1/2} + \min\Bigl\{\frac{\widecheck\gamma_{\mu,i}(t)^{2}}{(t_\ast-t)^{11/4}},\frac{\widecheck\gamma_{\mu,i}(t)}{t_\ast-t},\widecheck\gamma_{\mu,i}(t)^{1/3}\Bigr\}}\Bigr].
\end{split}
\end{equation}
Moreover, the quantiles of $\rho_{\alpha,t}$ are the convex combination
\begin{equation}\label{quantiles convex comb}
\wh\gamma_{\alpha,i}(t) = \alpha\wh\gamma_{\lambda,i}(t) + (1-\alpha)\wh\gamma_{\mu,i}(t),\quad \widecheck\gamma_{\alpha,i}(t) = \alpha\widecheck\gamma_{\lambda,i}(t) + (1-\alpha)\widecheck\gamma_{\mu,i}(t).
\end{equation}
\end{subequations}
\end{lemma}
\begin{proof}
The proof follows directly from the estimates in~\eqref{rho counting fcts gap} and~\eqref{rho conting fcts min}. The relation~\eqref{quantiles convex comb} follows directly from~\eqref{interpolating phi} in the definition of the $\alpha$-interpolating density.
\end{proof}

\subsection{Movement of edges, quantiles and minima}
For the analysis of the Dyson Brownian motion it is necessary to have a precise understanding of the movement of the reference points $\ed_{r,t}^\pm$ and $\mi_{r,t}$, $r=\lambda,\mu$.
 For technical reasons it is slightly easier to work with an auxiliary quantity $\wt\mi_{r,t}$ which is very close to $\mi_{r,t}$. According to~\cite[Lemma 5.1]{1809.03971} the minimum $\mi_{r,t}$ can approximately be found by solving the implicit equation 
\begin{subequations}
\begin{equation}\label{eq wt mi}
\wt \mi_{r,t} = \cu_r - (t-t_\ast)\Re m_{r,t}(\wt\mi_{r,t}),\qquad \wt\mi_{r,t}\in\R,\quad r=\lambda,\mu.
\end{equation} 
The explicit relation~\eqref{eq wt mi} is the main reason why it is more convenient to study the movement of $\wt\mi_t$ rather than the one of $\mi_t$. We claim that $\wt\mi_{r,t}$ is indeed a very good approximation for $\mi_{r,t}$ in the sense that
\begin{equation}\label{eq wt mi mi}
\abs{\mi_{r,t}-\wt\mi_{r,t}}\lesssim (t-t_\ast)^{3/2+1/4},\quad \Im m_{r,t}(\wt\mi_{r,t})=\gamma^2(t-t_\ast)^{1/2}+  \landauO{t-t_\ast},
\quad r=\lambda,\mu.
\end{equation}
\begin{proof}[Proof of~\eqref{eq wt mi mi}] 
The first claim in~\eqref{eq wt mi mi} is a direct consequence of~\cite[Lemma 5.1]{1809.03971}. For the second claim we refer to~\cite[Eq.~(89a)]{1809.03971} which implies
\[\Im m_{r,t}(\wt\mi_{r,t}) = (t-t_\ast)^{1/2}\gamma^2 \Bigl[1+\landauO{(t-t_\ast)^{1/3}[\Im m_{r,t}(\wt\mi_{r,t})]^{1/3}}\Bigr]=\gamma^2(t-t_\ast)^{1/2}+  \landauO{t-t_\ast}.\qedhere\]
\end{proof}
For the $t$-derivative of (semi-)quantiles \(\gamma_{r,t}\), i.e.~points such that \(\int_{-\infty}^{\gamma_{r,t}}\rho_{r,t}(x)\diff x\) is constant in \(t\), as well as for the minima $\wt\mi_{r,t}$ we have the explicit relations
\begin{equation}\label{diff m quant}
  \frac{\diff}{\diff t}\gamma_{r,t} = -\Re m_{r,t}(\gamma_{r,t}),
\end{equation}
\begin{equation}\label{diff m m}
\frac{\diff}{\diff t}\wt\mi_{r,t}= -\Re m_{r,t}(\wt\mi_{r,t}) + \landauO{t-t_\ast},\quad t_*\le t\le 2t_\ast.
\end{equation}
In particular, for the spectral edges it follows from~\eqref{diff m quant} that
\begin{equation}\label{diff e m}
  \frac{\diff}{\diff t}\ed_{r,t}^+ = -m_{r,t}(\ed_{r,t}^+), \quad 0\le t\le t_\ast.
  \end{equation}  
\end{subequations}
\begin{proof}[Proof of~\eqref{diff m quant}--\eqref{diff e m}] 
 For the proof of~\eqref{diff m quant} we first recall  that from the defining equation~\eqref{eq free sc conv def}
  of the semicircular flow it follows that 
  the Stieltjes transform $m= m_{t}(\zeta)$  of $\rho_{t}$ satisfies the Burgers equation
  \begin{equation}\label{burgers}
     \dot{m} = mm' =\frac{1}{2} (m^2)',
  \end{equation}
  where prime denotes the $\frac{\diff}{\diff \zeta}$ derivative and dot denotes the \(\frac{\diff}{\diff t}\) derivative. Thus   \[
    \dot{\gamma}_{r,t}= -\frac{1}{\rho_{r,t} (\gamma_{r,t})} \Im \int_{-\infty}^{\gamma_{r,t}}   \dot{m}_{r,t}(E) \diff E
    =-\frac{1}{2\rho_{r,t} (\gamma_{r,t})} \Im \int_{-\infty}^{\gamma_{r,t}}   (m_{r,t}^2)'(E) \diff E = - \frac{\Im m_{r,t}^2(\gamma_{r,t})}{2\Im m_{r,t}(\gamma_{r,t})} = -\Re m_{r,t}(\gamma_{r,t}).
  \] follows directly from differentiating \(\int_{-\infty}^{\gamma_{r,t}}\rho_{r,t}(x)\diff x\equiv\mathrm{const}\).

For~\eqref{diff m m} we begin by computing the integral 
\begin{equation}\label{m prime cusp comp} m_{r,t_\ast}'( \cu_r + \ii\eta ) = \int_{\R} \frac{\rho_{t_\ast}(\cu_r+ x)}{(x-\ii \eta)^2}\diff x = \int_\R\frac{\sqrt 3\gamma^{4/3}\abs{x}^{1/3}+\landauO[1]{\abs{x}^{2/3}}}{2\pi (x-\ii\eta)^2}\diff x = \frac{\gamma^{4/3}}{3 \eta^{2/3}}+\landauO{\eta^{-1/3}},\end{equation}
so that by definition $m_{r,t}(z)=m_{r,t_\ast}(z+(t-t_\ast) m_{r,t}(z))$ of the free semicircular flow, 
\[
\begin{split}
 \frac{\diff}{\diff t} m_{r,t}(\wt\mi_{r,t}) &= m_{r,t_\ast}'(\wt\mi_{r,t}+(t-t_\ast) m_{r,t}(\wt\mi_{r,t})) \Bigl[ \frac{\diff}{\diff t} \wt\mi_{r,t} + m_{r,t}(\wt\mi_{r,t})+(t-t_\ast)\frac{\diff}{\diff t}m_{r,t}(\wt\mi_{z,t})  \Bigr]\\
 &=\Bigl( \frac{1}{3 (t-t_\ast)}+\landauO{(t-t_\ast)^{-1/2}}\Bigr) \Bigl[ \frac{\diff}{\diff t} \wt\mi_{r,t} + m_{r,t}(\wt\mi_{r,t})+(t-t_\ast)\frac{\diff}{\diff t}m_{r,t}(\wt\mi_{r,t})  \Bigr]\\
 &=\ii\Bigl( \frac{1}{3 (t-t_\ast)}+\landauO{(t-t_\ast)^{-1/2}}\Bigr) \Bigl[  \Im m_{r,t}(\wt\mi_{r,t})+(t-t_\ast)\frac{\diff}{\diff t}\Im m_{r,t}(\wt\mi_{r,t})  \Bigr]\\
 &=\Bigl(\ii \frac{\gamma^2}{3(t-t_\ast)^{1/2}}+\frac{\ii}{3}\frac{\diff}{\diff t}\Im m_{r,t}(\wt\mi_{r,t})\Bigr)\Bigl[1+\landauO{(t-t_\ast)^{1/2}}\Bigr]. 
\end{split}
\]
Here we used~\eqref{eq wt mi},~\eqref{eq wt mi mi} together with~\eqref{m prime cusp comp} in the second step. The third step follows from taking the $t$-derivative of~\eqref{eq wt mi}. The ultimate inequality is again a consequence of~\eqref{eq wt mi mi}. 
By considering real and imaginary part separately it thus follows that 
\[ \frac{\diff}{\diff t} \Im m_{r,t}(\wt\mi_{r,t}) = \frac{\gamma^2}{2(t-t_\ast)^{1/2}}\Bigl[1+\landauO{(t-t_\ast)^{1/2}}\Bigr],\qquad \frac{\diff}{\diff t} \Re m_{r,t}(\wt\mi_{r,t}) = \landauO{1} \]
and therefore~\eqref{diff m m} follows by differentiating~\eqref{eq wt mi}.
\end{proof}

\subsection{Rigidity scales}
In this section we compute, up to leading order, the fluctuations of the eigenvalues around their classical locations, i.e. the quantiles defined in Section~\ref{sec:quantiles}. Indeed, the computation of the fluctuation scale for the particles $x_i(t)$, $y_i(t)$, defined in~\eqref{xxx},~\eqref{yyy}, will be one of the fundamental inputs
 to prove rigidity for the interpolated process in Section~\ref{sec:rigid}.
The fluctuation scale $\eta_\mathrm{f}^{\rho}(\tau)$ of any density function  $\rho(\omega)$ around $\tau$ is defined via 
\[ \int_{\tau-\eta_\mathrm{f}^{\rho}(\tau)}^{\tau+\eta_\mathrm{f}^{\rho}(\tau)} \rho(\omega)\diff\omega=\frac{1}{N}\]
for $\tau\in\supp\rho$ and by the value $\eta_\mathrm{f}(\tau)\defeq\eta_\mathrm{f}(\tau')$ where  $\tau'\in\supp\rho$ is the edge closest to $\tau$ for $\tau\not\in\supp\rho$. If this edge is not unique, an arbitrary choice can be made between the two possibilities. From~\eqref{gamma gap} we immediately obtain for $0\le t\le t_\ast$ and $1\le i\le N$, that
\begin{subequations}\label{eqs fluc scale} 
\begin{equation}\label{fluc scale gap}
 \eta_\mathrm{f}^{\rho_{r,t}}(\ed_{r,t}^++\wh\gamma_{r,\pm i}(t)) \sim \max\Bigl\{ \frac{\Delta_{r,t}^{1/9}}{N^{2/3}i^{1/3}},\frac{1}{N^{3/4}i^{1/4}} \Bigr\}\sim \max\Bigl\{ \frac{(t_\ast-t)^{1/6}}{N^{2/3}i^{1/3}},\frac{1}{N^{3/4}i^{1/4}} \Bigr\}, \quad r=\alpha,\lambda,\mu, 
\end{equation}
while for $t_*\le t\le 2t_\ast$, $1\le \abs{i}\ll N$ we obtain from~\eqref{gamma min} that
\begin{equation}\label{fluc scale min}
 \eta_\mathrm{f}^{\rho_{r,t}}(\mi_{r,t}+\widecheck\gamma_{r,i}(t))\sim \min\Bigl\{\frac{1}{N\rho_{r,t}(\mi_{r,t})},\frac{1}{N^{3/4}\abs{i}^{1/4}}\Bigr\}\sim \min\Bigl\{\frac{1}{N(t-t_\ast)^{1/2}},\frac{1}{N^{3/4}\abs{i}^{1/4}} \Bigr\}, \qquad r=\alpha,\lambda,\mu.
\end{equation}
\end{subequations}
In the second relations we used~\eqref{eq Delta size} and~\eqref{eq min size}.
For reference purposes we also list for $0<i,j \ll N$ the bounds 
\begin{equation}\label{gamma difference gap}
\abs{\wh\gamma_{r, i}(t)-\wh\gamma_{r, j}(t)} \sim \max\Big\{\frac{\Delta_{r,t}^{1/9}\abs{i-j}}{N^{2/3}(i+j)^{1/3}},\frac{\abs{i-j}}{N^{3/4}(i+j)^{1/4}} \Big\}\,, \qquad 
\end{equation}
in case $t \le t_*$ and 
\begin{equation}
\label{gamma difference min}
\abs{\widecheck\gamma_{r, i}(t)-\widecheck\gamma_{r, j}(t)} \sim \min\Big\{\frac{\abs{i-j}}{\rho_{r,t}(\mi_{r,t})N},\frac{\abs{i-j}}{N^{3/4}(i+j)^{1/4}} \Big\}
\end{equation}
in case $t > t_*$. Furthermore we have 
\begin{equation}\label{rhoquantile edge}
\rho_{r,t}(\ed_{r,t}^++\wh\gamma_{r, i}(t)) \sim  \min\Big\{\frac{i^{1/3}}{N^{1/3}(t_*-t)^{1/6}}, \frac{i^{1/4}}{N^{1/4}} \Big\}
\end{equation}
and
\begin{equation}\label{rhoquantile m}
\rho_{r,t}(\mi_{r,t}+\widecheck\gamma_{r, i}(t)) \sim  \max\Big\{ \rho_{r,t}(\mi_{r,t}), \frac{i^{1/4}}{N^{1/4}}\Big\}\,.
\end{equation}

\subsection{Stieltjes transform bounds}
It follows from~\eqref{rho rho gap} and~\eqref{rho rho min} that also the real parts of the Stieltjes transforms $m_{\alpha,t}$, $m_{\lambda,t}$, $m_{\mu,t}$ are close. We claim that for $r=\lambda,\alpha$, $\nu\in[-\delta_\ast,\delta_\ast]$ and $0\le t\le t_\ast$ we have
\begin{subequations}\label{Re m}
\begin{equation} 
\begin{split}\label{Re m gap}
&\abs{\Re \Bigl[\Bigl(m_{r,t}(\ed_{r,t}^++\nu)-m_{r,t}(\ed_{r,t}^+)\Bigr)-\Bigl(m_{\mu,t}(\ed_{\mu,t}^++\nu)-m_{\mu,t}(\ed_{\mu,t}^+)\Bigr)\Bigr]}
\\
&\qquad\qquad\lesssim \abs{\nu}^{1/3}\Bigl[\abs{\nu}^{1/3}+(t_\ast-t)^{1/3}\Bigr]\abs{\log\abs{\nu}}+(t_\ast-t)^{11/18}\bm{1}(\nu \le -\Delta_{\mu,t}/2),
\end{split}
 \end{equation}
while for $t_\ast\le t\le 2t_\ast$ we have
\begin{equation}\label{Re m min}
 \begin{split}
&\abs{\Re \Bigl[\Bigl(m_{r,t}(\mi_{r,t}+\nu)-m_{r,t}(\mi_{r,t})\Bigr)-\Bigl(m_{\mu,t}(\mi_{\mu,t}+\nu)-m_{\mu,t}(\mi_{\mu,t})\Bigr)\Bigr]}\\
&\qquad\qquad\lesssim \Bigl[\abs{\nu}^{1/3}(t-t_\ast)^{1/4}+(t_\ast-t)^{3/4}+\abs{\nu}^{2/3}\Bigr]\abs{\log\abs{\nu}}.
\end{split}  \end{equation}
\end{subequations}
\begin{proof}[Proof of~\eqref{Re m}]
We first recall from Lemma~\ref{lemma holderhsh} that also the density $\rho_{\alpha,t}$ is $1/3$-H\"older continuous which we will use repeatedly in the following proof. We begin with the proof of~\eqref{Re m gap} and compute for $r=\alpha,\lambda,\mu$
\begin{equation}\label{eq Rem}
\begin{split}
\Re \Bigl[m_{r,t}(\ed_{r,t}^++\nu)-m_{r,t}(\ed_{r,t}^+)\Bigr] &= \int_0^\infty  \frac{\nu \rho_{r,t}(\ed_{r,t}^++\omega)}{(\omega-\nu)\omega}\diff \omega + \int^{\infty}_{0} \frac{\nu \rho_{r,t}(\ed_{r,t}^--\omega)}{(\omega+\Delta_{r,t}+\nu)(\omega+\Delta_{r,t})}\diff \omega.
\end{split}
\end{equation}
For $\nu> 0$ the first of the two terms is the more critical one. Our goal is to obtain a bound on
\[ \int_0^\infty  \frac{\nu}{(\omega-\nu)\omega}\Bigl[\rho_{\lambda,t}(\ed_{\lambda,t}^++\omega)-\rho_{\mu,t}(\ed_{\mu,t}^++\omega)\Bigr]\diff \omega \]
by using~\eqref{rho rho gap}. Let $0<\epsilon<\nu/2$ be a small parameter for which we separately consider the two critical regimes $0\le \omega\le \epsilon$ and $\abs{\nu-\omega}\le\epsilon$. We use 
\begin{equation}\label{rho holder}\rho_{r,t}(\ed_{r,t}^++\omega)\lesssim \omega^{1/3}\quad\text{and}\quad\rho_{r,t}(\ed_{r,t}^++\omega)=\rho_{r,t}(\ed_{r,t}^++ \nu )+\landauO{\abs{\omega-\nu}^{1/3}}, \quad r=\lambda,\mu,\end{equation}
from the $1/3$-H\"older continuity of $\rho_{r,t}$ and the fact that the integral over $1/(\omega-\nu)$ from $\nu-\epsilon$ to $\nu+\epsilon$ vanishes by symmetry to estimate, for $r=\lambda,\mu$, 
\[\abs{ \int_0^\epsilon \frac{\nu}{(\omega-\nu)\omega}\rho_{r,t}(\ed_{r,t}^++\omega)\diff \omega} \lesssim \int_{0}^\epsilon \abs{\omega}^{-2/3}\diff \omega \lesssim \epsilon^{1/3} \]
and 
\[ \abs{ \int_{\nu-\epsilon}^{\nu+\epsilon} \biggl[\frac{\rho_{r,t}(\ed_{r,t}^++\omega)}{\omega-\nu}-\frac{\rho_{r,t}(\ed_{r,t}^++\omega)}{\omega}\biggr]\diff \omega} \lesssim \int_{\nu-\epsilon}^{\nu+\epsilon} \abs{\omega-\nu}^{-2/3}\diff \omega + \epsilon \nu^{-2/3}\lesssim \epsilon^{1/3}+\epsilon \nu^{-2/3}.\]
Next, we consider the remaining integration regimes where we use~\eqref{rho rho gap} and~\eqref{rho holder} to estimate 
\[ 
\begin{split}
&\abs{\int_{\epsilon}^{\nu-\epsilon} \frac{\nu}{(\omega-\nu)\omega}\Bigl[\rho_{r,t}(\ed_{r,t}^++\omega)-\rho_{\mu,t}(\ed_{\mu,t}^++\omega)\Bigr]\diff \omega}\\
&\quad\lesssim \int_{\epsilon}^{\nu/2}\frac{\omega^{1/3}(t_\ast-t)^{1/3}+\omega^{2/3}}{\omega}\diff\omega +\int_{\nu/2}^{\nu-\epsilon}\Bigl(\frac{\nu^{1/3}(t_\ast-t)^{1/3}}{\omega-\nu}+\frac{\nu^{2/3}}{\omega-\nu}\Bigr)\diff\omega\lesssim \nu^{1/3}\Bigl( (t_\ast-t)^{1/3}+ \nu^{1/3}\Bigr)\abs{\log\epsilon}
\end{split}
  \]
and similarly 
\[ \abs{\int_{\nu+\epsilon}^{\infty} \frac{\nu}{(\omega-\nu)\omega}\Bigl[\rho_{r,t}(\ed_{r,t}^++\omega)-\rho_{\mu,t}(\ed_{\mu,t}^++\omega)\Bigr]\diff \omega} \lesssim \nu^{1/3}\Bigl( (t_\ast-t)^{1/3}+ \nu^{1/3}\Bigr)\abs{\log\epsilon}.\]
We now consider the difference of the first terms in~\eqref{eq Rem} for $r=\lambda,\mu$ and for $\nu<0$ where the bound is simpler because the integration regime close to $\nu$ does not have to be singled out. Using~\eqref{rho rho gap} we find 
\[ \abs{\int_{0}^{\infty} \frac{\nu}{(\omega-\nu)\omega}\Bigl[\rho_{r,t}(\ed_{r,t}^++\omega)-\rho_{\mu,t}(\ed_{\mu,t}^++\omega)\Bigr]\diff \omega} \lesssim \abs{\nu}^{2/3} +(t_\ast-t)^{1/3}\abs{\nu}^{1/3}.\]

Finally, it remains to consider the difference of the second terms in~\eqref{eq Rem}. We first treat the regime where  $\nu \ge -\frac{3}{4}\Delta_{r,t}$ and  split the difference into the sum of two terms
\[ \begin{split}
&\abs{\int^{\infty}_{0} \biggl(\frac{\nu \rho_{r,t}(\ed_{r,t}^--\omega)}{(\omega+\Delta_{r,t}+\nu)(\omega+\Delta_{r,t})}-\frac{\nu \rho_{r,t}(\ed_{r,t}^--\omega)}{(\omega+\Delta_{\mu,t}+\nu)(\omega+\Delta_{\mu,t})}\biggr)\diff \omega}\\
&\,\,\le \abs{\nu} \abs{\Delta_{r,t}-\Delta_{\mu,t}} \int_0^\infty  \frac{\rho_{r,t}(\ed_{r,t}^--\omega)\bigl[2\Delta_{r,t}+2\omega+{\abs{\nu}}\bigr]}{(\omega+\Delta_{r,t}+\nu)^2(\omega+\Delta_{r,t})^2}\diff\omega \lesssim  \frac{\abs{\Delta_{r,t}-\Delta_{\mu,t}}}{\Delta_{r,t}^{2/3}}-\frac{\abs{\Delta_{r,t}-\Delta_{\mu,t}}}{(\Delta_{r,t}+{\abs{\nu}})^{2/3}}\lesssim (t_\ast-t)^{1/3} \abs{\nu}^{1/3}
\end{split}  \]
and 
\[ \abs{\int^{\infty}_{0} \biggl(\frac{\nu \rho_{r,t}(\ed_{r,t}^--\omega)}{(\omega+\Delta_{\mu,t}+\nu)(\omega+\Delta_{\mu,t})}-\frac{\nu \rho_{\mu,t}(\ed_{\mu,t}^--\omega)}{(\omega+\Delta_{\mu,t}+\nu)(\omega+\Delta_{\mu,t})}\biggr)\diff \omega}\lesssim \abs{\nu}^{2/3}+(t_\ast-t)^{1/3}\abs{\nu}^{1/3}.\]
Here we used $\rho_{r,t}(\ed_{r,t}^--\omega)\lesssim \omega^{1/3}$ as well as~\eqref{eq Delta size} for the first and~\eqref{eq Delta size},\eqref{rho rho gap} for the second computation. By collecting the various error terms and choosing  $\epsilon=\nu^2$ we conclude~\eqref{Re m gap}.

We define $\kappa\defeq -\nu-\Delta_{r,t}$. Then we are left with the regime $\nu < -\frac{3}{4}\Delta_{r,t}$ or equivalently $\kappa> -\frac{1}{4}\Delta_{r,t}$ and use
\[
m_{r,t}(\ed_{r,t}^++\nu)-m_{r,t}(\ed_{r,t}^+)= (m_{r,t}(\ed_{r,t}^--\kappa)-m_{r,t}(\ed_{r,t}^-)) + (m_{r,t}(\ed_{r,t}^-)-m_{r,t}(\ed_{r,t}^+))\,,
\]
as well as
\begin{equation}
\begin{split}
m_{\mu,t}(\ed_{\mu,t}^++\nu)-m_{\mu,t}(\ed_{\mu,t}^+) = &(m_{\mu,t}(\ed_{\mu,t}^--\kappa+\Delta_{\mu,t}-\Delta_{r,t})-m_{\mu,t}(\ed_{\mu,t}^--\kappa))+(m_{\mu,t}(\ed_{\mu,t}^--\kappa)-m_{\mu,t}(\ed_{\mu,t}^-))
\\
&+(m_{\mu,t}(\ed_{\mu,t}^-)-m_{\mu,t}(\ed_{\mu,t}^+))
\end{split}
\end{equation}
in the left hand side of~\eqref{Re m gap}. Thus we have to estimate the three expressions,
\begin{subequations}
\begin{align} \label{first term at e-}
&\abs{\Re \Bigl[\Bigl(m_{r,t}(\ed_{r,t}^--\kappa)-m_{r,t}(\ed_{r,t}^-)\Bigr)-\Bigl(m_{\mu,t}(\ed_{\mu,t}^--\kappa)-m_{\mu,t}(\ed_{\mu,t}^-)\Bigr)\Bigr]},
\\ \label{second term at e-}
&\abs{\Re \Bigl[\Bigl(m_{r,t}(\ed_{r,t}^-)-m_{r,t}(\ed_{r,t}^+)\Bigr)-\Bigl(m_{\mu,t}(\ed_{\mu,t}^-)-m_{\mu,t}(\ed_{\mu,t}^+)\Bigr)\Bigr]},
\\ \label{third term at e-}
&\abs{\Re \Bigl[m_{\mu,t}(\ed_{\mu,t}^--\kappa+\Delta_{\mu,t}-\Delta_{r,t})-m_{\mu,t}(\ed_{\mu,t}^--\kappa)\Bigr]}.
\end{align}
\end{subequations}
In order to bound the first term  we use that estimating~\eqref{first term at e-} for $\kappa \ge -\frac{3}{4}\Delta_{r,t}$ is equivalent to estimating the left hand side of~\eqref{Re m gap} for $\nu \ge -\frac{3}{4}\Delta_{r,t}$, i.e.~the regime we already considered above. This equivalence follows by using the reflection $A\to -A$ of the expectation (cf.~\eqref{MDE matrix form}) that turns every left edge $\ed_{z,t}^+$ into a right edge $\ed_{z,t}^-$. In particular, by the analysis that we already performed~\eqref{first term at e-} is bounded by 
 $\abs{\kappa}^{1/3}[\abs{\kappa}^{1/3}+(t_\ast-t)^{1/3}]\abs{\log\abs{\kappa}}$. Since $\abs{\kappa}\le \abs{\nu}$ this is the desired bound. 

For the  second term~\eqref{second term at e-} we see from~\eqref{eq Rem} that we have to estimate the difference  between the expressions
\begin{equation} \label{difference of m at edges}
\int_0^\infty  \frac{\Delta_{r,t} \rho_{r,t}(\ed_{r,t}^++\omega)}{\omega(\omega+\Delta_{r,t})}\diff \omega + \int^{\infty}_{0} \frac{\Delta_{r,t} \rho_{r,t}(\ed_{r,t}^--\omega)}{\omega(\omega+\Delta_{r,t})}\diff \omega,
\end{equation}
for $r=\alpha,\lambda,\mu$. The summands in~\eqref{difference of m at edges} are treated analogously, so we focus on the first summand. We split the integrand of the difference between the first summands and estimate
\[
 \frac{(\Delta_{r,t}-\Delta_{\mu,t})\rho_{r,t}(\ed_{r,t}^++\omega)}{(\omega+\Delta_{r,t})(\omega+\Delta_{\mu,t})}
+
\frac{\Delta_{\mu,t}}{\omega(\omega+\Delta_{\mu,t})}\big(  \rho_{r,t}(\ed_{r,t}^++\omega)- \rho_{\mu,t}(\ed_{\mu,t}^++\omega)\big)
\lesssim\frac{\Delta(\omega^{1/3}+(t_\ast-t)^{1/3}) }{\omega^{2/3}(\omega+\Delta)} 
\]
where $\Delta\defeq  \Delta_{r,t} \sim \Delta_{\mu,t}$ and we used~\eqref{eq Delta size},~\eqref{rho rho gap}  and the first inequality of~\eqref{rho holder}. Thus 
\[
\abs{\int_0^\infty  \frac{\Delta_{r,t} \rho_{r,t}(\ed_{r,t}^++\omega)}{\omega(\omega+\Delta_{r,t})}\diff \omega -\int_0^\infty  \frac{\Delta_{\mu,t} \rho_{\mu,t}(\ed_{\mu,t}^++\omega)}{\omega(\omega+\Delta_{\mu,t})}\diff \omega } \lesssim \Delta^{2/3} + \Delta^{1/3}(t_\ast-t)^{1/3}.
\]
Since $\abs{\nu} \gtrsim \Delta$ this finishes the estimate on~\eqref{second term at e-}. 

For~\eqref{third term at e-} we use the $1/3$-H\"older regularity of $m_{\mu,t}$ and~\eqref{eq Delta size} to get an upper bound $\Delta^{1/3}(t_\ast-t)^{1/9}\lesssim (t_\ast-t)^{11/18}$. This finishes the proof of~\eqref{Re m gap}.

We now turn to the case of a small local minimum in~\eqref{Re m min} and compute for $r=\alpha,\lambda,\mu$ and $\nu\ne 0$ that
\[ \Re \Bigl[m_{r,t}(\mi_{r,t}+\nu)-m_{r,t}(\mi_{r,t})\Bigr] = \int_\R \frac{\nu \rho_{r,t}(\mi_{r,t}+\omega)}{(\omega-\nu)\omega}\diff\omega.\] 
Without loss of generality, we consider the case $\nu>0$ as $\nu<0$ is completely analogous. As before, we first pick a threshold $\epsilon\le\nu/2$ and single out the integration over $[-\epsilon,\epsilon]$ and $[\nu-\epsilon,\nu+\epsilon]$. From the $1/3$-H\"older continuity of $\rho_{r,t}$ we have, for $r=\lambda,\mu$, 
\[ \rho_{r,t}(\mi_{r,t}+\omega) = \rho_{r,t}(\mi_{r,t}+\nu)+ \landauO{\abs{\nu-\omega}^{1/3}}\]
and therefore
\[ \abs{\int_{-\epsilon}^\epsilon \frac{\rho_{r,t}(\mi_{r,t}+\omega)}{\omega-\nu}\diff\omega} \lesssim \frac{\epsilon}{\nu},\qquad \abs{\int_{-\epsilon}^\epsilon \frac{\rho_{r,t}(\mi_{r,t}+\omega)}{\omega}\diff\omega} \lesssim \int_{-\epsilon}^\epsilon\abs{\omega}^{-2/3}\diff\omega\lesssim\epsilon^{1/3} \]
and 
\[ \abs{\int_{\nu-\epsilon}^{\nu+\epsilon} \frac{\rho_{r,t}(\mi_{r,t}+\omega)}{\omega-\nu}\diff\omega} \lesssim \int_{\nu-\epsilon}^{\nu+\epsilon}\abs{\omega-\nu}^{-2/3}\diff\omega\lesssim\epsilon^{1/3},\qquad \abs{\int_{\nu-\epsilon}^{\nu+\epsilon} \frac{\rho_{r,t}(\mi_{r,t}+\omega)}{\omega}\diff\omega} \lesssim \frac{\epsilon}{\nu}.\]
We now consider the difference between $\rho_{r,t}$ and $\rho_{\mu,t}$ for which we have 
\[\abs{\rho_{r,t}(\mi_{r,t}+\omega)-\rho_{\mu,t}(\mi_{\mu,t}+\omega)} \lesssim (t-t_\ast) \abs{\omega}^{1/3}(t-t_\ast)^{1/4}+(t-t_\ast)^{3/4} + \abs{\omega}^{2/3} \]
from~\eqref{rho rho min},~\eqref{eq min size} and the $1/3$-H\"older continuity of $\rho_{r,t}$. Thus we can estimate
\[ \begin{split}&
 \abs[3]{\biggl[\int_{-\infty}^{-\epsilon}+\int_{\epsilon}^{\nu-\epsilon}+\int_{\nu+\epsilon}^{\infty}\biggr]\frac{\nu \bigl(\rho_{\lambda,t}(\mi_{r,t}+\omega)-\rho_{r,t}(\mi_{r,t}+\omega)\bigr)}{(\omega-\nu)\omega}\diff\omega}\\
&\quad \lesssim \biggl[\int_{-\infty}^{-\epsilon}+\int_{\epsilon}^{\nu-\epsilon}+\int_{\nu+\epsilon}^{\infty}\biggr]\frac{\nu \bigl(\abs{\omega}^{1/3}(t-t_\ast)^{1/4}+(t-t_\ast)^{3/4} + \abs{\omega}^{2/3}\bigr)}{\abs{\omega-\nu}\omega}\diff\omega\\
&\quad\lesssim \abs{\log\epsilon}\Bigl[ \nu^{1/3}(t-t_\ast)^{1/4} + (t-t_\ast)^{3/4}+\nu^{2/3} \Bigr].
\end{split} \]
We again choose $\epsilon=\nu^2$ and by collecting the various error estimates can conclude~\eqref{Re m min}.
\end{proof}
 
\section{Index matching for two DBM}\label{sec:padding}
For two real symmetric matrix valued standard (GOE) Brownian motions $\mathfrak{B}_t^{(\lambda)},\mathfrak{B}_t^{(\mu)}\in\mathbb{R}^{N\times N}$ we define the matrix flows
 \begin{equation}
\label{defflow}
H_t^{(\lambda)}\defeq H^{(\lambda)}+\mathfrak{B}^{(\lambda)}_t, \,\,\,\,\,H_t^{(\mu)}\defeq H^{(\mu)}+\mathfrak{B}^{(\mu)}_t.
\end{equation} In particular, by~\eqref{defflow} it follows that \begin{equation}
\label{flow}
H^{(\lambda)}_t \stackrel{d}{=} H^{(\lambda)}+\sqrt{t}U^{(\lambda)},\,\,\,\,\,H^{(\mu)}_t \stackrel{d}{=} H^{(\mu)}+\sqrt{t}U^{(\mu)},
\end{equation} for any fixed $0\le t\le t_1$, where $U^{(\lambda)}$ and $U^{(\mu)}$ are GOE matrices.
 In~\eqref{flow} with $X \stackrel{d}{=}Y$ we denote that the two random variables $X$ and $Y$ are equal in distribution.
   
We will prove Proposition~\ref{DBM prop} by comparing the two Dyson Brownian motions 
for the eigenvalues of the matrices $H^{(\lambda)}_t$ and $H^{(\mu)}_t$ for $0\le t\le t_1$, see~\eqref{lambda}--\eqref{mu} below.
To do this, we will use the coupling idea of~\cite{MR3606475} and~\cite{MR3541852}, where the 
DBMs for the eigenvalues of $H^{(\lambda)}_t$ and $H^{(\mu)}_t$ are coupled in such a way that the difference of 
the two DBMs obeys a discrete parabolic  equation with good decay properties. 
 In order to analyse this equation we consider a \emph{short range approximation} for the DBM, first introduced in~\cite{MR3372074}. 
 Coupling only the short range approximation of the DBMs leads to a parabolic equation whose heat kernel has a rapid off diagonal decay
 by  \emph{finite speed of propagation} estimates.   In this way the kernels of both DBMs are locally determined  
 and thus can be directly compared by optimal rigidity since locally the two densities, hence their quantiles,  are close. 
 Technically it is much easier to work with a one parameter interpolation between the two DBM's and 
 consider its derivative with respect to the parameter, as introduced in~\cite{MR3606475}; the proof of the finite speed propagation for this dynamics
 does not require to establish level repulsion unlike in several previous works~\cite{MR3729630, MR3372074, MR3687212}.
 However, it requires to establish (almost) optimal rigidity for the interpolating dynamics as well. Note that 
 optimal rigidity is known for $H^{(\lambda)}_t$ and $H^{(\mu)}_t$ from~\cite{1809.03971},  see Lemma~\ref{diffx} later, 
 but not for the interpolation. For a complete picture, we mention that in  the works~\cite{MR3729630, MR3372074, MR3687212} on  \emph{bulk gap universality}, beyond 
 heat kernel and Sobolev estimates,  a version of De Giorgi-Nash-Moser parabolic regularity estimate, which used level repulsion in a more substantial way than finite speed of propagation, was also  necessary. \emph{Fixed energy universality} in the bulk 
 can be proven via homogenisation without De Giorgi-Nash-Moser estimates, hence level repulsion can
 also be avoided~\cite{MR3914908}. In a certain sense, the situation at the edge/cusp is easier  than the bulk regime since 
 relatively simple heat kernel bounds are
 sufficient for local relaxation to equilibrium. In another sense, due to singularities in the density,
 the edge and especially the cusp regime is more difficult.

 In Section~\ref{sec:rigid} we will establish rigidity for the interpolating process by DBM methods. Armed with this rigidity, 
 in Section~\ref{DBMS} we prove Proposition~\ref{DBM prop}  
 for the small gap and the exact cusp case, i.e.~$t_1\le t_*$. Some estimates
 are slightly different for the small minimum case, i.e.~$t_*\le t_1\le 2t_*$, the modifications are given in Section~\ref{UNMIN}.
 We recall that $t_*$ is the time at which both $H_{t_*}^{(\lambda)}$ and $H_{t_*}^{(\mu)}$ have an exact cusp.
Some technical details  on the corresponding Sobolev inequality and heat kernel estimates as well as finite speed of propagation
and short range approximation 
  are deferred to the Appendix: these are similar to the corresponding estimates for the edge
   case, see~\cite{MR3253704} and~\cite{1712.03881}, respectively.

In the rest of this section we prepare  the proof of Proposition~\ref{DBM prop} by setting up
 the appropriate framework.  While we are  interested
only in the eigenvalues near the physical cusp, the DBM is highly non-local, so
we need to define the dynamics for all eigenvalues.  
In the setup of Proposition~\ref{DBM prop} we could easily assume that the cusps for the two matrix flows are formed at  the same time
and their  slope parameters coincide -- these could be achieved by a rescaling and 
a trivial time shift. However, the number of eigenvalues  to the left of the cusp
may macroscopically differ for the two ensembles which would  mean that the 
labels of the ordered eigenvalues near the cusp would  not be constant along the interpolation.
To resolve this discrepancy, we will pad the system with $N$ fictitious particles in addition to the original flow of $N$ eigenvalues
similarly as in~\cite{MR3914908}, giving sufficient freedom to match the labels
of the  eigenvalues near the cusp.  These artificial  particles will be  placed
very far from the cusp regime and from each other so that their effect on the dynamics
of the relevant particles is negligible.

With the notation of Section~\ref{sec scflow},
we let $\rho_{\lambda,t}$, $\rho_{\mu,t}$ denote the (self-consistent) densities at time $0\le t\le t_1$ of $H_t^{(\lambda)}$ and $H_t^{(\mu)}$, respectively. 
In particular, $\rho_{\lambda,0}=\rho_\lambda$ and $\rho_{\mu,0}=\rho_\mu$, where $\rho_\lambda$, $\rho_\mu$ are the self consistent densities
 of $H^{(\lambda)}$ and $H^{(\mu)}$ and $\rho_{\lambda,t}$, $\rho_{\mu,t}$ are their semicircular evolutions.
 For each $0\le t\le t_*$ both densities $\rho_{\lambda,t}$, $\rho_{\mu,t}$
  have a small gap, denoted by $[\ed^-_{\lambda, t}, \ed^+_{\lambda, t}]$
  and $[\ed^-_{\mu, t}, \ed^+_{\mu, t}]$
   and we let 
  \[
  \Delta_{\lambda, t}\defeq \ed^+_{\lambda, t}- \ed^-_{\lambda, t}, \qquad \Delta_{\mu, t} \defeq \ed^+_{\mu, t}- \ed^-_{\mu, t} 
  \]
 denote the length of these gaps. In case of $t_*\le t\le 2t_*$ the densities $\rho_{\lambda,t}$, $\rho_{\mu,t}$
  have a small minimum denoted by $\mathfrak{m}_{\lambda, t}$ 
 and $\mathfrak{m}_{\mu, t}$ respectively. 
 Since we always assume $0\le t\le t_1\ll 1$, both $H^{(\lambda)}_t$ and $H^{(\mu)}_t$ will always have exactly one {physical} cusp near $\mathfrak{c}_\lambda$ 
 and $\mathfrak{c}_\mu$, respectively, using that  the Stieltjes transform of the density is a H\"older continuous function of $t$, see~\cite[Proposition 10.1]{1804.07752}.

Let $i_\lambda$ and $i_\mu$ be the indices defined by 
\[
\int_{-\infty}^{\ed_{\lambda, 0}^-} \rho_\lambda=\frac{i_\lambda-1}{N}, \,\,\,\,\,\int_{-\infty}^{\ed_{\mu, 0}^{-}}\rho_\mu=\frac{i_\mu-1}{N}.
\] 
By band rigidity (see Remark 2.6 in~\cite{1804.07744}) $i_\lambda$ and $i_\mu$ are integers. 
Note that by the explicit expression of the density in~\eqref{gamma def}-\eqref{gamma def edge} it follows that 
$cN\le i_\lambda, i_\mu\le (1-c)N$ with some small $c>0$, because the density on both sides of a {physical} cusp is macroscopic. 
 
We let  $\lambda_i(t)$ and $\mu_i(t)$ denote the eigenvalues of $H^{(\lambda)}_t$ and $H^{(\mu)}_t$, respectively. Let $\left\{B_i\right\}_{i\in[-N,N]\setminus\{0\}}$ be a family of independent standard (scalar) Brownian motions. It is well known~\cite{MR0148397} that  the eigenvalues of $H^{(\lambda)}_t$ satisfy
the equation for \emph{Dyson Brownian motion}, i.e.~the following system of coupled SDE's
\begin{equation}
\label{lambda}
\diff\lambda_i=\sqrt{\frac{2}{N}}\diff B_{i-i_\lambda+1}+\frac{1}{N}\sum_{j\ne i}\frac{1}{\lambda_i-\lambda_j}\diff t
\end{equation} with initial conditions $\lambda_i(0)=\lambda_i(H^{(\lambda)})$. Similarly, for the eigenvalues of $H^{(\mu)}_t$  we have
\begin{equation}
\label{mu}
\diff\mu_i=\sqrt{\frac{2}{N}}\diff B_{i-i_\mu+1}+\frac{1}{N}\sum_{j\ne i}\frac{1}{\mu_i-\mu_j}\diff t
\end{equation} with initial conditions $\mu_i(0)=\mu_i(H^{(\mu)})$.
Note that we chose the Brownian motions for $\lambda_i$ and 
$\mu_{i+i_\mu-i_\lambda}$ to be  identical.  This is the key ingredient for the coupling argument, since in this way
the stochastic differentials will cancel when we take the difference of the two DBMs or we differentiate it with respect to an additional parameter. 

For convenience of notation, we will shift the indices so that the same index  labels the last quantile before the
gap in $\rho_\lambda$ and $\rho_\mu$. 
This shift was already prepared by choosing the Brownian motions for $\mu_{i_\mu}$ and $\lambda_{i_\lambda}$ to be identical.
We achieve this shift by adding $N$ ``ghost'' particles very far away and relabelling,  as in~\cite{MR3914908}.
We thus embed $\lambda_i$ and $\mu_i$ into  the enlarged processes $\{x_i\}_{i\in[-N,N]\setminus\{ 0\} }$ and $\{y_i\}_{i\in[-N,N]\setminus\{0\}}$.
Note that the index $0$ is always omitted.

More precisely, the processes $x_i$ are defined by the following SDE  \emph{(extended Dyson Brownian motion) }
\begin{equation}
\label{xxx}
\diff x_i=\sqrt{\frac{2}{N}}\diff B_i+\frac{1}{N}\sum_{j\ne i}\frac{1}{x_i-x_j}\diff t, \qquad 1\le \abs{i}\le N,
\end{equation} with initial data \begin{equation}
\label{idxxx}
x_i(0)=\begin{cases}
-N^{200}+iN & \text{if} \,\,\,-N\le i\le -i_\lambda  \\
\lambda_{i+i_\lambda}(0) &\text{if}\,\,\, 1-i_\lambda\le i\le -1 \\
\lambda_{i+i_\lambda-1}(0) &\text{if}\,\,\, 1 \le i\le N+1-i_\lambda \\
N^{200}+iN &\text{if}\, \,N+2-i_\lambda\le i\le N,
\end{cases}
\end{equation} 
and the $y_i$ are defined by 
\begin{equation}
\label{yyy}
\diff y_i=\sqrt{\frac{2}{N}}\diff B_i+\frac{1}{N}\sum_{j\ne i}\frac{1}{y_i-y_j}\diff t, \qquad 1\le \abs{i}\le N,
\end{equation}
 with initial data 
\begin{equation}
\label{idyyy}
y_i(0)=\begin{cases}
-N^{200}+iN & \text{if}\,\,\, -N\le i\le -i_\mu  \\
\mu_{i+i_\mu}(0) &\text{if}\,\,\, 1-i_\mu\le i\le -1 \\
\mu_{i+i_\mu-1}(0) &\text{if} \,\,\, 1\le i\le N+1-i_\mu \\
N^{200}+iN &\text{if} \,\,\,N+2-i_\mu\le i\le N.
\end{cases}
\end{equation} 
The summations in~\eqref{xxx} and~\eqref{yyy} extend to all $j$ with  $1\le \abs{j}\le N$ except $j=i$.

The following lemma shows that the additional particles at distance $N^{200}$ have negligible effect on the dynamics of the re-indexed eigenvalues, thus we can study the processes $x_i$ and $y_i$ instead of the eigenvalues $\lambda_i$, $\mu_i$. The proof of this lemma follows by Appendix C of~\cite{MR3914908}.
 \begin{lemma}
\label{appr}
With very high probability the following estimates hold: 
\[
\sup_{0\le t\le 1}\sup_{1 \le i\le N+1-i_\lambda}\abs{x_i(t)-\lambda_{i+i_\lambda-1}(t)}\le N^{-100},
\] 
\[\sup_{0\le t\le 1}\sup_{1-i_\lambda\le i\le N+1-i_\lambda}\abs{x_i(t)-\lambda_{i+i_\lambda}(t)}\le N^{-100},\] 
\[\sup_{0\le t\le 1}\sup_{1 \le i\le N+1-i_\mu}\abs{y_i(t)-\mu_{i+i_\mu-1}(t)}\le N^{-100},\] 
\[\sup_{0\le t\le 1}\sup_{1-i_\mu\le i\le N+1-i_\mu}\abs{y_i(t)-\mu_{i+i_\mu}(t)}\le N^{-100},\] 
\[\sup_{0\le t\le 1}x_{-i_\lambda}(t)\lesssim - N^{200},\,\,\,\,\,\sup_{0\le t\le 1}x_{N+2-i_\lambda}(t)\gtrsim N^{200},\]  \[\sup_{0\le t\le 1}y_{-i_\mu}(t)\lesssim -N^{200},\,\,\,\,\,\sup_{0\le t\le 1}y_{N+2-i_\mu}(t)\gtrsim N^{200}.\]
\end{lemma}

\begin{remark}\label{notequal}
  For notational simplicity we assumed that $H^{(\lambda)}$ and $H^{(\mu)}$ have the same dimensions, but
 our proof works as long as the corresponding  dimensions $N_\lambda$ and $N_\mu$ 
 are merely comparable, say $\frac{2}{3}N_\lambda\le N_\mu \le \frac{3}{2} N_\lambda$.
 The only modification is that  the times in~\eqref{defflow} need to be scaled differently in order to keep the strength of the stochastic differential terms  in~\eqref{lambda}--\eqref{mu}  identical. In particular, we rescale the time in the process~\eqref{lambda} as $t'=(N_\mu/N_\lambda)t$, in such a way the $N$-scaling in front of the stochastic differential and in front of the potential term are exactly the same in both the processes~\eqref{lambda} and~\eqref{mu}; namely we may replace $N$ with $N_\mu$ in both~\eqref{lambda} and~\eqref{mu}. 
 Furthermore, the number of additional ``ghost'' particles  in the \emph{extended  Dyson Brownian motion} (see~\eqref{xxx} and~\eqref{yyy})
 will be different to ensure that we have the same total number of particles, i.e.~the total number of $x$ and $y$ particles will be $2N\defeq 2\max\{ N_\mu,N_\lambda\}$, after the extension.
 Hence, assuming that $N_\mu\ge N_\lambda$, there will be $N=N_\mu$ particles added to the DBM of the eigenvalues of $H^{(\mu)}$ and $2N_\mu-N_\lambda$ 
 particles added to the DBM of $H^{(\lambda)}$. In particular, under the assumption $N_\mu\ge N_\lambda$, we may replace~\eqref{idxxx} and~\eqref{idyyy} by
 \[
 x_i(0)=\begin{cases}
-N_\mu^{200}+iN_\mu & \text{if} \,\,\,-N_\mu\le i\le -i_\lambda  \\
\lambda_{i+i_\lambda}(0) &\text{if}\,\,\, 1-i_\lambda\le i\le -1 \\
\lambda_{i+i_\lambda-1}(0) &\text{if}\,\,\, 1 \le i\le N_\lambda+1-i_\lambda \\
N_\mu^{200}+iN_\mu &\text{if}\, \,N_\lambda+2-i_\lambda\le i\le N_\mu,
\end{cases}
\qquad
y_i(0)=\begin{cases}
-N_\mu^{200}+iN_\mu & \text{if}\,\,\, -N_\mu\le i\le -i_\mu  \\
\mu_{i+i_\mu}(0) &\text{if}\,\,\, 1-i_\mu\le i\le -1 \\
\mu_{i+i_\mu-1}(0) &\text{if} \,\,\, 1\le i\le N_\mu+1-i_\mu \\
N_\mu^{200}+iN_\mu &\text{if} \,\,\,N_\mu+2-i_\mu\le i\le N_\mu.
\end{cases}
 \]
Then, all the proofs of Section~\ref{sec:padding} and Section~\ref{sec:rigid} are exactly the same of the case $N\defeq N_\mu=N_\lambda$, since all the analysis of the latter sections is done in a small, order one neighborhood of the physical cusp. In particular, only the particles $x_i(t)$, $y_i(t)$ with $1\le |i|\le \epsilon \min\{N_\mu, N_\lambda\}$, for some small fixed $\epsilon>0$, will matter for our analysis. The far away particles in the case will be treated exactly as in~\eqref{rhox}--\eqref{stbound} replacing $N$ by $N_\mu$.
\end{remark}

We now construct the analogues of the self-consistent densities $\rho_{\lambda,t}$, $\rho_{\mu,t}$ for the $x(t)$ and $y(t)$ processes
as well as for their $\alpha$-interpolations.  We start with $\rho_{x,t}$. Recall $\rho_{\lambda, t}$ from Section~\ref{sec scflow},   and  set
\begin{equation}\label{rhox}
 \rho_{x,t}(E)\defeq  \rho_{\lambda,t} (E)  +     \frac{1}{N} \sum_{i=-N}^{-i_\lambda}  \psi(E-x_i(t))+
  \frac{1}{N} \sum_{i=N+2-i_\lambda}^N   \psi(E-x_i(t)), \qquad E\in\R,
 \end{equation}
where $\psi$ is a non-negative symmetric approximate delta-function on scale $N^{-1}$, i.e.~it is supported in an $N^{-1}$
neighbourhood of zero,  $\int\psi=1$, $\norm{\psi}_\infty\lesssim N$ and $\norm{\psi'}_\infty\lesssim N^2$. 
Note that the total mass is $\int_\R \rho_{x,t}=2$. For the Stieltjes transform $m_{x,t}$ of $\rho_{x,t}$, we have
$\sup_{z\in \C^+}\abs{ m_{x,t}(z)}\le C$ since the same bound holds for $\rho_{\lambda,t} $ by the shape analysis.
Note that $\rho_{\lambda,t} $ is the semicircular flow with initial condition $\rho_{\lambda, t=0}=\rho_\lambda$ 
by definition, but $\rho_{x,t}$ is not exactly the semicircular evolution of $\rho_{x, 0}$. We will not need
this information, but in fact, the effect 
of the far away padding particles on the density near the cusp  is very tiny.

Since $\rho_{x,t}$ coincides with $\rho_{\lambda, t}$ in a big finite interval, their edges and local minima near the cusp regime
 coincide, i.e we can identify
\[
 \ed_{x,t}^\pm= \ed_{\lambda, t}^\pm, \qquad \mi_{x, t} = \mi_{\lambda, t}.
\]
The shifted quantiles and semiquantiles $\wh\gamma_{x,i}(t), \widecheck\gamma_{x,i}(t)$ and $\wh\gamma_{x,i}^*(t), \widecheck\gamma_{x,i}^*(t)$  
of $\rho_{x,t}$ are defined by the obvious analogues of the formulas~\eqref{def:quan}--\eqref{def:semiquan2}
except that $r$ subscript is replaced with $x$ and the indices run over the entire range $1\le \abs{i}\le N$. As before,
$\gamma_{x,0}(t) = \ed_{x,t}^+$ . The unshifted quantiles are defined by 
\[
   \gamma_{x,i}(t) = \wh\gamma_{x,i}(t)+\ed_{x,t}^+, \quad 0\le t\le t_*, \qquad  \gamma_{x,i}(t) = \widecheck\gamma_{x,i}(t)+\mi_{x,t}, \quad t_*\le  t\le 2t_*
\]
 and similarly  for the semiquantiles.

So far we explained how to construct $\rho_{x,t}$ and its quantiles from $\rho_{\lambda, t}$, exactly in the same way we obtain $\rho_{y,t}$
from $\rho_{\mu, t}$ with straightforward notations.

Now for any $\alpha\in [0,1]$ we construct the $\alpha$-interpolation of $\rho_{x,t}$ and $\rho_{y,t}$ that we will denote by $\ov\rho_t$.
The bar will indicate quantities related to $\alpha$-interpolation that implicitly depend   on $\alpha$; a dependence that we often
omit from the notation. The interpolating measure will be constructed via its quantiles, i.e.~we define
\begin{equation}\label{barg}
    \ov \gamma_i (t)\defeq  \alpha \wh \gamma_{x,i}(t) + (1-\alpha)\wh \gamma_{y,i}(t), \qquad
      \ov\gamma_i^*(t)\defeq \alpha \wh \gamma_{x,i}^*(t) + (1-\alpha)\wh \gamma_{y,i}^*(t), \qquad 1\le \abs{i}\le N, \quad 0\le t\le t_*
\end{equation}
and similarly for $t_*\le t\le 2t_*$ involving $\widecheck\gamma$'s. We also set the interpolating edges
\begin{equation}\label{edgeinterpolation}
  \ov\ed_t^\pm = \alpha \ed_{x,t}^\pm + (1-\alpha) \ed_{y,t}^\pm.
\end{equation}

 Recall the parameter  $\delta_*$ 
describing the size of a  neighbourhood  around the physical cusp where the shape analysis for $\rho_\lambda$ and $\rho_\mu$ 
in Section~\ref{sec: main results} holds.
Choose 
 $i(\delta_*)\sim N$ such that $\abs[0]{\ov\gamma_{x, -i(\delta_*)}(t)}\le \delta_*$  as well as $\abs[0]{\ov \gamma_{x,i(\delta_*)}(t) }\le \delta_*$
 hold for all $0\le t\le 2t_*$.
Then define, for any $E\in\R$,  the function
\begin{equation}\label{rhozdef}
    \ov\rho_t(E)\defeq  \rho_{\alpha,t} (E) \cdot {\bf 1}\big( \ov\gamma_{-i(\delta_*)}(t)+\ov\ed_t^+ \le E
     \le \ov\gamma_{  i(\delta_*)}(t)+\ov\ed_t^+\big) +
     \frac{1}{N} \sum_{i(\delta_*)<\abs{i}\le N} \psi(E- \ov\ed_t^+- \ov\gamma_{i}^*(t) ),
\end{equation}
where $\rho_{\alpha, t}$ is the $\alpha$-interpolation, constructed in Definition~\ref{def interpolating density},
between $\rho_{\lambda,t}(E) =\rho_{x,t}(E)$ and $\rho_{\mu,t}(E) =\rho_{y,t}(E)$
for $\abs{E}\le \delta_*$. By this construction  (using also the symmetry of $\psi$) we know that all shifted semiquantiles of 
$\ov\rho_t$ are exactly  $\ov\gamma_{i}^*(t)$. The same holds for all shifted quantiles  $\ov\gamma_{i}(t)$
at least in the interval $[-\delta_*, \delta_*]$ since here $\ov\rho_t \equiv \rho_{\alpha, t}$ and
the latter was constructed exactly by the requirement of linearity of the quantiles~\eqref{barg}, see~\eqref{quantiles convex comb}.

We also record $\int \ov\rho_t=2$ and that for the Stieltjes transform $\ov m_t(z)$  of $\ov\rho_t$  we have
\begin{equation}\label{stbound}
  \max_{\abs[0]{\Re z-\ov\ed_t^+}\le \frac{1}{2}\delta_*} \abs{\ov m_t(z)}\le C 
\end{equation}
for all $0\le t\le 2t_*$.  The first bound follows easily from the same boundedness of the Stieltjes transform of  $\rho_{\alpha,t}$. 
Moreover, $\ov m_t(z)$ is $\frac{1}{3}$-H\"older continuous in the regime $\abs{\Re z-\ov\ed_t^+}\le \frac{1}{2}\delta_*$ since
in this regime $\ov\rho_t = \rho_{\alpha, t}$ and $\rho_{\alpha, t}$ is $\frac{1}{3}$-H\"older continuous  by Lemma~\ref{lemma holderhsh}.

\section{Rigidity for the short range approximation}\label{sec:rigid}

\newcommand{\rr}{\mathring{r}}

In this section we consider Dyson Brownian Motion (DBM), i.e.~a system of  $2N$ coupled stochastic differential equations 
for $z(t) =\{ z_i(t) \}_{[-N,N]\setminus \{0\}}$
of the form
\begin{equation}\label{req}
  \diff  z_i =\sqrt{ \frac{2}{N}} \diff B_i+ \frac{1}{N}\sum_j  \frac{1}{ z_i- z_j} \diff t,
 \qquad 1\le \abs{i}\le N,
\end{equation}
with some initial condition $z_i(t=0)=z_i(0)$, 
where $B(s)= (B_{-N}(s), \ldots, B_{-1}(s), B_1(s) \ldots, B_N(s))$ is the vector of $2N$  independent standard Brownian motions.
We use the indexing convention that all indices  $i,j$, etc., run from $-N$ to $N$ but zero index is excluded.

We will assume that $z_i(0)$ is an $\alpha$-linear interpolation of $x_i(0), y_i(0)$ for some $\alpha\in [0,1]$:
\begin{equation}\label{initalpha}
    z_i(0)= z_i(0,\alpha)\defeq  \alpha x_i(0) + (1-\alpha) y_i(0).
 \end{equation} In the following of this section we will refer to the process defined by~\eqref{req} using $z(t,\alpha)$ in order to underline the $\alpha$ dependence of the process.
Clearly for $\alpha=0,1$ we recover the original $y(t)$ and $x(t)$ processes, $z(t,\alpha=0)=y(t)$, $z(t,\alpha=1)=x(t)$.
For these processes we have
 the following optimal rigidity estimate that immediately follows from~\cite[Corollary 2.6]{1809.03971}
 and Lemma~\ref{appr}:
 \begin{lemma}
\label{diffx}
Let $r_i(t)=x_i(t)$ or $r_i(t)=y_i(t)$ and $r=x,y$. Then, there exists a fixed small $\epsilon>0$, depending only
on the model parameters,  such that for each $1\le \abs{i}\le \epsilon N$, we have 
\begin{equation}
\label{rzx}
\sup_{0\le t\le 2t_*}\abs{r_i(t)-\gamma_{r,i }(t)}\le N^\xi \eta_{\mathrm{f}}^{\rho_{r,t}}({\gamma}_{r,  i}(t)),
\end{equation} for any $\xi>0$ with very high probability, where we recall that the behavior of $\eta_{\mathrm{f}}^{\rho_{r,t}}({\gamma}_{r,  i }(t))$, 
with $r=x,y$, is given by~\eqref{fluc scale gap}. 
\end{lemma}

Note that, by~\eqref{eq Delta size},~\eqref{eq min size} and~\eqref{eqs fluc scale}, for all $1\le \abs{i}\le \epsilon N $ and for all $0\le t\le t_*$ we have that \begin{equation}
\label{boundeta}
\eta_{\mathrm{f}}^{\rho_{r,t}}({\gamma}_{r, i }(t))\lesssim \frac{N^\frac{\omega_1}{6}}{\abs{ i  }^\frac{1}{4}N^\frac{3}{4}},
\end{equation} with $r=x,y$.

In particular, we know that
$z(0,\alpha)$  lie close to the quantiles~\eqref{barg} of an \emph{$\alpha$-interpolating density}
 $\rho_z=\overline{\rho}_0$, see the definition in~\eqref{rhozdef}. 
This means that  $\rho_z$ has a small gap $[\ed_z^-, \ed_z^+]$ of size $\Delta_z\sim t_*^{3/2}$ (i.e.~it will develop a physical cusp in a time of order $t_*$) and it is an $\alpha$-interpolation between $\rho_{x,0}$ and $\rho_{y,0}$. Here interpolation refers
to the process introduced in Section~\ref{sec:padding} that guarantees that the corresponding quantiles are convex linear
combinations of the two initial densities with weights $\alpha$ and $1-\alpha$, i.e.~
\[
  \gamma_{z, i} = \alpha\gamma_{x,i} + (1-\alpha)\gamma_{y,i}.
\]

In this section we will prove rigidity results for $z(t,\alpha)$ and for its appropriate short range approximation. 

\begin{remark}
\label{diifothers}
Before we go into the details, we point out that
  we will prove  rigidity %
  \emph{dynamically}, i.e. 
using the DBM. The route chosen here is very different from the one in~\cite[Sec. 6]{1712.03881}, where the authors prove a local law for short times in order to get  rigidity for the short range approximation of the interpolated process. While it would be possible to follow the latter strategy in the cusp regime 
as well, the technical difficulties are overwhelming, in fact already in the much simpler edge regime 
a large part of~\cite{1712.03881} was devoted to this task. The current proof of the optimal law at the cusp regime~\cite{1809.03971}
heavily use an effective mean-field condition (called \emph{flatness}) that corresponds to large time in the DBM. Relaxing this condition would
require to adjust  not only~\cite{1809.03971} but also the necessary deterministic analysis from~\cite{1804.07752} to the short time case. Similar complications would have arisen if we had followed the strategy of~\cite{1810.08308,1612.06306} where rigidity is proven by analysing the characteristics of the McKean-Vlasov equation. The route chosen here is shorter and  more interesting.
\end{remark}

Since the group velocity of the entire cusp regime is different for $\rho_{x,t}$ and $\rho_{y,t}$,
the interpolated process will have an intermediate group velocity. Since we have to follow the process
for time scales $t\sim N^{-\frac{1}{2}+\om_1}$, much bigger than the relevant rigidity scale $N^{-\frac{3}{4}}$
we  have to determine the group velocity quite precisely. Technically, we will encode this information by
defining an appropriately  shifted process $\wt z(t,\alpha) = z(t,\alpha) - \mathrm{Shift}(t,\alpha)$. It is essential that the shift function
is independent of the indices $i$ to preserve the local statistics of the process. In the next section we 
explain how to choose the shift.

\subsection{Choice of the shifted process \texorpdfstring{$\wt z$}{z}}\label{sec:shiftchoice}

The remainder of Section~\ref{sec:rigid} is formulated for the small gap regime, i.e.~for $0\le t\le t_\ast$. We will comment on the modifications in the small minimum regime in Section~\ref{UNMIN}. To match the location of the gap, the natural guess would be  to study the shifted process $z_i(t,\alpha) - \ed_{z, t}^+$
where $[\ed_{z, t}^-, \ed_{z, t}^+]$ is the gap of the semicircular evolution $\rho_{z,t}$ of  $\rho_z$  near 
the physical cusp, and approximate $z_i(t,\alpha) - \ed_{z, t}^+$ by the shifted semiquantiles $\wh \gamma_{z,i}^*(t)$
of $\rho_{z,t}$. However, the evolution of the semicircular flow  $t\to \rho_{z,t}$ near the cusp is
not sufficiently well understood. We circumvent this technical problem by considering the quantiles of another 
approximating density $\ov \rho_t$  defined by the requirement that its quantiles are 
exactly the $\alpha$-linear combinations of the quantiles of $\rho_{x, t}$ and $\rho_{y,t}$ as described in Section~\ref{sec:padding}.
The necessary regularity properties of $\ov \rho_t$ follow directly from its construction.
The precise description below assumes that $0\le t\le 2t_*$, i.e.~we are in the small gap situation. For $t_*\le t\le t_*$
an identical construction works but the reference point $\ed_{r,t}^+$ is replaced with the approximate
minimum $\wt\mi_{r,t}$,  for $r=x,y$. For simplicity we present all formulas for $0\le t\le t^*$ and we will comment on the other case in Section~\ref{UNMIN}.

More concretely,  for any fixed $\alpha\in [0,1]$ recall the (semi)quantiles from~\eqref{barg}.
These are the (semi)quantiles of the  interpolating density $\ov \rho=\ov\rho_t$  defined in~\eqref{rhozdef}
and let  its Stieltjes transform be
denoted by $\ov m= \ov m_t$.  Bar will refer to quantities related to this interpolation; implicitly 
all quantities marked by bar depend on the interpolation
parameter $\alpha$, which dependence will be omitted from the notation. 
Notice that $\ov \rho_t$ has a gap $[\ov\ed_t^-, \ov\ed_t^+]$ near the cusp satisfying~\eqref{edgeinterpolation}.
Initially at $t=0$ we have $\ov \rho_{t=0}=\rho_z$, in particular $\ov\gamma_i(t=0) =\wh \gamma_{z,i}(t=0)$
and $\ov \ed_0^\pm = \ed_z^\pm$.
 We will choose the shift in the definition of the  $\wt z_i(t,\alpha)$ process so that
 we could use  $\ov\gamma_i^*(t)$
to trail it.

The semicircular flow and the $\alpha$-interpolation do not commute
hence  $\ov \gamma_{i}(t)$ are not the same as the quantiles $\wh \gamma_{z,i}(t)$ of
 the semicircular evolution $\rho_{z,t}$ of the initial density $\rho_z$.  We will, however, show
 that they are sufficiently close near the cusp and up to times relevant for us, modulo an irrelevant time dependent shift.
 Notice that the evolution of  $\wh \gamma_{z,i}(t)$  is hard to control since
analysing $\frac{\diff}{\diff t} \wh \gamma_{z,i}(t) = -\Re m_{z,t}( \gamma_{z,i}(t)) + \Re m_{z,t}(\ed_{z,t}^+)$
would involve knowing the evolved density $\rho_{z,t}$ quite precisely in the critical cusp regime. While
this necessary  information is in principle accessible from the explicit expression for the semicircular flow and the 
precise shape analysis of $\rho_z$ obtained from that of $\rho_x$ and $\rho_y$,  here we chose a different, technically lighter path
by using $\ov\gamma_i(t)$. 
Note that unlike  $\wh \gamma_{z,i}(t)$, the derivative of $\ov\gamma_i(t)$
 involves only the Stieltjes transform of the densities $\rho_{x,t}$ and $\rho_{y,t}$
 for which shape analysis is available.

However, the global group velocities of $\ov\gamma(t)$
and $\wh \gamma_z(t)$ are not the same near the cusp. 
We thus need to define $\wt z(t,\alpha)$ not as $z(t,\alpha)-\ov\ed_t^+$ but with a modified 
 time dependent shift 
to make up for this velocity difference so that
$\ov \gamma(t) $ indeed correctly follows $\wt z(t,\alpha)$. To determine this shift, we first define   the function 
\begin{equation}\label{h*def}
   h^*(t,\alpha)\defeq \Re \Big[  - \ov m_{t}(\ov \ed_{t}^+)+ (1-\alpha)  m_{y,t}( \ed_{y,t}^+) +
    \alpha   m_{x,t}( \ed_{x,t}^+) \Big],
\end{equation}
where recall that $\ov m_t$ is the Stieltjes transform of 
the measure $\ov\rho_t$. Note that $h^*(t)=O(1)$ following 
from the boundedness of the Stieltjes transforms  $m_{x,t}$, $m_{y,t}$ and $\ov m_t(\ov\ed_t^+)$. The boundedness of $m_{x,t}$ and $m_{y,t}$ follows by~\eqref{eq free sc conv def} and $\abs{\overline{m}_t(\overline{e}_t^+)}\le C$ by~\eqref{stbound}.

We note that
\[
      h^*(t, \alpha=0) = m_{y,t}( \ed_{y,t}^+) - \ov m_{t}(\ov \ed_{t}^+) = m_{y,t}( \ed_{y,t}^+) - \ov m_{t}(\ed_{y,t}^+)
\] 
since for $\alpha=0$ we have $\ed_{y,t}^+ = \ov \ed_{t}^+$ by construction. At $\alpha=0$ 
the measure $\ov\rho_t$  is given exactly by the density  $\rho_{y,t}$ in an $\mathcal{O}(1)$ neighbourhood of the cusp. Away from the
cusp,  depending on the precise construction
in the analogue of~\eqref{rhozdef},  the continuous $\rho_{y,t}$ is replaced by locally smoothed out
Dirac measures at the quantiles. A similar statement
holds at $\alpha=1$, i.e.~for the density $\rho_{x,t}$.
It is easy to see that the difference of the corresponding Stieltjes transforms evaluated at the cusp regime is of order $N^{-1}$, i.e.
\begin{equation}
\label{zetaest}
     |h^*(t, \alpha=0)|+ |h^*(t, \alpha=1)|= O(N^{-1}).
\end{equation}

Since later in~\eqref{zeta} we will need to give some very crude estimate on the $\alpha$-derivative of $h^*(t,\alpha)$, but
it actually blows up
since $\ov m_t'$ is singular at the edge,  we introduce a tiny regularization of $h^*$,  i.e.
we define the function 
\begin{equation}
\label{h**def}
h^{**}(t,\alpha)\defeq  \Re \Big[  - \ov m_{t}(\ov \ed_{t}^++\ii N^{-100})+ (1-\alpha)  m_{y,t}( \ed_{y,t}^+) +
    \alpha   m_{x,t}( \ed_{x,t}^+) \Big].
\end{equation} Note that by the $\frac{1}{3}$-H\"older continuity of $\overline{m}_t$ in the cusp regime, i.e.~for $z\in\mathbb{H}$ such that $\abs{\Re z-\overline{\mathfrak{e}}_t^+}\le\frac{\delta_*}{2}$, it follows that 
\begin{equation}
\label{close}
h^{**}(t,\alpha)=h^*(t,\alpha)+\mathcal{O}(N^{-30}).
\end{equation}

Then, we define
\begin{equation}\label{hdef}
   h(t)=h(t,\alpha)\defeq  h^{**}(t, \alpha) -\alpha h^{**}(t, 1) - (1-\alpha) h^{**}(t, 0) 
   =O(1)
\end{equation}
to ensure that 
\begin{equation}\label{hedge}
h(t,\alpha=0) =  h(t, \alpha=1) = 0.
\end{equation}
In particular, we  have
\begin{equation}\label{hhdef}
    h(t,\alpha)= \Re \Big[ - \ov m_{t}(\ov \ed_{t}^+)+  (1-\alpha)  m_{y,t}( \ed_{y,t}^+) +
    \alpha   m_{x,t}( \ed_{x,t}^+)\Big]+ O(N^{-1}).
\end{equation}
Define its antiderivative
\begin{equation}\label{Hdef}
    H(t,\alpha)\defeq  \int_0^t h(s,\alpha)\diff s, \qquad H(0,\alpha)= 0, \qquad \max_{0\le t\le t_*} \abs{H(t,\alpha)}\lesssim N^{-1/2+\om_1}.
\end{equation}
Now we are ready to define the correctly shifted process
\begin{equation}\label{def:rtilde}
\wt z_i (t) = \wt z_i(t, \alpha)\defeq z_i(t) -   \big[ \alpha \ed_{x,t}^+ + (1-\alpha) \ed_{y,t}^+\big] - H(t,\alpha),
\end{equation}
that will be trailed by $\ov\gamma_i(t)$.  It satisfies the shifted DBM
\begin{equation}\label{tildereq}
  \diff \wt z_i =\sqrt{ \frac{2}{N}} \diff B_i+ \Bigg[ \frac{1}{N}\sum_{j\ne i}  \frac{1}{\wt z_i-\wt z_j} +\Phi_\alpha(t) \Bigg]\diff t
\end{equation}
with
\begin{equation}\label{Phidef}
  \Phi(t) \defeq \Phi_\alpha(t)= \alpha \Re m_{x,t}(\ed_{x,t}^+) + (1-\alpha)\Re m_{y,t}(\ed_{y,t}^+)  - h(t,\alpha), 
\end{equation}
and with  initial conditions $\wt z(0)\defeq  z(0) - \ed_z^+$  by~\eqref{edgeinterpolation} and $H(0,\alpha)=0$.
The shift function satisfies  
\begin{equation}\label{Phiest}
\Phi_\alpha(t)=\Re[ \overline{m}_t(\overline{\mathfrak{e}}_t^+)]+\mathcal{O}(N^{-1}).
\end{equation}

Notice that
 for $\alpha=0,1$ this definition gives back  the naturally shifted $x(t)$ and $y(t)$ processes since we clearly have
\begin{equation}\label{xy}
 \wt z(t,\alpha=1)= \wt x(t)\defeq x(t) -\ed_{x,t}^+, \qquad \wt z(t,\alpha=0)= \wt y(t)\defeq y(t) -\ed_{y,t}^+,
\end{equation}
that are trailed by the shifted semiquantiles
\begin{equation}\label{xyq}
 \ov\gamma_i^*(t,\alpha=1)= \wh \gamma_{x,i}^*(t)\defeq \gamma_{x,i}^*(t) - \ed_{x,t}^+, \qquad 
 \ov\gamma_i^*(t,\alpha=0)= \wh \gamma_{y,i}^*(t)\defeq \gamma_{y,i}^*(t) - \ed_{y,t}^+.
 \end{equation}

 As we explained, the time dependent shift $H(t,\alpha)$ in~\eqref{def:rtilde} makes up for
the difference between the true edge velocity of the semicircular flow (which we do not compute
directly) and the naive guess which is $\frac{\diff}{\diff t}  \big[ \alpha \ed_{x,t}^+ + (1-\alpha) \ed_{y,t}^+\big]$
hinted by the linear combination procedure. The precise expression~\eqref{h*def} will come out of the
proof. The key point is that this adjustment is global, i.e.~it is only time dependent but independent of $i$
since this expresses a group velocity of the entire cusp regime.

\bigskip

\subsection{Plan of the proof.}

In the following three  subsections we prove an almost optimal rigidity not directly for $\widetilde{z}_i(t)$
but  for its appropriate  short range approximation $\widehat{z}_i(t)$. This will be sufficient for the proof of the universality.
The proof of the rigidity will be divided into three phases, which we first explain informally, as follows.

\begin{itemize}
\item[\textbf{Phase 1.}] (Subsection~\ref{martingalenew})
The main result is a rigidity for $\widetilde{z}_i(t)-\ov\gamma_i(t)$ for $1\le \abs{i}\lesssim \sqrt{N}$  
on scale $N^{-\frac{3}{4}+C\om_1}$ without $i$-dependence in the error term. 
First we prove a crude rigidity on scale $N^{-1/2+C\om_1}$ for all indices $i$. Using this rigidity, we can define
a short range approximation $\mathring{z}$ of the original dynamics $\wt z$ and show that $\wt z_i$ and $\mathring{z}_i$ 
are close by $N^{-\frac{3}{4}+C\om_1}$ for $1\le \abs{i}\lesssim \sqrt{N}$. Then we analyse the short range process $\mathring{z}$
that has a finite speed of propagation, so we can localize the dynamics. Finally, 
we can directly compare $\mathring{z}$ with a deterministic particle dynamics  because the effect of
the stochastic term $\sqrt{2/N}\diff B_i$, i.e.~$\sqrt{t_*/N} = N^{-3/4+\om_1/2}\ll N^{-3/4+ C\om_1}$,
 remains below the rigidity scale of interest in this Phase 1.

However, to understand this deterministic particle dynamics we need to compare it with  the corresponding continuum evolution;
this boils down to estimating the difference of a Stieltjes transform and its Riemann sum approximation  at the
semiquantiles. Since the Stieltjes transform is given by a singular integral, this approximation relies on quite delicate
cancellations which require some strong regularity properties of the density. We can easily guarantee this regularity 
by considering the density $\ov\rho_t$ of the linear interpolation between the quantiles of $\rho_{x,t}$ and $\rho_{y,t}$.

\item[\textbf{Phase 2.}]  (Subsection~\ref{sec:hatrig})  In this section we improve the rigidity from scale $N^{-\frac{3}{4}+C\om_1}$
to scale $N^{-\frac{3}{4}+\frac{1}{6}\om_1}$, for a smaller range of indices,
 but we can achieve this not for $\wt z$ directly, but for its short range approximation
$\wh z$. Unlike $\mathring{z}$ in Phase 1, this time we choose a very short scale approximation $\wh z$ on scale $N^{4\om_\ell}$ with $\om_1\ll \om_\ell \ll 1$.
As an input, we need
 the rigidity of $\wt z_i$ on scale $N^{-\frac{3}{4}+C\om_1}$ 
for $1\le \abs{i}\lesssim \sqrt{N}$ obtained in Phase 1. We use heat kernel contraction 
 for a  direct comparison with the $y_i(t)$ dynamics for which we know optimal rigidity by~\cite{1809.03971},
 with the precise matching of the indices (\emph{band rigidity}). In particular, when the gap is large, this 
 guarantees that band rigidity is transferred to  the $\wh z$ process from  the $\wh y$ process.

\item[\textbf{Phase 3.}] (Subsection~\ref{sec:correcti})  Finally, we establish  the optimal $i$-dependence 
in the rigidity estimate for $\wh z_i$ from Phase 2, i.e.~we get a precision $N^{-\frac{3}{4}+\frac{1}{6}\om_1} \abs{i}^{-1/4}$.
The main method we use in Phase 3 is maximum principle. We compare $\wh z_i$ with $\widehat{y}_{i-K}$, a slightly shifted element of the $\widehat{y}$ process,
where $K=N^\xi$ with some tiny $\xi$. This method allows us to prove the optimal $i$-dependent rigidity (with a factor $N^{\frac{1}{6}\om_1}$) but
only for indices $\abs{i}\gg K$ because otherwise $\wh z_i$ and $\wh y_{i-K}$ may be on different sides of the gap for small $i$.
For very small indices,  therefore,  we need to rely on band rigidity for $\wh z$ from Phase 2.

The optimal $i$-dependence  allows us to replace the random particles $\wh z$ by appropriate quantiles
with a precision 
so that 
\[ 
\abs{\wh z_i -\wh z_j}\lesssim  N^\frac{\omega_1}{6} \abs{\ov \gamma_{i} -\ov \gamma_j} 
\sim N^{-\frac{3}{4} +\frac{\omega_1}{6}}\abs{ \abs{i}^{\frac{3}{4}} - \abs{j}^{\frac{3}{4}}}.
\]
Such upper bound on $\abs{\wh z_i -\wh z_j}$, hence a lower bound on the interaction
 kernel $\mathcal{B}_{ij} =\abs{\wh z_i -\wh z_j}^{-2}$ of the differentiated DBM  (see~\eqref{11} later) with 
the correct dependence  on the indices $i,j$, is 
essential since this gives the heat kernel contraction which eventually drives the precision below
the rigidity scale in order to prove universality. On a time scale $t_*= N^{-\frac{1}{2}+\om_1}$ the 
$\ell^p\to \ell^\infty$ contraction of the heat kernel gains a factor $N^{-\frac{4}{15}\om_1}$
with the convenient choice of $p=5$. Notice that $\frac{4}{15} >\frac{1}{6}$, so the contraction
wins over the  imprecision in the rigidity $N^{\frac{1}{6}\om_1}$ from Phase 3, but not over $N^{C\om_1}$
from Phase 1, showing that both Phase 2 and Phase 3 are indeed necessary.
\end{itemize}

\subsection{Phase 1: Rigidity for \texorpdfstring{$\wt z$}{z} on scale \texorpdfstring{$N^{-3/4+C\om_1}$}{N-3/4+Cw1}.}\label{martingalenew}

The main result of this section is the following proposition:

\begin{proposition}\label{prop:option2}
Fix $\alpha\in [0,1]$. Let $\wt z(t,\alpha)$ solve~\eqref{tildereq} with initial condition $\wt z_i(0,\alpha)$ satisfying the crude rigidity bound for all indices
\begin{equation}\label{verycrudeinitial}
  \max_{1\le \abs{i}\le N}  \abs{ \wt z_i(0,\alpha) - \ov \gamma_{i}^*(0)}\lesssim N^{-1/2+2\om_1}.
\end{equation}
We also assume that
\begin{equation}\label{mmbound}
\norm{m_{x,0}}_\infty + \norm{m_{y,0}}_\infty + \abs{ \ov m_t(\ov \ed_t^\pm)}\le C.
\end{equation}  
Then we have a weak but uniform rigidity
\begin{equation}\label{crudinitial}
  \sup_{0\le t\le t_*}\max_{1\le \abs{i}\le N}  \abs{ \wt z_i(t,\alpha) - \ov \gamma_{i}^*(t) }\lesssim N^{-1/2+2\om_1}, 
\end{equation}
with very high probability.
Moreover, for small $\abs{i}$, i.e.~$1\le \abs{i}\le i_*$,  with $i_* \defeq  N^{1/2+ C_*\om_1}$ for some large $C_*>100$, we have a stronger rigidity:
\begin{equation}\label{rgtilde}
   \sup_{0\le t\le t_*} \max_{1\le \abs{i}\le i_*} \abs{\wt z_i(t,\alpha) -\ov \gamma_{i}^*(t)} \lesssim  \max_{1\le \abs{i}\le 2i_*}
    \abs{\wt z_i(0,\alpha)-\ov \gamma_{i}^*(0)} +\frac{N^{C\om_1}}{N^{3/4}} 
\end{equation}
with very high probability.
\end{proposition}

In our application,~\eqref{verycrudeinitial}   is satisfied
  and the right hand side of~\eqref{rgtilde} is simply $N^{-\frac{3}{4}+C\om_1}$ since
\begin{equation}\label{optimalinit}
\wt z_i(0,\alpha) -\ov \gamma_{i}^*(0) = \alpha \big(x_i(0) - \gamma_{x,i}(0)\big) + (1-\alpha) \big(y_i(0) - \gamma_{y,i}(0)\big)
= O\left( \frac{N^\xi N^\frac{\omega_1}{6}}{N^{\frac{3}{4}} \abs{ i}^{\frac{1}{4}}}\right),
\end{equation}
for any $\xi> 0$ with very high probability, by optimal rigidity for $x_i(0)$ and $y_i(0)$ from~\cite{1809.03971}. 
Similarly, the assumption~\eqref{mmbound} is
  trivially satisfied by~\eqref{stbound}.
 However, we stated Proposition~\ref{prop:option2}
under the slightly weaker conditions~\eqref{verycrudeinitial},~\eqref{mmbound} to highlight what is really needed for its proof.

Before starting the proof, we recall the formula 
\begin{equation}\label{gammader}
  \frac{\diff}{\diff t} \wh \gamma_{i,r}^*(t) = -\Re m_{r,t}(  \gamma_{r, i}^*(t) )+\Re m_{r,t}(\ed_{r,t}^+), \qquad r=x, y.
\end{equation} 
on the derivative of the (shifted) semiquantiles
of a density which evolves by the semicircular flow and follows directly from~\eqref{diff m quant} and~\eqref{diff e m}.

\begin{proof}[Proof of Proposition~\ref{prop:option2}.] 
We start with the proof of the crude rigidity~\eqref{crudinitial}, then
we introduce a short range approximation and finally, with its help, we prove the refined rigidity~\eqref{rgtilde}. 
The main technical input of the last step is a refined estimate on the forcing term.
These four steps will be presented in the next four subsections.

\subsubsection{Proof of the crude rigidity:}
For the proof of~\eqref{crudinitial}, using~\eqref{gammader} twice in~\eqref{barg}, we notice that
\[
   \frac{\diff}{\diff t} \ov \gamma_i^*(t) = \alpha\big[ -\Re m_{x,t}(\gamma_{x,i}^*(t)) + \Re m_{x,t}(\ed_{x,t}^+)\big]
   + (1-\alpha) \big[ -\Re m_{y,t}(\gamma_{y,i}^*(t)) + \Re m_{y,t}(\ed_{y,t}^+)\big] = O(1)
 \]
 since $m_{x, t}$ and $m_{y,t}$ are bounded  recalling 
  that the semicircular flow preserves (or reduces) the $\ell^\infty$ norm of the Stieltjes transform  by~\eqref{eq free sc conv def}, 
 so $\norm{ m_{x,t}}_\infty \le \norm{m_{x, 0}}_\infty \le C$, similarly for $m_{y,t}$. 
 This gives 
 \begin{equation}\label{ovgam}
 \abs{\ov \gamma_i^*(t)-\ov \gamma_i^*(0)}\lesssim N^{-1/2+\om_1}.
 \end{equation}
  Thus  in order to prove~\eqref{crudinitial} it is sufficient to prove 
  \begin{equation}\label{rrr}
  \norm{ \wt z(t,\alpha)-\wt z(0,\alpha)}_\infty\le N^{-1/2+2\om_1},
  \end{equation} for any fixed $\alpha\in [0,1]$.
 To do that,  we compare the dynamics of~\eqref{tildereq} with the dynamics of the $y$-semiquantiles, i.e.~
 set 
 \[
 u_i\defeq u_i(t,\alpha)=\wt z_i(t) -\wh\gamma_{y,i}^*(t),
 \] 
 for all $0\le t\le t_*$.
 
  Compute
 \begin{equation}\label{du}
     \diff u_i = \sqrt{ \frac{2}{N}} \diff B_i+ (\wt {\mathcal{B}}u)_i \diff t+\wt F_i(t) \diff t
 \end{equation}
 with 
 \begin{equation}\label{mcbtilde}
    (\wt {\mathcal{B}}f)_i\defeq  
   \frac{1}{N}\sum_{j\ne i}   \frac{f_j -f_i}{(\wt z_i-\wt z_j)(\wh \gamma_{y,i}^*- \wh \gamma_{y,j}^*)  }
\end{equation}
and 
\[
   \wt F_i(t)\defeq  \frac{1}{N} \sum_{j\ne i}
    \frac{1}{\wh \gamma_{y,i}^*- \wh \gamma_{y,j}^*} +\Re m_{y,t}(\gamma_{y,i}^*(t)) 
  +  \alpha\big[  \Re m_{x,t}(\ed_{x,t}^+) - \Re m_{y,t}(\ed_{y,t}^+) \Big] - h(t).
 \] 
 The operator $\wt{\mathcal{B}}$ is defined on $\C^{2N}$ 
 and we label the vectors $f\in \C^{2N}$ as $f=(f_{-N}, f_{-N+1}, \ldots, f_{-1}, f_1, \ldots, f_N)$,
i.e.~we omit the $i=0$ index. Accordingly, 
in the summations the $j=0$ term is always omitted since $\wt z_j$, $\wh z_j$ and $\wh\gamma_{y,j}^*$ are defined for $1\le \abs{j}\le N$.
Furthermore in  the summation of the interaction terms, the $j=i$ term is always omitted.
 
 We now show that 
 \begin{equation}\label{Rbound}
    \norm{ \wt F(t)}_\infty \lesssim \log N, \qquad 0\le t \le t^*.
 \end{equation}
 By the boundedness of $m_{x,t}, m_{y,t}$ and the $1/3$-H\"older continuity of $\overline{m}_t$ in the cusp regime,  it remains to control 
 \[
   \frac{1}{N} \sum_{j\ne i} \frac{1}{ \wh\gamma_{y,i}^*(t)-  \wh\gamma_{y,j}^*(t)} \lesssim \sum_{1\le \abs{j-i}\le N}\frac{1}{\abs{i-j}} \lesssim \log N
 \]
since  $\abs{\wh\gamma_{y,j}^*-\wh\gamma_{y,i}^*}\ge c\abs{i-j}/N$ as the density $\rho_{y,t}$ is bounded.

Let $\wt{\mathcal{U}}(s,t)$ be the fundamental solution of the heat  evolution with  kernel $\wt{\mathcal{B}}$
from~\eqref{mcbtilde}, i.e, for any $0\le s \le t$
\begin{equation}\label{defU}
    \partial_t \wt{\mathcal{U}}(s,t) = \wt{\mathcal{B}}(t) \wt{\mathcal{U}}(s,t) , \qquad \wt{ \mathcal{U}}(s,s)=I .
\end{equation}
Note that $\wt {\mathcal{U}}$ is a contraction on every $\ell^p$ space and the same is true for its adjoint $\wt {\mathcal{U}}^*(s,t)$.
In particular, for any indices $a, b$ and times $s,t$ we have
\begin{equation}\label{Utriv}
   \wt{\mathcal{U}}_{ab}(s,t) \le 1,\quad  \wt{\mathcal{U}}^*_{ab}(s,t) \le 1. 
\end{equation}
By Duhamel principle, the solution to the SDE~\eqref{du} is given by
\begin{equation}\label{duhamel}
      u (t) = \wt{\mathcal{U}}(0,t) u(0) + \sqrt{\frac{2}{N}} \int_0^t \wt{ \mathcal{U}}(s,t) \diff B(s) + \int_0^t  \wt{\mathcal{U}}(s,t) \wt F(s)\diff s,
\end{equation}
where $B(s)= (B_{-N}(s), \ldots, B_{-1}(s), B_1(s) \ldots, B_N(s))$ are the $2N$  independent Brownian motions from~\eqref{req}.

For the second term in~\eqref{duhamel} we fix an index $i$ and  consider the martingale
\[
   M_t \defeq \sqrt{\frac{2}{N}}\int_0^t  \sum_j \wt{\mathcal{U}}_{ij} (s,t) \diff B_j(s)
\]
with its quadratic variation  process
\[
  [M]_t\defeq   \frac{2}{N} \int_0^t \sum_j \big( \wt{\mathcal{U}}_{ij} (s,t) \big)^2 \diff s =  \frac{2}{N} \int_0^t 
    \norm{ \wt{\mathcal{U}}^*(s,t) \delta_i}^2_2\diff s \le \frac{2t}{N}.
\]
By the Burkholder maximal inequality for martingales, for any $p>1$ we have that
\[
    \E \sup_{0\le t\le T} \abs{M_t}^{2p} \le C_p \E [M]_T^p \le C_p\frac{T^p}{N^p}.
\]
By Markov inequality we obtain that 
\begin{equation}\label{marti}
  \sup_{0\le t\le T} \abs{M_t} \le N^\xi \sqrt{\frac{T}{N}}
\end{equation}
with probability more than $1-N^{-D}$, for any (large) $D>0$ and  (small) $\xi>0$. 

The last term in~\eqref{duhamel} is estimated, using~\eqref{Rbound}, by
\begin{equation}\label{Fduh}
\abs{\int_0^t  \wt{\mathcal{U}}(s,t) \wt F(s)\diff s} \le t \max_{s\le t} \norm{ \wt F (s)}_\infty \lesssim t\log N.
\end{equation} 
 This, together with~\eqref{marti} and the contraction property of $\wt {\mathcal{B}}$  implies
 from~\eqref{duhamel} that
  \[
      \norm{ u(t) -u(0)}_\infty \lesssim N^{-3/4+\om_1} + t\log N \lesssim N^{-1/2+2\om_1}
 \]
 with very high probability. Recalling the definition of $u$ and~\eqref{ovgam}, we get~\eqref{rrr} since
 \[
    \norm{\wt z(t) -\wt z (0)}_\infty \le \norm{ u(t) -u(0)}_\infty + \norm{ \wh \gamma_{y}^*(t) - \wh \gamma_y^*(0)}_\infty \lesssim N^{-1/2+2\om_1}.
 \]
This completes the proof of 
the crude rigidity bound~\eqref{crudinitial}.

\subsubsection{Crude short range approximation.}\label{sec:crudeshort}

\bigskip
Now we turn to the proof of~\eqref{rgtilde}  by introducing 
 a short range approximation of the dynamics~\eqref{tildereq}. Fix an integer $L$.
Let $\mathring{z}_i =\mathring{z}_i(t)$ solve the {\bf $L$-localized short scale DBM}
\begin{equation}\label{zstareqnew}
  \diff \mathring{z}_i =\sqrt{ \frac{2}{N}} \diff B_i+
   \frac{1}{N}\sum_{j: \abs{j-i}\le L}   \frac{1}{\mathring{z}_i-\mathring{z}_j}\diff t
 + \Bigg[  \frac{1}{N}\sum_{j: \abs{j-i}> L}   \frac{1}{\ov\gamma_i^*-\ov\gamma_j^*}
 +\Phi(t) 
 \Bigg] \diff t  
\end{equation}
for each $1\le \abs{i}\le N$ and
with initial data $\mathring{z}_i(0)\defeq  \wt  z_i(0)$, where we recall that \(\Phi\) was defined in~\eqref{Phidef}.
Then, we have the following comparison:
\begin{lemma}\label{lm:RR} Fix $\alpha\in [0,1]$.
Assume that
\begin{equation}\label{initialcrude}
\max_{1\le \abs{i}\le N}
    \abs{\wt z_i(0,\alpha)-\ov \gamma_{i}^*(0)}\lesssim N^{-1/2+2\om_1}.
\end{equation}
Consider the short scale DBM~\eqref{zstareqnew} with a range $L=N^{1/2+C_1\om_1}$
with a constant $10 \le C_1 \ll  C_*$, in particular $L$ is much smaller than $i_*$.
Then we have a weak uniform comparison
\begin{equation}\label{crudinitialhat}
  \sup_{0\le t\le t_*} \max_{1\le \abs{i}\le N}  \abs{ \mathring{z}_i(t,\alpha) - \wt z_i (t,\alpha)}\lesssim N^{-1/2+2\om_1}, 
\end{equation}
and a stronger comparison for small $i$:
\begin{equation}\label{goodrighat}
  \sup_{0\le t\le t_*} \max_{1\le \abs{i}\le i_*}  \abs{ \mathring{z}_i(t,\alpha) - \wt z_i(t,\alpha)}\lesssim  
  N^{-3/4+C\om_1}, 
\end{equation}
both with very high probability.
\end{lemma}

\begin{proof}
For any fixed $\alpha\in [0,1]$ and for all $0\le t\le t_*$, set $w\defeq w(t,\alpha)=\mathring{z}(t,\alpha) -\wt z(t,\alpha)$ and subtract~\eqref{zstareqnew} and~\eqref{tildereq} to get
\[
 \partial_t w = \mathring{\mathcal{B}}_1 w +\mathring{F},
\]
where
\[
(\mathring{\mathcal{B}}_1f)_i  \defeq  \frac{1}{N}\sum_{j: \abs{j-i}\le L}   \frac{f_j-f_i}{(\mathring{z}_i-\mathring{z}_j)(\wt z_i-\wt z_j)},
\qquad
       \mathring{F}_i\defeq  \frac{1}{N}\sum_{j: \abs{j-i}> L} \Big[ \frac{1}{\ov\gamma_i^*-\ov\gamma_j^*} -
\frac{1}{\wt z_i-\wt z_j}\Big].
\]
We estimate
\[  \abs{\mathring{F}_i}  \le \frac{1}{N}  \sum_{j: \abs{j-i}> L} \frac{\abs{\wt z_i-\ov \gamma_i^*}   + \abs{\wt z_j-\ov\gamma_j^*}}{(\ov\gamma_i^*-\ov\gamma_j^*)(\wt z_i-\wt z_j)}   \lesssim  \frac{N^{-1/2+2\om_1}}{N}  \sum_{j: \abs{j-i}> L} \frac{1}{(\ov\gamma_i^*-\ov\gamma_j^*)(\wt z_i-\wt z_j)}, \]
where we used the
 crude rigidity~\eqref{crudinitial} (applicable by~\eqref{initialcrude}), and we chose $C_1$ 
 in $L=N^{1/2+C_1\om_1}$ large enough so that $\abs{\ov\gamma_i^* -\ov\gamma_j^*}$ for 
 any $\abs{i-j}\ge L$ be much 
 bigger than the rigidity scale $N^{-1/2+2\om_1}$ in~\eqref{crudinitial}. This is guaranteed since
 \[
   \abs{\ov\gamma_i^* -\ov\gamma_j^*} 
   =\alpha \abs{\wh\gamma_{x,i}^* -\wh\gamma_{x,j}^*} +(1-\alpha) \abs{\wh\gamma_{y,i}^* -\wh\gamma_{y,j}^*} \gtrsim \frac{\abs{i-j}}{N} 
   \gtrsim N^{-1/2+C_1\om}
 \]
 with very high probability.
 By this choice of $L$ we have $\abs{\wt z_i-\wt z_j}\sim  \abs{\ov\gamma_i^* -\ov\gamma_j^*}$ and 
 therefore
 \begin{equation}\label{fcrude}
    \abs{\mathring{F}_i}  \lesssim  \frac{N^{-\frac{1}{2}+2\om_1}}{N}  \sum_{j: \abs{j-i}> L} \frac{1}{(\ov\gamma_i^*-\ov\gamma_j^*)^2}\lesssim N^{1/2+2\om_1}
    \sum_{j: \abs{j-i}> L} \frac{1}{\abs{i-j}^2} \lesssim N^{-(\frac{1}{2}C_1-2) \om_1}\le 1, \quad \forall \abs{i}\le N.
  \end{equation}
   Since $\mathcal{B}_1$ is positivity preserving, its evolution is a contraction, so by Duhamel formula, similarly to~\eqref{duhamel}, we get
\[
 \norm{\mathring{z} (t) -\wt z(t) }_\infty= \norm{w(t) }_\infty\le \norm{ w(0)}_\infty + t \max_{s\le t}\norm{ \mathring{F}(s)}_\infty\lesssim N^{-1/2+\om_1}
\]
with very high probability.
  
Next, we proceed with the proof of~\eqref{goodrighat}.
  
  In fact, for $1\le\abs{i}\le 2 i_*$, with $i_*$ much bigger than $L$, we have a better bound:
  \begin{equation}\label{Festim}
    \abs{\mathring{F}_i}   \lesssim  \frac{N^{-\frac{1}{2}+2\om_1}}{N}  \sum_{j: \abs{j-i}> L} \frac{1}{(\ov\gamma_i^*-\ov\gamma_j^*)^2}\lesssim 
    \sum_{j: \abs{j-i}> L} \frac{N^{2\om_1}}{\abs[1]{ \abs{i}^{3/4}-\abs{j}^{3/4}}^2} \lesssim N^{-\frac{1}{4} - (\frac{1}{2}C_1-2) \om_1}\le N^{-\frac{1}{4}}, \quad \abs{i}\le 2i_*,
  \end{equation}
 which we can use  to get the better bound~\eqref{goodrighat}.  
To do so, we define a continuous interpolation $v(t,\beta)$  between $\wt z$ and $\mathring{z}$. More precisely, for any fixed $\beta\in [0,1]$
we set $v(t,\beta)= \{ v(t,\beta)_i\}_{i=-N}^N$ as the solution to the 
SDE
\begin{equation}\label{uu1}
\begin{split}
 \diff v_i = & \sqrt{ \frac{2}{N}} \diff B_i+  \frac{1}{N}\sum_{j: \abs{j-i}\le L}   \frac{1}{ v_i- v_j}\diff t
 + \Phi_\alpha(t)
 \diff t \\
& + \frac{1-\beta}{N}\sum_{j: \abs{j-i}> L}   \frac{1}{\wt z_i-\wt z_j} \diff t +
  \frac{\beta}{N}\sum_{j: \abs{j-i}> L}  \frac{1}{\ov\gamma_i^*-\ov\gamma_j^*}  \diff t
\end{split}
\end{equation}
with initial condition $v(t=0, \beta)= (1-\beta) \wt z_i(0) + \beta\mathring{z}_i(0)$. Clearly $v(t, \beta=0) = \wt z(t)$
and $v(t, \beta=1) = \mathring{z}(t)$.

Differentiating in $\beta$, for $u\defeq u(t,\beta)= \partial_\beta v(t,\beta)$ we obtain the SDE
\begin{equation}\label{v1}
 \diff u_i =   (\mathcal{B}^v u)_i\diff t + \mathring{F}_i\diff t ,
 \qquad \mbox{with} \quad (\mathcal{B}^v f)_i\defeq  \frac{1}{N}\sum_{j: \abs{j-i}\le L}   \frac{f_j-f_i}{ (v_i- v_j)^2},
\end{equation}
with initial condition $u(t=0, \beta) = \mathring{z}(0)-\wt z(0)=0$. By the contraction property 
of the heat evolution kernel $\mathcal{U}^v$ of $\mathcal{B}^v$, 
with a simple Duhamel formula, we have for any fixed $\beta$
\begin{equation}\label{rigu}
   \sup_{0\le t\le t_{*}} \norm{ u(t,\beta)}_\infty \le
       t_{*} \norm[1]{\mathring{F} }_\infty \le N^{-1/2+\frac{3}{2}\om_1},
\end{equation}
with very high probability,
where we used~\eqref{fcrude}. 
After integration in $\beta$ we  get
\begin{equation}\label{rigv}
  \norm{ v(t,\beta) -\ov \gamma^*(t)}_\infty \le \norm{ v(t,0) -\ov \gamma^*(t)}_\infty + 
\norm{ \int_0^\beta u(t,\beta')\diff \beta' }_\infty, \qquad 0\le t\le t_*, \quad \beta\in [0,1].
\end{equation}
From~\eqref{rigu} we have
\begin{equation}\label{markov}
   \E \norm{   \int_0^\beta u(t,\beta')\diff \beta' }_\infty^p \le  \int_0^\beta \E \norm{ u(t,\beta')}^p \diff \beta' \lesssim \big(N^{-1/2+\frac{3}{2}\om_1}\big)^p
\end{equation}
for any exponent $p$. Hence, using a high moment Markov inequality, we have
\begin{equation}\label{markov1}   
   \P \Bigg(\norm{   \int_0^\beta u(t,\beta')\diff \beta' }_\infty  \ge N^{-1/2+\frac{3}{2}\om_1+\xi}  \Bigg) \le N^{-D}
\end{equation}
for any (large) $D>0$ and (small) $\xi>0$ by choosing $p$ large enough.
Since $v(t,0) = \wt z(t)$, for which we have rigidity in~\eqref{crudinitial}, by~\eqref{rigv} and~\eqref{markov1} we conclude that
\begin{equation}\label{rigv1}
  \sup_{0\le t\le t_*}\norm{ v(t,\beta) -\ov \gamma^*(t)}_\infty  \lesssim N^{-\frac{1}{2}+2\om_1}
\end{equation}
with very high probability for any $\beta\in [0,1]$.

In particular $L$ is much larger than the rigidity scale of $v=v(t,\beta)$. This means that
\[
  \abs{  \abs{ v_i-v_j} - \abs{ \ov \gamma^*_i - \ov \gamma^*_j} }   \lesssim N^{-\frac{1}{2}+2\om_1}
\]
and $\abs{ \ov \gamma^*_i - \ov \gamma^*_j}  \gtrsim \frac{\abs{i-j}}{N}\ge N^{-\frac{1}{2} + C_1\om_1}
\gg N^{-\frac{1}{2}+2\om_1}$ whenever $\abs{i-j}\ge L$, so we have
\begin{equation}\label{vvgg}
   \abs{ v_i-v_j} \sim \abs{ \ov \gamma^*_i - \ov \gamma^*_j}, \qquad \abs{i-j}\ge L.
\end{equation}
Since $i_*$ is much bigger than $L$ and $L$ is much larger
than the rigidity scale of $v_i(t,\beta)$ in the sense of~\eqref{vvgg},
 the heat evolution kernel $\mathcal{U}^{v}$ satisfies the following finite speed of propagation estimate (the proof is given in Appendix~\ref{FSPS}):
\begin{lemma}\label{FSP} With the notations above we have
\begin{equation}\label{b2fs}
  \sup_{0\le s\le t\le t_*}\big[   \mathcal{U}^{v}_{pi} +  \mathcal{U}^{v}_{ip}\big]  \le N^{-D}, \qquad
  1\le \abs{i}\le i_*, \quad \abs{p}\ge 2 i_*
\end{equation}
for any $D$ if $N$ is sufficiently large.
\end{lemma}

Using a Duhamel formula again, for any fixed $\beta$,  we have
\[
   u_i(t) = \sum_p \mathcal{U}^{v}_{ip} u_p(0) + \int_0^t \sum_p \mathcal{U}^{v}_{ip}(s,t) \mathring{F}_p(s)\diff s.
\]
We can split the summation and estimate
\[
   \abs{u_i(t)} \le \Big[ \sum_{\abs{p}\le 2i_*} + \sum_{\abs{p}> 2i_*}\Big]  \mathcal{U}^{v}_{ip} \abs{u_p(0)}
    + \int_0^t \Big[\sum_{\abs{p}\le 2i_*} +  \sum_{\abs{p}> 2i_*} \Big]\mathcal{U}^{v}_{ip}(s,t) \abs[1]{\mathring{F}_p(s)}  \diff s.
\]
For $\abs{i}\le i_*$,
the terms with $\abs{p}> 2i_*$ are negligible by~\eqref{b2fs} and the trivial bounds~\eqref{fcrude} and~\eqref{rigu}.
For $1\le \abs{p}\le 2i_*$ we use the improved bound~\eqref{Festim}.
This gives
\[
   \abs{u_i(t,\beta)} \le \max_{1\le \abs{j}\le 2i_*} \abs{u_j(0,\beta)} +  N^{-3/4 + \om_1} =N^{-3/4 + \om_1} , \qquad \abs{i}\le i_*,
\]
since $u(t=0, \beta)=0$.
Integrating from $\beta=0$ to $\beta=1$, and recalling that $v(\beta=0)= \wt z$ and $v(\beta=1)=\mathring{z}$, by high moment Markov inequality, we conclude
\[
   \abs{\wt z_i(t) -\mathring{z}_i(t) }\lesssim 
     N^{-\frac{3}{4} + \om_1}, \qquad 1\le \abs{i}\le i_*,
\] with very high probability.
This yields~\eqref{goodrighat} and completes the proof of Lemma~\ref{lm:RR}. 

We remark that it would have been sufficient to require that  $\abs{\wt z_j(0) -\mathring{z}_j(0) }\le  N^{-\frac{3}{4} + \om_1}$ for
all $1\le \abs{j}\le 2i_*$ instead of setting $\mathring{z}(0)\defeq \wt z(0)$ initially.  Later in Section~\ref{sec:hatrig} we will use 
a similar finite speed of propagation mechanism to show that changing  the initial condition for large indices has negligible effect.
\end{proof}

\subsubsection{Refined rigidity for small \texorpdfstring{$\lvert i\vert$}{i}.}

Finally, in the last but main step of the proof of~\eqref{rgtilde} in 
Proposition~\ref{prop:option2} we compare $\mathring{z}_i$ with $\ov \gamma^*_i$ for small $\abs{i}$ with a  much
higher precision than the crude bound $N^{-\frac{1}{2}+C\om_1}$ which directly follows from~\eqref{crudinitialhat} and~\eqref{crudinitial}.
Notice that we use the semiquantiles for
comparison since $\ov \gamma_{i}^* \in [\ov \gamma_{i-1}, \ov \gamma_{i}]$   and $\ov \gamma_{i}^*$ is typically close to the midpoint of this interval.
In particular,   $\ov \rho_{t}(\ov \gamma_{i}^*(t))$ is never zero, in fact we have $\ov \rho_{t}(\ov \gamma_{i}^*(t))\ge cN^{-1/3}$, because
by band rigidity quantiles may fall exactly at spectral edges, but semiquantiles cannot. This lower bound makes 
the semiquantiles much more convenient reference points than the quantiles.

\begin{proposition}\label{option2prop} Fix $\alpha\in [0,1]$, then with  the notations above for the localized DBM $\mathring{z}(t,\alpha)$ on short scale $L= N^{1/2+C_1\om_1}$
with $10\le C_1 \le \frac{1}{10} C_*$, defined in~\eqref{zstareqnew}, we have
\begin{equation}\label{rgy}
\abs{ \big( \mathring{z}_i(t,\alpha) -\ov \gamma_{i}^*(t) \big) - \big( \mathring{z}_i(0,\alpha) -\ov \gamma_{i}^*(0) \big)} \le 
N^{-3/4+C\om_1}, \qquad  1\le \abs{i}\le i_* = N^{\frac{1}{2}+C_* \om_1}
\end{equation}
with very high probability. 
\end{proposition}

Combining~\eqref{rgy} with 
\eqref{goodrighat} and noticing that
\[   
\mathring{z}_i(0,\alpha) -\ov \gamma_{i}^*(0) = \wt z_i(0,\alpha) -\ov \gamma_{i}^*(0) 
= O\left(\frac{N^\xi N^\frac{\omega_1}{6}}{N^\frac{3}{4} \abs{i}^\frac{1}{4}}\right)
\]
for any $\xi> 0$ with very high probability by~\eqref{optimalinit}, 
we obtain~\eqref{rgtilde} and complete the proof of Proposition~\ref{prop:option2}.
\end{proof}

\begin{proof}[Proof of Proposition~\ref{option2prop}]
We recall from~\eqref{gammader} that
\begin{equation}\label{gammader1}
   \frac{\diff}{\diff t} \ov \gamma_{i}^*(t) = \alpha \big[-\Re m_{x,t}(  \gamma_{x, i}^*(t) )+\Re m_{x,t}(\ed_{x,t}^+)\big]
   +(1-\alpha)\big[ -\Re m_{y,t}(  \gamma_{y, i}^*(t) )+\Re m_{y,t}(\ed_{y,t}^+)\big].
\end{equation}
Next, we define a dynamics that interpolates between $\mathring{z}_i(t,\alpha)$ and $\ov\gamma_{i}^*(t)$, i.e.
between~\eqref{zstareqnew} and~\eqref{gammader1}. Let $\beta \in [0,1]$
and for any fixed $\beta$ define the process $v=v(t, \beta)=\{ v_i(t, \beta)\}_{i=-N}^N$ as the solution of the
following interpolating DBM
\begin{equation}\label{veq}
\begin{split}
  \diff v_i =&\beta \sqrt{\frac{2}{N}} \diff B_i+
   \frac{1}{N}\sum_{j: \abs{j-i}\le L}   \frac{1}{v_i-v_j}\diff t +\beta\Bigg[  \frac{1}{N}\sum_{j: \abs{j-i}> L}   \frac{1}{\ov\gamma_i^*-\ov\gamma_j^*}\diff t
 +\Phi(t) 
 \Bigg]\diff t \\
& + (1-\beta) \Bigg[ \frac{\diff}{\diff t} \ov \gamma_{i}^*(t)  
-\frac{1}{N}\sum_{j: \abs{j-i}\le L}   \frac{1}{\ov\gamma_{i}^*-\ov\gamma_{j}^*}
\Bigg]\diff t, \qquad 1\le\abs{i}\le N,
 \end{split}
\end{equation}
with initial condition $v_i(0, \beta)\defeq  \beta \mathring{z}_i(0) + (1-\beta)\ov\gamma_{i}^*(0)$.
Notice that 
\begin{equation}\label{extreme}
v_i(t, \beta=0)= \ov \gamma_{i}^*(t), \qquad v_i(t, \beta=1) = \mathring{z}_i(t).
\end{equation}
Here we use the same letter $v$ as in~\eqref{uu1} within the proof of Lemma~\ref{lm:RR},  but this is now a new interpolation.
Since both appearances of the letter $v$ are used only  within the proofs of separate lemmas, this should not cause any confusion.
The same remark applies to the letter $u$ below.

Let $u\defeq u(t,\beta)= \partial_\beta v(t,\beta)$, then it satisfies the equation
\begin{equation}\label{SDEu}
   \diff u_i = \sqrt{\frac{2}{N}} \diff B_i+ \sum_{j\ne i} \mathcal{B}_{ij}(u_i-u_j)\diff t + F_i \diff t, \qquad 1\le\abs{i}\le N,
\end{equation}
with a time dependent short range kernel  (omitting the time argument and the $\beta$ parameter)
\begin{equation}\label{mcb}
 \mathcal{B}_{ij}(t)= \mathcal{B}_{ij}\defeq  -\frac{1}{N}\frac{{\bf 1}(\abs{i-j}\le L)}{(v_i-v_j)^2}
\end{equation}
and external force
\begin{equation}\label{Fdef}
\begin{split}
 F_i= F_i(t)\defeq  &  - \frac{1}{N}\sum_{j} \frac{1}{\ov \gamma_{j}^*(t) -\ov \gamma_{i}^*(t)} 
     + \alpha \Re m_{x,t}(  \gamma_{x, i}^*(t) ) +(1-\alpha)\Re m_{y,t}(  \gamma_{y, i}^*(t) ) -h(t,\alpha),
\qquad 1\le \abs{i}\le N.
\end{split}
\end{equation}
 Since the density $\ov\rho$ is regular, at least near the cusp regime, 
 we can replace the sum over $j$ with 
 an integral with very high precision for small $i$;
 this integral is $\Re \ov m(\ov \ed^++ \ov \gamma_i^*) $.
 A simple rearrangement of various terms yields
 \begin{equation}\label{Fdef3}
\begin{split}
F_i= &  \Bigg[ \Re \ov m(\ov \ed^++\ov \gamma_i^*)- \frac{1}{N}\sum_{j} \frac{1}{\ov \gamma_{j\ne i}^* -\ov \gamma_{i}^*} \Bigg]
 - (1-\alpha) D_{y,i} - \alpha D_{x,i} + O(N^{-1}),
 \end{split}
\end{equation}
with
\[
   D_{r, i} \defeq \Re \Big[ \big( \ov m(\ov \ed^++\ov \gamma_i^*) - 
   \ov m(\ov \ed^+)\big) - \big(  m_r( \gamma_{r,i}^*) -  m_r(\ed_{r}^+)\big)\Big] , \qquad r=x,y,
\]
 where we used the formula for  $h$ from~\eqref{hhdef} and the definition of $\Phi$ from~\eqref{Phidef}.
The choice of the shift $h$ was governed by  the idea to replace the last three terms in~\eqref{Fdef} by $\Re \ov m(\ov \ed^++\ov \gamma_i^*) $.
However, the shift cannot be $i$ dependent as it would result in an $i$-dependent shift in the definition of $\wt z_i$, see
\eqref{def:rtilde}, which would mean that the differences (gaps)  of  the processes $z_i$ and $\wt z_i$  are not the same.
Therefore, we defined the shift $h(t)$ by the similar formula evaluated at the edge, justifying the choice~\eqref{hhdef}.
The discrepancy is expressed by $D_{x,i}$  and $D_{y,i}$ which are small.
Indeed we have, for $r=x,y$ and $1\le \abs{i}\le 2i_*$ that
\begin{equation}
\begin{split}
 \abs{ D_{r,i}} \le & \abs{\Re \Big[ \big( \ov m(\ov \ed^+ + \wh\gamma_{r,i}^*) - \ov m(\ov \ed^+)\big) - 
 \big(  m_r(\ed_{r}^+ + \wh\gamma_{r,i}^*) -  m_r(\ed_{r}^+)\big)\Big]} + \abs{  \ov m( \ov \ed^++\wh\gamma_{r,i}^*)
 - \ov m(\ov \ed^++\ov \gamma_i^*)} \\
 \lesssim & \abs{\wh\gamma_{r,i}^*}^{1/3}\Bigl[\abs{\wh\gamma_{r,i}^*}^{1/3}+ N^{-\frac{1}{6} +\frac{\om_1}{3}}\Bigr]
 \abs{\log\abs{\wh\gamma_{r,i}^*}}+N^{-\frac{11}{36}+\om_1}
 +\frac{ \abs{\wh\gamma_{r,i}^*-\ov \gamma_i^*}}{\ov\rho( \ov \gamma_i^*)^2} \label{Dest} \\
 \lesssim & \left[ \left(\frac{\abs{i}}{N}\right)^{1/2} + \left(\frac{\abs{i}}{N}\right)^{1/4}N^{-\frac{1}{6} +\frac{\om_1}{3}} \right](\log N)  +N^{-\frac{11}{36}+\om_1}
 + \frac{ \Big(\frac{\abs{i}}{N}\Big) + \Big(\frac{\abs{i}}{N}\Big)^{3/4}N^{-\frac{1}{6} +\om_1}}{ \Big(\frac{\abs{i}}{N}\Big)^{1/2}}
 \lesssim N^{-\frac{1}{4}+C\om_1},
 \end{split}
\end{equation}
where  from the first to the second line we used~\eqref{Re m gap} 
and the bound on the  derivative  of $\ov m$, see~\eqref{weakm}.
In the last inequality we used~\eqref{gamma gap} to estimate $\abs{\wh\gamma_{r,i}^*} \lesssim (\abs{i}/N)^{3/4} N^{C\om_1}$
and similarly $\abs{\wh\gamma_{r,i}^*-\ov \gamma_i^*} $  in
the regime $\abs{i}\le i_* = N^{\frac{1}{2}+C_*\om_1}$, furthermore we used that
 $\ov\rho( \ov \gamma_i^*) \ge (\abs{i}/N)^{1/4}$ and also
 $\abs{\ov \gamma_i^*}\ge c/N$, since a semiquantile is always away from the edge.

Let $\mathcal{U}(s,t)$ be the fundamental solution of the heat  evolution with  kernel $\mathcal{B}$
from~\eqref{mcb}. Similarly to~\eqref{duhamel},
 the solution to the SDE~\eqref{SDEu} is given by
\begin{equation}\label{duhamel1}
      u (t) = \mathcal{U}(0,t) u + \sqrt{\frac{2}{N}} \int_0^t  \mathcal{U}(s,t) \diff B(s) + \int_0^t  \mathcal{U}(s,t) F(s)\diff s.
\end{equation}
The middle martingale term can be estimated as in~\eqref{marti}.
The last term in~\eqref{duhamel1} is estimated by
\begin{equation}\label{Fduh11}
\abs{\int_0^t  \mathcal{U}(s,t) F(s)\diff s} \le t \max_{0\le s\le t} \norm{ F (s)}_\infty.
\end{equation}

First we use these simple Duhamel bounds to obtain a 
crude rigidity bound on $v_i(t,\beta)$ by integrating the bound on $u$
\begin{equation}\label{56}
    \abs{ v_i(t, \beta) - v_i(t,\beta=0)} \le \beta \max_{\beta'\in[0,\beta]}  \abs{u_i(t, \beta')} \le \max_{\beta'\in[0,1]} \norm{ u(0,\beta')}_\infty
    + N^{-1/2+\om_1+\xi}, \qquad  1\le\abs{i}\le N,
\end{equation}
for any $\xi>0$ with very high probability, using~\eqref{marti},~\eqref{duhamel1},~\eqref{Fduh11} and that $\mathcal{U}$ is a
contraction. Note that in the first inequality of~\eqref{56} we used that it holds with very high probability by Markov inequality as in~\eqref{markov}-\eqref{markov1}.
We also used the trivial bound
\begin{equation}\label{Ftrivbound}
  \max_{0\le s\le t_*} \norm{F(s)}_\infty\lesssim \log L \sim \log N, 
\end{equation}
which easily follows from~\eqref{Fdef},\eqref{Dest} and the fact that $\abs{\ov \gamma_{j}^*(t) -\ov \gamma_{i}^*(t)}\gtrsim \abs{i-j}/N$.

Recalling that $v_i(t,\beta=0)=\ov \gamma_{i}^*(t)$ and $u_i(0,\beta')= \mathring{z}_i(0) -\ov \gamma_i^*(0)$,
together with~\eqref{crudinitialhat} and~\eqref{crudinitial}, by~\eqref{56}, we obtain the  crude rigidity
\begin{equation}\label{crudrig}
 \abs{ v_i(t, \beta) -\ov \gamma_{i}^*(t)} \le N^{-\frac{1}{2}+2\om_1}, \qquad 1\le \abs{i}\le N,
\end{equation} with very high probability.

The main technical result is a considerable improvement of the bound~\eqref{crudrig} at least for $i$ near the cusp regime.
This is the content of
 the following proposition whose proof is postponed:

\begin{proposition}\label{prop:F}
The vector  $F$  defined in~\eqref{Fdef} satisfies the bound
\begin{equation}\label{Fest}
\max_{s\le t_*}\abs{ F_i (s)} \le N^{-\frac{1}{4}+C\om_1}, \qquad  1\le \abs{i}\le 2i_*.
 \end{equation}
\end{proposition}

Since $i_*$ is much bigger than $L=N^{\frac{1}{2} +C_1\om_1}$ with a large $C_1$,  and we have the rigidity~\eqref{crudrig} on scale much smaller than $L$,
similarly to Lemma~\ref{FSP}, we
have the following finite speed of propagation result. The proof is identical to that of Lemma~\ref{FSP}.

\begin{proposition}\label{finsp}
 For the short range dynamics $\mathcal{U}=\mathcal{U}^\mathcal{B}$ 
defined by the operator~\eqref{mcb}:
\begin{equation}\label{Uest}
   \sup_{0\le s\le t\le t_*} \Big[ \mathcal{U}_{pi}(s,t) +\mathcal{U}_{ip}(s,t)\Big]\le N^{-D}, \qquad 1\le \abs{i}\le  i_*, \quad \abs{p}\ge 2i_*.
\end{equation}
for any $D$ if $N$ is sufficiently large. \qed
\end{proposition}

Armed with these two propositions, we can easily complete the proof of Proposition~\ref{option2prop}. For any $1\le \abs{i}\le i_*$ we have
from~\eqref{duhamel}, using~\eqref{Utriv},~\eqref{marti},~\eqref{Uest} and that $\mathcal{U}$ is a contraction 
on $\ell^\infty$  that
\begin{equation}\label{ub}
\begin{split}
  \abs{u_i(t)}\le & N^{-3/4+\om_1 +\xi} + \sum_p \mathcal{U}_{ip} \abs{u_p(0)} +\int_0^t \sum_p \mathcal{U}_{ip}(s,t) \abs{F_p(s)} \diff s \\
  \le &  N^{-3/4+\om_1 +\xi}  + \max_{\abs{p}\le 2i_*} \abs{u_p(0)} + t \max_{0\le s\le t_*}\max_{\abs{p}\le 2i_*}\abs{ F_p (s)} + N^{-D}\max_{0\le s\le t}\norm{ F(s)}_\infty.
\end{split}
\end{equation}
The trivial bound~\eqref{Ftrivbound} together with~\eqref{Fest}  completes the proof of~\eqref{rgy}
by integrating back the bound~\eqref{ub} for $u=\partial_\beta v$ in $\beta$, using a high moment Markov inequality similar to~\eqref{markov}-\eqref{markov1}, and recalling~\eqref{extreme}.
This completes the proof  of Proposition~\ref{option2prop}.
\end{proof}

\subsubsection{Estimate of the forcing term.}

\begin{proof}[Proof of Proposition~\ref{prop:F}]
Within this proof we will  use  $\gamma_i\defeq  \ov\gamma_{i}(t)$, 
$\gamma_i^*\defeq \ov \gamma_{i}^*(t)$, $\rho=\ov \rho_{t}$, $m=\ov m_{t}$ and $\ed^+= \ov \ed_{t}^+$ for brevity.
For notational simplicity we may assume within this proof that $\ed^+=0$ by a simple shift.
 The key input is the following bound on the derivative of the density, proven in 
\cite{1804.07752} for self-consistent densities of Wigner type matrices 
\begin{equation}\label{rhoder}
  \abs{\rho'(x)}\le \frac{C}{\rho(x)[\rho(x) + \Delta^{1/3}]}, \qquad \abs{x}\le \delta_*
\end{equation}
where $\Delta=\ov\Delta_t$ is the length of the unique gap in the support of $\rho=\ov \rho_t$
 in a small neighbourhood of size $\delta_*\sim 1$ around $\ed^+=0$.
If there is no such gap, then we set $\Delta=0$ in~\eqref{rhoder}. 
 By the definition of the interpolated density $\overline{\rho}_t$ in~\eqref{rhozdef} clearly follows that it satisfies~\eqref{rhoder} by~\eqref{lemma holderhsh}.
Notice that~\eqref{rhoder} implies local H\"older continuity, i.e.~
\begin{equation}\label{goodholder}
   \abs{\rho(x)-\rho(y)}\le \min\big\{ \abs{x-y}^{1/3}, \abs{x-y}^{1/2} \Delta^{-1/6}\big\}
\end{equation}
for any $x, y$ in a small neighbourhood of the gap or the local minimum.

Throughout the entire proof we fix an $i$ with $1\le \abs{i}\le 2i_*$. For simplicity, we assume $i>0$, the case $i<0$ is analogous.
We rewrite $F_i$  from~\eqref{Fdef3} as follows
\begin{equation}\label{Fdef1}
  F_i = G_1+ G_2 +G_3+G_4
\end{equation}
with 
\[
    G_1\defeq \sum_{1\le \abs{j- i}\le L} \int_{\gamma_{j-1}}^{\gamma_{j}} \left[\frac{1}{x- \gamma_i^*} -
 \frac{1}{\gamma_j^* -\gamma_i^*} \right]\rho(x)\diff x, \qquad
  G_2\defeq  \int_{\gamma_{i-1}}^{\gamma_{i}} \frac{\rho(x)\diff x}{x- \gamma_i^*},
 \]
 \[
    G_3\defeq \sum_{\abs{j- i}>L} \int_{\gamma_{j-1}}^{\gamma_{j}} \left[\frac{1}{x- \gamma_i^*} -
 \frac{1}{\gamma_j^* -\gamma_i^*} \right]\rho(x)\diff x, 
   \qquad G_4\defeq - (1-\alpha) D_{y,i} - \alpha D_{x,i} + O(N^{-1}).
 \]
 The term $G_4$ was already estimated in~\eqref{Dest}.
In the following we will show separately that $\abs{G_a}\lesssim N^{-1/4}$, $a=1,2,3$.

\bigskip

\emph{Estimate of $G_3$}.  By elementary computations, using the crude rigidity~\eqref{crudinitial}, it follows that \[\abs{G_3}\lesssim \frac{N^{-\frac{1}{2}+2\omega_1}}{N} \sum _{j:\abs{j-i}>L} \frac{1}{(\gamma_i^*-\gamma_j^*)^2}. \]

Then, the estimate $\abs{G_3}\lesssim N^{-\frac{1}{4}}$ follows using the same computations as in~\eqref{Festim}.

\bigskip 

\emph{Estimate of $G_2$}. 
We write
\begin{equation}\label{secondterm}
 G_2=  \int_{\gamma_{i-1}}^{\gamma_{i}}  \frac{ \rho(x)\diff x} {x- \gamma_i^*}  = \int_{\gamma_{i-1}}^{\gamma_{i}}
    \frac{ \rho(x)-\rho(\gamma_i^*)} {x- \gamma_i^*} \diff x + 
    \rho(\gamma_i^*) \int_{\gamma_{i-1}}^{\gamma_{i}}    \frac{\diff x } {x- \gamma_i^*}
\end{equation}
and we will show that both summands are bounded by $CN^{-1/4}$.
We make the convention that if $\gamma_{i-1}$ is exactly at the left edge of a gap, then 
for the purpose of this proof we redefine it to be the right edge of the same gap
and similarly, if $\gamma_{i}$ is exactly at the right edge of the gap, then we set it
to be left edge.  This is just to make sure that $[\gamma_{i-1}, \gamma_{i}]$ is always included in the
support of $\rho$.

In the first integral we  use~\eqref{goodholder}
to get
\begin{equation}\label{ff}
 \abs{ \int_{\gamma_{i-1}}^{\gamma_{i}}     \frac{ \rho(x)-\rho(\gamma_i^*)} {x- \gamma_i^*} \diff x } \lesssim \min\big\{
  ( \gamma_{i}-\gamma_{i-1})^{1/3}, ( \gamma_{i}-\gamma_{i-1})^{1/2}\Delta^{-1/6} \big\}
  = O(N^{-1/4}). 
\end{equation}
Here we used that the local eigenvalue spacing (with the convention above) is bounded by
\begin{equation}\label{onegap}
   \gamma_{i} -\gamma_{i-1}\lesssim \max\Big\{ \frac{\Delta^{1/9}}{N^{2/3}}, \frac{1}{N^{3/4}} \Big\}.
\end{equation}

For the second integral in~\eqref{secondterm} 
 is an explicit calculation
\begin{equation}\label{f22}
    \rho (\gamma_i^*) \int_{\gamma_{i-1}}^{\gamma_{i}} 
    \frac{\diff x } {x- \gamma_i^*}  =  \rho(\gamma_i^*) \log \frac{\gamma_{i} - \gamma_i^*}{\gamma_i^*- \gamma_{i-1}}.
\end{equation}
Using the definition of the quantiles and~\eqref{goodholder}, we have
\[
 \frac{1}{2N} = \int_{\gamma_{i-1}}^{\gamma_i^*} \rho(x)\diff x =  \rho(\gamma_i^*) ( \gamma_i^*- \gamma_{i-1})
 + O\Big( \min\big\{ \abs{\gamma_i^* -\gamma_{i-1}}^{4/3}, \abs{\gamma_i^* -\gamma_{i-1}}^{3/2}\Delta^{-1/6} \big\}\Big),
\]
and similarly
\[
 \frac{1}{2N} = \int^{\gamma_{i}}_{\gamma_i^*} \rho(x)\diff x =  \rho(\gamma_i^*) ( \gamma_{i}-\gamma_i^*)
 + O\Big( \min\big\{ \abs{\gamma_i^* -\gamma_{i}}^{4/3}, \abs{\gamma_i^* -\gamma_{i}}^{3/2}\Delta^{-1/6} \big\}\Big).
\]
The error terms are comparable  and they are $O(N^{-1})$ using~\eqref{onegap}, thus, subtracting these two equations, we have
\[
  \abs{ (\gamma_{i} - \gamma_i^*) - (\gamma_i^*- \gamma_{i-1})} \lesssim 
  \frac{\min\big\{ \abs{\gamma_i^* -\gamma_{i}}^{4/3}, \abs{\gamma_i^* -\gamma_{i}}^{3/2}\Delta^{-1/6} \big\} }{\rho(\gamma_i^*)}.
\]
Expanding the logarithm in~\eqref{f22}, we have
\[
  \abs{  \rho (\gamma_i^*) \int_{\gamma_{i-1}}^{\gamma_{i}} 
    \frac{\diff x } {x- \gamma_i^*}}\lesssim \rho(\gamma_i^*) 
    \frac{\abs{ (\gamma_{i} - \gamma_i^*) - (\gamma_i^*- \gamma_{i-1})} }{\gamma_i^*- \gamma_{i-1}}
    \lesssim  \min\big\{ \abs{\gamma_i^* -\gamma_{i}}^{1/3}, \abs{\gamma_i^* -\gamma_{i}}^{1/2}\Delta^{-1/6} \big\}\lesssim N^{-1/4}
\]
as in~\eqref{ff}. This completes the estimate 
\begin{equation}\label{G2}
  \abs{G_2}\lesssim N^{-1/4}.
\end{equation}

\bigskip
\emph{Estimate of $G_1$}.
Fix $i>0$  and set $n=n(i)$ as follows
\begin{equation}\label{niidef}
  n(i)\defeq  \min \Big\{ n\in \N\; : \;  \min\big\{  \abs{\gamma_{i-n-1}-\gamma_i^*} , \abs{\gamma_{i+n}-\gamma_i^*}  \big\} \ge cN^{-3/4}
  \Big\}
\end{equation}
with some small constant $c>0$.

Next, we estimate $n(i)$. Notice that for $i=1$ we have $n(i)=0$.
If $i\ge 2$, then 
we notice that one can choose $c$ sufficiently small depending only on the model parameters, such
that
\begin{equation}\label{rhocomp}
  \frac{1}{2}\le \frac{\rho(x)}{\rho(\gamma_i^*)} \le 2 \; : \; \forall x \in 
  [\gamma_{i-n(i)-1}, \gamma_{i+n(i)}], \quad i\ge 2.
\end{equation}

Let
\[
  m(i)\defeq  \max \Big\{ m\in \N\; : \;  \frac{1}{2}\le \frac{\rho(x)}{\rho(\gamma_i^*)} \le 2\; : \; \forall x \in 
  [\gamma_{i-m-1}, \gamma_{i+m}]
  \Big\},
\]
then, in order to verify~\eqref{rhocomp}, we need to prove that $m(i)\ge n(i)$.

Then by a case by case calculation it follows that
\begin{equation}\label{ni}
  m(i) \ge c_1\abs{i},
\end{equation}
and thus
\begin{equation}\label{gammani}
  \min\Big\{  \abs{\gamma_{i-m(i)-1}-\gamma_i^*} , \abs{\gamma_{i+m(i)}-\gamma_i^*}  \Big\} \gtrsim 
  \max \Big\{ \Big(\frac{i}{N}\Big)^{2/3}\Delta^{1/9}, \Big(\frac{i}{N}\Big)^{3/4}\Big\} \ge c_2N^{-3/4}.
  \end{equation}
  with some $c_1, c_2$. Hence~\eqref{rhocomp} will hold if $c\le c_2$ is chosen  in the definition~\eqref{niidef}.
Notice that in these estimates it is important that the semiquantiles are always at a certain distance away
  from the quantiles.

Now we give an upper bound on $n(i)$ when $\gamma_i^*$ is near a (possible small) gap as in the proof above. 
The local eigenvalue spacing is
\begin{equation}\label{ggloc}
   \gamma_i-\gamma_i^* \sim \max\Big\{ \frac{\Delta^{1/9}}{N^{2/3}(i)^{1/3}}, \frac{1}{N^{3/4}(i)^{1/4}} \Big\},
\end{equation}
which is bigger than $cN^{-3/4}$ if $i\le \Delta^{1/3} N^{1/4}$. So in this case $n(i)=0$ and 
we may now assume that $i  \ge \Delta^{1/3} N^{1/4}$ and still $i\ge 2$.

Consider first the so-called \emph{cusp case}  when  $i\ge N \Delta^{4/3}$, in this case, as long as $n\le \frac{1}{2}i$, we have
\[ 
   \gamma_{i+n}-\gamma_i^* \sim \frac{n}{N^{3/4}(i+1)^{1/4}}.
\]
 This is bigger than $cN^{-3/4}$ if  $n\ge i^{1/4}$, 
thus we have $n(i)\le i^{1/4}$ in this case.

In the opposite case, the so-called \emph{edge case}, 
 $i\le N \Delta^{4/3}$, which together with the above assumption $i  \ge \Delta^{1/3} N^{1/4}$
also implies that $\Delta\ge N^{-3/4}$. In this case, as long as $n\le \frac{1}{2}i$, we have
\[ 
   \gamma_{i+n}-\gamma_i^* \sim \frac{n\Delta^{1/9}}{N^{2/3}i^{1/3}}.
\]
This is bigger than $cN^{-3/4}$ if  $n\ge \Delta^{-1/9} N^{-1/12}i^{1/3} $.
So we have $n(i) \le \Delta^{-1/9} N^{-1/12} i^{1/3} \le i^{1/3}$ in this case.

We split the sum in the definition of $G_1$, see~\eqref{Fdef1},  as follows:
\begin{equation}\label{nidef}
  G_1=\sum_{1\le \abs{j-i}\le L} \int_{\gamma_{j-1}}^{\gamma_{j}}
     \frac{x-\gamma_j^*}{(\gamma_i^*-\gamma_j^*)  (x-\gamma_i^*)}  \rho(x)\diff x =
     \Big(  \sum_{ n(i)<\abs{j-i}\le L} +  \sum_{  1\le \abs{j-i}\le  n(i)} \Big)\defqe  S_1+S_2.
\end{equation}
For the first sum we use $\abs{x-\gamma_j^*}\le \gamma_{j+1}^*-\gamma_{j}^*$,  $\abs{\gamma_i^*-x}\sim \abs{\gamma_i^* -\gamma_j^*}$. 
Moreover, we have
\begin{equation}\label{rg1}
   \rho(\gamma_i^*) (\gamma_{i}-\gamma_{i-1}) \sim \frac{1}{N}
\end{equation}
from the definition of the semiquantiles. Thus 
we restore the integration in the first sum $S_1$ and estimate
\begin{equation}\label{S1}
  \abs{S_1}\lesssim  \frac{1}{N} \Big[ \int_{-\infty}^{\gamma_{i-n(i)-1}} + \int_{\gamma_{i+n(i)}}^\infty \Big]\frac{\diff x}{\abs{x-\gamma_i^*}^2}
  \lesssim \frac{1}{N}\Big[ \frac{1}{\abs{\gamma_{i-n(i)-1}-\gamma_i^*}} + \frac{1}{\abs{\gamma_{i+n(i)}-\gamma_i^*}} \Big]
  \le CN^{-1/4}.
\end{equation}
In the last step we used  the definition of $n(i)$.

Now we consider $S_2$. Notice that this sum is non-empty only if $n(i)\ne 0$
 In this case to estimate $S_2$  we have to 
symmetrize. Fix $1\le n\le n(i)$, assume $i> n$ and consider together
\[
\int_{\gamma_{i-n-1}}^{\gamma_{i-n}}
     \frac{x-\gamma_{i-n}^*}{(\gamma_i^*-\gamma_{i-n}^*)  (x-\gamma_i^*)}  \rho(x)\diff x +
  \int_{\gamma_{i+n-1}}^{\gamma_{i+n}}
     \frac{x-\gamma_{i+n}^*}{(\gamma_i^*-\gamma_{i+n}^*)  (x-\gamma_i^*)}  \rho(x)\diff x   
\]
\begin{equation} \label{symm}
=  \frac{1}{\gamma_i^*-\gamma_{i-n}^*} \int_{\gamma_{i-n-1}}^{\gamma_{i-n}}
     \frac{x-\gamma_{i-n}^*}{x-\gamma_i^*}  \rho(x)\diff x +
     \frac{1}{\gamma_i^*-\gamma_{i+n}^*} \int_{\gamma_{i+n-1}}^{\gamma_{i+n}}
     \frac{x-\gamma_{i+n}^*}{x-\gamma_i^*}  \rho(x)\diff x
\end{equation}
\[
  =  \frac{1}{N}\Big[ \frac{1}{\gamma_i^*-\gamma_{i-n}^*} + \frac{1}{\gamma_i^*-\gamma_{i+n}^*}\Big]
+  \Bigg[ \int_{\gamma_{i-n-1}}^{\gamma_{i-n}}
     \frac{ \rho(x)\diff y}{x-\gamma_i^*} +  \int_{\gamma_{i+n-1}}^{\gamma_{i+n}} \frac{ \rho(x)\diff x}{x-\gamma_i^*}\Bigg] \defqe  B_1(n)+B_2(n).
\]
We now use $\frac{1}{3}$-H\"older regularity
\[
    \rho(x) =\rho(\gamma_i^*)+ O\Big( \abs{x-\gamma_i^*}^{1/3}\Big).
\]
We thus have
\begin{equation}\label{gg0}
\sum_{n\le n(i)} \int_{\gamma_{i-n-1}}^{\gamma_{i-n}}
     \frac{ \rho(x)\diff y}{x-\gamma_i^*}  =\sum_{n\le n(i)}
     \rho(\gamma_i^*) \log \frac{\gamma_{i-n-1}-\gamma_i^*}{\gamma_{i-n}-\gamma_i^*} + 
     O\Big( \int_{\gamma_{i-n(i)-1}}^{\gamma_{i+n(i)}}  \frac{\diff x}{\abs{x-\gamma_i^*}^{2/3}}  \Big)
\end{equation}
and similarly
\begin{equation}\label{gg1}
\sum_{n\le n(i)} \int_{\gamma_{i+n-1}}^{\gamma_{i+n}}
     \frac{ \rho(x)\diff y}{x-\gamma_i^*}  =\sum_{n\le n(i)}
     \rho(\gamma_i^*) \log \frac{\gamma_{i+n-1}-\gamma_i^*}{\gamma_{i+n}-\gamma_i^*} + 
      O\Big( \int_{\gamma_{i-n(i)-1}}^{\gamma_{i+n(i)}}  \frac{\diff x}{\abs{x-\gamma_i^*}^{2/3}}  \Big).
\end{equation}
The error terms are bounded by $CN^{-1/4}$  using~\eqref{niidef} and therefore we have
\[
  \sum_{n\le n(i)}B_2(n)=
    \sum_{n\le n(i)} \rho(\gamma_i^*)\Big[  \log \frac{\gamma_i^*-\gamma_{i-n-1}}{\gamma_i^*-\gamma_{i-n}}  -  \log \frac{\gamma_{i+n}
    -\gamma_i^*}{\gamma_{i+n-1}-\gamma_i^*}\Big] 
     +O(N^{-1/4})
\]
\[
  =  \sum_{n\le n(i)} \rho(\gamma_i^*)\Big[  \log \frac{\gamma_i^*-\gamma_{i-n-1}}{\gamma_{i+n} - \gamma_i^*} 
   +  \log \frac{\gamma_{i+n-1}-\gamma_i^*}{\gamma_i^*-\gamma_{i-n}}\Big]
  +O(N^{-1/4}).
\]
We now use the bound
\begin{equation}\label{hcont}
    \abs{\rho(x)-\rho(\gamma_i^*)} \lesssim \frac{\abs{x-\gamma_i^*}}{\rho(\gamma_i^*)^2+ \rho(\gamma_i^*)\Delta^{1/3}}, 
    \qquad x\in [\gamma_{i-n(i)-1}, \gamma_{i+n(i)}],
\end{equation}
which  follows from the derivative bound~\eqref{rhoder} if $\epsilon$
in the definition of $i_*=\epsilon N$ is chosen sufficiently small, depending on $\delta$
since throughout the proof $1\le \abs{i}\le 2i_*$ and $n(i)\ll  i_*$.
 
Note that
\begin{equation}\label{applh}
   \frac{n}{N} = \int_{\gamma_{i-n}}^{\gamma_i} \rho(x) \diff x = \rho(\gamma_i^*)[\gamma_i -\gamma_{i-n}] + 
   O\Big(\frac{\abs{\gamma_{i-n}-\gamma_i^*}^2}{\rho(\gamma_i^*)^2+ \rho(\gamma_i^*)\Delta^{1/3}}\Big)
 \end{equation}

Thus, using~\eqref{applh} also for $\gamma_{i+n} -\gamma_{i}$, equating the two equations and dividing by $\rho(\gamma_i^*)$, we have
\begin{equation}\label{gg4}
   \gamma_i -\gamma_{i-n} =  \gamma_{i+n} -\gamma_{i} 
   +O\Big(\frac{\abs{\gamma_{i-n}-\gamma_i^*}^2}{\rho(\gamma_i^*)^3+ \rho(\gamma_i^*)^2\Delta^{1/3}}\Big).
   \end{equation}
Similar relation holds for the semiquantiles:
\begin{equation}\label{gg5}
   \gamma_i^* -\gamma_{i-n}^* =  \gamma_{i+n}^* -\gamma_{i}^*
   +O\Big(\frac{\abs{\gamma_{i-n}^*-\gamma_i^*}^2}{\rho(\gamma_i^*)^3+ \rho(\gamma_i^*)^2\Delta^{1/3}}\Big)
   \end{equation}
and for the mixed relations among quantiles and semiquantiles:
\[
   \gamma_i^* - \gamma_{i-n} =  \gamma_{i+n-1} -\gamma_{i}^* 
    +O\Big(\frac{\abs{\gamma_{i-n}-\gamma_i^*}^2}{\rho(\gamma_i^*)^3+ \rho(\gamma_i^*)^2\Delta^{1/3}}\Big)
\]
\[
   \gamma_i^* - \gamma_{i-n-1} =\gamma_{ i+n}-\gamma_i^* 
    +O\Big(\frac{\abs{\gamma_{i-n}-\gamma_i^*}^2}{\rho(\gamma_i^*)^3+ \rho(\gamma_i^*)^2\Delta^{1/3}}\Big).
 \]

Thus, using $\gamma_i^*-\gamma_{i-n-1}\sim \gamma_{i+n} - \gamma_i^*$, we have
\begin{equation}\label{firs}
\rho(\gamma_i^*) \abs{ \log \frac{\gamma_i^*-\gamma_{i-n-1}}{\gamma_{i+n} - \gamma_i^*}}  \lesssim \frac{ \rho(\gamma_i^*)}{\gamma_{i+n}- \gamma_i^*}O\Big(\frac{\abs{\gamma_{i-n-1}-\gamma_i^*}^2}{\rho(\gamma_i^*)^3+ \rho(\gamma_i^*)^2\Delta^{1/3}}\Big)
\lesssim \frac{\abs{\gamma_{i-n-1}-\gamma_i^*}}{\rho(\gamma_i^*)^2+ \rho(\gamma_i^*)\Delta^{1/3}}.
\end{equation}
Using $n\le n(i)$ and~\eqref{niidef}, we have $\abs{\gamma_{i-n-1}-\gamma_i^*}\lesssim N^{-3/4}$.
The contribution of this term to $\sum_n B_2(n)$ is thus
\begin{equation}\label{gg7}
N^{-3/4}   \sum_{n\le n(i)} \frac{1}{\rho(\gamma_i^*)^2+ \rho(\gamma_i^*)\Delta^{1/3}}
  \le \frac{n(i) N^{-3/4}}{\rho(\gamma_i^*)^2+ \rho(\gamma_i^*)\Delta^{1/3}}.
\end{equation}
In the bulk regime we have $\rho(\gamma_i^*)\sim 1$ and $n(i)\sim N^{1/4}$, so this contribution is much smaller than $N^{-1/4}$.

In the cusp   regime, i.e.~when $\Delta\le (i/N)^{3/4}$, then
 we have $\gamma_i^* \sim (i /N)^{3/4}$ and $\rho(\gamma_i^*)\sim( i/N)^{1/4}$, thus
we get
\[
\eqref{gg7}\le
\frac{n(i) N^{-3/4}}{\rho(\gamma_i^*)^2+ \rho(\gamma_i^*)\Delta^{1/3}}
\le  \frac{n(i) N^{-3/4} }{\rho(\gamma_i^*)^2} \lesssim N^{-1/4} n(i) i^{-1/2} \lesssim N^{-1/4}    
\]
since $n(i) \le i^{1/4}$.

In the edge regime,  i.e.~when $\Delta\ge (i /N)^{3/4}$, then
we have $\gamma_i^* \sim \Delta^{1/9} (i /N)^{2/3}$ and $\rho(\gamma_i^*)\sim
\Delta^{-1/9}( i/N)^{1/3}$,   thus
we get
\[
   \eqref{gg7}\le \frac{n(i) N^{-3/4}}{\rho(\gamma_i^*)^2+ \rho(\gamma_i^*)\Delta^{1/3}}\le
     \frac{ n(i) N^{-3/4} }{\rho(\gamma_i^*)\Delta^{1/3}}\lesssim
     \frac{n(i) N^{-5/12}  }{ \Delta^{2/9}i^{1/3}} \le 
     \frac{  N^{-5/12} }{ \Delta^{2/9}}\le N^{-1/4}
\]
since $n(i) \le i^{1/3}$ and $\Delta\ge N^{-3/4}$.
This completes the proof of  $\sum_n B_2(n)\lesssim N^{-1/4}$.

Finally the $\sum_n B_1(n)$ term  from~\eqref{symm} is estimated as follows by using~\eqref{gg5}:
\begin{equation}\label{gg8}
\sum_n \frac{1}{N}\Big[ \frac{1}{\gamma_i^*-\gamma_{i-n-1}^*} + \frac{1}{\gamma_i^*-\gamma_{i+n-1}^*}\Big]
=\sum_n \frac{1}{N} \frac{1}{(\gamma_i^*-\gamma_{i-n}^*)^2} 
O\Big(\frac{(\gamma_i-\gamma_{i-n-1})^2}{\rho(\gamma_i^*)^2 [\rho(\gamma_i^*) +\Delta^{1/3}]  } \Big)
 \lesssim \frac{n(i)}{N \rho(\gamma_i^*)^2 [\rho(\gamma_i^*) +\Delta^{1/3}] }.
\end{equation}
In the bulk regime this is trivially bounded by $CN^{-3/4}$. 
In the cusp regime, $\Delta\le (i /N)^{3/4}$, we have
\[
  \frac{n(i)}{N \rho(\gamma_i^*)^2 [\rho(\gamma_i^*) +\Delta^{1/3}] } \le 
  \frac{n(i)  }{N\rho(\gamma_i^*)^3} \lesssim  \frac{n(i)}{ N^{1/4} i^{3/4}} \lesssim N^{-1/4}    
\]
since $n(i) \le i^{1/4}$.

Finally, in the edge   regime,  $\Delta\ge (i /N)^{3/4}$, we just use
\[
 \frac{n(i)}{N \rho(\gamma_i^*)^2 [\rho(\gamma_i^*) +\Delta^{1/3}] } \le 
  \frac{n(i)  }{N\rho(\gamma_i^*)^2\Delta^{1/3}} \lesssim \frac{n(i)}{ N^{1/4} i^{3/4}} \lesssim N^{-1/4}    
\]
since $n(i) \le i^{1/3}$.
This gives $\sum_n B_1(n)\lesssim N^{-1/4}$. Together with the estimate on $\sum_n B_2(n)$ we get 
$\abs{S_2}\lesssim N^{-1/4}$, see~\eqref{nidef} and~\eqref{symm}.
This completes the estimate of $G_1$ in~\eqref{Fdef1}, which, together with~\eqref{G2} and~\eqref{Dest}
finishes the proof of Proposition~\ref{prop:F}.
\end{proof}

\subsection{Phase 2: Rigidity of \texorpdfstring{$\wh z$}{z} on scale \texorpdfstring{$N^{-3/4+\om_1/6}$}{N-3/4+w1/6},  without \texorpdfstring{$i$}{i} dependence}\label{sec:hatrig}
For any fixed $\alpha\in[0,1]$ recall the definition of the shifted process $\wt z(t,\alpha)$~\eqref{tildereq} and
the shifted $\alpha$-interpolating  semiquantiles $\ov\gamma_i^*(t)$ from~\eqref{barg} that trail $\wt z$. 
Furthermore, for all $0\le t\le t_*$ we consider the interpolated density $\overline{\rho}_{t}$ with a small gap $[\overline{\mathfrak{e}}_t^-,\overline{\mathfrak{e}}_t^+]$, and its Stieltjes transform $\overline{m}_t$. In particular, \[\overline{\mathfrak{e}}_t^\pm=\alpha \mathfrak{e}_{x,t}^\pm+(1-\alpha)\mathfrak{e}_{y,t}^\pm.\]
We recall that by Proposition~\ref{prop:option2} and~\eqref{optimalinit}
we have that 
\begin{equation}
\label{21}
\sup_{0\le t\le t_*} \max_{1\le \abs{i}\le i_*} \abs{\widetilde{z}_i(t,\alpha)-\overline{\gamma}_i^*(t)}\le  N^{-\frac{3}{4}+C\omega_1},
\end{equation} 
holds with very high probability for some $i_*= N^{\frac{1}{2} + C_*\om_1}$.

In this section we improve the rigidity~\eqref{21} from scale  $N^{-\frac{3}{4}+C\omega_1}$ to the 
almost optimal, but still $i$-independent rigidity of order $ N^{-\frac{3}{4}+\frac{\omega_1}{6}+\xi}$  but only 
for a  new short range approximation $\widehat{z}_i(t,\alpha)$ of $\widetilde{z}_i(t,\alpha)$. The range of this new approximation $\ell^4 =N^{4\om_\ell}$
with some $\om_\ell\ll 1$
is much shorter than that of $\mathring{z}_i(t,\alpha)$ in Section~\ref{martingalenew}. Furthermore,  the result will hold
 only for $1\le \abs{i}\le N^{4\omega_\ell+\delta_1}$, for some small $\delta_1>0$. The rigorous statement is in Proposition~\ref{G3}
 below, after we give the definition of the short range approximation.

\subsubsection*{Short range approximation on fine scale.}
Adapting the idea of~\cite{1712.03881} to the cusp regime,
we now introduce a new short range approximation process $\wh z (t,\alpha)$ for the solution to~\eqref{tildereq}.  
The short range approximation in this section will always be denoted by hat, $\wh z$,  in distinction to the other short range 
approximation $\mathring{z}$ used in Section~\ref{martingalenew}, see~\eqref{zstareqnew}.  Not only the length scale 
is shorter for $\wh z$, but the definition of  $\wh z$ is more subtle than in~\eqref{zstareqnew}

The new short scale approximation
 is characterized by two exponents $\omega_\ell$ and $\omega_A$. In particular, we will always assume that $\omega_1\ll \omega_\ell\ll \omega_A\ll 1$, where 
 recall that  $t_*\sim N^{-\frac{1}{2}+\omega_1}$ is defined in such a way $\ov\rho_{t_*}$ has  an exact cusp.
The key quantity is $\ell \defeq N^{\om_\ell}$ that  determines the scale
on which the interaction term in~\eqref{tildereq} will be cut off and replaced by its mean-field value.
This scale is not
constant, it increases away from the cusp at a certain rate. The cutoff will be effective only near the cusp, 
for indices beyond $\frac{i_*}{2}$, with $i_*= N^{\frac{1}{2}+C_*\omega_1}$, no cutoff is made. 
Finally, the intermediate scale $N^{\om_A}$
is used for a technical reason:  closer to the cusp, for indices less than $N^{\om_A}$,
 we always use the density $\rho_{y,t}$ of the reference process $y(t)$  to define the mean field approximation of the cutoff  long range terms.
 Beyond this scale we use the actual density $\ov\rho_t$. In this way we can exploit the closeness of the density $\ov\rho_t$ to 
 the reference density $\rho_{y,t}$ near the cusp and simplify the estimate. This choice will guarantee 
 that the error term $\zeta_0$ in~\eqref{weq} below is non zero only for $\abs{i}>N^{\omega_A}$.

Now we define the $\wh z$ process precisely.
Let 
 \begin{equation}
\label{shortrange}
\mathcal{A}\defeq \left\{(i,j):\abs{i-j}\le \ell(10\ell^3+\abs{i}^\frac{3}{4}+\abs{j}^\frac{3}{4})\right\}\cup\left\{(i,j): \abs{i},\abs{j}>\frac{i_*}{2}\right\}.
\end{equation}
One can easily check that for each $i$ with $1\le \abs{i} \le \frac{i_*}{2}$ the set $\{j:(i,j)\in\mathcal{A}\}$ is an interval of the nonzero integers and that $(i,j)\in\mathcal{A}$ if and only if $(j,i)\in\mathcal{A}$. For each such fixed $i$ we denote the  smallest and the biggest $j$ such that $(i,j)\in\mathcal{A}$ by $j_-(i)$ and $j_+(i)$,  respectively.
We will use the notation 
\[
\sum_j^{\mathcal{A},(i)}\defeq \sum_{\substack{j:(i,j)\in\mathcal{A}\\ i\ne j}},\,\,\,\,\, \qquad \sum_j^{\mathcal{A}^c,(i)}\defeq \sum_{j:(i,j)\notin \mathcal{A}}.
\] 
Assuming for simplicity that $i_*$  
 is divisible by four, 
  we introduce the intervals
\begin{equation}
\label{Bigint}
\mathcal{J}_z(t)\defeq \left[\overline{\gamma}_{-\frac{3i_*}{4}}(t),\overline{\gamma}_{\frac{3i_*}{4}}(t)\right],
\end{equation}
and for each $0<\abs{i}\le \frac{i_*}{2}$ we define 
\begin{equation}
\label{intmaxj}
\mathcal{I}_{z,i}(t)\defeq [\overline{\gamma}_{j_-(i)}(t),\overline{\gamma}_{j_+(i)}(t)].
\end{equation} 

For  a fixed $\alpha\in [0,1]$  and $N\ge \abs{i}> \frac{i_*}{2}$ we let
 \begin{equation}
\label{333}
\begin{split}
\diff\widehat{z}_i(t,\alpha)&=\sqrt{\frac{2}{N}}\diff B_i+ \Bigg[ \frac{1}{N}\sum_j^{\mathcal{A},(i)}\frac{1}{\widehat{z}_i(t,\alpha)-\widehat{z}_j(t,\alpha)}+\frac{1}{N}\sum_j^{\mathcal{A}^c,(i)}\frac{1}{\widetilde{z}_i(t,\alpha)-\widetilde{z}_j(t,\alpha)}  + \Phi_\alpha(t)
\Bigg]\diff t
\end{split}\end{equation} for $0<\abs{i}\le N^{\omega_A}$ \begin{equation}
\label{444}
\diff \widehat{z}_i(t,\alpha)=\sqrt{\frac{2}{N}}\diff B_i+\Bigg[ \frac{1}{N}\sum_j^{\mathcal{A},(i)}\frac{1}{\widehat{z}_i(t,\alpha)-\widehat{z}_j(t,\alpha)}+\int_{\mathcal{I}_{y,i}(t)^c}\frac{\rho_{y,t}(E+\mathfrak{e}_{y,t}^+)}{\widehat{z}_i(t,\alpha)-E}\diff E+\Re [m_{y,t}(\mathfrak{e}_{y,t}^+)]
\Bigg]\diff t,
\end{equation} and for $N^{\omega_A}< \abs{i}\le \frac{i_*}{2}$
 \begin{equation}
\label{555}
\begin{split}
\diff \widehat{z}_i(t,\alpha)=\sqrt{\frac{2}{N}}\diff B_i+\Bigg[ & \frac{1}{N}\sum_j^{\mathcal{A},(i)}\frac{1}{\widehat{z}_i(t,\alpha)-\widehat{z}_j(t)}+
\int_{\mathcal{I}_{z,i}(t)^c\cap\mathcal{J}_z(t)}\frac{\overline{\rho}_t(E+\overline{\mathfrak{e}}_t^+)}{\widehat{z}_i(t,\alpha)-E}\diff  E\\
&+\frac{1}{N}\sum_{\abs{j}\ge \frac{3}{4}i_*}\frac{1}{\widetilde{z}_i(t,\alpha)-\widetilde{z}_j(t,\alpha)}
+\Phi_\alpha(t)
\Bigg]\diff t,
\end{split}\end{equation} 
with initial data 
\begin{equation}
\label{666}
\widehat{z}_i(0,\alpha)\defeq \widetilde{z}_i(0,\alpha),
\end{equation} 
where we recall that $\widetilde{z}_i(0,\alpha)=\alpha\widetilde{x}_i(0)+(1-\alpha)\widetilde{y}_i(0)$ for any $\alpha\in [0,1]$. In particular, $\widehat{z}(t,1)=\widehat{x}(t)$ and $\widehat{z}(t,0)=\widehat{y}(t)$, that are the short range approximations of the 
$\wt x(t)\defeq  x(t) - \ed_{x,t}^+$ and $\wt y(t)\defeq  x(t) - \ed_{y,t}^+$ processes.

Using the rigidity estimates in~\eqref{crudinitial} and~\eqref{21} we will prove the following lemma in Appendix~\ref{SLL}.

\begin{lemma}\label{shortlong2}
Assuming that the rigidity estimates~\eqref{crudinitial} and~\eqref{21} hold. Then, for any fixed $\alpha\in [0,1]$ we have
\begin{equation}
\label{shortestimateba}
\sup_{1\le \abs{i}\le N}\sup_{0\le t\le t_*} \abs{\widehat{z}_i(t,\alpha)-\widetilde{z}_i(t,\alpha)}\le \frac{N^{C\omega_1}}{N^\frac{3}{4}},
\end{equation}
 with very high probability.
\end{lemma}

In particular, since~\eqref{crudinitial} and~\eqref{21} have already been proven,
 we conclude from~\eqref{21} and~\eqref{shortestimateba} that
  \begin{equation}
\label{badrighat}
\sup_{0\le t\le t_*}\abs{\widehat{z}_i(t,\alpha)-\overline{\gamma}_i(t)}\le \frac{N^{C\omega_1}}{N^\frac{3}{4}},\qquad 1\le \abs{i}\le i_*,
\end{equation} for any fixed $\alpha\in [0,1]$.

Now we state the improved rigidity for $\wh z$, the main result of Section~\ref{sec:hatrig}:
\begin{proposition}
\label{G3}
Fix any $\alpha\in [0,1]$. There exists a constant $C>0$ such that if $0<\delta_1<C\omega_\ell$ then \begin{equation}
\label{10}
\sup_{0\le t\le t_*}\abs{\widehat{z}_i(t,\alpha)-\overline{\gamma}_i(t)}\lesssim \frac{N^\xi N^\frac{\omega_1}{6}}{N^\frac{3}{4}},\qquad 1\le \abs{i}\le N^{4\omega_\ell+\delta_1}
\end{equation} for any $\xi>0$ with very high probability.

\proof
Recall that initially $\widetilde{z}_i(0,\alpha)$ is a linear interpolation between $\widetilde{x}_i(0)$ and $\widetilde{y}_i(0)$
and thus for $\widetilde{z}_i(0,\alpha)$ optimal rigidity~\eqref{optimalinit} holds.
We define the derivative process
\begin{equation}\label{wdef}
w_i(t,\alpha)\defeq \partial_\alpha \widehat{z}_i(t,\alpha).
\end{equation}
In particular, we find that $w= w(t,\alpha)$ is the solution of 
\begin{equation}\label{weq}
\partial_t w=\mathcal{L}w +\zeta^{(0)}, \qquad \mathcal{L}\defeq \mathcal{B}+\mathcal{V},
\end{equation}
with initial data 
\[w_i(0,\alpha)=\widehat{x}_i(0)-\widehat{y}_i(0).
\]
Here, for any $1\le \abs{i}\le N$, the  (short range) operator $\mathcal{B}$ is defined on any vector $f\in \C^{2N}$ as 
\begin{equation}
\label{11}
(\mathcal{B}f)_i\defeq \sum_j^{\mathcal{A},(i)} \mathcal{B}_{ij} (f_i-f_j), \qquad \mathcal{B}_{ij}\defeq 
-\frac{1}{N}\frac{1}{(\widehat{z}_i(t,\alpha)-\widehat{z}_j(t,\alpha))^2}.
\end{equation} 
Moreover,  $\mathcal{V}$ is a multiplication operator, 
i.e.~$(\mathcal{V}f)_i=\mathcal{V}_if_i$, where $\mathcal{V}_i$ is defined in different regimes of $i$ as follows:
   \begin{equation} 
   \begin{split}
\mathcal{V}_i\defeq & -\int_{\mathcal{I}_{y,i}(t)^c}\frac{\rho_{y,t}(E+\mathfrak{e}_{y,t}^+)}{(\widehat{z}_i(t,\alpha)-E)^2}\diff E,
\qquad  1\le \abs{i}\le N^{\omega_A} \\
\label{V222rig1}
\mathcal{V}_i\defeq & -\int_{\mathcal{I}_{z,i}(t)^c\cap\mathcal{J}_z(t)}\frac{\overline{\rho}_t(E+\overline{\mathfrak{e}}_t^+)}{(\widehat{z}_i(t,\alpha)-E)^2}\diff  E,
\qquad  N^{\omega_A}< \abs{i}\le \frac{i_*}{2}
\end{split}
\end{equation} 
and $\mathcal{V}_i=0$ for $\abs{i}>\frac{ i_*}{2}$. The error term $\zeta_i^{(0)}= \zeta_i^{(0)}(t)$ in~\eqref{weq} is defined as follows: for 
 $\abs{i}>\frac{i_*}{2}$ we have
\begin{equation}
\label{AAAAA}
\zeta_i^{(0)}\defeq \frac{1}{N} \sum_j^{\mathcal{A}^c,(i)}\frac{\partial_\alpha \widetilde{z}_j(t,\alpha)-\partial_\alpha\widetilde{z}_i(t,\alpha)}{(\widetilde{z}_i(t,\alpha)-\widetilde{z}_j(t,\alpha))^2}
+\partial_\alpha\Phi_{\alpha}(t)\defqe Z_1+\partial_\alpha\Phi_{\alpha}(t) 
\end{equation}
 for $N^{\omega_A}<\abs{i}\le \frac{i_*}{2}$ we have
 \begin{equation}
 \label{AAA}
 \begin{split}\zeta_i^{(0)}&\defeq \frac{1}{N} \sum_{\abs{j}\ge \frac{3i_*}{4}}\frac{\partial_\alpha \widetilde{z}_j(t,\alpha)-\partial_\alpha\widetilde{z}_i(t,\alpha)}{(\widetilde{z}_i(t,\alpha)-\widetilde{z}_j(t,\alpha))^2}
 +\int_{\mathcal{I}_{z,i}(t)^c\cap\mathcal{J}_z(t)}\frac{\partial_\alpha\big[\overline{\rho}_t(E+\overline{\mathfrak{e}}_t^+)\big]}{\widehat{z}_i(t,\alpha)-E}\, \diff E\\
&\qquad +\Big(\partial_\alpha \int_{\mathcal{I}_{z,i}(t)^c\cap\mathcal{J}_z(t)} \Big)
\frac{\overline{\rho}_t(E+\overline{\mathfrak{e}}_t^+)}{\widehat{z}_i(t,\alpha)-E}\, \diff E
+\partial_\alpha\Phi_\alpha(t)  \defqe Z_2+Z_3+Z_4+\partial_\alpha\Phi_{\alpha}(t),
\end{split}
\end{equation} 
and finally for $1\le \abs{i}\le N^{\omega_A}$ we have
$\zeta_i^{(0)}=0$. We recall that $\mathcal{I}_{z,i}(t)$ and $\mathcal{J}_z(t)$ in~\eqref{AAA} are defined by~\eqref{intmaxj} and~\eqref{Bigint} respectively.
Next, we prove that the error term $\zeta^{(0)}$ in~\eqref{weq} is bounded by some large power of $N$.

\begin{lemma}\label{lm:zeta}
There exists a large constant $C>0$ such that 
\begin{equation}\label{zeta}
\sup_{0\le t\le t_*}\max_{1\le \abs{i}\le N} \abs{\zeta_i^{(0)}(t)}\le N^C.
\end{equation}

\proof[Proof of Lemma~\ref{lm:zeta}] 
By~\eqref{Phidef}, it follows that 
\[ 
\partial_\alpha\Phi_\alpha(t)=\partial_\alpha \Re[\overline{m}_t(\overline{\mathfrak{e}}_t^+
+\ii N^{-100})]+h^{**}(t,1)-h^{**}(t,0),
\] 
with $h^{**}(t,\alpha)$ defined by~\eqref{h**def}. Since the two $h^{**}$  terms are small  by~\eqref{zetaest}, for each fixed $t$, we have that
\begin{equation}\label{PP}
\abs{\partial_\alpha\Phi_{\alpha}(t)}\lesssim\abs{  \partial_\alpha
\int_\mathbb{R}\frac{\overline{\rho}_t(\overline{\mathfrak{e}}_t^+ + E)}{E-\ii N^{-100}}\, \diff E}+N^{-1} = U_1+U_2+N^{-1},
\end{equation}
where
\[
\begin{split}
U_1\defeq &\abs{  \partial_\alpha  \int_{\ov\gamma_{-i(\delta_*)}}^{\ov\gamma_{i(\delta_*)}}
\frac{\overline{\rho}_t (\overline{\mathfrak{e}}_t^+  +E) }{E-\ii N^{-100}}   \, \diff E}
=  \abs{\partial_\alpha\int_{I_*}
\frac{\overline{\rho}_t(\overline{\mathfrak{e}}_t^++\varphi_{\alpha,t}(s))}{\varphi_{\alpha,t}(s)-\ii N^{-100}} \varphi_{\alpha,t}'(s)\, \diff s} \\
U_2\defeq & \abs{  \frac{1}{N}\sum_{i_*(\delta)<\abs{i}\le N} \partial_\alpha
\int_\R\frac{\psi (E-\ov\gamma_i^*(t))}{E-\ii N^{-100}}\,
 \diff E},
\end{split}
\] using the notation $\overline{\gamma}_{i(\delta_*)}=\overline{\gamma}_{i(\delta_*)}(t)$ and the definition of $\ov\rho_t$ from~\eqref{rhozdef}. In $U_1$ we changed variables, i.e.~ $E=\varphi_{\alpha,t}(s)$, using that $s\to \varphi_{\alpha,t}(s)$ is strictly increasing. In particular, in order to compute the limits of integration we used that  $\varphi_{\alpha,t}(i/N) = \ov\gamma_i (t)$ by~\eqref{interpolating phi} and defined the $\alpha$-independent interval $I_*\defeq  [ - i(\delta_*)/N, i(\delta_*)/N]$. Furthermore, in $U_1$ we denoted by prime the $s$-derivative.

For $U_1$  we have that (omitting the $t$ dependence, $\ov\rho=\ov\rho_t$, etc.)
\begin{equation}\begin{split}
\label{3int}
U_1&\lesssim \abs{\int_{I_*}
\frac{\partial_\alpha[\overline{\rho}(\overline{\mathfrak{e}}^++\varphi_\alpha(s))]}{\varphi_\alpha(s)-\ii N^{-100}} \varphi_\alpha'(s)\, \diff s}+\abs{\int_{I_*}\frac{\overline{\rho}(\overline{\mathfrak{e}}^++\varphi_\alpha(s))}{(\varphi_\alpha(s)-\ii N^{-100})^2} (\varphi_\alpha'(s))^2\, \diff s}\\
&\quad+\abs{\int_{I_*}\frac{\overline{\rho}(\overline{\mathfrak{e}}^++\varphi_\alpha(s))}{\varphi_\alpha(s)-\ii N^{-100}} \partial_\alpha\varphi_\alpha'(s)\, \diff s}.
\end{split}\end{equation} 

For $s\in I_*$, by the definition of $\varphi_\alpha(s)$ and~\eqref{dernal} it follows that 
\[1=n_{\alpha}'(\varphi_\alpha(s))\varphi_{\alpha}'(s)=\rho_\alpha(\varphi_\alpha(s))\varphi_\alpha(s),\] 
and so that 
\begin{equation}
\label{dervarphial}
\varphi_\alpha'(s)=\frac{1}{\rho_\alpha(\varphi_\alpha(s))}\lesssim  s^{-\frac{1}{4}},
\end{equation} 
where in the last inequality we used that $\rho_\alpha(\om) \sim \min\{ \om^{1/3}, \om^{1/2}\Delta^{-1/6}\}$ and 
 $\varphi_\alpha(s)\sim \max\{ s^\frac{3}{4}, s^\frac{2}{3}\Delta^{1/9}\} $ by~\eqref{rho counting fcts gap}.

In the first integral in~\eqref{3int} we use that
\[ \ov\rho (\overline{\mathfrak{e}}^++\varphi_\alpha(s))
=\rho_\alpha(\overline{\mathfrak{e}}^++\varphi_\alpha(s)), \qquad s\in I_*
\]
by~\eqref{rhozdef} and that
 $\partial_\alpha[\rho_\alpha(\overline{\mathfrak{e}}^++\varphi_\alpha(s))]$ is bounded
by the  explicit relation in~\eqref{ealphder}.
 For the other two integrals in~\eqref{3int} we use that
 $\overline{\rho}$ is bounded on the integration domain and  that $(\varphi_\alpha'(s))^2\lesssim s^{-1/2}$ 
 from~\eqref{dervarphial}, hence it  is integrable. 
 In the third integral we also observe that 
 \[ 
 \partial_\alpha \varphi_\alpha(s) =\varphi_\lambda(s)-\varphi_\mu(s)
\]
by~\eqref{interpolating phi}, thus $\abs{\partial_\alpha \varphi_\alpha'(s)}\lesssim s^{-1/4}$ similarly to~\eqref{dervarphial}.
 Using $\abs{\varphi_\alpha(s)-\ii N^{-100}}\gtrsim N^{-100}$, we
  thus conclude that 
\[
U_1\lesssim N^{200}.
\]
 
Next, we proceed with the estimate for $U_2$.
 
Notice that $\abs{\partial_\alpha \psi (E-\ov\gamma_i^*(t))}\le\norm{\psi'}_\infty \abs{\wh\gamma_{x,i}(t)-\wh\gamma_{y,i}(t)}$
 by~\eqref{barg}.
 Furthermore, 
 since $\abs{E-\ii N^{-100}}\gtrsim \delta_*$
on the domain of integration of $U_2$, we conclude that 
\[
U_2\lesssim N^{200}\norm{\psi'}_\infty,
\] 
and therefore  from~\eqref{PP} we have
\begin{equation}
\label{ZZZ222}
\abs{\partial_\alpha\Phi_{\alpha}(t)}\lesssim N^{202}. 
\end{equation}
since $\norm{ \psi'}_\infty\lesssim N^2$ by the choice of $\psi$, see below~\eqref{rhox}.

Similarly, we conclude that 
\begin{equation}
\abs{Z_3}\lesssim N^{200}\lVert \psi'\rVert_\infty.
\end{equation}

To estimate $Z_2$, by~\eqref{tildereq}, it follows that 
\[ \diff (\partial_\alpha \wt z_i) = 
\Bigg[ \frac{1}{N}\sum_j  \frac{\partial_\alpha \widetilde{z}_j-\partial_\alpha \widetilde{z}_i}{(\wt z_i-\wt z_j)^2} 
+\partial_\alpha\Phi_\alpha(t) \Bigg]\diff t,
\] 
with initial data 
\[\partial_\alpha\widetilde{z}_i(0,\alpha)=\widetilde{x}_i(0)-\widetilde{y}_i(0),\] 
for all $1\le \abs{i}\le N$. Since $\abs{\partial_\alpha\widetilde{z}_i(0,\alpha)}\lesssim N^{200}$ for all $1\le \abs{i}\le N$, by Duhamel principle
and contraction, it follows that \begin{equation}
\label{tildeder}
\abs{\partial_\alpha\widetilde{z}_i(t,\alpha)}\lesssim N^{200}+t_* \max_{0\le \tau\le t_*}\abs{  \partial_\alpha\Phi_\alpha(\tau) } \lesssim N^{202}
\end{equation} 
for all $0\le t\le t_*$. In particular, by~\eqref{tildeder} it follows that 
\begin{equation}
\label{ZZZ111}
\abs{Z_2}\lesssim N^{202}\sqrt{N}
\end{equation} since for all $j$ in the summation in $Z_2$ we have that $\abs{i-j}\gtrsim  i_*\sim N^\frac{1}{2}$ 
and thus $\abs{\wt z_i-\wt z_j}\gtrsim \abs{i-j}/N\gtrsim N^{-1/2}$.

Finally, we estimate $Z_4$ using the fact that the endpoints of $\mathcal{I}_{z,i}(t)^c\cap\mathcal{J}_z(t)$ are
quantiles $\ov\gamma_i(t)$ whose $\alpha$-derivatives are bounded by~\eqref{barg}. Hence
\begin{equation}
\label{ZZZ555}
\abs{Z_4}\lesssim \abs{ \frac{\overline{\rho}_t(\overline{\gamma}_{j_+} +\overline{\mathfrak{e}}_t^+)}{\widehat{z}_i-\overline{\gamma}_{j_+}} }
+ \abs{\frac{\overline{\rho}_t(\overline{\gamma}_{j_-}+\overline{\mathfrak{e}}_t^+)}{\widehat{z}_i-\overline{\gamma}_{j_-}} }
+ \abs{\frac{\overline{\rho}_t(\overline{\gamma}_{\frac{3i_*}{4}}+\overline{\mathfrak{e}}_t^+)}{\widehat{z}_i-\overline{\gamma}_{\frac{3i_*}{4}}} }\lesssim N
\end{equation}
by rigidity.
Combining~\eqref{ZZZ222}-\eqref{ZZZ555} we conclude~\eqref{zeta}, completing the proof of Lemma~\ref{lm:zeta}.
\endproof
\end{lemma}

Continuing the analysis of the equation~\eqref{weq}, 
for any fixed $\alpha$ let us define $w^\#=w^\#(t,\alpha)$ as the solution of \begin{equation}
\label{13}
\partial_tw^\#=\mathcal{L}w^\#,
\end{equation} 
with cutoff initial data
\[
w_i^\#(0,\alpha)=\mathbf{1}_{\{\abs{i}\le N^{4\omega_\ell+\delta}\}}w_i(0,\alpha),
\] with some $0<\delta<C\omega_\ell$  where $C>10$ a constant such that $(4+C)\omega_\ell<\omega_A$.

By the rigidity~\eqref{badrighat} the finite speed estimate~\eqref{unifs}, with $\delta'\defeq \delta$,  for the propagator $\mathcal{U}^\mathcal{L}$ 
of $\mathcal{L}$ holds.
Let $0< \delta_1<\frac{\delta}{2}$, then, using Duhamel principle, that the error term $\zeta_i^{(0)}$ is bounded by~\eqref{zeta} and that $\zeta_i^{(0)}=0$ for any $1\le |i|\le N^{\omega_A}$, it easily follows that \begin{equation}
\label{12}
\sup_{0\le t\le t_*}\max_{\abs{i}\le N^{4\omega_\ell+\delta_1}}\abs{w_i^\#(t,\alpha)-w_i(t,\alpha)}\le N^{-100},
\end{equation} for any $\alpha\in [0,1]$.  In other words, the initial conditions far away do not
influence the $w$-dynamics, hence they can be set zero.

Next, we use the heat kernel contraction for the equation in~\eqref{13}.
By the optimal rigidity of $\widehat{x}_i(0)$ and $\widehat{y}_i(0)$, since $w_i^\#(0,\alpha)$ is non zero only for $1\le \abs{i}\le N^{4\omega_\ell+\delta}$, it follows that \begin{equation}
\label{14}
\max_{1\le \abs{i}\le N} \abs{w_i^\#(0,\alpha)}\le \frac{N^\xi N^\frac{\omega_1}{6}}{N^\frac{3}{4}},
\end{equation} and so, by heat kernel contraction and Duhamel principle
 \begin{equation}
\label{15}
\sup_{0\le t\le t_*}\max_{1\le\abs{i}\le N} \abs{w_i^\#(t,\alpha)}\le \frac{N^\xi N^\frac{\omega_1}{6}}{N^\frac{3}{4}}.
\end{equation} Next, we recall that $\widehat{z}(t,\alpha=0)=\widehat{y}(t)$.

Combining~\eqref{12} and~\eqref{15}, integrating 
$w_i(t,\alpha')$ over $\alpha'\in[0,\alpha]$, by high moment Markov inequality as in~\eqref{markov}-\eqref{markov1}, we conclude that 
\[
\sup_{0\le t\le t_*}\abs{\widehat{z}_i(t,\alpha)-\widehat{y}_i(t)}
\le \frac{N^\xi N^\frac{\omega_1}{6}}{N^\frac{3}{4}}, \qquad 1\le \abs{i}\le N^{4\omega_\ell+\delta_1},
\] for any fixed $\alpha\in [0,1]$ with very high probability for any $\xi>0$.
Since 
\[
\abs{\widehat{z}_i(t,\alpha)-\overline{\gamma}_i(t)}\le \abs{\widehat{y}_i(t)-\widehat{\gamma}_{y,i}(t)}
+ \abs{\overline{\gamma}_i(t)-\widehat{\gamma}_{y,i}(t)}+\frac{N^\xi N^\frac{\omega_1}{6}}{N^\frac{3}{4}},
\]
 for all $1\le \abs{i}\le N^{4\omega_\ell+\delta_1}$ and $\alpha\in [0,1]$, by~\eqref{gamma difference gap} and the optimal rigidity of $\widehat{y}_i(t)$, see~\eqref{rzx}, 
 we conclude that 
 \begin{equation}
\label{16}
\sup_{0\le t\le t_*}\abs{\widehat{z}_i(t,\alpha)-\overline{\gamma}_i(t)}\le \frac{N^\xi N^\frac{\omega_1}{6}}{N^\frac{3}{4}},\qquad 1\le \abs{i}\le N^{4\omega_\ell+\delta_1}
\end{equation} 
for any fixed $\alpha\in [0,1]$, for any $\xi>0$ with very high probability. This concludes the proof of~\eqref{10}.
\endproof
\end{proposition}

\subsection{Phase 3: Rigidity for \texorpdfstring{$\wh z$}{z} with the correct \texorpdfstring{$i$}{i}-dependence.}\label{sec:correcti}
In this subsection we will prove almost optimal $i$-dependent rigidity for the short range approximation $\widehat{z}_i(t,\alpha)$ 
(see~\eqref{333}--\eqref{666})
for $1\le \abs{i}\le N^{4\omega_\ell+\delta_1}$. 

\begin{proposition}
\label{G33}
Let $\delta_1$ be defined in Proposition~\ref{G3}, then, for any fixed $\alpha\in [0,1]$, we have that \begin{equation}
\label{100}
\sup_{0\le t\le t_*}\abs{\widehat{z}_i(t,\alpha)-\overline{\gamma}_i(t)}\lesssim \frac{N^\xi N^\frac{\omega_1}{6}}{N^\frac{3}{4}\abs{i}^\frac{1}{4}},\qquad 1\le \abs{i}\le N^{4\omega_\ell+\delta_1},
\end{equation} for any $\xi>0$ with very high probability.
\proof

Define
\[K\defeq \lceil N^{\xi}\rceil,
\]
then~\eqref{10} (with $\xi\to \xi/2$) implies~\eqref{100} for all $1\le \abs{i}\le2 K$.
Next, we prove~\eqref{100} for all $2K\le \abs{i}\le N^{4\omega_\ell+\delta_1}$
by coupling $\widetilde{x}_i(t)$ with  $\widetilde{y}_{\langle i-K\rangle }(t)$, where
we make
the following notational convention:
\begin{equation}\label{langle def}
    \langle i-K \rangle\defeq  i-K \qquad \mbox{if} \;  i\in [K+1,  N]\cup [-N, -1], \qquad  \langle i-K \rangle\defeq  i-K-1 \qquad \mbox{if}\; i\in [1,K].
\end{equation}
This slight complication is due to 
 our indexing convention that excludes $i=0$.

In order to couple the Brownian motion of $\widetilde{x}_i(t)$ with the one of $\widetilde{y}_{\langle i-K\rangle }(t)$ we construct a new process $\widetilde{z}^*(t,\alpha)$ satisfying
\begin{equation}\label{rst}
\diff \widetilde{z}^*_i(t,\alpha)=\sqrt{\frac{2}{N}}\diff B_{\langle i-K\rangle }+\Bigg[ \frac{1}{N}\sum_{j\ne i}\frac{1}{\widetilde{z}_i^*(t,\alpha)-\widetilde{z}^*_j(t,\alpha)} +\Phi_\alpha(t)
\Bigg]\diff t, \quad 1\le \abs{i}\le N
\end{equation}
 with initial data 
 \begin{equation}
  \widetilde{z}_i^*(0,\alpha)=\alpha \widetilde{x}_i(0)+(1-\alpha)\widetilde{y}_{\langle i-K\rangle }(0),  
  \end{equation}
   for any $\alpha\in [0,1]$.
   Notice that the only difference with respect to $\widetilde{z}_i(t,\alpha)$ from~\eqref{tildereq}
 is a shift in the index of the Brownian motion, i.e.~$\wt z$ and $\wt z^*$ (almost) coincide in distribution, 
 but their coupling to the $y$-process is different. The slight discrepancy comes from the effect of the few extreme
 indices. Indeed, to make the definition~\eqref{rst} unambiguous even for extreme indices,
 $i\in [-N, -N+K-1]$, additionally we need to define independent 
 Brownian motions $B_{j}$ and initial padding particles $\wt y_j(0)= - j N^{300}$
 for $j=-N-1, \ldots -N-K$. Similarly to Lemma~\ref{appr}, the
  effect of these very distant additional particles 
   is negligible on the dynamics of the particles for $1\le \abs{i}\le \epsilon N$ for some small $\epsilon$.

Next, we define the process $\widehat{z}^*(t,\alpha)$ as the short range approximation of $\widetilde{z}^*(t,\alpha)$,
given by~\eqref{333}--\eqref{555} but $B_i$ replaced with $B_{\langle i-K\rangle }$  and  we use
 initial data $\widehat{z}^*(0,\alpha)=\widetilde{z}^*(0,\alpha)$.
 In particular, 
 \begin{equation}
 \widehat{z}_i^*(t,1)= \widehat{x}_i(t) + O(N^{-100}), \qquad   
 \widehat{z}_i^*(t,0)= \widehat{y}_{\langle i-K\rangle }(t) + O(N^{-100}), \qquad  1\le \abs{i}\le \epsilon N,
 \end{equation}
  the discrepancy again coming from the negligible effect of the additional $K$ distant particles on the particles near the cusp regime.

Let $w_i^*(t,\alpha)\defeq \partial_\alpha \widehat{z}_i^*(t,\alpha)$, i.e.~$w^*=w^*(t,\alpha)$ is a solution of \[\partial_tw^*=\mathcal{B}w^*+\mathcal{V}w^*+\zeta^{(0)}\] with initial data \[w_i^*(0,\alpha)=\widehat{x}^*_i(0)-\widehat{y}_{\langle i-K\rangle }(0).\]  
The operators $\mathcal{B}$, $\mathcal{L}$ and the error term $\zeta^{(0)}$ are defined as in~\eqref{11}-\eqref{AAA} with all $\wt z$ and $\wh z$ replaced by $\widetilde{z}^*$ and $\widehat{z}^*$, respectively.

We now define $(w^*)^\#$ as the solution of \begin{equation}
\label{133}
\partial_t(w^*)^\#=\mathcal{L}(w^*)^\#,
\end{equation} 
with cutoff  initial data \[(w_i^*)^\#(0,\alpha)=\mathbf{1}_{\{\abs{i}\le N^{4\omega_\ell+\delta}\}}w_i^*(0,\alpha),\] with $0<\delta< C\omega_\ell$ with $C>10$ such that $(4+C)\omega_\ell<\omega_A$.

We claim that 
\begin{equation}\label{ws}
(w_i^*)^\#(0,\alpha)\ge 0, \qquad  1\le \abs{i}\le N.
\end{equation}
We need to check it for $1\le \abs{i}\le N^{4\omega_\ell+\delta}$, otherwise $(w_i^*)^\#(0,\alpha)=0$ by the cutoff. 
In the regime $1\le \abs{i}\le N^{4\omega_\ell+\delta}$ we use the optimal rigidity (Lemma~\ref{diffx}
with $\xi\to\xi/10$) for  $\widehat{x}^*_i(0)$ and $\widehat{y}_{\langle i-K\rangle}(0)$ that yields
 \begin{equation}
\label{secondmax}
(w_i^*)^\#(0,\alpha)=\widehat{x}^*_i(0)-\widehat{y}_{\langle i-K\rangle }(0)\ge-N^{\frac{\xi}{10}}\eta_f({\gamma}^*_{x,i}(0))
+\widehat{\gamma}_{x,i}(0)-\widehat{\gamma}_{y,\langle i-K\rangle }(0)-N^{\frac{\xi}{10}}\eta_f({\gamma}^*_{y, \langle i-K\rangle }(0)).
\end{equation}  
We now check that $\widehat{\gamma}_{x,i}(0)-\widehat{\gamma}_{y,\langle i-K\rangle }(0)$ is sufficiently positive to compensate
for the $N^{\frac{\xi}{10}}\eta_f$ error terms.
Indeed,  by~\eqref{gamma gap}  and~\eqref{gamma difference gap}, for all $\abs{i}\ge 2K$ we have 
\[
\widehat{\gamma}_{x,i}(t)-\widehat{\gamma}_{y, \langle i-K\rangle }(t)\gtrsim K\eta_f({\gamma}^*_{x,i}(t)) \gg N^{\frac{\xi}{10}}\eta_f({\gamma}^*_{x,i}(t))
\] 
and that 
\[\eta_f({\gamma}^*_{y,\langle i-K\rangle }(t))\sim\eta_f({\gamma}^*_{x,i}(t)).
\]
This shows~\eqref{ws} in the $2K\le \abs{i}\le   N^{4\omega_\ell+\delta} $ regime. 
 If $K\le \abs{i}\le 2K$ or $-K\le i\le -1$ we have that $(w_i^*)^\#(0,\alpha)\ge 0$ since
  \[  
 \widehat{\gamma}_{x,i}(0)-\widehat{\gamma}_{y,\langle i-K\rangle }(0)  \gtrsim \max\Big\{\frac{K^{3/4}}{N^{3/4}},(t_*-t)^{1/6}\frac{K^{2/3}}{N^{2/3}}\Big\}
 \gtrsim K\max\big\{\eta_f({\gamma}^*_{x,i}(0)),
 \eta_f({\gamma}^*_{y, \langle i-K\rangle }(0))
 \big\}
 ,
  \] 
  so $\widehat{\gamma}_{x,i}(0)-\widehat{\gamma}_{y,\langle i-K\rangle }(0)$ beats the error terms $N^{\frac{\xi}{10}}\eta_f$ as well.
  Finally, if  $1\le i\le K-1$ the bound in~\eqref{secondmax} is easy since $\widehat{\gamma}_{x,i}(0)$ 
  and $\widehat{\gamma}_{y, \langle i-K\rangle }(0)$ have opposite sign, i.e.~they are in two different sides of the small gap
  and one of them is at least of order $(K/N)^{3/4}$, beating $N^{\frac{\xi}{10}}\eta_f$. This proves~\eqref{ws}.
Hence, by the maximum principle we conclude that 
\begin{equation}\label{wpos}
(w_i^*)^\#(t,\alpha)\ge 0, \qquad 0\le t\le t_*, \qquad \alpha\in [0,1]. 
\end{equation}

Let $\delta_1<\frac{\delta}{2}$ be defined in Proposition~\ref{G3}. The rigidity estimate in~\eqref{badrighat} holds for $\widehat{z}^*$ as well, since $\widehat{z}$ and $\widehat{z}^*$ have the same distribution. Furthermore, by~\eqref{badrighat} the  propagator  $\mathcal{U}$ of $\mathcal{L}\defeq \mathcal{B}+\mathcal{V}$
satisfies the finite speed estimate in Lemma~\ref{fsl}.
 Then, using Duhamel principle and~\eqref{zeta}, we obtain
 \begin{equation}
\label{122}
\sup_{0\le t\le t_*}\max_{1\le \abs{i}\le N^{4\omega_\ell+\delta_1}}\abs{(w_i^*)^\#(t,\alpha)-w_i^*(t,\alpha)}\le N^{-100},
\end{equation} for any $\alpha\in [0,1]$ with very high probability.

By~\eqref{122}, integrating $w_i^*(t,\alpha')$ over $\alpha'\in[0,\alpha]$,
 we conclude that 
 \begin{equation}
 \label{FFBB100} \widehat{z}^*_i(t,\alpha)-\widehat{y}_{\langle i-K\rangle }(t)\ge - N^{-100}, \qquad 1\le \abs{i}\le N,^{4\omega_\ell+\delta_1}
 \end{equation}
 for all $\alpha\in [0,1]$ and $0\le t\le t_*$ with very high probability. Note that in order to prove~\eqref{FFBB100} with very high probability we used a Markov inequality as in~\eqref{markov}-\eqref{markov1}. 
  Hence,
  \begin{equation}
\label{12345}
\begin{split}
\widehat{z}_i^*(t,\alpha)-\overline{\gamma}_i(t)\ge   & \big[ \widehat{y}_{\langle i-K\rangle}(t)-\widehat{\gamma}_{y,\langle i-K\rangle}(t)\big] +
\big[ \widehat{\gamma}_{y,\langle i-K\rangle }(t) -  \widehat{\gamma}_{y, i }(t)\big] +\big[ \widehat{\gamma}_{y, i }(t)
-\overline{\gamma}_i(t)\big]
-N^{-100}  \\ 
\gtrsim & -K (\eta_f({\gamma}^*_{y,\langle i-K\rangle }(t))+\eta_f({\gamma}^*_{y,i}(t)))- \gamma^*_i(t) t_*^{1/3} 
  \ge -2K(\eta_f({\gamma}^*_{y,\langle i-K\rangle }(t))+\eta_f({\gamma}^*_{y,i}(t)))
\end{split}
\end{equation} 
for all $1\le \abs{i}\le N^{4\omega_\ell+\delta_1}$, where we used the optimal rigidity~\eqref{rzx}  and~\eqref{gamma difference gap} in going to the second line. 
In particular, since for $\abs{i}\ge 2K$ we have that 
$\eta_f({\gamma}^*_{y,i}(t))\sim\eta_f({\gamma}^*_{y,i-K}(t))$, we conclude that
 \begin{equation}
\label{lowerbound}
\widehat{z}^*_i(t,\alpha)-\overline{\gamma}_i(t)\ge -\frac{CK N^\frac{\omega_1}{6}}{N^\frac{3}{4}\abs{i}^\frac{1}{4}}, \qquad 2K\le \abs{i}\le N^{4\omega_\ell+\delta_1},
\end{equation} for all $0\le t\le t_*$ and for any $\alpha\in [0,1]$. This implies the lower bound in~\eqref{100}.

In order to prove the upper bound in~\eqref{100} we consider a very similar  process $\widetilde{z}_i^*(t,\alpha)$ (we continue to denote it by star) where the index shift in $y$ is in the other direction. More precisely, it is defined as a solution of 
\[ \diff \widetilde{z}^*_i(t,\alpha)=\sqrt{\frac{2}{N}}\diff B_{\langle i+K\rangle }+
\Bigg[ \frac{1}{N}\sum_{j\ne i}\frac{1}{\widetilde{z}^*_i(t,\alpha)-\widetilde{z}^*_j(t,\alpha)} 
+\Phi_{\alpha}(t)
\Bigg] \diff t \] with initial data \[\widetilde{z}_i(0,\alpha)=\alpha\widetilde{y}_{\langle i+K\rangle }(0)+(1-\alpha)\widetilde{x}_i(0),\] for any $\alpha\in [0,1]$.
Here $\langle i+K \rangle$ is defined analogously to~\eqref{langle def}.
Then, by similar computations, we conclude  that  \begin{equation}
\label{upperbound}
\widehat{z}^*_i(t,\alpha)-\overline{\gamma}_i(t)\le \frac{K N^\frac{\omega_1}{6}}{N^\frac{3}{4}\abs{i}^\frac{1}{4}}, \qquad 2K\le \abs{i}\le N^{4\omega_\ell+\delta_1}, 
\end{equation} for all $0\le t\le t_*$ and for any $\alpha\in [0,1]$. 
Combining~\eqref{lowerbound} and~\eqref{upperbound} we conclude~\eqref{100} and complete the proof of Proposition~\ref{G33}.
\endproof
\end{proposition}

\section{Proof of Proposition~\ref{DBM prop}: Dyson Brownian motion near the cusp} 
\label{DBMS}

In this section $t_1\le t_*$, indicating that we are before the cusp formation, we recall that $t_1$ is defined as follows
\[
t_1\defeq \frac{N^{\omega_1}}{N^{1/2}},
\]
for a small fixed $\omega_1>0$ and $t_*$ is the time of the formation of the exact cusp.
The main result of this section is the following proposition from which we can quickly 
 prove Proposition~\ref{DBM prop} for $t_1\le t_*$. If $t_1>t_*$ we conclude Proposition~\ref{DBM prop} 
using the analogous Proposition~\ref{cutmin} instead of Proposition~\ref{cut} exactly in the same way.

\begin{proposition}
\label{cut}
For $t_1\le t_*$, with very high probability we have that 
\begin{equation}
\label{cu}
\abs{(\lambda_j(t_1)-\mathfrak{e}^{+}_{\lambda, t_1})-(\mu_{j+i_\mu-i_\lambda}(t_1)-\mathfrak{e}^{+}_{\mu, t_1})}\le N^{-\frac{3}{4}-c\omega_1}
\end{equation} 
for some small constant $c>0$ and for any $j$ such that $\abs{j-i_\lambda}\le N^{\omega_1}$.
\end{proposition}

Note that if $t_1=t_*$ then $\mathfrak{e}_{r,t_*}^+=\mathfrak{e}_{r,t_*}^-=\mathfrak{c}_r$, for $r=\lambda,\mu$, with $\mathfrak{c}_r$ being the exact cusp point of the scDOSs $\rho_{r,t_*}$. The proof of Proposition~\ref{cut} will be given at the end of the section after several auxiliary lemmas. 

\begin{proof}[Proof of Proposition~\ref{DBM prop}]
Firstly, we recall the definition of the physical cusp\[
  \mathfrak{b}_{r,t_1}\defeq \begin{cases}
  \frac{1}{2}(\mathfrak{e}_{r,t_1}^++\mathfrak{e}_{r,t_1}^-) &\text{if} \quad t_1< t_*, \\
  \mathfrak{c}_r &\text{if} \quad t_1= t_*, \\
  \mathfrak{m}_{r,t_1} &\text{if} \quad t_1> t_*.
  \end{cases}
  \]
of \(\rho_{r,t_1}\) as in~\eqref{eq pearcey param choice}, for $r=\lambda,\mu$. Then, using the change of variables $\bx=N^\frac{3}{4}(\bx'-\mathfrak{b}_{r,t_1})$, for $r=\lambda,\mu$, 
and the definition of correlation function, for each Lipschitz continuous and compactly supported test function $F$, we have that 
\begin{equation}
\label{corrpro}
\begin{split}&\int_{\R^k} F(\bx)\left[ N^{k/4} p_{k,t_1}^{(N,\lambda)}\left( \bu_{\lambda,t_1} + \frac{\bx}{ N^{3/4}}\right)
-N^{k/4} p_{k,t_1}^{(N,\mu)}\left( \bu_{\mu,t_1} + \frac{\bx}{ N^{3/4}}\right)\right] \diff \bx \\
&\qquad=N^k\binom{N}{k}^{-1} \sum_{\{i_1,\dots,i_k\}\subset[N]}
\left[\E_{H^{(\lambda)}_{t_1}}F\left(N^\frac{3}{4}(\lambda_{i_1}-\mathfrak{b}_{\lambda,t_1}),\dots,
N^\frac{3}{4}(\lambda_{i_k}-\mathfrak{b}_{\lambda,t_1})\right) -\E_{H^{(\mu)}_{t_1}}F(\lambda \to \mu)\right],
\end{split}\end{equation} 
where $\lambda_1,\dots,\lambda_N$ and $\mu_1,\dots, \mu_N$ are the eigenvalues, labelled in increasing order, of $H^{(\lambda)}_{t_1}$ 
and $H^{(\mu)}_{t_1}$ respectively. 
In $\E_{H^{(\mu)}_{t_1}}F(\lambda \to \mu)$ we also replace $\mathfrak{b}_{\lambda,t_1}$ by $\mathfrak{b}_{\mu,t_1}$.

In order to apply Proposition~\ref{cut} we split the sum in the rhs. of~\eqref{corrpro} into two sums: \begin{equation}
\label{sssum}
 \sum_{\substack{\{i_1,\dots,i_k\}\subset[N] \\ \abs{i_1-i_\lambda},\dots,\abs{i_k-i_\lambda}< N^\epsilon}}\qquad \text{and its complement}\qquad  \sum',
\end{equation} where $\epsilon$ is a positive exponent with $\epsilon\ll \omega_1$.

We start with the estimate for the second sum of~\eqref{sssum}. In particular, we will estimate only 
the term $\E_{H_{t_1}^{(\lambda)}}(\cdot)$, the estimate for $\E_{H_{t_1}^{(\mu)}}(\cdot)$ will follow in an analogous way.

Since the test function $F$ is compactly supported in some set $\Omega\subset \R^k$ and in $\sum'$ there is at least one index $i_l$ 
such that $\abs{i_l-i_\lambda}\ge N^\epsilon$, we have that 
\begin{equation}
\label{s1}
\sum'\E_{H^{(\lambda)}_{t_1}}F\left(N^\frac{3}{4}(\lambda_{i_1}-\mathfrak{b}_{\lambda,t_1}),\dots,
N^\frac{3}{4}(\lambda_{i_k}-\mathfrak{b}_{\lambda,t_1})\right)\lesssim N^{k-1}  \| F \|_\infty  \sum_{i_l:\,\abs{i_l-i_\lambda}\ge N^\epsilon} 
\mathbb{P}_{H^{(\lambda)}_{t_1}}\left(\abs{\lambda_{i_l}-\mathfrak{b}_{\lambda, t_1}}\lesssim C_\Omega N^{-\frac{3}{4}}\right).
\end{equation} 
where $C_\Omega$ is the diameter of $\Omega$. Let $\gamma_{\lambda,i}=\widehat{\gamma}_{\lambda,i}+\mathfrak{e}_{\lambda,t_1}^+$ 
be the classical eigenvalue locations of $\rho_\lambda(t_1)$ defined by~\eqref{def:quan} for all $1-i_\lambda\le i \le N+1-i_\lambda$. 
Then, by the rigidity estimate from~\cite[Corollary 2.6]{1809.03971}, we have that 
\begin{equation}
\label{s2}
\mathbb{P}_{H^{(\lambda)}_{t_1}}\left(\abs{\lambda_{i_l}-\mathfrak{b}_{\lambda,t_1}}\lesssim C_\Omega N^{-\frac{3}{4}}, \abs{i_l-i_\lambda}\ge N^\epsilon\right)\le N^{-D},
\end{equation} 
for each $D>0$ if $N$ is large enough, depending on $C_\Omega$. Indeed, by rigidity it follows that 
\begin{equation}
\label{bch}
\abs{\lambda_{i_l}-\mathfrak{b}_{\lambda,t_1}}\ge
 \abs{\gamma_{\lambda,i_l}-\gamma_{\lambda,i_\lambda}}-\abs{\lambda_{i_l}-\gamma_{\lambda,i_l}}-\abs{\mathfrak{b}_{\lambda,t_1}-\gamma_{\lambda,i_\lambda}}
 \gtrsim \frac{N^{c\epsilon}}{N^\frac{3}{4}}-\frac{N^{c\xi}}{N^\frac{3}{4}}\gtrsim\frac{N^{c\epsilon}}{N^\frac{3}{4}}
\end{equation} 
with very high probability, if $N^\epsilon\le \abs{i_l-i_\lambda}\le \tilde{c}N$, for some $0<\tilde{c}<1$. In~\eqref{bch} we used the rigidity from~\cite[Corollary 2.6]{1809.03971} in the form
 \[
 \abs{\lambda_i-\gamma_{\lambda,i}}\le \frac{N^\xi}{N^\frac{3}{4}},
 \]  
 for any $\xi>0$, with very high probability. Note that~\eqref{s2} and~\eqref{bch} hold for any 
 $\epsilon\gtrsim \xi$. If $\abs{i_l-i_\lambda}\ge \tilde{c}N$, then $\abs{\gamma_{i_l}-\gamma_{i_\lambda}}\sim 1$ 
 and the bound in~\eqref{bch} clearly holds. A similar estimate holds for $H^{(\mu)}_{t_1}$, hence, 
 choosing $D>k+1$ we conclude that the second sum in~\eqref{sssum} is negligible.

Next, we consider the first sum in~\eqref{sssum}. For $t_1\le t_*$ we have, by~\eqref{eq Delta size} that 
\[
\abs{ (\mathfrak{e}_{\lambda,t_1}^+-\mathfrak{b}_{\lambda,t_1})-(\mathfrak{e}_{\mu,t_1}^+-\mathfrak{b}_{\mu,t_1})} 
= \frac{1}{2}\abs{ \Delta_{\lambda,t_1}-\Delta_{\mu,t_1}}\lesssim \Delta_{\mu,t_1}(t_\ast-t_1)^{1/3}
 \le N^{-\frac{3}{4}-\frac{1}{6}+ C\om_1}. 
\]
Hence, by~\eqref{cu}, using that $\abs{F(\bx)-F(\bx')}\lesssim  \| F\|_{C^1}  \lVert \bx-\bx'\rVert$, 
we conclude that 
\begin{equation}
\label{s3}
\sum_{\substack{\{i_1,\dots,i_k\}\subset[N] \\ \abs{i_1-i_\lambda},\dots,\abs{i_k-i_\lambda}\le N^\epsilon}}
\left[\E_{H^{(\lambda)}_{t_1}}F\left(N^\frac{3}{4}(\lambda_{i_1}-\mathfrak{b}_{\lambda,t_1}),\dots,
N^\frac{3}{4}(\lambda_{i_k}-\mathfrak{b}_{\lambda,t_1})\right) -\E_{H^{(\mu)}_{t_1}}F(\lambda\to \mu)\right]
\le C_k  \| F\|_{C^1} \frac{N^{k\epsilon}}{N^{c\omega_1}},
\end{equation} 
for some $c>0$. Then, using that \[\frac{N^k(N-k)!}{N!}=1+\mathcal{O}_k(N^{-1}),\] 
we easily conclude the proof of Proposition~\ref{DBM prop}.
\end{proof}

\subsection{Interpolation.}
In order to prove Proposition~\ref{cut} we recall a few concepts introduced previously. In Section~\ref{sec:padding} we introduced the padding 
particles $x_i(t)$, $y_i(t)$, for $1\le \abs{i}\le N$, that are good approximations of the eigenvalues 
$\lambda_j(t)$, $\mu_j(t)$ respectively, for $1\le j \le N$, in the sense of Lemma~\ref{appr}.
They satisfy a Dyson Brownian Motion equation~\eqref{xxx},~\eqref{yyy} mimicking the DBM of
genuine eigenvalue processes~\eqref{lambda},~\eqref{mu}. It is more convenient to 
consider shifted processes where the edge motion is subtracted.

More precisely, 
for $r=x,y$ and $r(t)=x(t),y(t)$, we defined
\[
\widetilde{r}_i(t)\defeq r_i(t)-\mathfrak{e}_{r,t}^+, \qquad 1\le \abs{i}\le N,
\] 
for all $0\le t\le t_*$. In particular, $\widetilde{r}(t)$ is a solution of 
\begin{equation}
\label{paddingshiftDBM}
\diff\widetilde{r}_i(t)=\sqrt{\frac{2}{N}}\diff B_i+\Bigg(\frac{1}{N}\sum_{j\ne i}\frac{1}{\widetilde{r}_i(t)-\widetilde{r}_j(t)}+\Re [m_{r,t}(\mathfrak{e}_{r,t}^+)]\Bigg)\diff t,
\end{equation} with initial data \begin{equation}
\label{paddid}
\widetilde{r}_i(0)=r_i(0)-\mathfrak{e}_{r,0}^+,
\end{equation} for all $1\le \abs{i}\le N$.

Next, following a similar idea of~\cite{1712.03881}, we also introduced in~\eqref{tildereq}
an interpolation process  between $\widetilde{x}(t)$ and $\widetilde{y}(t)$.
 For any $\alpha\in[0,1]$ we defined the process  $\widetilde{z}(t,\alpha)$  as the solution of 
 \begin{equation}
\label{SDEshift111}
\diff\widetilde{z}_i(t,\alpha)=\sqrt{\frac{2}{N}}\diff B_i+\Bigg(\frac{1}{N}\sum_{j\ne i}\frac{1}{\widetilde{z}_i(t,\alpha)-\widetilde{z}_j(t,\alpha)}+\Phi_\alpha(t)\Bigg)\diff t,
\end{equation} 
with initial data \[\widetilde{z}_i(0,\alpha)=\alpha \widetilde{x}_i(0)+(1-\alpha)\widetilde{y}_i(0),\]for each $1\le \abs{i}\le N$. Recall that
$\Phi_\alpha(t)$ was defined in~\eqref{Phidef} and it is such that $\Phi_0(t)=\Re [m_{y,t}(\mathfrak{e}_{y,t}^+)]$ 
and $\Phi_1(t)=\Re [m_{x,t}(\mathfrak{e}_{x,t}^+)]$.
 Note that $\widetilde{z}_i(t,1)=\widetilde{x}_i(t)$ and $\widetilde{z}_i(t,0)=\widetilde{y}_i(t)$ for all $1\le \abs{i}\le N$ and $0\le t\le t_*$.

We recall the definition of the interpolated quantiles from~\eqref{barg} of Section~\ref{sec:padding};
 \begin{equation}
\label{Mqua}
\overline{\gamma}_i(t)\defeq \alpha \widehat{\gamma}_{x,i}(t)+(1-\alpha)\widehat{\gamma}_{y,i}(t),\qquad \alpha\in [0,1],
\end{equation} 
where $\widehat{\gamma}_{x,i}$ and $\widehat{\gamma}_{y,i}$ are the shifted
 quantiles of $\rho_{x,t}$ and $\rho_{y,t}$ respectively, defined in Section~\ref{sec:padding}. 
 In particular, 
 \[
 \overline{\mathfrak{e}}_t^\pm=\alpha \mathfrak{e}_{x,t}^\pm+(1-\alpha)\mathfrak{e}_{y,t}^\pm,\qquad \alpha\in [0,1].
 \] 
 We denoted the interpolated density, whose quantiles are the $\overline{\gamma}_i(t)$, by $\overline{\rho}_t$
~\eqref{rhozdef}, 
  and its Stieltjes transform by $\overline{m}_t$.

Let $\widehat{z}(t,\alpha)$ be the short range approximation of $\widetilde{z}(t,\alpha)$ defined by~\eqref{333}-\eqref{555},
with exponents $\omega_1\ll \omega_\ell\ll \omega_A\ll 1$ and with initial data $\widehat{z}(0,\alpha)=\widetilde{z}(0,\alpha)$ and $i_*=N^{\frac{1}{2}+C_*\omega_1}$, for some large constant $C_*>0$.
In particular, $\widehat{x}(t)=\widehat{z}(t,1)$ and $\widehat{y}(t)=\widehat{z}(t,0)$. 
Assuming optimal rigidity in~\eqref{rzx} for $\widetilde{r}_i(t)=\widetilde{x}_i(t),\widetilde{y}_i(t)$, 
the following lemma shows that the process $\widetilde{r}$ and its short range approximation $\widehat{r}=\widehat{x},\widehat{y}$ 
stay very close to each other, i.e.~$\abs{\widehat{r}_i-\widetilde{r}_i}\le N^{-\frac{3}{4}-c}$, for some small $c>0$. 
This is the analogue of Lemma 3.7 in~\cite{1712.03881} and its proof, given in Appendix~\ref{SLL},
 follows similar lines.  
It assumes the optimal rigidity,  see~\eqref{new:rzx} below, which is ensured by~\cite[Corollary 2.6]{1809.03971},
see Lemma~\ref{diffx}.

\begin{lemma}
\label{se}
Let $i_*= N^{\frac{1}{2}+C_*\om_1}$. Assume  that $\widetilde{z}(t,0)$ and $\widetilde{z}(t,1)$ satisfy the optimal rigidity 
\begin{equation}
\label{new:rzx}
\sup_{0\le t\le t_1}\abs{\widetilde{z}_i(t,\alpha)-\wh\gamma_{r,i}(t)}\le N^\xi \eta_{\mathrm{f}}^{\rho_{r,t}}(e_{r,t}^+ + \widehat{\gamma}_{r,\pm i}(t)),
 \qquad 1\le \abs{i}\le i_*, \quad \alpha=0, 1,
\end{equation} with $r=x,y$, for any $\xi>0$, with very high probability. Then, for $\alpha=0$ or $\alpha=1$ we have that
\begin{equation}
\label{shortestimate1a}
\begin{split}
&\sup_{1\le\abs{i}\le N} \sup_{0\le t\le t_1}\abs{\widetilde{z}_i(t,\alpha)-\widehat{z}_i(t,\alpha)}\\
&\qquad\qquad \lesssim \frac{N^{\frac{\omega_1}{6}}N^\xi}{N^\frac{3}{4}}\left(\frac{N^{\omega_1}}{N^{3\omega_\ell}}+\frac{N^{\omega_1}}{N^\frac{1}{8}}+\frac{N^{C\omega_1}N^\frac{\omega_A}{2}}{N^\frac{1}{6}}+\frac{N^\frac{\omega_A}{2}N^{C\omega_1}}{N^\frac{1}{4}}+\frac{N^{C\omega_1}}{N^\frac{1}{18}}\right),
\end{split}\end{equation} for any $\xi>0$, with very high probability. 
\end{lemma}

In particular,~\eqref{shortestimate1a} implies that there exists a small fixed universal constant $c>0$ such that
 \begin{equation}
\label{goodsl}
\sup_{1\le \abs{i}\le N} \sup_{0\le t\le t_1}\abs{\widetilde{z}_i(t,\alpha)-\widehat{z}_i(t,\alpha)}\lesssim N^{-\frac{3}{4}-c}, \qquad \alpha=0,1
\end{equation} with very high probability.

\begin{remark}\label{3power}
Note the denominator in the first error term in~\eqref{shortestimate1a}: the factor $N^{3\om_\ell}$ is better than $N^{2\om_\ell}$ in
Lemma 3.7 in~\cite{1712.03881}, this
is because of the natural  cusp scaling. The fact that this power is at least $N^{(1+\epsilon)\om_\ell}$ was essential
in~\cite{1712.03881} since this allowed to transfer the optimal rigidity from $\wt z$ to the  $\wh z$ process for all $\alpha\in [0,1]$.
Optimal rigidity for $\wh z$ is essential   (i) for the heat kernel bound for the propagator of $\mathcal{L}$, see~\eqref{weq}--\eqref{11},
 and (ii) for a good  $\ell^p$-norm for the initial condition in~\eqref{v0norm}.  
With our approach, however, this power in~\eqref{shortestimate1a} is not critical since we  have already obtained an even better, $i$-dependent rigidity for
the $\wh z$ process for any $\alpha$ by using maximum principle, see Proposition~\ref{G33}. 
We still need~\eqref{shortestimate1a} for the $x$ and $y$ processes (i.e.~only for $\alpha=0,1$), but only with  a precision
below the rigidity scale, therefore the denominator in the first term has only to beat $N^{\frac{7}{6}\om_1+\xi}$.
\end{remark}

\subsection{Differentiation.}
Next, we consider the $\alpha$-derivative of the process $\widehat{z}(t,\alpha)$. Let 
\[
 u_i(t)=u_i(t, \alpha)\defeq \partial_\alpha \widehat{z}_i(t,\alpha),  \qquad 1\le \abs{i}\le N, 
 \]
 then $u$ is a solution of the equation
  \begin{equation}
\label{eq1}
\partial_tu=\mathcal{L}u+\zeta^{(0)},
\end{equation} where $\zeta^{(0)}$, defined by~\eqref{AAAAA}-\eqref{AAA}, is an error term that is non zero only for $\abs{i}> N^{\omega_A}$ and such that $\abs{\zeta^{(0)}_i}\lesssim N^C$, for some large constant $C>0$ with very high probability, by~\eqref{zeta}, and the operator $\mathcal{L}=\mathcal{B}+\mathcal{V}$ acting on $\mathbb{R}^{2N}$ is defined by~\eqref{11}-\eqref{V222rig1}.

In the following with $\mathcal{U}^\mathcal{L}$ we denote the semigroup associated to~\eqref{eq1}, i.e.
by Duhamel principle 
\[
u(t)=\mathcal{U}^\mathcal{L}(0,t)u(0)+\int_0^t\mathcal{U}^\mathcal{L}(s,t)\zeta^{(0)}(s)\,\diff s
\]
 and $\mathcal{U}^\mathcal{L}(s,s)=\mbox{Id}$ for all $0\le s\le t$. 
 Furthermore, for each $a,b$ such that $\abs{a},\abs{b}\le N$, with $\mathcal{U}^\mathcal{L}_{ab}$ 
 we denote the entries of $\mathcal{U}^\mathcal{L}$, which can be either seen as the solution of the equation~\eqref{eq1} 
 with initial condition $u_a(0)=\delta_{ab}$.

By Proposition~\ref{prop:option2} and Lemma~\ref{seb111}, for any fixed $\alpha\in [0,1]$, it follows that   
 
\begin{equation}
\label{RZH1}
\sup_{0\le t\le t_*}\abs{\widehat{z}_i(t,\alpha)-\overline{\gamma}_i(t)}\lesssim \frac{N^{C\omega_1}}{N^\frac{1}{2}},\qquad 1\le \abs{i}\le  N,
\end{equation}  and 
 \begin{equation}
\label{RZH}
\sup_{0\le t\le t_*}\abs{\widehat{z}_i(t,\alpha)-\overline{\gamma}_i(t)}\lesssim \frac{N^{C\omega_1}}{N^\frac{3}{4}},\qquad 1\le \abs{i}\le  i_*,
\end{equation} with very high probability. Then, using~\eqref{RZH}, as a consequence of Lemma~\ref{fsl} we have the following: \begin{lemma}
\label{fsl2}
There exists a constant $C>0$ such that for any $0<\delta< C\omega_\ell$, if $1\le \abs{a}\le \frac{1}{2} N^{4\omega_\ell+\delta}$ and $\abs{b}\ge N^{4\omega_\ell+\delta}$, then \begin{equation}
\label{finitespeed3}
\sup_{0\le s\le t\le t_*} \mathcal{U}^\mathcal{L}_{ab}(s,t)+\mathcal{U}^\mathcal{L}_{ba}(s,t)\le N^{-D}
\end{equation} for any $D>0$ with very high probability.
\end{lemma}

Furthermore, by Proposition~\ref{G33}, for any fixed $\alpha\in [0,1]$, we have that \begin{equation}
\label{ABC}
\sup_{0\le t\le t_*}\abs{\widehat{z}_i(t,\alpha)-\overline{\gamma}_i(t)}\lesssim \frac{N^\xi N^\frac{\omega_1}{6}}{N^\frac{3}{4}\abs{i}^\frac{1}{4}},
\qquad 1\le \abs{i}\le N^{4\omega_\ell+\delta_1},
\end{equation} for some small fixed $\delta_1>0$ and for any $\xi>0$ with very high probability.

Next, we introduce the $\ell^p$ norms \[ \lVert u\rVert_p\defeq \left(\sum_i\abs{u_i}^p\right)^\frac{1}{p}, \,\,\,\,\, \lVert u\rVert_\infty\defeq \max_i\abs{u_i}.\] Following a similar scheme to~\cite{MR3253704},~\cite{MR3372074} with some minor modifications we will prove the following Sobolev-type inequalities in Appendix~\ref{STI}.

\begin{lemma}\label{lm:Sobolev} For any small $\eta>0$ there exists $c_\eta>0$ such that
\begin{equation}
\label{Sobolev}
\sum_{i\ne j\in\mathbb{Z}_+}\frac{(u_i-u_j)^2}{\abs[1]{i^\frac{3}{4}-j^\frac{3}{4}}^{2-\eta}}\ge c_\eta \left(\sum_{i\ge1}\abs{u_i}^p\right)^\frac{2}{p},\,\,\, \sum_{i\ne j\in\mathbb{Z}_-}\frac{(u_i-u_j)^2}{\abs[1]{\abs{i}^\frac{3}{4}-\abs{j}^\frac{3}{4}}^{2-\eta}}\ge c_\eta \left(\sum_{i\le -1}\abs{u_i}^p\right)^\frac{2}{p}
\end{equation}
hold, with $p=\frac{8}{2+3\eta}$, for any function  $\norm{ u}_p<\infty$.
\end{lemma}

Using the Sobolev inequality in~\eqref{Sobolev} and the finite speed estimate of Lemma~\ref{fsl2}, in Appendix~\ref{hke} we prove the energy estimates for the heat kernel in Lemma~\ref{ee} via a Nash-type argument.

\begin{lemma}
\label{ee}
Assume~\eqref{RZH1},~\eqref{RZH} and~\eqref{ABC}. Let $0<\delta_4<\frac{\delta_1}{10}$ and $w_0\in\mathbb{R}^{2N}$ such that $\abs{(w_0)_i}\le N^{-100}\lVert w_0\rVert_1,$ for $\abs{i}\ge \ell^4N^{\delta_4}$. Then, for any small $\eta>0$ there exists a constant $C>0$ independent of $\eta$ and a constant $c_\eta$ such that for all $0\le s\le t\le t_*$ \begin{equation}
\label{Ee1}
\lVert\mathcal{U}^\mathcal{L}(s,t)w_0\rVert_2\le \left(\frac{N^{C\eta+\frac{\omega_1}{3}}}{c_\eta N^\frac{1}{2} (t-s)}\right)^{1-3\eta}\lVert w_0\rVert_1,
\end{equation} and  \begin{equation}
\label{Ee2}
\lVert\mathcal{U}^\mathcal{L}(0,t)w_0\rVert_\infty\le \left(\frac{N^{C\eta+\frac{\omega_1}{3}}}{c_\eta N^\frac{1}{2}t}\right)^{\frac{2(1-3\eta)}{p}}\lVert w_0\rVert_p,
\end{equation} for each $p\ge 1$. 
\end{lemma}

Let $0<\delta_v<\frac{\delta_4}{2}$. Define $v_i= v_i(t,\alpha)$ to be the solution of \begin{equation}
\label{v}
\partial_tv=\mathcal{L}v,\,\,\,\,\,\,\,\,v_i(0,\alpha)=u_i(0,\alpha)\bm{1}_{\{\abs{i}\le N^{4\omega_\ell+\delta_v}\}}.
\end{equation}

Then, by Lemma~\ref{fsl2} the next lemma follows. \begin{lemma}
\label{uv}
Let $u$ be the solution of the equation in~\eqref{eq1} and $v$ defined by~\eqref{v}, then we have that \begin{equation}
\label{uveq}
\sup_{0\le t\le t_1}\sup_{\abs{i}\le \ell^4}\abs{u_i(t)-v_i(t)}\le N^{-100},
\end{equation} with very high probability.
\proof
By~\eqref{eq1} and~\eqref{v} have that \[u_i(t)-v_i(t)=\sum_{j=-N}^N\mathcal{U}^\mathcal{L}_{ij}(0,t)u_j(0)-\sum_{j=-N^{4\omega_\ell+\delta_v}}^{N^{4\omega_\ell+\delta_v}}\mathcal{U}^\mathcal{L}_{ij}(0,t)u_j(0)+\int_0^t \sum_{\abs{j}\ge N^{\omega_A}}\mathcal{U}^\mathcal{L}_{ij}(s,t)\zeta_j^{(0)}(s)\,\diff s.\] Then, using that $\zeta_i^{(0)}=0$ for $1\le \abs{i}\le N^{\omega_A}$ and~\eqref{zeta}, the bound in~\eqref{uveq} follows by Lemma~\ref{fsl2}.
\endproof
\end{lemma}

\begin{proof}[Proof of Proposition~\ref{cut}.]
We consider only the $j=i_\lambda$ case. By Lemma~\ref{appr} and~\eqref{goodsl} we have that 
\[
\begin{split}
\abs{(\lambda_{i_\lambda}(t_1)-\mathfrak{e}^{+}_{\lambda,t_1})-(\mu_{i_\mu}(t_1)-\mathfrak{e}^{+}_{\mu, t_1})}
&\le \abs{\widetilde{x}_1(t_1)-\widehat{x}_1(t_1)}+\abs{\widehat{x}_1(t_1)-\widehat{y}_1(t_1)}+\abs{\widehat{y}_1(t_1)-\widetilde{y}_1(t_1)}\\
&\le \abs{\widehat{x}_1(t_1)-\widehat{y}_1(t_1)}+N^{-\frac{3}{4}-c}
\end{split}
\] 
with very high probability.

Since $\widehat{z}_i(t_1,1)=\widehat{x}_i(t_1)$ and $\widehat{z}_i(t_1,0)=\widehat{y}_i(t_1)$ for all $1\le \abs{i}\le N$, by the definition of $u_i(t,\alpha)$, it
 follows that \[\widehat{x}_1(t_1)-\widehat{y}_1(t_1)=\int_0^1u_1(t_1,\alpha)\,\diff \alpha.\]

Furthermore, by a high moment Markov inequality as in~\eqref{markov}-\eqref{markov1} and Lemma~\ref{uv}, we get 
\[\int_0^1\abs{ u_1(t_1,\alpha)}\, \diff \alpha\lesssim N^{-100}+
\int_0^1 \abs{ v_1(t_1,\alpha)}\, \diff\alpha.
\] 
Since $v_i(0)=u_i(0)\bm{1}_{\{\abs{i}\le  N^{4\omega_\ell+\delta_v}\}}$ and, by~\eqref{gamma difference gap} and~\eqref{rzx}, for $1\le \abs{i}\le N^{4\omega_\ell+\delta_v}$ we have that \[\begin{split}
\abs{u_i(0)}&\lesssim \abs{\widehat{x}_i(0)-\widehat{\gamma}_{x,i}(0)}+\abs{\widehat{y}_i(0)-\widehat{\gamma}_{y,i}(0)}+\abs{\widehat{\gamma}_{x,i}(0)-\widehat{\gamma}_{y,i}(0)} \\
&\lesssim \frac{N^\frac{\omega_1}{6}}{\abs{i}^\frac{1}{4}N^\frac{3}{4}}+\frac{\abs{i}^\frac{3}{4}N^\frac{\omega_1}{2}}{N^\frac{11}{12}}\lesssim \frac{N^\frac{\omega_1}{6}}{\abs{i}^\frac{1}{4}N^\frac{3}{4}},\end{split}\] we conclude that 
\begin{equation}\label{v0norm}
\lVert v(0)\rVert_5\lesssim \frac{N^{\frac{\omega_1}{6}}}{N^\frac{3}{4}}
\end{equation}
 with very high probability. Hence, reaclling that $t_1= N^{-1/2+\omega_1}$, by~\eqref{Ee2} and Markov's inequality again, we get \begin{equation}
\label{fefe}
\int_0^1\abs{v_1(t_1,\alpha)}\,\diff\alpha \le \sup_{\alpha\in [0,1]}
\lVert v(t_1,\alpha)\rVert_\infty \le \left(\frac{N^{C\eta+\omega_1/3}}{N^{1/2}t_1}\right)^\frac{2(1-3\eta)}{5}\lVert v(0)\rVert_5\lesssim \frac{N^{\frac{\omega_1}{6}+\frac{\eta}{5}(2C+3\omega_1-6\eta C)}}{N^\frac{3}{4}N^{\frac{4\omega_1}{15}}}=\frac{1}{N^\frac{3}{4}N^\frac{\omega_1}{20}}, 
\end{equation}  with very high probability, for $\eta$ small enough, say $\eta\le \omega_1(8C+12\omega_1)^{-1}$. 
Notice that the constant in front of the $\omega_1$ in the exponents play a crucial role: eventually the constant $\big( 1-\frac{1}{3}\big)\frac{2}{5} =\frac{4}{15}$
from the Nash estimate beats the constant $\frac{1}{6}$ from~\eqref{v0norm}.
This completes the proof of Proposition~\ref{cut}.
\end{proof}

\section{Case of \texorpdfstring{$t\ge t_*$}{t>=t*}: small minimum}
\label{UNMIN}

In this section we consider the case when the densities $\rho_{x,t}$, $\rho_{y,t}$, hence their interpolation $\overline{\rho}_t$ as well,
 have a small minimum,
 i.e.~$t_*\le t\le 2t_*$. We deal with the small minimum case in this separate section mainly for notational reasons:
  for $t_*\le t \le 2t_*$ the processes $x(t)$ and $y(t)$, and consequently the associated quantiles and densities, 
 are shifted by $\widetilde{\mathfrak{m}}_{r,t}$, for $r=x,y$, instead of $\mathfrak{e}_{r,t}^+$. 
 We recall that $\widetilde{\mathfrak{m}}_{r,t}$, defined in~\eqref{eq wt mi},
  denotes a close approximation of the actual local minimum $\mathfrak{m}_{r,t}$ near the physical cusp. 
 We chose to shift $x(t)$ and $y(t)$ by the tilde approximation of the minimum instead of the minimum itself for technical reasons, namely because
  the $t$-derivative of $\widetilde{\mathfrak{m}}_{r,t}$, $r=x,y$, satisfies the convenient relation in~\eqref{diff m m}.

As we explained at the beginning of Section~\ref{DBMS},
in order to prove universality, i.e.~Proposition~\ref{DBM prop} at time $t_1\ge t_*$, 
 it is enough to prove the following: 
 \begin{proposition}
\label{cutmin} For $t_1\ge t_*$,  we have, 
with very high probability,  that
 \begin{equation}
\label{cumin}
\abs{(\lambda_j(t_1)-\mathfrak{m}_{\lambda, t_1})-(\mu_{j+i_\mu-i_\lambda}(t_1)-\mathfrak{m}_{\mu, t_1})}\le N^{-\frac{3}{4}-c}
\end{equation} for some small constant $c>0$ and for any $j$ such that $\abs{j-i_\lambda}\le N^{\omega_1}$. Here $\mathfrak{m}_{\lambda, t_1}$ 
and $\mathfrak{m}_{\mu, t_1}$ are the local  minimum of $\rho_{\lambda, t_1}$ and $\rho_{\mu, t_1}$, respectively.
\end{proposition}

We  introduce the shifted process $\widetilde{r}_i(t)=\widetilde{x}_i(t),\widetilde{y}_i(t)$ for $t\ge t_*$. Let us define \begin{equation}
\label{tildemin}
\widetilde{r}_i(t)\defeq r_i(t)-\widetilde{\mathfrak{m}}_{r,t}, \qquad 1\le \abs{i}\le N,
\end{equation} for $r=x,y$, hence, by~\eqref{diff m m}, the shifted points satisfy the following DBM \begin{equation}
\label{SDEshiftmin}
\diff\widetilde{r}_i(t)=\sqrt{\frac{2}{N}}\diff B_i+\frac{1}{N}\sum_{j\ne i}\frac{1}{\widetilde{r}_i(t)-\widetilde{r}_j(t)}\diff t
-\Big(\frac{\diff}{\diff t}\widetilde{\mathfrak{m}}_{r,t}\Big) \diff t.
\end{equation} 
Furthermore we recall that by $\widehat{\gamma}_{r,i}(t)$ we the 
denote the quantiles of $\rho_{r,t}$, with $r=x,y$, for all $t_*\le t\le 2t_*$, i.e.~
\[ \widehat{\gamma}_{r,i}=\gamma_{r,i}-\widetilde{\mathfrak{m}}_{r,t},\qquad 1\le \abs{i}\le N.\]

By the rigidity estimate of~\cite[Corollary 2.6]{1809.03971}, using Lemma~\ref{appr} and the fluctuation scale estimate in~\eqref{fluc scale gap} the proof of  the following lemma is immediate.

\begin{lemma}
\label{diffxmin}
Let $\widetilde{r}(t)=\widetilde{x}(t),\widetilde{y}(t)$. There exists a fixed small $\epsilon>0$, such that for each $1\le \abs{i}\le \epsilon N$, we have \begin{equation}
\label{rzxmin}
\sup_{t_*\le t\le t_1}\abs{\widetilde{r}_i(t)-\widehat{\gamma}_{r,i}(t)}\le N^\xi \eta_\mathrm{f}^{\rho_{r,t}}(\gamma_{r,i}(t)),
\end{equation} for any $\xi>0$ with very high probability, where we recall that the behavior of $\eta_{\mathrm{f}}^{\rho_{r,t}}(\ed_{r,t}^++\widehat{\gamma}_{r,\pm i}(t))$, with $r=x,y$, is given by~\eqref{fluc scale min}.
\end{lemma}

In order to prove Proposition~\ref{cutmin}, by Lemma~\ref{appr} and~\eqref{eq wt mi mi}, it is enough to prove the following proposition.

\begin{proposition}
\label{cutmin2} For $t_1\ge t_*$ we have, with very high probability, that
 \begin{equation}
\label{cumin2}
\abs{(x_i(t_1)-\widetilde{\mathfrak{m}}_{x,t_1})-(y_i(t_1)-\widetilde{\mathfrak{m}}_{y,t_1})}\le N^{-\frac{3}{4}-c}
\end{equation} for some small constant $c>0$ and for any $1\le \abs{i}\le N^{\omega_1}$.
\end{proposition}

The remaining part of this section is devoted to the proof of Proposition~\ref{cutmin2}.  We start with some preparatory lemmas.
We recall the definition of the interpolated quantiles given in Section~\ref{sec:padding}, \begin{equation}
\label{hatquaminint}
\overline{\gamma}_i(t)\defeq \alpha \widehat{\gamma}_{x,i}(t)+(1-\alpha)\widehat{\gamma}_{y,i}(t),
\end{equation} for all $\alpha\in [0,1]$ and $t_*\le t\le 2t_*$, as well as
 \[\overline{\mathfrak{m}}_t\defeq \alpha \widetilde{\mathfrak{m}}_{x,t}+(1-\alpha)\widetilde{\mathfrak{m}}_{y,t},\]
  for all $\alpha\in [0,1]$ and $t_*\le t\le 2t_*$. Furthermore by $\overline{\rho}_t$ from~\eqref{rhozdef} 
  we denote the interpolated density between $\rho_{x,t}$ and $\rho_{y,t}$ and by $\overline{m}_t$ its Stieltjes transform.

We now define the process $\widetilde{z}_i(t,\alpha)$ whose initial data are given by the linear interpolation of 
$\widetilde{x}(0)$ and $\widetilde{y}(0)$. Analogously to the small gap case, we define the function $\Psi_\alpha(t)$, for $t_*\le t\le 2t_*$,
 that represents the correct shift of the process $\widetilde{z}(t,\alpha)$, in order to compensate 
 the discrepancy of our choice of the interpolation for $\overline{\rho}_t$ with 
 respect to the semicircular flow evolution of the density $\overline{\rho}_0$.

Analogously to the edge case, see~\eqref{h*def}-\eqref{hhdef}, we define $h(t,\alpha)$ with the following properties
\begin{equation}
\label{hmin}
h(t,\alpha)=\alpha\Re[m_{x,t}(\widetilde{\mathfrak{m}}_{x,t})]+(1-\alpha)\Re[m_{y,t}(\widetilde{\mathfrak{m}}_{y,t})]-\Re[\overline{m}_t(\overline{\mathfrak{m}}_t+\ii N^{-100})]+\mathcal{O}\left(N^{-1}\right)
\end{equation} 
and $h(t,0)=h(t,1)=0$. 
Then, similarly to the edge case, we define 
\begin{equation}
\label{Psimin}
\Psi_\alpha(t)\defeq -\alpha \frac{\diff}{\diff t}[m_{x,t}(\widetilde{\mathfrak{m}}_{x,t})]-
(1-\alpha)\frac{\diff}{\diff t}[m_{y,t}(\widetilde{\mathfrak{m}}_{y,t})]-h(t,\alpha).
\end{equation} 
In particular, by our definition of $h(t,\alpha)$ in~\eqref{hmin} it follows that 
$\Psi_0(t)=\frac{\diff}{\diff t}\widetilde{\mathfrak{m}}_{y,t}$, 
$\Psi_1(t)=\frac{\diff}{\diff t}\widetilde{\mathfrak{m}}_{x,t}$ and that
 \begin{equation}\label{Psider}
 \Psi_\alpha(t)=\Re[\overline{m}_t(\overline{\mathfrak{m}}_t)]+\mathcal{O}(N^{-\frac{1}{2}+\om_1}).
 \end{equation}
Note that the error in~\eqref{Psider} is somewhat weaker than in the analogous equation~\eqref{Phiest}
due to the additional error  in~\eqref{diff m m} compared with~\eqref{diff e m}.

More precisely, the process $\widetilde{z}(t,\alpha)$ is defined by \begin{equation}
\label{interminz}
\diff\widetilde{z}_i(t,\alpha)=\sqrt{\frac{2}{N}}\diff B_i+\Bigg[\frac{1}{N}\sum_{j\ne i}\frac{1}{\widetilde{z}_i(t,\alpha)-\widetilde{z}_j(t,\alpha)}+\Psi_\alpha(t)\Bigg]\diff t,
\end{equation} with initial data \begin{equation}
\label{indatamin}
\widetilde{z}_i(t_*,\alpha)\defeq \alpha\widetilde{x}_i(t_*)+(1-\alpha)\widetilde{y}_i(t_*),
\end{equation} 
for all $1\le \abs{i}\le N$ and for all $\alpha\in [0,1]$.

We recall that $\omega_1\ll \omega_\ell\ll \omega_A\ll 1$ and that $i_* = N^{\frac{1}{2}+C_*\om_1}$ with some large constant $C_*$.

Next, we define the analogue of $\mathcal{J}_{z}(t)$ and $\mathcal{I}_{z,i}(t)$ for the small minimum by~\eqref{Bigint} and~\eqref{intmaxj} using the definition in~\eqref{hatquaminint} for the quantiles. Then, for each $t_*\le t\le t_1$, we define the short range approximation $\widehat{z}_i(t,\alpha)$ of $\widetilde{z}(t,\alpha)$ by the following SDE.

For $\abs{i}> \frac{i_*}{2}$ we let \begin{equation}
\label{SDE1min}
\diff\widehat{z}_i(t,\alpha)=\sqrt{\frac{2}{N}}\diff B_i+\Bigg[\frac{1}{N}\sum_j^{\mathcal{A},(i)}\frac{1}{\widehat{z}_i(t,\alpha)-\widehat{z}_j(t,\alpha)}+\frac{1}{N}\sum_j^{\mathcal{A}^c,(i)}\frac{1}{\widetilde{z}_i(t,\alpha)-\widetilde{z}_j(t,\alpha)}+\Psi_\alpha(t)\Bigg] \diff t,
\end{equation} for $\abs{i}\le N^{\omega_A}$ \begin{equation}
\label{SDE2min}
\diff\widehat{z}_i(t,\alpha)=\sqrt{\frac{2}{N}}\diff B_i+\Bigg[\frac{1}{N}\sum_j^{\mathcal{A},(i)}\frac{1}{\widehat{z}_i(t,\alpha)-\widehat{z}_j(t,\alpha)}+\int_{\mathcal{I}_{y,i}(t)^c}\frac{\rho_{y,t}(E+\widetilde{\mathfrak{m}}_{y,t}^+)}{\widehat{z}_i(t,\alpha)-E}\diff E\bigg]\diff t
-\Big(\frac{\diff}{\diff t}\widetilde{\mathfrak{m}}_{r,t}\Big) \diff t,
\end{equation} and for $N^{\omega_A}< \abs{i}\le \frac{i_*}{2}$ \begin{equation}
\label{SDE3min}
\begin{split}
\diff\widehat{z}_i(t,\alpha)&=\sqrt{\frac{2}{N}}\diff B_i+\Bigg[\frac{1}{N}\sum_j^{\mathcal{A},(i)}\frac{1}{\widehat{z}_i(t,\alpha)-\widehat{z}_j(t,\alpha)}+\int_{\mathcal{I}_{z,i}(t)^c\cap\mathcal{J}_z(t)}\frac{\overline{\rho}_t(E+\overline{\mathfrak{m}}_t^+)}{\widehat{z}_i(t,\alpha)-E}\diff E\\
&\quad +\sum_{\abs{j}\ge \frac{3}{4}i_*}\frac{1}{\widetilde{z}_i(t,\alpha)-\widetilde{z}_j(t,\alpha)}+\Psi_\alpha(t)\Bigg] \diff t,
\end{split}\end{equation} with initial data \begin{equation}
\label{idmin}
\widehat{z}_i(t_*,\alpha)\defeq \widetilde{z}_i(t_*,\alpha).
\end{equation} 

Next, by Lemma~\ref{semin111aaa}, by the optimal rigidity in~\eqref{rzxmin} for $\widetilde{x}(t)$ and $\widetilde{y}(t)$, the next lemma follows immediately.

\begin{lemma}
\label{semin}
For $\alpha=0$ and $\alpha=1$, with very high probability, we have 
\begin{equation}
\label{shortestimatemin}
\sup_{1\le \abs{i}\le N} \sup_{t_*\le t\le t_1}\abs{\widetilde{z}_i(t,\alpha)-\widehat{z}_i(t,\alpha)}\lesssim \frac{N^\xi}{N^\frac{3}{4}}\left(\frac{N^{\omega_1}}{N^{3\omega_\ell}}+\frac{N^{{C}\omega_1}}{N^\frac{1}{24}}\right),
\end{equation} for any $\xi>0$ and ${C}>0$ a large universal constant.
\end{lemma}

In order to proceed with the heat-kernel estimates we need an
 optimal $i$-dependent rigidity for $\widehat{z}_i(t,\alpha)$ for $1\le \abs{i}\le N^{4\omega_\ell+\delta}$,
  for some $0<\delta< C\omega_\ell$. 
   In particular, analogously to Proposition~\ref{G33}  we have:

\begin{proposition}
\label{optimrigmin}
Fix any $\alpha\in [0,1]$. There exists a small fixed $0<\delta_1<C\omega_\ell$, for some constant $C>0$, such that \begin{equation}
\label{RHZM}
\sup_{t_*\le t\le 2t_*}\abs{\widehat{z}_i(t,\alpha)-\overline{\gamma}_i(t)}
\lesssim \frac{N^\xi N^{\frac{\om_1}{6}}}{N^\frac{3}{4}\abs{i}^\frac{1}{4}},\qquad 1\le \abs{i}\le N^{4\om_\ell +\delta_1}
\end{equation} for any $\xi>0$ with very high probability.
\end{proposition}

\begin{proof} We can adapt the arguments in Section~\ref{sec:rigid}  to the case of the  small minimum, $t\ge t_*$, 
in a straightforward way. In Section~\ref{sec:rigid}, 
as the main input, we used the precise estimates on the density $\rho_{r,t}$~\eqref{rho rho gap},~\eqref{rhoquantile edge},
on the quantiles $\wh\gamma_{r,i}(t)$
\eqref{gamma gap}, on the quantile gaps~\eqref{gamma difference gap}, on
the fluctuation scale~\eqref{fluc scale gap} and on the Stieltjes transform~\eqref{Re m gap};
all formulated for the small gap case, $0\le t\le t_*$. In the small minimum case, $t\ge t_*$,  the corresponding estimates
are all available in Section~\ref{sec scflow}, see~\eqref{rho rho min},\eqref{rhoquantile m},\eqref{gamma min},
\eqref{gamma difference min},\eqref{fluc scale min} and~\eqref{Re m min}, respectively.
 In fact, the  semicircular flow is more regular after
the cusp formation, see e.g.~the better (larger) exponent in the $(t-t_*)$ error terms 
when comparing~\eqref{rho rho gap} with~\eqref{rho rho min}. This makes handling the  small minimum  case easier.
The most critical part 
in Section~\ref{sec:rigid} is the estimate of the forcing term (Proposition~\ref{prop:F}),
where the  derivative of the density~\eqref{rho' difference hash} was heavily used.
The main mechanism of this proof is the delicate cancellation between the contributions
to $S_2$ from the intervals $[\gamma_{i-n-1}, \gamma_{i-n}]$ and $[\gamma_{i+n-1}, \gamma_{i+n}]$,
see~\eqref{symm}. This cancellation takes place away from the edge. The proof is divided
into two cases; the so-called ``edge regime'', where the gap length $\Delta$ is relatively large and 
the ``cusp regime'', where $\Delta$ is small or zero. The adaptation of this argument
to the  small minimum case, $t\ge t_*$, will be identical to  the proof for the small gap case  in
the ``cusp regime''. In this regime the derivative bound~\eqref{rho' difference hash}
is used only in the form $\abs{\rho'} \lesssim \rho^{-2}$ which is available in the
small minimum case, $t\ge t_*$, as well, see~\eqref{rho' difference min}.  This  proves Proposition~\ref{prop:F}
for $t\ge t_*$. The rest of the argument is identical to the proof in the small minimum case up to  obvious
 notational changes; the
details are  left to the reader.
\end{proof}

Let us define $u_i(t,\alpha)\defeq \partial_\alpha\widehat{z}_i(t,\alpha)$, for $t_*\le t\le 2t_*$. In particular, $u$ is a solution of the equation 
\begin{equation}
\label{eq1min}
\partial_tu=\mathcal{L}u+\zeta^{(0)}
\end{equation} 
with initial condition $u(t_*,\alpha) = \wt x(t_*)-\wt y(t_*)$ from~\eqref{indatamin}.
The error term $\zeta^{(0)}$ is  defined analogously to~\eqref{AAAAA}-\eqref{AAA} but replacing $\Phi_\alpha$ and $\ov\ed_t^+$
with $\Psi_\alpha$ and $\wt\mi_t$, respectively. Note that 
this error term is non zero only for $\abs{i}\ge N^{\omega_A}$ and for any $i$ we have $\abs{\zeta^{(0)}_i}\le N^C$ with very high probability,
 for some large $C>0$. Furthermore, 
  $\mathcal{L}=\mathcal{B}+\mathcal{V}$ is defined as in~\eqref{11}-\eqref{V222rig1} replacing $\mathfrak{e}_{y,t}^+$ and $\overline{\ed}_t^+$ 
  by $\widetilde{\mathfrak{m}}_{y,t}$ and $\overline{\mathfrak{m}}_t$, respectively. 
  In the following by $\mathcal{U}^\mathcal{L}$ we denote the propagator of the operator $\mathcal{L}$.

Let $0<\delta_v<\frac{\delta_4}{2}$, with $\delta_4$ defined in Lemma~\ref{ee}. Define $v_i=v_i(t,\alpha)$ to be the solution of \begin{equation}
\label{vmin}
\partial_tv=\mathcal{L}v,\,\,\,\,\,\,\,\,v_i(t_*,\alpha)=u_i(t_*,\alpha)\bm{1}_{\{\abs{i}\le N^{4\omega_\ell+\delta_v}\}}.
\end{equation}
By the finite speed of propagation  estimate in Lemma~\ref{fsl}, similarly to the proof of Lemma~\ref{uv},
we immediately obtain the following:
 \begin{lemma}
\label{uvmin}
Let $u$ be the solution of the equation in~\eqref{eq1min} and $v$ defined by~\eqref{vmin}, then we have that \begin{equation}
\label{uveqmin}
\sup_{t_*\le t\le 2t_*}\sup_{1\le \abs{i}\le \ell^4}\abs{u_i(t)-v_i(t)}\le N^{-100}
\end{equation}
with very high probability.
\end{lemma}

Collecting all the previous lemmas we conclude this section  with the proof of Proposition~\ref{cutmin2}.

\begin{proof}[Proof of Proposition~\ref{cutmin2}.]
We consider only the $i=1$ case.  By Lemma~\ref{appr} and Lemma~\ref{semin} we have that
 \[\begin{split}\abs{(x_1(t_1)-\widetilde{\mathfrak{m}}_{x,t_1})-(y_1(t_1)-\widetilde{\mathfrak{m}}_{y,t_1})}&\le \abs{\widetilde{x}_1(t_1)-\widehat{x}_1(t_1)}+\abs{\widehat{x}_1(t_1)-\widehat{y}_1(t_1)}+\abs{\widehat{y}_1(t_1)-\widetilde{y}_1(t_1)}\\
&\le \abs{\widehat{x}_1(t_1)-\widehat{y}_1(t_1)}+\frac{1}{N^{\frac{3}{4}+c}}\end{split}\] with very high probability. 
Since $u(t,\alpha)=\partial_\alpha\widehat{z}(t,\alpha)$,  using
$\widehat{x}_1(t_1)-\widehat{y}_1(t_1) =\int_0^1 u(t_1,\alpha)\diff \alpha$ and 
 Lemma~\ref{uvmin} it will be sufficient to estimate $\int_0^1\abs{v_1(t_1,\alpha)}\,\diff\alpha$.
 By rigidity from~\eqref{rzxmin}, we have
  \[ \abs{v_i(t_*,\alpha)}  =  \abs{u_i(t_*,\alpha)}  = \abs{\wt y_i(t_*) -\wt x_i(t_*)}
  \lesssim \frac{N^\xi}{N^\frac{3}{4}\abs{i}^\frac{1}{4}},
 \] 
 for any $1\le \abs{i}\le N^{4\omega_\ell+\delta_v}$ hence
 \[
  \lVert v(t_*,\alpha)\rVert_5\lesssim \frac{N^\xi}{N^\frac{3}{4}},
  \] 
  for any $\xi>0$ with very high probability.

 Finally,  using the heat kernel estimate in~\eqref{Ee2} for $\mathcal{U}^\mathcal{L}(0,t)$ for
  $t_*\le t\le 2t_*$,  we conclude, after a Markov inequality as in~\eqref{markov}-\eqref{markov1},
  \begin{equation}
\label{fefemin}
\int_0^1\abs{v_1(t_1,\alpha)}\,\diff\alpha\lesssim \frac{N^\xi}{N^\frac{3}{4}N^{\frac{4\omega_1}{15}}},
\end{equation} 
with very high probability.
\end{proof}

\appendix
\section{Proof of Theorem~\ref{thr:time cusp universality}}\label{sec proof space time}
We now briefly outline the changes required for the proof of Theorem~\ref{thr:time cusp universality} compared to the proof of Theorem~\ref{thr:cusp universality}. We first note that for $0\le\tau_1\le \dots\le\tau_k\lesssim N^{-1/2}$ in distribution $(H^{(\tau_1)},\dots,H^{(\tau_k)})$ agrees with 
\begin{equation}\label{eq H+tau U} \Bigl(H + \sqrt{\tau_1} U_1, H+\sqrt{\tau_1}U_1+\sqrt{\tau_2-\tau_1}U_2,\dots, H+\sqrt{\tau_1}U_1+\dots+\sqrt{\tau_k-\tau_{k-1}}U_k\Bigr),\end{equation}
where $U_1,\dots,U_k$ are independent GOE matrices. Next, we claim and prove later by Green function comparison that the time-dependent $k$-point correlation function of~\eqref{eq H+tau U} asymptotically agrees with the one of
\begin{equation}\label{eq wtH+tau U} \Bigl(\widetilde{H}_t + \sqrt{\tau_1} U_1, \widetilde{H}_t+\sqrt{\tau_1}U_1+\sqrt{\tau_2-\tau_1}U_2,\dots, \widetilde{H}_t+\sqrt{\tau_1}U_1+\dots+\sqrt{\tau_k-\tau_{k-1}}U_k\Bigr),\end{equation}
and thereby also with the one of
\begin{equation}\label{eq Ht+tau U} \Bigl(H_t +\sqrt{ct}U+ \sqrt{\tau_1} U_1, H_t+\sqrt{ct}U+\sqrt{\tau_1}U_1+\sqrt{\tau_2-\tau_1}U_2,\dots, H_t+\sqrt{ct}U+\sqrt{\tau_1}U_1+\dots+\sqrt{\tau_k-\tau_{k-1}}U_k\Bigr),\end{equation}
for any fixed $t\le N^{-1/4-\epsilon}$, where we recall that $\wt H_t$ and $H_t$ constructed as in Section~\ref{sec OU flow} (see~\eqref{OU decomp}). Finally, we notice that the joint eigenvalue distribution of the matrices in~\eqref{eq Ht+tau U} is precisely given by the joint distribution of 
\[\Bigl(\lambda_i(ct+\tau_1),\dots,\lambda_i(ct+\tau_k),\,\,i\in[N]\Bigr)\]
of eigenvalues $\lambda_i^s$ evolved according to the DBM 
\begin{equation} \label{DBM app}\diff \lambda_i(s) = \sqrt{\frac{2}{N}}\diff B_i + \sum_{j\ne i}\frac{1}{\lambda_i(s)-\lambda_j(s)}\diff s,\quad \lambda_i(0)= \lambda_i(H_t).
\end{equation}
The high probability control on the eigenvalues evolved according to~\eqref{DBM app} in Propositions~\ref{cut} and~\ref{cutmin} allows to simultaneously compare eigenvalues at different times with those of the Gaussian reference ensemble automatically. 

In order to establish Theorem~\ref{thr:time cusp universality} it thus only remains to argue that the $k$-point functions of~\eqref{eq H+tau U} and~\eqref{eq wtH+tau U} are asymptotically equal. For the sake of this argument we consider only the randomness in $H$ and condition on the randomness in $U_1,\dots,U_k$. Then the OU-flow 
\[\diff \widetilde{H}'_s = -\frac{1}{2}\bigl(\widetilde H'_s-A-\sqrt{\tau_1}U_1-\dots-\sqrt{\tau_l-\tau_{l-1}} U_l\bigr) \diff s+\Sigma^{1/2}[\diff\mathfrak B_s] \] 
with initial conditions
\[ \widetilde H'_0 = H + \sqrt{\tau_1} U_1+\dots +\sqrt{\tau_l-\tau_{l-1}} U_l\]
for fixed $U_1,\dots,U_l$ is given by
\[ \widetilde{H}'_s = \widetilde{H}_s+ \sqrt{\tau_1} U_1+\dots +\sqrt{\tau_l-\tau_{l-1}} U_l,\]
i.e.~we view $\sqrt{\tau_1} U_1+\dots +\sqrt{\tau_l-\tau_{l-1}} U_l$ as an additional expectation matrix. Thus we can appeal to the standard Green function comparison technique already used in Section~\ref{sec OU flow} to compare the $k$-point functions of~\eqref{eq H+tau U} and~\eqref{eq wtH+tau U}. Here we can follow the standard resolvent expansion argument from~\cite[Eq.~(116)]{1809.03971} and note that the proof therein verbatim also allows us to compare products of traces of resolvents with differing expectations. Finally we then take the $\E_{U_1}\dots\E_{U_k}$ expectation to conclude that not only the conditioned $k$-point functions of~\eqref{eq H+tau U} and~\eqref{eq wtH+tau U} asymptotically agree, but also the $k$-point functions themselves.

\section{Finite speed of propagation estimate}
\label{FSPS}
In this section we prove a finite speed of propagation estimate for the time evolution of the $\alpha$-derivative 
of the short range dynamics  defined in
\eqref{333}--\eqref{555}. It is an adjustment to the analogous proof of Lemma 4.1 in~\cite{1712.03881}.
For concreteness, we present the proof for the propagator $\mathcal{U}^\mathcal{L}$ where
$\mathcal{L}=\mathcal{B}+\mathcal{V}$ is defined in~\eqref{weq}--\eqref{V222rig1}.
 The point is that once the dynamics is localized, i.e.~the
range of the interaction term $\mathcal{B}$ is restricted to a local scale $\abs{i-j}\le \abs{j_+(i)-j_-(i)}$, with $\abs{j_+(i)-j_-(i)} \gtrsim N^{4\omega_\ell}\defqe L$, 
and the time is also restricted, $0\le t\le 2t_*\lesssim N^{-\frac{1}{2}+\om_1}$, then the propagation cannot go beyond a scale that is much bigger 
than the interaction scale.  This mechanism   is very general and will also be used
in a slightly different (simpler) setup of Lemma~\ref{FSP} and Proposition~\ref{finsp} where the interaction scale is much bigger $L\sim \sqrt{N}$.
We will give the necessary changes for the proof of  Lemma~\ref{FSP} and Proposition~\ref{finsp} at the end of this section.

\begin{lemma}
\label{fsl1}
Let $\wh z(t)=\wh z(t,\alpha)$ be the solution to the short range dynamics~\eqref{333}--\eqref{555}
with $i_*=N^{\frac{1}{2}+C_*\omega_1}$, exponents $\omega_1\ll \omega_\ell\ll \omega_A\ll 1$ and propagator
$\mathcal{L}=\mathcal{B}+\mathcal{V}$ from~\eqref{weq}--\eqref{V222rig1}. Let us assume that 
\begin{equation}
\label{rrrrr}
\sup_{0\le t\le t_*}\abs{\widehat{z}_i(t)-\overline{\gamma}_i(t)}\le \frac{N^{C\omega_1}}{N^\frac{3}{4}},\qquad 1\le \abs{i}\le i_*,
\end{equation} 
where $\ov\gamma_i(t)$ are the quantiles from~\eqref{barg}.
Then, there exists a constant $C'>0$ such that for any $0<\delta<C'\omega_\ell$, $\abs{a}\ge L N^{\delta}$ and $\abs{b}\le \frac{3}{4}LN^{\delta}$, 
for any fixed $0\le s\le t_*$, we have that \begin{equation}
\label{finitespeed}
\sup_{s\le t\le t_*} \mathcal{U}^\mathcal{L}_{ab}(s,t)+\mathcal{U}^\mathcal{L}_{ba}(s,t)\le N^{-D}
\end{equation} for any $D>0$, with very high probability.  The same result holds for the short range dynamics after the cusp
defined in~\eqref{eq1min} for $t_*\le s\le 2t_*$. 
\end{lemma}

\begin{proof}[Proof of Lemma~\ref{fsl1}.]
For concreteness we assume that $0\le s\le t\le t_*$, i.e.~we are in the small gap regime.
For $t_*\le s\le t\le 2t_*$ the proof is analogous using the definition~\eqref{hatquaminint} for the $\overline{\gamma}_i(t)$, the definition of the short range approximation in~\eqref{SDE1min}-\eqref{idmin} for the $\widehat{z}_i(t,\alpha)$ and replacing $\overline{\mathfrak{e}}_t^+$ by $\overline{\mathfrak{m}}_t$. 
With these adjustments the proof follows in the same way except for~\eqref{g10} below, 
where we have to use the estimates in~\eqref{Re m min} instead of~\eqref{Re m gap}.

First we consider the  $s=0$ case, then in Lemma~\ref{fsl} below we extend the proof for all $0\le s\le t$.
  Let $\psi(x)$ be an even $1$-Lipschitz real function, i.e.~$\psi(x)=\psi(-x)$,  $\norm{\psi'}_\infty\le 1$
 such that
 \begin{equation}
\label{def1}
\psi(x)=\abs{x} \qquad \text{for} \quad \abs{x}\le \frac{L^\frac{3}{4}N^{\frac{3}{4}\delta}}{N^\frac{3}{4}},
\qquad
\psi'(x)=0 \qquad \text{for}\quad \abs{x}\ge 2\frac{L^\frac{3}{4} N^{\frac{3}{4}\delta}}{N^\frac{3}{4}}.
\end{equation} 
and 
\begin{equation}
\label{def3}
\norm{\psi''}_\infty\lesssim \frac{N^\frac{3}{4}}{L^\frac{3}{4}N^\frac{3\delta}{4}}.
\end{equation}

We consider a solution of the equation \[\partial_t f=\mathcal{L}f, \qquad 0\le t\le t_*\] with some discrete Dirac delta  initial condition $f_i(0)=\delta_{ip_*}$ 
at $p_*$  for any $\abs{p_*}\ge N^{4\omega_\ell}N^{\delta}$. For concreteness, assume $p_*>0$ and 
set $p\defeq N^{4\omega_\ell}N^{\delta}$.
Define 
\begin{equation}
\label{nu}
  \phi_i = \phi_i(t,\alpha)\defeq e^{\frac{1}{2}\nu\psi(\widehat{z}_i(t,\alpha)-\overline{\gamma}_p(t))}, \qquad m_i=m_i(t,\alpha)\defeq f_i(t,\alpha) \phi_i(t,\alpha),
  \qquad \nu=\frac{N^\frac{3}{4}}{L^\frac{3}{4} N^{\delta'}}
\end{equation}
with 
 some $\delta'\ge \frac{\delta}{2}$  to be chosen later. Let $\widehat{z}_i=\widehat{z}_i(t,\alpha)$ and set
\begin{equation}
\label{def4}
F(t)\defeq \sum_if_i^2e^{\nu\psi(\widehat{z}_i-\overline{\gamma}_p(t))} =\sum_i m_i^2.
\end{equation} Since 
\[2\sum_i f_i(\mathcal{B}f)_i\phi_i^2=\sum_{(i,j)\in\mathcal{A}}\mathcal{B}_{ij}(m_i-m_j)^2-\sum_{(i,j)\in\mathcal{A}}\mathcal{B}_{ij}m_im_j\left(\frac{\phi_i}{\phi_j}+\frac{\phi_j}{\phi_i}-2\right),
\]
 using Ito's formula, we conclude that 
 \begin{align}
\label{derivative1}
\diff F&=\sum_{(i,j)\in \mathcal{A}}\mathcal{B}_{ij}(m_i-m_j)^2\diff t+2\sum_i\mathcal{V}_im_i^2\diff t\\
\label{derivative2}
&\quad -\sum_{(i,j)\in\mathcal{A}}\mathcal{B}_{ij}m_im_j\left(\frac{\phi_i}{\phi_j}+\frac{\phi_j}{\phi_i}-2\right)\diff t \\
\label{derivative3}
&\quad+  \sum_i\nu m_i^2 \psi'(\widehat{z}_i-\ov \gamma_{p})\diff (\widehat{z}_i-\overline{\gamma}_p) \\
\label{derivative4}
&\quad +\sum_i m_i^2\left(\frac{\nu^2}{N}\psi'(\widehat{z}_i-\overline{\gamma}_p)^2+\frac{\nu}{N}\psi''(\widehat{z}_i-\overline{\gamma}_p)\right)\diff t.
\end{align}

Let $\tau_1\le t_*$ be the first time such that $F\ge 5$ and let $\tau_2$ be stopping time so that 
the estimate~\eqref{rrrrr} holds with $t\le \tau_2$ instead of $t\le t_*$; the condition~\eqref{rrrrr}
then says that $\tau_2=t_*$ with very high probability.
Define $\tau\defeq \tau_1\wedge \tau_2\wedge t_*$, our goal is to show that $\tau=t_*$. 
In the following we assume $t\le \tau$.

Now we estimate the terms in~\eqref{derivative1}--\eqref{derivative4} one by one.
We start with~\eqref{derivative2}.
Note that the rigidity scale  $N^{-\frac{3}{4}+C\om_1}$ in~\eqref{rrrrr}
 is much smaller than $N^{-\frac{3}{4}(1-\delta) + 3\om_\ell  }$,  the range of the support of $\psi'$,
which, in turn, is comparable  with $\abs{\ov \gamma_i -\ov \gamma_p} \gtrsim (p/N)^{3/4}$ for any
 $i\ge 2 p =2L N^{\delta}$.
Therefore $\psi'(\widehat{z}_i-\overline{\gamma}_p)=0$ unless $\abs{i}\lesssim L N^{\delta}$. 
Moreover, if $\abs{i}\lesssim L N^{\delta}$ and $(i,j)\in\mathcal{A}$, then $\abs{j}\lesssim L N^{\delta}$. 
Hence, the nonzero terms in the sum in~\eqref{derivative2} have   both indices  $\abs{i},\abs{j}\lesssim N^{4\omega_\ell+\delta}$. 
By~\eqref{rrrrr} and  $C\omega_1\ll \omega_\ell$, for such terms we have 
\begin{equation}
\label{rige}
\abs{\widehat{z}_i-\widehat{z}_j}\lesssim \frac{\abs{i-j}}{N^\frac{3}{4}\min\{\abs{i},\abs{j}\}^\frac{1}{4}}
+ \frac{N^{C\omega_1}}{N^\frac{3}{4}}\lesssim \frac{L^\frac{3}{4}N^\frac{\delta}{2}}{N^\frac{3}{4}}.
\end{equation}
Note that $\nu \abs{\widehat{z}_i-\widehat{z}_j}\lesssim 1$. Therefore, by Taylor expanding
in the exponent, we have 
\[
\abs{\frac{\phi_i}{\phi_j}+\frac{\phi_j}{\phi_i}-2}=\left(e^{\frac{\nu}{2}(\psi(\widehat{z}_j-\overline{\gamma}_p)-\psi(\widehat{z}_i-\overline{\gamma}_p))}-e^{\frac{\nu}{2}(\psi(\widehat{z}_i-\overline{\gamma}_p)-\psi(\widehat{z}_j-\overline{\gamma}_p))}\right)^2\lesssim \nu^2 \abs{\psi(\widehat{z}_i-\overline{\gamma}_p)-\psi(\widehat{z}_j-\overline{\gamma}_p)}^2,
\]
and thus
 \begin{equation}
\label{newker}
\abs{\mathcal{B}_{ij}\left(\frac{\phi_i}{\phi_j}+\frac{\phi_j}{\phi_i}-2\right)}\lesssim \nu^2 \frac{\abs{\psi(\widehat{z}_i-\overline{\gamma}_p)-\psi(\widehat{z}_j-\overline{\gamma}_p)}^2}{N(\widehat{z}_i-\widehat{z}_j)^2}\lesssim \frac{\nu^2}{N},
\end{equation} 
where in the last inequality we used that $\psi$ is Lipschitz continuous. 
Hence  we conclude the estimate of~\eqref{derivative2} as
 \begin{equation}
\label{g1}
\abs{\sum_{(i,j)\in\mathcal{A}}\mathcal{B}_{ij}m_im_j\left(\frac{\phi_i}{\phi_j}+\frac{\phi_j}{\phi_i}-2\right)} \lesssim \frac{\nu^2}{N}\sum_im_i^2\sum_j^{\mathcal{A},(i)}\bm{1}_{\{\phi_j\ne \phi_i\}}\lesssim \frac{\nu^2L N^\frac{3}{4}\delta}{N}F(t),
\end{equation}
since the number of $j$'s in the summation is at most 
\begin{equation}\label{jj}
\abs{j_+(i)- j_-(i)}\le  \ell^4  +\ell  \abs{i}^{3/4} \le L N^{3\delta/4}.
\end{equation}

By~\eqref{def3} and since $\abs{\psi'(x)}\le 1$,~\eqref{derivative4} is bounded as follows \begin{equation}
\label{g2}
\abs{\sum_i m_i^2\left(\frac{\nu^2}{N}\psi'(\widehat{z}_i-\overline{\gamma}_p)^2+\frac{\nu}{N}\psi''(\widehat{z}_i-\overline{\gamma}_p)\right)}\lesssim \left(\frac{\nu^2}{N}+\frac{\nu}{N^\frac{1}{4}L^\frac{3}{4}N^{\frac{3}{4}\delta}}\right)F(t).
\end{equation} 

The next step is to get a bound for~\eqref{derivative3}. Since $\psi'(\widehat{z}_i-\overline{\gamma}_p)=0$ unless $\abs{i}\lesssim N^{4\omega_\ell+\delta}\ll N^{\omega_A}$, choosing $C>0$ such that $(4+C)\omega_\ell<\omega_A$ and
using~\eqref{444} we get 
\begin{equation}\label{drg}
\diff (\widehat{z}_i(t)-\overline{\gamma}_p(t))=\sqrt{\frac{2}{N}}\diff  B_i+\frac{1}{N}\sum_j^{\mathcal{A},(i)}\frac{1}{\widehat{z}_i(t)-\widehat{z}_j(t)}+Q_i(t)
\end{equation}
with
\begin{equation}
\label{QQQ}
\begin{split}
Q_i(t):&=\int_{\mathcal{I}_{y,i}(t)^c}\frac{\rho_{y,t}(E+\mathfrak{e}_{y,t}^+)}{\widehat{z}_i(t)-E}\diff E+\Re [m_{y,t}(\mathfrak{e}_{y,t}^+)]+\alpha \Big(\Re[m_{x,t}(\widehat{\gamma}_{x,p}(t)+\mathfrak{e}_{x,t}^+)-m_{x,t}(\mathfrak{e}_{x,t}^+)]\Big)\\
&\quad+(1-\alpha)\Big(\Re[m_{y,t}(\widehat{\gamma}_{y,p}(t)+\mathfrak{e}_{y,t}^+)-m_{y,t}(\mathfrak{e}_{y,t}^+)]\Big).
\end{split}
\end{equation}

We insert~\eqref{drg} into~\eqref{derivative3} and estimate all three terms separately in the regime $\abs{i}\lesssim LN^\delta$.
 For the stochastic differential,
by the definition of $\tau\le t_*$ and the Burkholder-Davis-Gundy inequality we have that 
\begin{equation}
\label{g3}
\sup_{0\le t\le \tau}\int_0^t\sqrt{\frac{2}{N}}\nu \sum_im_i^2\psi'(\widehat{z}_i-\overline{\gamma}_p)\, \diff B_i
\le N^{\epsilon'} \frac{\nu}{\sqrt{N}}  \sqrt{t_*} \sup_{0\le t\le \tau}F(t)
\lesssim \nu N^{\epsilon'} N^{-\frac{3}{4} +\frac{1}{2}\om_1}, 
\end{equation} 
for any $\epsilon'>0$, with very high probability. In~\eqref{g3} we used that 
$\tau\le t_*\sim N^{-\frac{1}{2}+\omega_1}$, and that, by the definition of $\tau$, 
$F(t)$ is bounded for all $0\le t\le \tau$.

The contribution of the second term in~\eqref{drg} to
\eqref{derivative3} is written, after symmetrisation,  as follows
 \begin{equation}
\label{g4}
\frac{\nu}{N}\sum_{(i,j)\in\mathcal{A}}\frac{\psi'(\widehat{z}_i-\overline{\gamma}_p)m_i^2}{\widehat{z}_j-\widehat{z}_i}=\frac{\nu}{2N}\sum_{(i,j)\in\mathcal{A}}\frac{\psi'(\widehat{z}_i-\overline{\gamma}_p)(m_i^2-m_j^2)}{\widehat{z}_j-\widehat{z}_i}+\frac{\nu}{2N}\sum_{(i,j)\in\mathcal{A}}m_i^2\frac{\psi'(\widehat{z}_i-\overline{\gamma}_p)-\psi'(\widehat{z}_j-\overline{\gamma}_p)}{\widehat{z}_j-\widehat{z}_i}.
\end{equation} 
Using~\eqref{def3} and~\eqref{jj}, the second sum in~\eqref{g4} is bounded by 
\begin{equation}
\label{g5}
\begin{split}
\abs{\frac{\nu}{2N}\sum_{(i,j)\in\mathcal{A}}m_i^2\frac{\psi'(\widehat{z}_i-\overline{\gamma}_p)-\psi'(\widehat{z}_j-\overline{\gamma}_p)}{\widehat{z}_j-\widehat{z}_i}}&\lesssim \frac{\nu}{N^\frac{1}{4}L^\frac{3}{4}N^{\frac{3}{4}\delta}}\sum_im_i^2\sum_j^{\mathcal{A},(i)}\bm{1}_{\{\psi'(\widehat{z}_i-\overline{\gamma}_p)\ne\psi'(\widehat{z}_j-\overline{\gamma}_p)\}} 
\lesssim \frac{\nu L^\frac{1}{4} }{N^\frac{1}{4}}F.
\end{split}\end{equation} 
Using $m_i^2-m_j^2 = (m_i-m_j)(m_i+m_j)$ and  Schwarz inequality, the first sum in~\eqref{g4} is bounded as follows 
\begin{equation}
\label{g6}
\begin{split}
\frac{\nu}{2N}\sum_{(i,j)\in\mathcal{A}}\frac{\psi'(\widehat{z}_i-\overline{\gamma}_p)(m_i^2-m_j^2)}{\widehat{z}_j-\widehat{z}_i} 
\le -\frac{1}{100}\sum_{(i,j)\in\mathcal{A}}  \mathcal{B}_{ij} (m_i-m_j)^2
+\frac{C\nu^2}{2N}\sum_{(i,j)\in\mathcal{A}}\psi'(\widehat{z}_i-\overline{\gamma}_p)^2(m_i^2+m_j^2).
\end{split}
\end{equation} 
The second sum in~\eqref{g6}, using~\eqref{jj},   is bounded by
\begin{equation}
\label{g7}
\frac{C\nu^2}{2N}\sum_{(i,j)\in\mathcal{A}}\psi'(\widehat{z}_i-\overline{\gamma}_p)(m_i^2+m_j^2)\le 
\frac{C\nu^2 L N^\frac{3\delta}{4}}{2N}F,
\end{equation} 
hence we conclude that 
\begin{equation}
\label{g8}
\frac{\nu}{N}\sum_{(i,j)\in\mathcal{A}}\frac{\psi'(\widehat{z}_i-\overline{\gamma}_p)m_i^2}{\widehat{z}_j-\widehat{z}_i}\le -\frac{1}{100}\sum_{(i,j)\in\mathcal{A}}\mathcal{B}_{ij}(m_i-m_j)^2+C\left(\frac{\nu L^\frac{1}{4} }{N^\frac{1}{4}}+\frac{\nu^2 L N^\frac{3\delta}{4}}{N}\right)F.
\end{equation} 
Note that the first term on the right-hand side of~\eqref{g8} can be incorporated in the first, dissipative term
  in~\eqref{derivative1}.

To conclude the estimate of~\eqref{derivative3} we write the third term in~\eqref{drg}
\begin{equation}
\label{g9}
\begin{split}
Q_i(t)&=\left(\int_{\mathcal{I}_{y,i}(t)^c}\frac{\rho_{y,t}(E+\mathfrak{e}_{y,t}^+)}{\widehat{z}_i(t)-E}\diff E+\Re [m_{y,t}(\overline{\gamma}_p(t)+\mathfrak{e}_{y,t}^+)]\right)\\
&\quad+\alpha\Big(\Re[m_{x,t}(\widehat{\gamma}_{x,p}(t)+\mathfrak{e}_{x,t}^+)-m_{x,t}(\mathfrak{e}_{x,t}^+)]
- \Re[m_{y,t}(\widehat{\gamma}_{x,p}(t)+\mathfrak{e}_{y,t}^+)-m_{y,t}(\mathfrak{e}_{y,t}^+)] \Big) \\
&\quad+\alpha \Big(\Re[m_{y,t}(\widehat{\gamma}_{x,p}(t)+\mathfrak{e}_{y,t}^+)]-\Re[m_{y,t}(\overline{\gamma}_p(t)
+\mathfrak{e}_{y,t}^+)]
\Big)
\\
&\quad+(1-\alpha)\Big(\Re[m_{y,t}(\widehat{\gamma}_{y,p}(t)+\mathfrak{e}_{y,t}^+)]
-\Re[m_{y,t}(\overline{\gamma}_p(t)+\mathfrak{e}_{y,t}^+)\Big)\defqe A_1+A_2+A_3+A_4.
\end{split}\end{equation}

Similarly to the estimates in~\eqref{Dest}, for $A_2$ we use~\eqref{Re m gap} while for $A_3, A_4$ we use~\eqref{weakm}, then 
we use the asymptotic behavior of $\wh \gamma_p, \ov\gamma_p$
by~\eqref{gamma gap} and $p=LN^\delta$ to conclude that  
\begin{equation}
\label{g10}
\abs{A_2}+\abs{A_3}+\abs{A_4}\lesssim
\frac{L^\frac{1}{4}N^\frac{\delta}{4}N^{C\omega_1}\log N}{N^\frac{1}{4}N^\frac{1}{6}}.
\end{equation} 

For the $A_1$ term we write it as
\begin{equation}\label{A1}
A_1= \int_{\mathcal{I}_{y,i}(t)^c}\frac{\overline{\gamma}_p(t)-\widehat{z}_i(t) }{(\widehat{z}_i(t)-E)(\overline{\gamma}_p(t)-E)}\rho_{y,t}(E+\mathfrak{e}_{y,t}^+)\diff E+
\int_{\mathcal{I}_{y,i}(t)}\frac{\rho_{y,t}(E+\mathfrak{e}_{y,t}^+) }{\overline{\gamma}_p(t)-E}\diff E.
\end{equation}
Since $i\le Cp$, we have
 $\rho_{y,t}(E+\mathfrak{e}_{y,t}^+)\le \rho_{y,t}(\overline{\gamma}_{Cp}(t)+\mathfrak{e}_{y,t}^+)\lesssim
 L^\frac{1}{4} N^{-\frac{1}{4}+\frac{\delta}{4}}$ for any $E\in \mathcal{I}_{y,i}(t)$, the second
 term in~\eqref{A1} is bounded by $L^\frac{1}{4} N^{-\frac{1}{4}+\frac{\delta}{4}}\log N$.
 In the first term in~\eqref{A1} we  use that 
 \[
  \abs{\widehat{z}_i(t)-E} \ge  \abs{\ov\gamma_i(t)-E} -  \abs{\widehat{z}_i(t)-\ov\gamma_i(t)} \gtrsim \ov\gamma_p(t)
  \]
   for $E\not\in \mathcal{I}_{y,i}(t)$, by rigidity~\eqref{rrrrr} and by the fact that in the $i\le Cp$ regime
   $\abs{\ov\gamma_i(t)-\ov\gamma_{i\pm j_\pm(i)}(t) } \gtrsim \ov\gamma_p(t) \gg N^{-\frac{3}{4}+C\om_1}$
   since $\om_1\ll \om_\ell$ and $=LN^\delta= N^{4\om_\ell +\om_1}$.

We thus conclude that the first term in~\eqref{A1} is bounded by
\[
 \abs{\widehat{z}_i(t)-\overline{\gamma}_p(t)}\frac{\Im [m_{y,t}(\mathfrak{e}_{y,t}^++\ii \overline{\gamma}_p(t))]}{\overline{\gamma}_p(t)}
 \lesssim \overline{\gamma}_p^\frac{1}{3} \lesssim 
L^\frac{1}{4} N^{-\frac{1}{4}+\frac{\delta}{4}},
\]
where we used again the rigidity~\eqref{rrrrr}.
In summary, we have
\begin{equation}
\label{g12}
\abs{A_1}\lesssim 
L^\frac{1}{4} N^{-\frac{1}{4}+\frac{\delta}{4}}\log N.
\end{equation}

In particular~\eqref{g9}-\eqref{g12} imply that \begin{equation}
\label{estQ}
Q\defeq \sup_{0\le t\le t_*}\sup_{\abs{i}\lesssim L N^{\delta}} \abs{Q_i(t)}\lesssim L^\frac{1}{4} N^{-\frac{1}{4}+\frac{\delta}{4}}\log N.
\end{equation}

Collecting all the previous estimates using the choice of $\nu$ from~\eqref{nu}
with $\delta'\ge \frac{\delta}{2}$ and that $F$ is bounded up to $t\le \tau$, we
integrate~\eqref{derivative1}--\eqref{derivative4} from $0$ up to time $0\le t\le t_*$ and 
 conclude that 
 \begin{equation}
\label{fF}
\begin{split}
\sup_{0\le t\le \tau} F(t) -F(0) &\lesssim \left(\frac{\nu^2L N^{\frac{3\delta}{4}+\omega_1}}{N^\frac{3}{2}}
+\frac{\nu L^\frac{1}{4} N^{\omega_1}}{N^\frac{3}{4}}+\frac{\nu  Q N^{\om_1}}{N^\frac{1}{2}}\right)\\
&\lesssim\frac{N^{\frac{3\delta}{4}+\omega_1}}{L^\frac{1}{2}N^{2\delta'}}+\frac{N^{\omega_1}}{L^\frac{1}{2}N^{\delta'}}+\frac{N^{\omega_1+\frac{\delta}{4}}}{L^\frac{1}{2}N^{\delta'}}\log N \le 1
\end{split}\end{equation} 
for large $N$ and
with very high probability, where we used the choice of $\nu$~\eqref{nu}
and that $\om_1\ll \om_\ell$ in the last line. Since $F(0)=1$, we get that
$\tau=t_*$ with very high probability, and so
 \begin{equation}
\label{fb}
\sup_{0\le t\le t_*}F(t)\le 5,
\end{equation} 
with very high probability.

Furthermore, since $p=L N^{\delta}$, if $i\le \frac{3}{4}L N^{\delta}$, choosing 
 $\delta'=\frac{3\delta}{4}-\epsilon_1$, with $\epsilon_1\le \frac{\delta}{4}$, then by Proposition~\ref{prop:option2} we have that 
 \[\nu\psi(\widehat{z}_i(t)-\overline{\gamma}_p) =
 \nu \abs{\widehat{z}_i(t)-\overline{\gamma}_p}\gtrsim \nu \frac{\abs{i-p}}{N^\frac{3}{4}\abs{p}^\frac{1}{4}}
 \gtrsim \frac{N^\frac{3\delta}{4}}{N^{\delta'}}=N^{\epsilon_1}
 \] with very high probability.
 
   Note that~\eqref{fb} implies
 \[
     f_i(t)  \le  5e^{-\frac{1}{2} \nu\psi(\widehat{z}_i(t)-\overline{\gamma}_p)}.
 \]
Therefore, if $i\le \frac{3L N^{\delta}}{4}$ and $p_*\ge p$, then for each fixed $0\le t\le t_*$ we have that
 \begin{equation}
\label{fb1}
\mathcal{U}^\mathcal{L}_{ip_*}(0,t)\le N^{-D},
\end{equation} 
for any $D>0$ with very high probability. Similar estimate holds
if $i$ and $p_*$ are negative or have opposite sign. This proves the estimate on the first term in
\eqref{finitespeed} for any fixed $s$.
The estimate for $\mathcal{U}^\mathcal{L}_{p_*i}(s,t)$ is analogous
with initial condition $ f = \delta_{i}$. 
This proves Lemma~\ref{fsl1}. 
\end{proof}

Next, we enhance this result to a bound  uniform in $0\le s\le t_*$. We first have:
\begin{lemma}
\label{gronwall}
Let $u$ be a solution of \begin{equation}
\label{egl}
\partial_t u=\mathcal{L}u,
\end{equation} with non-negative initial condition $u_i(0)\ge 0$. Then, for each $0\le t\le t_*$ we have \begin{equation} 
\label{grobounds}
\frac{1}{2}\sum_i u_i(0)\le \sum_iu_i(t)\le\sum_i u_i(0)
\end{equation} with very high probability.
\proof
Since $\mathcal{U}^\mathcal{L}$ is a contraction semigroup the upper bound in~\eqref{grobounds} is trivial.
Notice that $\partial_t\sum_i u_i=\sum_i\mathcal{V}_iu_i$. Thus
the lower bound will follow once we prove  
$-\mathcal{V}_i\lesssim N^\frac{1}{2} L^{-\frac{1}{2}}$ with very high probability since $t_*  N^\frac{1}{2} L^{-\frac{1}{2}}$ is much smaller than $1$
by $\om_1\ll \om_\ell$.

The estimate $-\mathcal{V}_i\lesssim N^\frac{1}{2} L^{-\frac{1}{2}}$ proceeds similarly to~\eqref{A1}.
Indeed, for $1\le\abs{i}<N^{\omega_A}$, we use $\rho_{y,t}(E+\mathfrak{e}_{y,t}^+)\lesssim \abs{E}^\frac{1}{3} $ 
and that $\abs{\widehat{z}_i(t)-E}\sim \abs{\overline{\gamma}_i(t)-E}$ by rigidity~\eqref{rrrrr} and by the fact that
\[\abs{j_+(i)-i}, \,\abs{j_-(i)-i}\gtrsim N^{4\omega_\ell}+N^{\omega_\ell}\abs{i}^\frac{3}{4}\]
is much bigger than the rigidity scale.
Therefore, we have
\[\begin{split}-\mathcal{V}_i=\int_{\mathcal{I}_{y,i}(t)^c}\frac{\rho_{r,t}(E+\mathfrak{e}_{r,t}^+)}{(\widehat{z}_i(t)-E)^2}\, \diff E
&\lesssim \int_{\mathcal{I}_{y,i}(t)^c}\frac{1}{\abs{E-\overline{\gamma}_i(t)}^\frac{5}{3}}\,\diff E+\int_{\mathcal{I}_{y,i}(t)^c}\frac{\abs{\overline{\gamma}_i}^\frac{1}{3}}{(E-\overline{\gamma}_i(t))^2}\,\diff E
\lesssim \frac{N^\frac{1}{2}}{N^{2\omega_\ell}}=\frac{N^\frac{1}{2}}{L^\frac{1}{2}}.\end{split}
\]
The estimate of $-\mathcal{V}_i$ for $N^{\omega_A}<\abs{i}\le \frac{i_*}{2}$ is similar.
This concludes the proof of Lemma~\ref{gronwall}.
\endproof
\end{lemma}

Finally we prove the following version of  Lemma~\ref{fsl1} that is uniform in $s$:
\begin{lemma}
\label{fsl}
Under the same hypotheses of Lemma~\ref{fsl1}, for any $\delta'>0$, such that $\delta'<C'\omega_\ell$, 
with $C'>0$ the constant defined in Lemma~\ref{fsl1}, $\abs{a}\le \frac{L N^{\delta'}}{2}$ and $\abs{b}\ge L N^{\delta'}$ we have that 
\begin{equation}
\label{unifs}
\sup_{0\le s\le t\le t_*}\mathcal{U}_{ab}^\mathcal{L}(s,t)+\mathcal{U}_{ba}^\mathcal{L}(s,t)\le N^{-D}
\end{equation} with very high probability. The same result holds for $t_*\le s\le t\le 2t_*$ as well.
\proof
By the semigroup property for any $0\le s\le t\le t_*$ and any $j$ we have that
\begin{equation}
\label{spul}
\mathcal{U}_{aj}^\mathcal{L}(0,t)\ge\mathcal{U}_{ab}^\mathcal{L}(s,t)\mathcal{U}_{bj}^\mathcal{L}(0,s).
\end{equation} 
Furthermore, by Lemma~\ref{gronwall}  for the dual dynamics we have that 
\[\frac{1}{2}\sum_ju_j(0)\le\sum_ju_j(s)=\sum_i\sum_j \big(\mathcal{U}_{ji}^\mathcal{L}(0,s)\big)^Tu_i(0),
\]
and so, by choosing $u(0)=\delta_{b}$ 
we conclude that 
\[
\sum_j\mathcal{U}_{bj}^\mathcal{L}(0,s)\ge\frac{1}{2}, \qquad \forall\, 0\le s\le t_*.
\]
 From the last inequality and since $\sup_{s\le t_*}\mathcal{U}_{bj}^\mathcal{L}(0,s)\le N^{-100}$ 
 with very high probability for any 
 $\abs{j}\le \frac{3}{4} LN^{\delta'}$  by Lemma~\ref{fsl1}, 
 it follows that there exists an $j_*=j_*(s)$, maybe depending on $s$, 
  with $\abs{j_*(s)}\ge \frac{3}{4} L N^{\delta'}$, such that $\mathcal{U}_{bj_*(s)}^\mathcal{L}(0,s)\ge \frac{1}{4N}$. 
  Furthermore, by the finite speed propagation estimate in Lemma~\ref{fsl1} (this time with $\abs{a}\ge \frac{3}{4}LN^\delta$
  and $\abs{b}\le\frac{1}{2} LN^\delta$;  note that its proof only used that $\abs{a-b}\gtrsim LN^{\delta}$), we have that 
 \[
 \sup_{t\le t_*}  \mathcal{U}_{aj_*}^\mathcal{L}(0,t)\le N^{-D}, \qquad \forall \abs{j_*}\ge \frac{3}{4}L N^{\delta'}
 \]
  with very high probability. Hence we get from~\eqref{spul} with $j=j_*(s)$  that 
  $\sup_{s\le t}\mathcal{U}_{ab}^\mathcal{L}(s,t)\lesssim N^{-D+1}$
  with very high probability.
  The estimate for $\mathcal{U}_{ba}^\mathcal{L}(s,t)$ follows in a similar way.
 This concludes the proof of Lemma~\ref{fsl}.
 \endproof
\end{lemma}

Finally, we prove   Lemma~\ref{FSP} and Proposition~\ref{finsp} which are versions of Lemma~\ref{fsl} but for the short range approximation
on scale $L= N^{1/2+C_1\om_1}$ needed in Section~\ref{sec:crudeshort}.

\begin{proof}[Proof of Lemma~\ref{FSP}]
Choosing $L=N^{\frac{1}{2}+C_1\omega_1}$, the proof of Lemma~\ref{fsl1} 
is exactly the same except for the estimate of $Q$ in~\eqref{estQ}, since, for any $\alpha \in [0,1]$, $Q_i(t)$  from~\eqref{uu1} is now defined as \begin{equation}
\label{newQi}
\begin{split}
Q_i(t)\defeq  \frac{\beta}{N}\sum_{j: \abs{j-i}> L}  \frac{1}{\ov\gamma_i^*-\ov\gamma_j^*} +
 \frac{1-\beta}{N}\sum_{j: \abs{j-i}> L}   \frac{1}{\wt z_i-\wt z_j} \diff t+\Phi_\alpha(t),
\end{split}
\end{equation}
 with $\Phi_\alpha(t)$ given in~\eqref{Phidef} instead of~\eqref{QQQ}.
  Then Lemma~\ref{gronwall} and Lemma~\ref{fsl} follow exactly in the same way.

By~\eqref{newQi} it easily follows that 
\begin{equation}
\label{12AB}
Q\defeq \sup_{0\le t\le t_*}\sup_{\abs{i}\le LN^{\delta'}}\abs{Q_i(t)}\lesssim \log N.
\end{equation} Hence, by an estimate  similar  to~\eqref{fF}, we conclude that 
\begin{equation}
\label{FinFin}
\begin{split}
\sup_{0\le t\le \tau} F(t)- F(0)&\lesssim \left(\frac{\nu^2L N^{\frac{3\delta}{4}+\omega_1}}{N^\frac{3}{2}}+\frac{\nu L^\frac{1}{4} 
N^{\omega_1}}{N^\frac{3}{4}}+\frac{\nu  QN^{\om_1}}{N^\frac{1}{2}}\right)\\
&\lesssim\frac{N^{\frac{3\delta}{4}+\omega_1}}{L^\frac{1}{2}N^{\delta'}}+\frac{N^{\omega_1}}{L^\frac{1}{2}N^{\delta'}}
+\frac{N^{\frac{3}{4}+\omega_1}}{L^\frac{3}{4}N^\frac{1}{2} N^{\delta'}}\log N\le 1,
\end{split}\end{equation} with very high probability. 
Note that in the last inequality we used that $L=N^{\frac{1}{2}+C_1\omega_1}$.
\end{proof}

\begin{proof}[Proof of Proposition~\ref{finsp}] This proof is almost identical to the previous one, except that $Q_i(t)$ is now defined from  
\eqref{veq} as 
\[
Q_i(t)\defeq  \beta\Bigg[ \frac{1}{N}\sum_{j: \abs{j-i}> L}   \frac{1}{\ov\gamma_i^*-\ov\gamma_j^*} +\Phi(t)\Bigg]
 + (1-\beta) \Bigg[ \frac{\diff}{\diff t} \ov \gamma_{i}^*(t) 
 -\frac{1}{N}\sum_{j: \abs{j-i}\le L}   \frac{1}{\ov\gamma_{i}^*-\ov\gamma_{j}^*}\Bigg],
\]
which satisfies the same bound~\eqref{12AB}. The rest of the proof is unchanged.
\end{proof}

\section{Short-long approximation}
\label{SLL}

In this section we estimate the difference of the solution of the DBM $\widetilde{z}(t,\alpha)$ and its short range
 approximation $\widehat{z}(t,\alpha)$, closely following the proof of Lemma 3.7  in~\cite{1712.03881} and adapting
 it to the more complicated cusp situation.
  In particular, in Section~\ref{SLLG} we estimate $\abs{\widetilde{z}(t,\alpha)-\widehat{z}(t,\alpha)}$ for $0\le t\le t_*$, i.e.~until the formation of an exact cusp;
 in Section~\ref{SLLM}, instead, we estimate $\abs{\widetilde{z}(t,\alpha)-\widehat{z}(t,\alpha)}$ for $t_*< t\le 2t_*$,
  i.e.~after the formation of a small minimum. The precision of this approximation depends on the rigidity bounds we put as a condition. 
We consider a two-scale rigidity assumption, a weaker rigidity valid for all indices and
a stronger rigidity valid for $1\le \abs{i}\lesssim i_*=N^{\frac{1}{2}+C_*\om_1}$; both  described by an exponent.

\subsection{Short-long approximation: Small gap and exact cusp.}
\label{SLLG}
In this subsection we estimate the difference of the solution of the DBM $\widetilde{z}(t,\alpha)$ defined in~\eqref{tildereq} and its short range approximation $\widehat{z}(t,\alpha)$ defined by~\eqref{333}-\eqref{666} for $0\le t\le t_*$. We formulate Lemma~\ref{seb111} (for $0\le t\le t_*$) 
 below a bit more generally than we need  in order to indicate the dependence of the approximation precision on these two exponents.
For our actual application in Lemma~\ref{shortlong2}  and Lemma~\ref{se} we use specific exponents.

\begin{lemma} 
\label{seb111}
Let $\omega_1\ll \omega_\ell\ll \omega_A\ll 1$. Let $0<  a_0 \le \frac{1}{4}+C\omega_1$, $C>0$ 
a universal constant and $0< a\le C\omega_1$. Let $i_*\defeq N^{\frac{1}{2}+C_*\om_1}$ with $C_*$ defined in Proposition~\ref{prop:option2}.
 We assume that \begin{equation}
\label{zerorig}
\abs{\widetilde{z}_i(t,\alpha)-\overline{\gamma}_i(t)}\le \frac{N^{a_0}}{N^\frac{3}{4}},\qquad 1\le \abs{i}\le N,\quad 0\le t\le t_*
\end{equation} and that \begin{equation}
\label{riglemma111}
\abs{\widetilde{z}_i(t,\alpha)-\overline{\gamma}_i(t)}\le \frac{N^a}{N^\frac{3}{4}},\qquad 1\le \abs{i}\le i_*,\quad 0\le t\le t_*.
\end{equation} 
 Then, for any $\alpha\in [0,1]$, we have that
\begin{equation}
\label{shortestimateb}
\begin{split}
&\sup_{1\le \abs{i}\le N}\sup_{0\le t\le t_*} \abs{\widehat{z}_i(t,\alpha)-\widetilde{z}_i(t,\alpha)}\\
&\qquad\le \frac{N^aN^{C\omega_1}}{N^\frac{3}{4}}\left(\frac{1}{N^{2\omega_\ell}}
+\frac{N^\frac{\omega_A}{2}\log N}{N^\frac{1}{6}N^a}
+\frac{N^\frac{\omega_A}{2}\log N}{N^\frac{1}{4}N^a}
+  \frac{1}{N^\frac{2a}{5} i_*^\frac{1}{5}}
+\frac{N^{a_0}}{N^ai_*^\frac{1}{2}} 
+\frac{1}{N^\frac{1}{18} N^a}
\right),
\end{split}\end{equation} with very high probability.

\begin{proof}[Proof of Lemma~\ref{shortlong2}] Use Lemma~\ref{seb111} with the choice 
$a_0=\frac{1}{4}+C\omega_1$ and $a=C\omega_1$, for some universal constant $C>0$. The conditions~\eqref{zerorig} and~\eqref{riglemma111} are guaranteed by~\eqref{crudinitial} and~\eqref{rgtilde}. 
\end{proof}

\proof[Proof of Lemma~\ref{seb111}]
Let $w_i\defeq \widehat{z}_i-\widetilde{z}_i$, hence $w$ is a solution of \begin{equation}
\label{vsolb}
\partial_tw=\mathcal{B}_1w+\mathcal{V}_1w+\zeta,
\end{equation} where the operator $\mathcal{B}_1$ is defined for any  $f\in\mathbb{C}^{2N}$ by 
\begin{equation}
\label{B1b}
(B_1f)_i=\frac{1}{N}\sum_j^{\mathcal{A},(i)}\frac{f_j-f_i}{(\widetilde{z}_i(t,\alpha)-\widetilde{z}_j(t,\alpha))(\widehat{z}_i(t,\alpha)-\widehat{z}_j(t,\alpha))}.
\end{equation} 
The diagonal operator $\mathcal{V}_1$ is defined by $(\mathcal{V}_1f)_i=\mathcal{V}_1(i)f_i$, where 
\begin{equation}
\label{V1b}
\mathcal{V}_1(i)\defeq-\int_{\mathcal{I}_{y,i}(t)^c}\frac{\rho_{y,t}(E+\mathfrak{e}_{y,t}^+)}{(\widetilde{z}_i(t,\alpha)-E)(\widehat{z}_i(t,\alpha)-E)}\diff E,
\qquad \mbox{for}\quad 0<\abs{i}\le N^{\omega_A},
\end{equation} and
 \begin{equation}
\label{V2b}
\mathcal{V}_1(i)\defeq-\int_{\mathcal{I}_{z,i}(t)^c\cap\mathcal{J}_z(t)}
\frac{\overline{\rho}_t(E+\overline{\mathfrak{e}}_t^+)}{(\widetilde{z}_i(t,\alpha)-E)(\widehat{z}_i(t,\alpha)-E)}\diff E,  \qquad
\mbox{for} \quad N^{\omega_A}< \abs{i}\le \frac{i_*}{2}.
\end{equation}
Finally, $\mathcal{V}_1(i)=0$ for $\abs{i}\ge\frac{i_*}{2}$.
The vector $\zeta$ in~\eqref{vsolb} collects various error terms.

We define the stopping time 
\begin{equation}
\label{bigtimeb}
T\defeq \max\Set{t\in [0,t_*]|\sup_{0\le s\le t}\abs{\widetilde{z}_i(s,\alpha)-\widehat{z}_i(s,\alpha)}\le \frac{1}{2}
\min\{\abs{\mathcal{I}_{z,i}(t)},\abs{\mathcal{I}_{y,i}(t)},\}\, \forall \alpha\in [0,1]},
\end{equation}
where we recall that $\abs{\mathcal{I}_{z,i}(t)}\sim \abs{\mathcal{I}_{y,i}(t)}\sim N^{-\frac{3}{4}+3\om_\ell}$.

For $0\le t\le T$ we have that $\mathcal{V}_1\le 0$. Therefore,  since $\sum_i(\mathcal{B}f)_i=0$, by the symmetry of $\mathcal{A}$, 
the semigroup of $\mathcal{B}_1+\mathcal{V}_1$, denoted by $\mathcal{U}^{\mathcal{B}_1+\mathcal{V}_1}$, 
is a contraction on every $\ell^p$ space. Hence, since $w(0)=0$ by~\eqref{666}, 
we have that 
\[
w(t)=\int_0^t\mathcal{U}^{\mathcal{B}_1+\mathcal{V}_1}(s,t)\zeta(s)\,\diff s
 \]
  and so 
\begin{equation}
\label{infinityvb}
\lVert w(t)\rVert_\infty\le t\sup_{0\le s\le t}\lVert \zeta(s)\rVert_\infty \le N^{-\frac{1}{2}+\om_1}\sup_{0\le s\le t}\lVert \zeta(s)\rVert_\infty.
\end{equation} 
Thus,  to prove~\eqref{shortestimateb} it is enough to estimate $\lVert \zeta(s)\rVert_\infty$, for all $0\le s\le t_*$.

The error term $\zeta$ is given by $\zeta_i=0$ for $\abs{i}> \frac{i_*}{2}$, then for $1\le \abs{i}\le N^{\omega_A}$, $\zeta_i$ is defined as \begin{equation}
\label{z1b}
\zeta_i=\int_{\mathcal{I}_{y,i}(t)^c}\frac{\rho_{y,t}(E+\mathfrak{e}_{y,t}^+)}{\widetilde{z}_i(t,\alpha)-E}\diff E-\frac{1}{N}\sum_j^{\mathcal{A}^c,(i)}\frac{1}{\widetilde{z}_i(t,\alpha)-\widetilde{z}_j(t,\alpha)}+\Phi_\alpha(t) -\Re [m_{y,t}(\mathfrak{e}_{y,t}^+)],
\end{equation} with $\Phi_\alpha(t)$ defined in~\eqref{Phidef}, and for $N^{\omega_A}< \abs{i}\le \frac{i_*}{2}$ as \begin{equation}
\label{z2b}
\zeta_i=\int_{\mathcal{I}_{z,i}(t)^c\cap\mathcal{J}_z(t)}\frac{\overline{\rho}_t(E+\overline{\mathfrak{e}}_t^+)}{\widetilde{z}_i(t,\alpha)-E}\diff E-\frac{1}{N}\sum_{1\le \abs{j}<\frac{3i_*}{4}}^{\mathcal{A}^c,(i)}\frac{1}{\widetilde{z}_i(t,\alpha)-\widetilde{z}_j(t,\alpha)}.
\end{equation} Note that in the sum in~\eqref{z2b} we do not have the summation over $\abs{j}\ge \frac{3i_*}{4}$ since if $1\le \abs{i}\le \frac{i_*}{2}$ and $\abs{j}\ge \frac{3i_*}{4}$ then $(i,j)\in\mathcal{A}^c$. 

 In the following we will often omit the $t$ and the $\alpha$ arguments from $\widetilde{z}_i$ and $\overline{\gamma}_i$ for notational simplicity.

First, we consider the error term~\eqref{z2b} for $N^{\omega_A}< \abs{i}\le \frac{i_*}{2}$. We start with the estimate \begin{equation}
\label{E133b}
\begin{split}
\abs{\zeta_i}&=\abs{\int_{\mathcal{I}_{z,i}^c(t)\cap\mathcal{J}_z(t)}\frac{\overline{\rho}_t(E+\overline{\mathfrak{e}}_t^+)}{\widetilde{z}_i-E}\diff E-\frac{1}{N}\sum_{1\le \abs{j}< \frac{3i_*}{4}}^{\mathcal{A}^c,(i)}\frac{1}{\widetilde{z}_i-\widetilde{z}_j}} \\
& \lesssim \abs{\sum_{1\le \abs{j}< \frac{3i_*}{4}}^{\mathcal{A}^c,(i)}\int_{\overline{\gamma}_j}^{\overline{\gamma}_{j+1}}\frac{\overline{\rho}_t(E+\overline{\mathfrak{e}}_t^+)(E-\overline{\gamma}_j)}{(\widetilde{z}_i-E)(\widetilde{z}_i-\overline{\gamma}_j)}\diff E}
 + \abs{\frac{1}{N}\sum_{1\le \abs{j}< \frac{3i_*}{4}}^{\mathcal{A}^c,(i)}\frac{\widetilde{z}_j-\overline{\gamma}_j}{(\widetilde{z}_i-\widetilde{z}_j)(\widetilde{z}_i-\overline{\gamma}_j)}} \\
&\quad+\abs{\int_{\overline{\gamma}_{j_+}}^{\overline{\gamma}_{j_++1}}\frac{\overline{\rho}_t(E+\overline{\mathfrak{e}}_t^+)}{\widetilde{z}_i-E}\diff E
}+\abs{\int_{\overline{\gamma}_{-\frac{3i_*}{4}}}^{\overline{\gamma}_{-\frac{3i_*}{4}+1}}\frac{\overline{\rho}_t(E+\overline{\mathfrak{e}}_t^+)}{\widetilde{z}_i-E}\diff E}+\abs{\int_0^{\overline{\gamma}_1}\frac{\overline{\rho}_t(E+\overline{\mathfrak{e}}_t^+)}{\widetilde{z}_i-E}\diff E}.
\end{split}
\end{equation} 
Since $\abs{j_+-i}\ge N^{4\omega_\ell}+N^{\omega_\ell}\abs{i}^\frac{3}{4}$ and $N^{\omega_A}$, i.e.~
\[\abs{\overline{\gamma}_{j_+}-\overline{\gamma}_i}\ge \frac{N^{\omega_\ell}\abs{i}^\frac{1}{2}}{N^\frac{3}{4}}
\]
 is bigger than the rigidity scale~\eqref{riglemma111}, all terms in the last line of~\eqref{E133b} are bounded by $N^{-\frac{1}{4}-3\omega_\ell}$.

Then, using the rigidity estimate in~\eqref{riglemma111} for the first and the second term of the rhs.~of~\eqref{E133b}, we conclude that 
\begin{equation}
\label{E122b}
\abs{\zeta_i}\lesssim \frac{N^a}{N^\frac{7}{4}}\sum_{1\le \abs{j}< \frac{3i_*}{4}}^{\mathcal{A}^c,(i)}\frac{1}{(\overline{\gamma}_i-\overline{\gamma}_j)^2}
+ N^{-\frac{1}{4}-3\omega_\ell}.
\end{equation} The sum on the rhs. of~\eqref{E122b} is over all the $j$, negative and positive, 
but the main contribution comes from $i$ and $j$ with the same sign, because if $i$ and $j$ have opposite sign then
 \[
 \frac{1}{(\overline{\gamma}_i-\overline{\gamma}_j)^2}\le\frac{1}{(\overline{\gamma}_{-i}-\overline{\gamma}_j)^2}.
 \]
  Hence, assuming that $i$ is positive (for negative $i$'s we proceed exactly in the same way), we conclude that 
  \begin{equation}
\label{E144b}
\abs{\zeta_i}\lesssim \frac{N^a}{N^\frac{7}{4}}\sum_{1\le j < \frac{3i_*}{4}}^{\mathcal{A}^c,(i)}\frac{1}{(\overline{\gamma}_i-\overline{\gamma}_j)^2}
+N^{-\frac{1}{4}-3\omega_\ell}.
\end{equation}

From now  we assume that both $i$ and $j$ are positive. 
In order to estimate~\eqref{E144b} we use the explicit expression of the quantiles from~\eqref{gamma gap}, i.e.~
\[\overline{\gamma}_j\sim\max\left\{ \Big(\frac{j}{N}\Big)^{2/3} \ov\Delta_t^\frac{1}{9}, \Big(\frac{j}{N}\Big)^{3/4}\right\},
\]
 where $\overline{\Delta}_t\lesssim t_*^{3/2}$ denotes the length of the small gap of $\overline{\rho}_t$, for all $\abs{j}\le i_*\sim N^\frac{1}{2}$. 
 A simple calculation from~\eqref{gamma gap} shows that
in the regime $i\ge N^{\om_A}$ and $j\in \mathcal{A}^c$ we may replace  $\abs{\ov \gamma_i -\ov \gamma_j}\sim \abs{\gamma_{y,i}(t)-\gamma_{y,j}(t)}
 \sim \abs{i^{3/4} -j^{3/4}}/N^{3/4}$, hence 
\begin{equation}
\label{E1b}
\abs{\zeta_i}\lesssim  \frac{N^a}{N^\frac{1}{4}}\sum_{1\le j< \frac{3i_*}{4}}^{\mathcal{A}^c,(i)}\frac{i^\frac{1}{2}+j^\frac{1}{2}}{(i-j)^2}
+N^{-\frac{1}{4}-3\omega_\ell}.
\end{equation}
In fact, the same replacement works if either $i\ge N^{4\om_\ell}$ or $j\ge N^{4\om_\ell}$ and at least one
of these two inequalities always hold as $(i,j)\in\mathcal{A}^c$.
Using  $i\le \frac{i_*}{2}$ and that by the restriction $(i,j)\in\mathcal{A}^c$ we have  $\abs{j- i}\ge \ell(\ell^3+i^\frac{3}{4})$,
elementary calculation gives
 \begin{equation}
\label{fe1b}
\abs{\zeta_i}\lesssim 
 \frac{N^a}{N^\frac{1}{4}N^{2\omega_\ell}}.
\end{equation} Since analogous computations hold for $i$ and $j$ both negative, we have 
 \begin{equation}
\label{fe2b}
\abs{\zeta_i}\lesssim \frac{N^a}{N^\frac{1}{4}N^{2\omega_\ell}}, \qquad \mbox{for any}\quad N^{\omega_A}< \abs{i}\le \frac{i_*}{2}.
\end{equation} with very high probability.

Next, we proceed with the bound for $\zeta_i$ for $\abs{i}\le N^{\omega_A}$. From~\eqref{z1b}  we have
 \begin{equation}
\label{E2b}
\begin{split}
\zeta_i&=
\left( \int_{\mathcal{I}_{z,i}(t)^c\cap\mathcal{J}_z(t)}\frac{\ov\rho_{t}(E+\ov{\mathfrak{e}}_{t}^+)}{\widetilde{z}_i-E}\diff E 
- \frac{1}{N}\sum_{\abs{j}<\frac{3i_*}{4}}^{\mathcal{A}^c,(i)}\frac{1}{\widetilde{z}_i-\widetilde{z}_j}\right) \\
&\quad +\left( \int_{\mathcal{J}_z(t)^c}\frac{\overline{\rho}_t(E+\overline{\mathfrak{e}}_t^+)}{\widetilde{z}_i-E}\diff E- \frac{1}{N}\sum_{\abs{j}\ge\frac{3i_*}{4}}^{\mathcal{A}^c,(i)}\frac{1}{\widetilde{z}_i-\widetilde{z}_j}\right) \\
&\quad +\Phi_\alpha(t)- \Re [\overline{m}_t(\widetilde{z}_i+\overline{\mathfrak{e}}_t^+)]+\Re[m_{y,t}(\widetilde{z}_i+\mathfrak{e}_{y,t}^+)] -\Re [m_{y,t}(\mathfrak{e}_{y,t}^+)]\\
&\quad +\left(\int_{\mathcal{I}_{z,i}(t)}\frac{\overline{\rho}_t(E+\overline{\mathfrak{e}}_t^+)}{\widetilde{z}_i-E}\diff E-\int_{\mathcal{I}_{y,i}(t)}\frac{\rho_{y,t}(E+\mathfrak{e}_{y,t}^+)}{\widetilde{z}_i-E}\diff E\right)\defqe   A_1+ A_2 + A_3 + A_4.
\end{split}\end{equation}

By the remark after~\eqref{E1b}, the estimate of $A_1$ proceeds as in~\eqref{E1b} and so we conclude that
 \begin{equation}\label{B11}
 \abs{A_1}\lesssim\frac{N^a}{N^\frac{1}{4} N^{2\omega_\ell}}.
 \end{equation}

To estimate $A_2$, we first notice  that the restriction $(i,j)\in\mathcal{A}^c$ in the summation is superfluous 
for $\abs{i}\le N^{\om_A}$ and $\abs{j}\ge \frac{3}{4}i_*$. Let $\eta_1 \in [N^{-\frac{3}{4}+\frac{3}{4}\om_A}, N^{-\delta}]$, for some small fixed $\delta>0$, be an auxiliary scale we will determine later in the proof, then we write
 $A_2$  as follows:
 \begin{equation}
\label{B121b}
\begin{split}
A_2=&\left( \int_{\mathcal{J}_z(t)^c}\frac{\overline{\rho}_t(E+\overline{\mathfrak{e}}_t^+)}{\widetilde{z}_i-E}\diff E -\int_{\mathcal{J}_z(t)^c}\frac{\overline{\rho}_t(E+\overline{\mathfrak{e}}_t^+)}{\widetilde{z}_i-E+\ii \eta_1}\diff E\right) \\
&\quad +\left(\frac{1}{N}\sum_{\abs{j}\ge\frac{3i_*}{4}}\frac{1}{\widetilde{z}_i-\widetilde{z}_j+\ii \eta_1}
-\frac{1}{N}\sum_{\abs{j}\ge\frac{3i_*}{4}}\frac{1}{\widetilde{z}_i-\widetilde{z}_j}\right) \\
&\quad + \left(\frac{1}{N}\sum_{\abs{j}<\frac{3i_*}{4}}\frac{1}{\widetilde{z}_i-\widetilde{z}_j+\ii \eta_1}-\int_{\mathcal{J}_z(t)}\frac{\overline{\rho}_t(E+\overline{\mathfrak{e}}_t^+)}{\widetilde{z}_i-E+\ii \eta_1}\diff E\right)\\
&\quad+ (\overline{m}_t(\widetilde{z}_i+\ii\eta_1)-m_{2N}(\widetilde{z}_i+\ii\eta_1,t,\alpha))\defqe A_{2,1}+A_{2,2}+A_{2,3}+A_{2,4},
\end{split}\end{equation}
 where we introduced
\[
m_{2N}(z,t,\alpha)\defeq \frac{1}{N}\sum_{\abs{j}\le N}\frac{1}{z_j(t,\alpha)-z}, \qquad z\in \mathbb{H}.
\] 

For $1\le \abs{i} \le N^{\omega_A}$ and $\abs{j}>\frac{3i_*}{4}$, the term $A_{2,2}$ is bounded by
the crude rigidity~\eqref{zerorig} as
 \begin{equation}
\label{B1A2b} 
\abs{A_{2,2}}\le \frac{1}{N}\sum_{\abs{j}>\frac{3i_*}{4}}\frac{\eta_1}{(\widetilde{z}_i-\widetilde{z}_j)^2}\lesssim \frac{N^\frac{1}{2}\eta_1}{i_*^\frac{1}{2}}.
\end{equation} 
Exactly the same estimate holds for $A_{2,1}$.

Next, using the rigidity estimates in~\eqref{zerorig} and~\eqref{riglemma111}  we conclude that \begin{equation}
\label{B1A4b}
\begin{split}
\abs{A_{2,4}}&\lesssim\frac{1}{N}\sum_{1\le \abs{j}\le i_*}\frac{\abs{\widetilde{z}_j-\overline{\gamma}_j}}{\abs{\widetilde{z}_i-\widetilde{z}_j+\ii \eta_1}^2}+\frac{1}{N}\sum_{i_*\le \abs{j}\le N}\frac{\abs{\widetilde{z}_j-\overline{\gamma}_j}}{\abs{\widetilde{z}_i-\widetilde{z}_j+\ii \eta_1}^2}\\
&\lesssim \frac{N^a}{N^\frac{3}{4}\eta_1}\Im m_{N}(\overline{\gamma}_i+\ii \eta_1)+\frac{N^{a_0}}{N^\frac{7}{4}}\sum_{i_*\le \abs{j}\le N}\frac{1}{(\overline{\gamma}_i-\overline{\gamma}_j)^2}\\
&\lesssim \frac{N^a}{N^\frac{3}{4}\eta_1}\left(\frac{N^\frac{3\omega_A}{4}}{N^\frac{3}{4}}+\eta_1\right)^\frac{1}{3}+\frac{N^{a_0}}{N^\frac{1}{4}i_*^\frac{1}{2}} \lesssim \frac{N^a}{N^\frac{3}{4}\eta_1^\frac{2}{3}} +\frac{N^{a_0}}{i_*^{\frac{1}{2}}N^\frac{1}{4}}.
\end{split}\end{equation} 
Here we used that the rigidity scale near $i$ for $1\le \abs{i}\le N^{\om_A}$ is much smaller than  $\eta_1\ge N^{-\frac{3}{4}+\frac{3}{4}\om_A}$. 
In particular, we know that  $\Im m_N(\overline{\gamma}_i+\ii \eta_1)$ can be bounded by the density $\ov\rho_t (\overline{\gamma}_i+ \eta_1)$
 which in turn is bounded by $(\overline{\gamma}_i+ \eta_1)^{1/3}$.
Similarly we conclude that \[
\abs{A_{2,3}}\le \frac{N^a}{N^\frac{3}{4}\eta_1^\frac{2}{3}}.
\]

Optimizing~\eqref{B1A2b} and~\eqref{B1A4b} for $\eta_1$, 
 we choose \(\eta_1 =(i_\ast^{1/2}N^{a-5/4})^{3/5}\) which falls into the required interval for $\eta_1$. Collecting all estimates for the parts of $A_2$ in~\eqref{B121b}, we
  therefore conclude that 
 \begin{equation}
\label{B1,2}
\abs{A_2}\le  \frac{N^\frac{3a}{5}}{i_*^\frac{1}{5}N^{\frac{1}{4}}}+\frac{N^{a_0}}{i_*^\frac{1}{2}N^\frac{1}{4}}.
\end{equation}

Next, we treat $A_3$ from~\eqref{E2b}. $\Phi_\alpha(t)=\Re[\overline{m}_t(\overline{\mathfrak{e}}_t^+)]+\mathcal{O}(N^{-1})$ by~\eqref{Phiest}, then by~\eqref{Re m gap} we conclude that \begin{equation}
\label{E4b}
\begin{split}
\abs{A_3}&=\abs{\Re [\overline{m}_t(\overline{\mathfrak{e}}_t^+)]- \Re [\overline{m}_t(\widetilde{z}_i+\overline{\mathfrak{e}}_t^+)]+\Re[m_{y,t}(\widetilde{z}_i+\mathfrak{e}_{y,t}^+)] -\Re [m_{y,t}(\mathfrak{e}_{y,t}^+)]} \\
&\lesssim\left(\frac{\abs{i}^\frac{1}{4}N^\frac{7\omega_1}{18}}{N^\frac{1}{4}N^\frac{1}{6}}+\frac{\abs{i}^\frac{1}{2}}{N^\frac{1}{2}}\right)\abs{\log \abs{\overline{\gamma}_i}}\lesssim \frac{N^\frac{\omega_A}{4}N^\frac{7\omega_1}{18}\log N}{N^\frac{1}{4} N^\frac{1}{6}}+\frac{N^\frac{\omega_A}{2}\log N}{N^\frac{1}{2}}.
\end{split}\end{equation}

We proceed writing $A_4$ as 
\begin{equation}
\label{B3b}
\begin{split}
A_4&=\left(\int_{\mathcal{I}_{z,i}(t)}\frac{\overline{\rho}_t(E+\overline{\mathfrak{e}}_t^+)}{\widetilde{z}_i-E}\diff E-\int_{\mathcal{I}_{z,i}(t)}\frac{\rho_{y,t}(E+\mathfrak{e}_{y,t}^+)}{\widetilde{z}_i-E}\diff E \right) \\
&\quad + \left(\int_{\mathcal{I}_{z,i}(t)}\frac{\rho_{y,t}(E+\mathfrak{e}_{y,t}^+)}{\widetilde{z}_i-E}\diff E-\int_{\mathcal{I}_{y,i}(t)}\frac{\rho_{y,t}(E+\mathfrak{e}_{y,t}^+)}{\widetilde{z}_i-E}\diff E\right)\defqe A_{4,1}+A_{4,2}.
\end{split}\end{equation} 

We start with the estimate for $A_{4,2}$.
By~\eqref{intmaxj}  and the comparison estimates between $\ov\gamma_{z,i}$ and $\wh\gamma_{y,i}$ by~\eqref{gamma difference gap}
we have that 
\begin{equation}
\label{bsdb}
\abs{\mathcal{I}_{z,i}(t)\Delta\mathcal{I}_{y,i}(t)}\lesssim
\abs{\ov \gamma_{z, i- j_-(i)} - \wh \gamma_{y, i-j_-(i)}} + \abs{\ov \gamma_{z, i+ j_+(i)} - \wh \gamma_{y, i+j_+(i)}}
\lesssim
 \frac{N^\frac{\omega_1}{2}(\ell^3+\abs{i}^\frac{3}{4})}{N^\frac{11}{12}}, 
\end{equation} 
where $\Delta$ is the symmetric difference. In the second inequality of~\eqref{bsdb} we used that 
$\abs{i\pm j_\pm(i)}\lesssim N^{\omega_A}$ and $\om_A\ll 1$.  For $E\in \mathcal{I}_{z,i}\Delta\mathcal{I}_{y,i}$
we have that 
\begin{equation}
\label{bosdb}
\abs{\frac{\rho_{y,t}(E+\mathfrak{e}_{y,t}^+)}{\widetilde{z}_i-E}}\lesssim \frac{N^\frac{1}{2}(\ell^2+\abs{i}^\frac{1}{2})}{\ell^3+\abs{i}^\frac{3}{4}}, 
\end{equation} and so, using $\abs{i}\le N^{\om_A}$, 
 \begin{equation}
\label{E5b}
\abs{A_{4,2}}\lesssim \frac{N^\frac{\omega_1}{2}N^\frac{\omega_A}{2}}{N^\frac{5}{12}}= \frac{N^\frac{\omega_1}{2}N^\frac{\omega_A}{2}}{N^\frac{1}{4}N^\frac{1}{6}}
\end{equation} with very high probability.

To estimate the integral in $A_{4,1}$ we have to deal with the logarithmic singularity due to the values of $E$ close to $\widetilde{z}_i(t)$.
For $\max\{\overline{\mathfrak{e}}_t^-,\mathfrak{e}_{y,t}^-\}< E\le 0$ we have that \begin{equation}
\label{AQZ}
\rho_{y,t}(E+\mathfrak{e}_{y,t}^+)=\overline{\rho}_t(E+\overline{\mathfrak{e}}_t^+)=0.
\end{equation}
For $\min\{\overline{\mathfrak{e}}_t^-,\mathfrak{e}_{y,t}^-\}\le E \le \max\{\overline{\mathfrak{e}}_t^-,\mathfrak{e}_{y,t}^-\}$, using the 
$\frac{1}{3}$-H\"older continuity of $\overline{\rho}_t$ and $\rho_{y,t}$ and~\eqref{eq Delta size} we have that
 \begin{equation}
\label{negativep1b}
\abs{\rho_{y,t}(E+\mathfrak{e}_{y,t}^+)-\overline{\rho}_t(E+\overline{\mathfrak{e}}_t^+)}\lesssim \Delta_{y,t}^\frac{1}{3}(t_*-t)^\frac{1}{9}\lesssim\frac{N^\frac{11\omega_1}{18}}{N^\frac{11}{36}},
\end{equation} 
for all $0\le t\le t_*$. In the last inequality we used that
 $\Delta_{y,t}\le \Delta_{y,0}\lesssim  N^{-\frac{3}{4}+\frac{3\omega_1}{2}}$  
 for all $t\le t_*$. 
Similarly, for $E\le \min\{\overline{\mathfrak{e}}_t^-,\mathfrak{e}_{y,t}^-\}$ we have that \begin{equation}
\label{negativepb}
\abs{\rho_{y,t}(E+\mathfrak{e}_{y,t}^+)-\overline{\rho}_t(E+\overline{\mathfrak{e}}_t^+)}\lesssim\abs{\rho_{y,t}(E'+\mathfrak{e}_{y,t}^-)-\overline{\rho}_t(E'+\overline{\mathfrak{e}}_t^-)}+\Delta_{y,t}^\frac{1}{3}(t_*-t)^\frac{1}{9},
\end{equation} with $E'\le 0$.

Using~\eqref{rho rho gap} for $ E\ge 0$ and combining~\eqref{rho rho gap} with~\eqref{AQZ}-\eqref{negativepb} for $E<0$, we have that \begin{equation}
\label{DD1b}
\begin{split}
\abs{A_{4,1}}&\lesssim  \left(\frac{(\ell+\abs{i}^\frac{1}{4})N^\frac{\omega_1}{3}}{N^\frac{1}{4}N^\frac{1}{6}}+\frac{(\ell^2+\abs{i}^\frac{1}{2})}{N^\frac{1}{2}}+\frac{N^\frac{11\omega_1}{18}}{N^\frac{11}{36}}\right)\int_{\mathcal{I}_{,i}(t)\cap \{\abs{E-\widetilde{z}_i}>N^{-60}\}}\frac{1}{\abs{\widetilde{z}_i-E}}\diff E \\
&\quad +\abs{\int_{\abs{E-\widetilde{z}_i}\le N^{-60} }\frac{\overline{\rho}_t(E+\overline{\mathfrak{e}}_t^+)-\rho_{y,t}(E+\mathfrak{e}_{y,t}^+)}{\widetilde{z}_i-E}\diff E}.
\end{split}\end{equation} 
The two singular integrals in the second line are estimated separately.
By the $\frac{1}{3}$-H\"older continuity $\rho_{y,t}$ we conclude that 
\[\begin{split}
\abs{\int_{\abs{E-\widetilde{z}_i}\le N^{-60}}\frac{\rho_{y,t}(E+\mathfrak{e}_{y,t}^+)}{\widetilde{z}_i-E}\diff E}&=\abs{\int_{\abs{E-\widetilde{z}_i}\le N^{-60}}\frac{\rho_{y,t}(E+\mathfrak{e}_{y,t}^+)-\rho_{y,t}(\widetilde{z}_i+\mathfrak{e}_{y,t}^+)}{\widetilde{z}_i-E}\diff E}\\
&\lesssim \int_{\abs{E-\widetilde{z}_i}\le N^{-60}}\frac{1}{\abs{\widetilde{z}_i-E}^\frac{2}{3}}\diff E\lesssim N^{-20}.
\end{split}\]
The same bound holds for the other singular integral in~\eqref{DD1b} by using the $\frac{1}{3}$-H\"older continuity of $\overline{\rho}_t$.
Hence, 
  for $1\le \abs{i}\le N^{\omega_A}$, by~\eqref{DD1b} we have that 
\begin{equation}
\label{DD12b}
\abs{A_{4,1}}\le \frac{N^\frac{\omega_A}{4}N^\frac{\omega_1}{3}\log N}{N^\frac{1}{4}N^\frac{1}{6}}+
\frac{ N^\frac{\omega_A}{2}\log N}{N^\frac{1}{2}}+\frac{N^\frac{11\omega_1}{18}\log N}{N^\frac{11}{36}},
\end{equation} 
 with very high probability.

Collecting all the estimates~\eqref{fe2b},~\eqref{B11},~\eqref{B1,2},~\eqref{E4b},~\eqref{E5b} and~\eqref{DD12b}, 
and recalling $\om_1\ll \om_\ell\ll \om_A\ll 1$, we see
that~\eqref{B11} is the largest term and thus $\abs{\zeta}\lesssim N^{-\frac{1}{4}-2\omega_\ell}N^{C\omega_1}$ as $a\le C\om_1$. 
Thus, using~\eqref{infinityvb}, 
 we conclude that the estimate in~\eqref{shortestimateb} is satisfied for all $0\le t\le T$. In particular, this means that 
\[\abs{\widehat{z}_i(t,\alpha)-\widetilde{z}_i(t,\alpha)}\le N^{-\frac{3}{4}+C\omega_1}, \qquad 0\le t\le T,\] 
for some small constant $C>0$. We conclude the proof of this lemma showing that $T\ge t_*$. 

Suppose by contradiction that $T<t_*$, then, since the solution of the DBM have continuous paths (see Theorem 12.2 of~\cite{MR3699468}),
 we have that 
\[ 
\abs{\widehat{z}_i(T+\tilde{t},\alpha)-\widetilde{z}_i(T+\tilde{t},\alpha)}\le 
\frac{N^a N^{c\omega_1}}{N^\frac{3}{4}N^{2\omega_\ell}},
\] for some tiny
 $\tilde{t}>0$ and for any $\alpha\in [0,1]$. This bound is much smaller than the threshold 
 $\abs{\mathcal{I}_{y,i}(t)},\abs{\mathcal{I}_{z,i}(t)}\sim N^{-\frac{3}{4}+3\omega\ell}$ in the definition of $T$.
 But this is a contradiction by the maximality in the definition of $T$, hence $T=t_*$, proving 
~\eqref{shortestimateb}  for all $0\le t\le t_*$. This completes the proof of Lemma~\ref{seb111}.
\endproof
\end{lemma}

\begin{proof}[Proof of Lemma~\ref{se}.]
The proof of this lemma  is very similar to that of Lemma~\ref{seb111}, 
hence we will only  sketch the proof by indicating the differences. 
The main difference is that in this lemma we have optimal $i$-dependent rigidity for all $1\le \abs{i}\le i_*$. Hence, we can give a better estimate 
on the first two terms in~\eqref{E133b} as follows  (recall that $N^{\om_A}\le i\le \frac{i_*}{2}$)
\[
\abs{\zeta_i} \lesssim \frac{N^ \xi N^\frac{\omega_1}{6}}{N^\frac{3}{4}}\sum_{\abs{j}<\frac{3i_*}{4}}\frac{1}{(\overline{\gamma}_i-\overline{\gamma}_j)^2\abs{j}^\frac{1}{4}}\lesssim \frac{N^ \xi N^\frac{\omega_1}{6}}{N^\frac{3}{4}}\sum_{\abs{j}<\frac{3i_*}{4}}\frac{\abs{i}^\frac{1}{2}+\abs{j}^\frac{1}{2}}{(\abs{i}-\abs{j})^2\abs{j}^\frac{1}{4}}
\lesssim \frac{N^\xi N^\frac{\omega_1}{6}}{N^\frac{1}{4}N^{3\omega_\ell}}.
\] 
Compared with~\eqref{fe1b}, the additional $N^{\om_\ell}$ factor in the denominator comes from the $\abs{j}^{1/4}$ factor beforehand
that is due to the optimal dependence of the rigidity on the index. Consequently,  using the optimal rigidity in~\eqref{rzx},
 we improve the denominator in the
  first term on the rhs. of~\eqref{shortestimateb} from $N^{2\omega_\ell}$ to $N^{3\omega_\ell}$ with respect  Lemma~\ref{seb111}.

Furthermore, by~\eqref{rzx}, 
\[
\abs{A_{2,3}},\abs{A_{2,4}}\le \frac{N^\xi}{N\eta_1},\qquad\mbox{and}\qquad
\abs{A_{2,1}},\abs{A_{2,2}}\lesssim \frac{N^\frac{1}{2}\eta_1}{i_*^\frac{1}{2}}\lesssim N^{\frac{1}{4}-\frac{C_*\omega_1}{2}}\eta_1,
\] 
since $i_*= N^{\frac{1}{2}+C_*\omega_1}$, hence, choosing $\eta_1=N^{-\frac{5}{8}}$, 
we conclude that 
\[
\abs{A_1}+\abs{A_2}\lesssim \frac{N^\xi N^\frac{\omega_1}{6}}{N^\frac{1}{4}N^{3\omega_\ell}}+\frac{N^\xi}{N^\frac{3}{8}}.
\]
All other estimates follow exactly in the same way of the proof of Lemma~\ref{seb111}.
This concludes the proof of Lemma~\ref{se}.
\end{proof}

\subsection{Short-long approximation: Small minimum.}
\label{SLLM}
In this subsection we estimate the difference of the solution of the DBM $\widetilde{z}(t,\alpha)$ defined by~\eqref{interminz} and its short range approximation $\widehat{z}(t,\alpha)$ defined by~\eqref{SDE1min}-\eqref{idmin} for $t_*\le t\le 2t_*$.

\begin{lemma}
\label{semin111aaa}
Under the same assumption of Lemma~\ref{SLLG} and assuming that the rigidity
bounds~\eqref{zerorig} and~\eqref{riglemma111} hold for the $\wt z(t,\alpha)$ dynamics~\eqref{interminz} for all $t_*\le t\le 2t_*$,
 we conclude that \begin{equation}
\label{shortestimateminaaa}
\sup_{1\le \abs{i}\le N} \sup_{t_*\le t\le 2t_*}\abs{\widetilde{z}_i(t,\alpha)-\widehat{z}_i(t,\alpha)}\lesssim \frac{N^aN^{C\omega_1}}{N^\frac{3}{4}}\left(\frac{1}{N^{2\omega_\ell}}+ \frac{1}{N^\frac{2a}{5}i_*^\frac{1}{5}} 
+\frac{N^{a_0}}{N^ai_*^\frac{1}{2}}+\frac{1}{N^aN^\frac{1}{24}}\right),
\end{equation} with very high probability, for any $\alpha\in [0,1]$.
\proof
The proof of this lemma is similar to the proof of Lemma~\ref{SLLG}, but some estimates for the semicircular flow are slightly
different mainly because in this lemma the $\widetilde{z}_i(t,\alpha)$ 
are shifted by $\overline{\mathfrak{m}}_t$ instead of $\overline{\mathfrak{e}}_t^+$.
 Hence, we will skip some details in this proof, describing carefully only the estimates that are different respect to Lemma~\ref{SLLG}.

Let $w_i\defeq \widehat{z}_i-\widetilde{z}_i$, hence $w$ is a solution of \[ \partial_t=\mathcal{B}_1w+\mathcal{V}_1 w+\zeta,\] where $\mathcal{B}_1$ and $\mathcal{V}_1$ are defined as in~\eqref{B1b}-\eqref{V2b} substituting $\overline{\mathfrak{e}}_t^+$ with $\overline{\mathfrak{m}}_t$.

Without loss of generality we assume that $\mathcal{V}_1\le 0$ for all $t_*\le t\le T$ (see~\eqref{bigtimeb} in the proof of Lemma~\ref{SLLG}
but now we have $t_*\le t\le 2t_*$ in the definition of the stopping time).
 This implies that $\mathcal{U}^{\mathcal{B}_1+\mathcal{V}_1}$ is a contraction semigroup and 
 so in order to prove~\eqref{shortestimateminaaa} it is enough to estimate 
 \[\sup_{t_*\le t\le T}\lVert \zeta(s)\rVert_\infty.\] 
At the end, exactly as at the end of the proof of Lemma~\ref{seb111},
  by continuity of the paths, we can easily establish $T=2t_*$ for the stopping time.

The error term $\zeta$ is given by $\zeta_i=0$ for $\abs{i}> \frac{i_*}{2}$, then $\zeta_i$ for $1\le \abs{i}\le N^{\omega_A}$ is defined as \begin{equation}
\label{z1minaa}
\zeta_i=\int_{\mathcal{I}_{y,i}(t)^c}\frac{\rho_{y,t}(E+\widetilde{\mathfrak{m}}_{y,t})}{\widetilde{z}_i-E}\diff E-\frac{1}{N}\sum_j^{\mathcal{A}^c,(i)}\frac{1}{\widetilde{z}_i-\widetilde{z}_j}+\Psi_\alpha(t)+\frac{\diff}{\diff t}\widetilde{\mathfrak{m}}_{y,t},
\end{equation} with $\Psi_\alpha(t)$ defined in~\eqref{Psimin}, and for $N^{\omega_A}< \abs{i}\le \frac{i_*}{2}$ as \begin{equation}
\label{z2minaaa}
\zeta_i=\int_{\mathcal{I}_{z,i}(t)^c\cap\mathcal{J}_z(t)}\frac{\overline{\rho}_t(E+\overline{\mathfrak{m}}_t)}{\widetilde{z}_i-E}\diff E-\frac{1}{N}\sum_{\abs{j}<\frac{3i_*}{4}}^{\mathcal{A}^c,(i)}\frac{1}{\widetilde{z}_i-\widetilde{z}_j}.
\end{equation}

We start to estimate the error term for $N^{\omega_A}< \abs{i}\le \frac{i_*}{2}$. 
A similar computation as the one leading to~\eqref{fe2b} in Lemma~\ref{SLLG}, using~\eqref{riglemma111}, we conclude that 
\begin{equation}
\label{E133min}
\abs{\zeta_i}=\abs{\int_{\mathcal{I}_{i,z}^c(t)\cap\mathcal{J}_z(t)}\frac{\overline{\rho}_t(E+\overline{\mathfrak{m}}_t)}{\widetilde{z}_i-E}\diff E-\frac{1}{N}\sum_{\abs{j}< \frac{3i_*}{4}}^{\mathcal{A}^c,(i)}\frac{1}{\widetilde{z}_i-\widetilde{z}_j}}\lesssim \frac{N^a}{N^\frac{1}{4}N^{2\omega_\ell}},
\qquad N^{\omega_A}< \abs{i}\le \frac{i_*}{2}.
\end{equation}

Next, we proceed with the bound for $\zeta_i$ for $1\le \abs{i}\le N^{\omega_A}$. We rewrite $\zeta_i$ as 
\begin{equation}
\label{E2minaaa}
\begin{split}
\zeta_i&=\left(\int_{\mathcal{I}_{i,z}^c(t)}\frac{\overline{\rho}_t(E+\overline{\mathfrak{m}}_t)}{\widetilde{z}_i-E}\diff E-\frac{1}{N}\sum_j^{\mathcal{A}^c,(i)}\frac{1}{\widetilde{z}_i-\widetilde{z}_j}\right)\\
&\quad+\Re[m_{y,t}(\widetilde{z}_i+\widetilde{\mathfrak{m}}_{y,t})]+\frac{\diff}{\diff t}\widetilde{\mathfrak{m}}_{y,t} +\Psi_\alpha(t)-\Re[\overline{m}_t(\widetilde{z}_i+\overline{\mathfrak{m}}_t)] \\
&\quad +\left(\int_{\mathcal{I}_{z,i}(t)}\frac{\overline{\rho}_t(E+\overline{\mathfrak{m}}_t)}{\widetilde{z}_i-E}\diff E-\int_{\mathcal{I}_{y,i}(t)}\frac{\rho_{y,t}(E+\widetilde{\mathfrak{m}}_{y,t})}{\widetilde{z}_i-E}\diff E\right)\defqe (A_1+A_2)+A_3+A_4.
\end{split}\end{equation}
where $(A_1+A_2)$ indicates that for the actual estimates we split the first line in~\eqref{E2minaaa} into two terms 
as in~\eqref{E2b}.
By similar computations as in Lemma~\ref{SLLG}, see~\eqref{B11} and~\eqref{B1,2}, we conclude that 
\begin{equation}
\label{E3minaaa}
\abs{A_1}+\abs{A_2}\lesssim \frac{N^a}{N^\frac{1}{4}N^{2\omega_\ell}}+ \frac{N^\frac{3a}{5}}{N^\frac{1}{4}i_*^\frac{1}{5}} 
+\frac{N^{a_0}}{i_*^{\frac{1}{2}}N^\frac{1}{4}}.
\end{equation}

By~\eqref{eq wt mi mi},~\eqref{diff m m},~\eqref{Re m min} and the definition of $\Psi_\alpha(t)$ in~\eqref{Psimin} it follows that \begin{equation}
\label{E4minaaa}
\begin{split}
\abs{A_3}&\lesssim \abs{\Re[m_{y,t}(\widetilde{z}_i+\widetilde{\mathfrak{m}}_{y,t})-m_{y,t}(\widetilde{\mathfrak{m}}_{y,t})]-\Re[\overline{m}_t(\overline{\mathfrak{m}}_t)-\overline{m}_t(\widetilde{z}_i+\overline{\mathfrak{m}}_t)]}+\frac{N^{\omega_1}}{N}\\
&\lesssim \left(\frac{N^\frac{\omega_A}{4}N^\frac{\omega_1}{4}}{N^\frac{1}{4}N^\frac{1}{8}}+\frac{N^\frac{3\omega_1}{4}}{N^\frac{3}{8}}+\frac{N^\frac{\omega_A}{2}}{N^\frac{1}{2}}\right) \abs{\log \abs{\widehat{\gamma}_i(t)}}+\frac{N^\frac{7\omega_1}{12}}{N^\frac{7}{24}}\lesssim\frac{N^\frac{7\omega_1}{12}}{N^\frac{7}{24}}.
\end{split}\end{equation}

We proceed writing $A_4$ as \begin{equation}
\label{B3minaaa}
\begin{split}
A_4&=\left(\int_{\mathcal{I}_{z,i}(t)}\frac{\overline{\rho}_t(E+\overline{\mathfrak{m}}_t)}{\widetilde{z}_i-E}\diff E-\int_{\mathcal{I}_{z,i}(t)}\frac{\rho_{y,t}(E+\widetilde{\mathfrak{m}}_{y,t})}{\widetilde{z}_i-E}\diff E \right) \\
&\quad + \left(\int_{\mathcal{I}_{z,i}(t)}\frac{\rho_{y,t}(E+\widetilde{\mathfrak{m}}_{y,t})}{\widetilde{z}_i-E}\diff E-\int_{\mathcal{I}_{y,i}(t)}\frac{\rho_{y,t}(E+\widetilde{\mathfrak{m}}_{y,t})}{\widetilde{z}_i-E}\diff E\right)\defqe A_{4,1}+A_{4,2}.
\end{split}\end{equation} We start with the estimate for $A_{4,2}$.

By~\eqref{gamma difference min} we have that 
\begin{equation}
\label{bsd.1aaa}
\abs{\mathcal{I}_{z,i}(t)\Delta\mathcal{I}_{y,i}(t)}\lesssim \frac{N^\xi(\ell+\abs{i})}{N},
\end{equation} where $\Delta$ is the symmetric difference. Note that this bound is somewhat better than the analogous~\eqref{bsdb}
due to the better bound in~\eqref{gamma difference min} compared with~\eqref{gamma difference gap}.
For $E\in \mathcal{I}_{z,i}(t)\Delta\mathcal{I}_{y,i}(t)$  we have that 
 \begin{equation}
\label{bosd.1aaa}
\abs{\frac{\rho_{y,t}(E+\overline{\mathfrak{m}}_t)}{\widetilde{z}_i-E}}\lesssim \frac{N^\frac{1}{2}(\ell^2+\abs{i}^\frac{1}{2})}{\ell^3+\abs{i}^\frac{3}{4}}, 
\end{equation} and so \begin{equation}
\label{E5.1}
\abs{A_{4,2}}\lesssim \frac{N^\frac{3\omega_A}{4}}{N^\frac{1}{2}}
\end{equation} with very high probability.

To estimate the integral in $A_{4,1}$, we 
combine~\eqref{rho rho min} and~\eqref{eq wt mi mi} to  obtain that
 \begin{equation}
\label{dendiffholdaaa}
\abs{\overline{\rho}_t(\overline{\mathfrak{m}}_t+E)-\rho_{y,t}(\mathfrak{\widetilde{m}}_{y,t}+E) }\le\abs{\rho_{x,t}(\alpha\mathfrak{m}_{x,t}+(1-\alpha)\mathfrak{m}_{y,t}+E)-\rho_{y,t}(\mathfrak{m}_{y,t}+E)}+(t-t_*)^\frac{7}{12}.
\end{equation} 
Proceeding similarly to the estimate of $\abs{A_{4,1}}$ at the end of the proof of Lemma~\ref{SLLG}, we conclude that \begin{equation}
\label{DD1minaaa}
\begin{split}
\abs{A_{4,1}}&\lesssim\left(\frac{N^\xi (\ell^2+\abs{i}^\frac{1}{2})}{N^\frac{1}{2}}+\frac{N^\frac{7\omega_1}{12}}{N^\frac{7}{24}}\right)\int_{\mathcal{I}_{z,i}(t)\cap \{\abs{E-\widetilde{z}_i}>N^{-60}\}}\frac{1}{\abs{\widetilde{z}_i-E}}\diff E \\
&\quad +\abs{\int_{\abs{E-\widetilde{z}_i}\le N^{-60}}\frac{\overline{\rho}_t(E+\overline{\mathfrak{m}}_t)-\rho_{y,t}(E+\widetilde{\mathfrak{m}}_{y,t})}{\widetilde{z}_i-E}\diff E}.
\end{split}\end{equation} 
Furthermore, similarly to the estimate in the singular integral in~\eqref{DD1b}, but substituting $\overline{\mathfrak{e}}_t^+$ and $\mathfrak{e}_{y,t}^+$ by $\overline{\mathfrak{m}}_t$ and $\widetilde{\mathfrak{m}}_{y,t}$ respectively, we conclude that
that the last term in~\eqref{DD1minaaa} is bounded by $N^{-20}$.
Therefore, 
\begin{equation}
\label{DD13minaaa}
\abs{A_{41}}\lesssim \frac{N^\xi (\ell^2+\abs{i}^\frac{1}{2})}{N^\frac{1}{2}}+
\frac{N^\frac{7\omega_1}{12}}{N^\frac{7}{24}}\lesssim \frac{N^\frac{7\omega_1}{12}}{N^\frac{7}{24}},
\end{equation} for any $\abs{i}\le N^{\omega_A}$.
Collecting~\eqref{E3minaaa},~\eqref{E4minaaa},~\eqref{E5.1} and~\eqref{DD13minaaa} completes the proof of Lemma~\ref{semin111aaa}.
\endproof
\end{lemma}

\section{Sobolev-type inequality}
\label{STI}

The proof of the Sobolev-type inequality in the cusp case is essentially identical to that in the edge case
presented in Appendix B of~\cite{MR3253704}; only the exponents need adjustment
to the cusp scaling. We give some details for completeness.

\begin{proof}[Proof of Lemma~\ref{lm:Sobolev}]
We will prove only the first inequality in~\eqref{Sobolev}. The proof for the second one is exactly the same.
We start  by proving a continuous version of~\eqref{Sobolev} and then we will conclude the proof by linear interpolation.
We claim that for any small $\eta$ there exists a constant $c_\eta>0$ such that for any real function 
$f\in L^p(\mathbb{R}_+)$ we have that
\begin{equation}
\label{cSobolev}
\int_0^{+\infty}\int_0^{+\infty}\frac{(f(x)-f(y))^2}{\abs{x^\frac{3}{4}-y^\frac{3}{4}}^{2-\eta}}\diff x\diff y\ge c_\eta\left(\int_0^{+\infty}\abs{f(x)}^p\diff x\right)^\frac{2}{p}.
\end{equation}  

We recall the representation formula for fractional powers 
of the Laplacian: for any $0<\alpha<2$ and for any function  $f\in L^p(\mathbb{R})$ for
some $p\in [1,\infty)$ 
we have 
\begin{equation}
\label{laplace}
\langle f, \abs{p}^\alpha f\rangle=C(\alpha)\int_\mathbb{R}\int_\mathbb{R}\frac{(f(x)-f(y))^2}{\abs{x-y}^{1+\alpha}}\diff x\diff y,
\end{equation} with some explicit constant $C(\alpha)$, where $\abs{p}\defeq \sqrt{-\Delta}$. 

Since for $0<x<y$ we have that 
\[
y^\frac{3}{4}-x^\frac{3}{4}=\frac{4}{3}\int_x^ys^{-\frac{1}{4}}\, \diff s\le C(y-x)(xy)^{-\frac{1}{8}},
\] in order to prove~\eqref{cSobolev} it is enough to show that \begin{equation}
\label{ccSobolev}
\int_0^{+\infty}\int_0^{+\infty}\frac{(f(x)-f(y))^2}{\abs{x-y}^{2-\eta}} (xy)^q\diff x\diff y\ge c_\eta\left(\int_0^{+\infty}\abs{f(x)}^p\diff x\right)^\frac{2}{p},
\end{equation} where $q\defeq \frac{1}{4}-\frac{\eta}{8}$ and $p\defeq \frac{8}{2+3\eta}$. 
Let $\tilde{f}(x)$ be the symmetric extension of $f$ to the whole real line, i.e.
  $\tilde{f}(x)\defeq f(x)$ for $x>0$ and $\tilde{f}(x)\defeq f(-x)$ for $x<0$. Then, by a simple calculation we have
  \[
  \begin{split}
  4\int_0^{+\infty}\int_0^{+\infty}\frac{(f(x)-f(y))^2}{\abs{x-y}^{2-\eta}} (xy)^q\diff x\diff y 
  &\ge \int_\mathbb{R}\int_\mathbb{R}\frac{(\tilde{f}(x)-\tilde{f}(y))^2}{\abs{x-y}^{2-\eta}} \abs{xy}^q\diff x\diff y.
\end{split}
\] 
Introducing $\phi(x)\defeq \abs{x}^q$ and  dropping the tilde for $f$ 
the estimate in~\eqref{ccSobolev} would follow from
\begin{equation}
\label{ccccSobolev}
\int_\mathbb{R}\int_\mathbb{R}\frac{(f(x)-f(y))^2}{\abs{x-y}^{2-\eta}} \phi(x)\phi(y)\diff x\diff y\ge c_\eta'\left(\int_\mathbb{R}\abs{f(x)}^p\diff x\right)^\frac{2}{p}.
\end{equation}

By the same computation as in the proof of Proposition 10.5 in~\cite{MR3253704} we conclude that
 \[
 \begin{split}
\int_\mathbb{R}\int_\mathbb{R}\frac{(f(x)-f(y))^2}{\abs{x-y}^{2-\eta}}\phi(x)\phi(y)\diff x\diff y 
=\langle \phi f, \abs{p}^{1-\eta}\phi f\rangle+C_0(\eta)\int_\mathbb{R}\frac{\abs{\phi(x)f(x)}^2}{\abs{x}^{1-\eta}}\diff x
\end{split}
\]
with some $C_0(\eta)>0$, 
hence for the proof of~\eqref{ccccSobolev} it is enough to show that 
\[\langle \phi f, \abs{p}^{1-\eta}\phi f\rangle \ge c_\eta\left(\int_\mathbb{R} \abs{f}^p\right)^\frac{2}{p}.\] 
Let $g\defeq \abs{p}^{\frac{1}{2}(1-\eta)}\abs{x}^qf$, we need to prove that \[ \lVert g\rVert_2\ge c_\eta\lVert \abs{x}^{-q}\abs{p}^{-\frac{1}{2}(1-\eta)}g\rVert_p.\] By the $n$-dimensional Hardy-Littlewood-Sobolev inequality in~\cite{MR0098285} we have that \[\left\lVert\abs{x}^{-q}\int \abs{x-y}^{-a}g(y)\,\diff y\right\rVert_p\le C\lVert g\rVert_r,\] where $\frac{1}{r}+\frac{a+q}{n}=1+\frac{1}{p}$, $0\le q<\frac{n}{p}$ and $0<a<n$. In our case $a=\frac{1+\eta}{2}$, $r=2$, $n=1$ and all the conditions are satisfied if we take $0<\eta<1$. This completes the proof of~\eqref{cSobolev}.

Next,  in order to prove~\eqref{Sobolev}, we proceed by linear interpolation as in Proposition B.2 in~\cite{MR3372074}. Given $u:\mathbb{Z}\to\mathbb{R}$, let $\psi:\mathbb{R}\to \mathbb{R}$ be its linear interpolation, i.e.~$\psi(i)\defeq u_i$ for $i\in\mathbb{Z}$ and \begin{equation}
\label{linint}
\psi(x)\defeq u_i+(u_{i+1}-u_i)(x-i)=u_{i+1}-(u_{i+1}-u_i)(i+1-x),
\end{equation} for $x\in [i,i+1]$. It is easy to see that for each $p\in[2,+\infty]$ (i.e.~$\eta\le 2/3$),
there exists a constant $C_p$ such that 
\begin{equation}
\label{estnorm}
C_p^{-1}\lVert \psi\rVert_{L^p(\mathbb{R})}\le\lVert u\rVert_{L^p(\mathbb{Z})}\le C_p\lVert \psi\rVert_{L^p(\mathbb{R})}.
\end{equation} 
In order to prove~\eqref{Sobolev} we claim that
 \begin{equation}
\label{discont}
\int_0^{+\infty}\int_0^{+\infty}\frac{\abs{\psi(x)-\psi(y)}^2}{\abs{x^\frac{3}{4}-y^\frac{3}{4}}^{2-\eta}}\diff x\diff y\le c_\eta \sum_{i\ne j\in\mathbb{Z}_+}\frac{(u_i-u_j)^2}{\abs{i^\frac{3}{4}-j^\frac{3}{4}}^{2-\eta}},
\end{equation} for some constant $c_\eta>0$. Indeed, combining~\eqref{estnorm} and~\eqref{discont} with~\eqref{cSobolev} we conclude~\eqref{Sobolev}.  Finally, the proof of~\eqref{discont} is a simple exercise along the lines of the proof of 
Proposition B.2 in~\cite{MR3372074}.
\end{proof}

\section{Heat-kernel estimates}
\label{hke}

The proof of the heat kernel estimates  relies on the Nash method. In the edge scaling regime  a similar bound was proven in
~\cite{MR3253704} for a compact interval, extended to non-compact interval but with compactly supported initial
 data $w_0$ in~\cite{1712.03881}. Here we closely follow the latter proof, adjusted to the cusp regime, where interactions
 on both sides of the cusp play a role unlike in the edge regime.  
 
 \begin{proof}[Proof of Lemma~\ref{ee}.]
We start proving~\eqref{Ee1}, then~\eqref{Ee2} follows by~\eqref{Ee1} by duality.
Without loss of generality we assume $\norm{ w_0}_1=1$ and that \begin{equation}
\label{lbp}
\lVert w(\tilde{s})\rVert_p\ge N^{-100}
\end{equation} for each $s\le \tilde{s} \le t$, where  $w(\tilde{s})=\mathcal{U}^{\mathcal{L}}(s, \tilde{s})w_0$.
Otherwise, by $\ell^p$-contraction we had $\norm{ w(\tilde{s})}_p\le N^{-100}$ implying~\eqref{Ee1} directly.

In the following we use the convention $w\defeq w(\tilde{s})$ if there is no confusion.
By~\eqref{Sobolev}, we have that \[\lVert w\rVert_p^2 \lesssim \sum_{\substack{i,j\ge 1 \\ i\ne j}}\frac{(w_i-w_j)^2}{\abs{i^\frac{3}{4}-j^\frac{3}{4}}^{2-\eta}}+\sum_{\substack{i,j\le -1 \\ i\ne j}}\frac{(w_i-w_j)^2}{\abs{\abs{i}^\frac{3}{4}-\abs{j}^\frac{3}{4}}^{2-\eta}}.\] First we assume that both $i$ and $j$ are positive.
Let $\delta_4<\delta_2<\delta_3< \frac{\delta_1}{2}$.
We start with the following estimate \begin{equation}
\label{Z1}
\sum_{\substack{i,j \ge 1 \\ i\ne j}}\frac{(w_i-w_j)^2}{\abs{i^\frac{3}{4}-j^\frac{3}{4}}^{2-\eta}}\lesssim\sum_{\substack{(i,j)\in\mathcal{A} \\ i,j\ge 1}}\frac{(w_i-w_j)^2}{\abs{i^\frac{3}{4}-j^\frac{3}{4}}^{2-\eta}}+\sum_{i\ge 1}\sum_{j\ge 1}^{\mathcal{A}^c,(i)}\frac{w_i^2}{\abs{i^\frac{3}{4}-j^\frac{3}{4}}^{2-\eta}}.
\end{equation}
We proceed by writing 
\begin{equation}
\label{Z2}
\sum_{\substack{(i,j)\in\mathcal{A}\\ i,j\ge 1}}\frac{(w_i-w_j)^2}{\abs{i^\frac{3}{4}-j^\frac{3}{4}}^{2-\eta}}\lesssim \sum_{\substack{(i,j)\in\mathcal{A}: \,i,j\ge 1\\ i \,\text{or}\, j\le \ell^4N^{\delta_2}}}\frac{(w_i-w_j)^2}{\abs{i^\frac{3}{4}-j^\frac{3}{4}}^{2-\eta}}+\sum_{\substack{(i,j)\in\mathcal{A}\\ i,j\ge \ell^4 N^{\delta_2}}}\frac{(w_i-w_j)^2}{\abs{i^\frac{3}{4}-j^\frac{3}{4}}^{2-\eta}}.
\end{equation}

By Lemma~\ref{fsl} we have that \begin{equation}
\label{1}
\sum_{\substack{(i,j)\in\mathcal{A}\\ i,j\ge \ell^4N^{\delta_2}}}\frac{(w_i-w_j)^2}{\abs{i^\frac{3}{4}-j^\frac{3}{4}}^{2-\eta}}\lesssim N^{-200},
\end{equation} since $i\ge \ell^4N^{\delta_2}$ and $\abs{(w_0)_j}\le N^{-100}$ for $j\ge \ell^4N^{\delta_4}$ by our hypotheses. Indeed, for $i\ge \ell^4N^{\delta_2}$, we have that \begin{equation}
\label{decay}
w_i=\left(\mathcal{U}^\mathcal{L}(s,\tilde{s})w_0\right)_i=\sum_{j=-N}^N \mathcal{U}^\mathcal{L}_{ij}(w_0)_j=\sum_{j=-\ell^4N^{\delta_4}}^{\ell^4N^{\delta_4}}\mathcal{U}^\mathcal{L}_{ij}(w_0)_j+N^{-100}\lesssim N^{-100},
\end{equation} with very high probability. If $(i,j)\in\mathcal{A}$, $i,j\ge 1$ and $i$ or $j$ are smaller than $\ell^4N^{\delta_2}$ then both $i$ and $j$ are smaller than $\ell^4N^{\delta_3}$. Hence, for such $i$ and $j$, by~\eqref{ABC}, we have that \begin{equation}
\label{rig}
\abs{\widehat{z}_i(t,\alpha)-\widehat{z}_j(t,\alpha)}\lesssim\frac{N^\frac{\omega_1}{6}\abs{i^\frac{3}{4}-j^\frac{3}{4}}}{N^\frac{3}{4}},
\end{equation} for any fixed $\alpha\in [0,1]$ and for all $0\le t\le t_*$, where $\widehat{z}_i(t,\alpha)$ is defined by~\eqref{11}-\eqref{V222rig1}.

If $i$ and $j$ are both negative the estimates in~\eqref{Z1}-\eqref{rig} follow in the same way.

In the following of the proof $\mathcal{B}$, $\mathcal{B}_{ij}$ and $\mathcal{V}_i$ are defined in~\eqref{11}-\eqref{V222rig1}. By~\eqref{rig} it follows that
 \begin{equation}
\label{2}
\begin{split}
\sum_{\substack{(i,j)\in\mathcal{A}:\, i,j\ge 1\\ i\text{ or } j\le \ell^4N^{\delta_2}}}\frac{(w_i-w_j)^2}{\abs[1]{i^\frac{3}{4}-j^\frac{3}{4}}^{2-\eta}}+\sum_{\substack{(i,j)\in\mathcal{A}:\, i,j\le -1\\  i \,\text{or}\, j\ge -\ell^4N^{\delta_2}}}\frac{(w_i-w_j)^2}{\abs[1]{i^\frac{3}{4}-j^\frac{3}{4}}^{2-\eta}}&\lesssim -N^{-\frac{1}{2}}N^{\frac{\omega_1}{3}+C\eta}\sum_{(i,j)\in\mathcal{A}}\mathcal{B}_{ij}(w_i-w_j)^2\\
&= -2N^{-\frac{1}{2}}N^{\frac{\omega_1}{3}+C\eta}\langle w, \mathcal{B} w\rangle.
\end{split}\end{equation}

Furthermore, since $1\le \abs{i}\le \ell^4N^{\delta_3}$, we have that 
\begin{equation}
\label{ADFR}
\sum_j^{\mathcal{A}^c,(i)}\frac{1}{\abs[1]{\abs{i}^\frac{3}{4}-\abs{j}^\frac{3}{4}}^{2-\eta}}\lesssim \frac{N^{\frac{\omega_1}{3}+C\eta}}{N^\frac{3}{2}}\sum_j^{\mathcal{A}^c,(i)}\frac{1}{(\widehat{z}_i-\widehat{z}_j)^2}.
\end{equation}

By the rigidity~\eqref{RZH1},~\eqref{RZH} and~\eqref{ABC}, we can replace $\widehat{z}_j$ by $\overline{\gamma}_j$ in the sum on the rhs. of~\eqref{ADFR} and so approximate it by an integral, then using that $\overline{\rho}_t(E)\lesssim\rho_{y,t}(E)$ in the cusp regime, i.e.~$\abs{E}\le \delta_*$, with $\delta_*$ defined in Definition~\ref{def interpolating density}, we conclude that 
\begin{equation}
\label{concl}
\frac{1}{N}\sum_j^{\mathcal{A}^c,(i)}\frac{1}{(\widehat{z}_i(t)-\widehat{z}_j(t))^2}\lesssim
 \int_{I_{i,y}(t)^c}\frac{\rho_{y,t}(E+\mathfrak{e}_{y,t}^+)}{(\wh{z}_i(t)-E)^2}\diff E= -\mathcal{V}_i.
\end{equation}

Hence, by~\eqref{concl}, we conclude that \begin{equation}
\label{3}
\begin{split}
\sum_i\sum_j^{\mathcal{A}^c,(i)}\frac{w_i^2}{\abs[1]{\abs{i}^\frac{3}{4}-\abs{j}^\frac{3}{4}}^{2-\eta}}&\lesssim \sum_{1\le\abs{i}\le \ell^4N^{\delta_3}}\sum_j^{\mathcal{A}^c,(i)}\frac{w_i^2}{\abs[1]{\abs{i}^\frac{3}{4}-\abs{j}^\frac{3}{4}}^{2-\eta}}+N^{-200}\\
&\lesssim -N^{-\frac{1}{2}}N^{\frac{\omega_1}{3}+C\eta}\sum_{\abs{i}\le \ell^4N^{\delta_3}}w_i^2\mathcal{V}_i+N^{-200}\\
&\lesssim  -N^{-\frac{1}{2}}N^{\frac{\omega_1}{3}+C\eta}\langle w, \mathcal{V}w\rangle+N^{-200}.\end{split}
\end{equation} Note that in the first inequality of~\eqref{3} we used~\eqref{decay}.

Summarizing~\eqref{1},~\eqref{2} and~\eqref{3} and rewriting $N^{-200}$ into an $\ell^p$-norm using
~\eqref{lbp},  we obtain
  \[\lVert w\rVert_p^2\le -N^{-\frac{1}{2}}N^{\frac{\omega_1}{3}+C\eta}\langle w, \mathcal{L} w\rangle+\frac{1}{10}\lVert w\rVert_p^2.\] Hence, using H\"older inequality, we have that \begin{equation}
\label{4}
\begin{split}
\partial_t\lVert w \rVert_2^2&=\langle w, \mathcal{L} w \rangle\le  -c_\eta N^\frac{1}{2}N^{-\frac{\omega_1}{3}-C\eta}\lVert w\rVert_p^2 \\
&\le -c_\eta N^\frac{1}{2}N^{-\frac{\omega_1}{3}-C\eta}\lVert w\rVert_2^\frac{6-3\eta}{2}\lVert w\rVert_1^{-\frac{2-3\eta}{2}} \\
&\le -c_\eta N^\frac{1}{2}N^{-\frac{\omega_1}{3}-C\eta}\lVert w\rVert_2^\frac{6-3\eta}{2}\lVert w_0\rVert_1^{-\frac{2-3\eta}{2}}.
\end{split}\end{equation} In the last inequality of~\eqref{4} we used the $\ell^1$-contraction of $\mathcal{U}^\mathcal{L}$.
Integrating~\eqref{4} back in time, it easily follows that \begin{equation}
\label{5}
\lVert\mathcal{U}^\mathcal{L}(s,t)w_0\rVert_2\le \left(\frac{N^{C\eta+\frac{\omega_1}{3}}}{c_\eta N^\frac{1}{2}(t-s)}\right)^{1-3\eta}\lVert w_0\rVert_1,
\end{equation} proving~\eqref{Ee1}. The same bound also holds for the transpose operator $(\mathcal{U}^\mathcal{L})^T$.

In order to prove~\eqref{Ee2} we follow Lemma 3.11 of~\cite{1712.03881}.
Let $\chi(i)\defeq \bm{1}_{\{\abs{i} \le \ell^4N^{\delta_5}\}}$, with $\delta_4<\delta_5<\frac{\delta_1}{2}$, and $v\in \mathbb{R}^{2N}$. Then, we have that \[\langle \mathcal{U}^\mathcal{L}(0,t)w_0,v\rangle=\langle w_0,(\mathcal{U}^\mathcal{L})^T\chi v\rangle+\langle w_0,(\mathcal{U}^\mathcal{L})^T(1-\chi)v\rangle.\] 
By Lemma~\ref{fsl} we have that 
\begin{equation}
\label{f1}
\abs{\langle w_0, (\mathcal{U}^\mathcal{L})^T(1-\chi)v\rangle}\le N^{-100}\lVert w_0 \rVert_2\lVert v\rVert_1.
\end{equation} By~\eqref{Ee1} and Cauchy-Schwarz inequality we have that \begin{equation}
\label{f2}
\abs{\langle w_0, (\mathcal{U}^\mathcal{L})^T\chi v\rangle}\le \lVert w_0\rVert_2 \lVert (\mathcal{U}^\mathcal{L})^T\chi v\rVert_2\le \lVert w_0\rVert_2\left(\frac{N^{C\eta+\frac{\omega_1}{3}}}{c_\eta N^\frac{1}{2}t}\right)^{1-3\eta}\lVert v\rVert_1.
\end{equation} Hence, combining~\eqref{f1} and~\eqref{f2}, we conclude that \begin{equation}
\label{f3}
\lVert\mathcal{U}^\mathcal{L}(0,t)w_0\rVert_\infty\le\left(\frac{N^{C\eta+\frac{\omega_1}{3}}}{c_\eta N^\frac{1}{2}t}\right)^{1-3\eta}\lVert w_0\rVert_2,
\end{equation} and so, by~\eqref{5}, that 
\begin{equation}
\label{f4}
\begin{split}
\lVert\mathcal{U}^\mathcal{L}(0,t)w_0\rVert_\infty=\lVert\mathcal{U}^\mathcal{L}(t/2,t)\mathcal{U}^\mathcal{L}(0,t/2)w_0\rVert_\infty &\lesssim \left(\frac{N^{C\eta+\frac{\omega_1}{3}}}{c_\eta N^\frac{1}{2}t}\right)^{1-3\eta} \lVert \mathcal{U}^\mathcal{L}(0,t/2)w_0\rVert_2 \\
&\lesssim\left(\frac{N^{C\eta+\frac{\omega_1}{3}}}{c_\eta N^\frac{1}{2}t}\right)^{2(1-3\eta)}\lVert w_0\rVert_1,
\end{split}\end{equation}
where in the first inequality we used that $\mathcal{U}^\mathcal{L}(0,t/2)w_0$ satisfies the hypothesis of Lemma~\ref{ee}, since $\abs{(\mathcal{U}^\mathcal{L}(0,t/2)w_0)_i}\le N^{-100}$ for $\abs{i}\ge \ell^4 N^{2\delta_4}$ by the finite speed estimate of Lemma~\ref{fsl}. Combining~\eqref{f3} and~\eqref{f4} then~\eqref{Ee2} follows by interpolation.
\end{proof}

\printbibliography

\end{document}